\documentclass{amsart}
\usepackage[utf8]{inputenc}
\usepackage[english]{babel}

\usepackage[mathscr]{eucal}
\usepackage{amsmath}
\usepackage{amsfonts}
\usepackage{amssymb}
\usepackage{amsthm}
\usepackage{cmap}
\usepackage{hyperref}
\usepackage{mathrsfs}
\usepackage{tikz-cd}
\usepackage{bbm}

\makeatletter
\newcommand{\leqnomode}{\tagsleft@true}
\newcommand{\reqnomode}{\tagsleft@false}
\makeatother

\newcommand{\sVect}{\mathsf{sVect}}

\newcommand{\soVect}{\mathsf{s}_0\mathsf{Vect}}
\newcommand{\spaces}{\EuScript{S}}

\newcommand{\rLie}{\mathsf{Lie}^\mathrm{r}}

\newcommand{\Vect}{\mathsf{Vect}}

\DeclareMathOperator{\Alg}{\mathsf{Alg}}

\newcommand{\FB}{\mathsf{Fin}^{\cong}}
\newcommand{\cCAlg}{\mathsf{CoAlg}}
\newcommand{\CoAlg}{\mathsf{CoAlg}}
\newcommand{\trcoalg}{\mathsf{CoAlg}^{\mathrm{tr}}}
\newcommand{\CAlg}{\mathsf{CAlg}}
\newcommand{\ProVect}{\mathsf{Pro}(D\Vect_{\kk})}
\newcommand{\Pro}{\mathsf{Pro}}
\newcommand{\sseqpro}{\sseq_{\mathrm{pro}}}

\newcommand{\sotrcoalg}{\mathsf{s}_{0}\mathsf{CoAlg}^{\mathrm{tr}}}

\newcommand{\sL}{\mathit{s}\EuScript{L}}
\newcommand{\sCA}{\mathit{s}\EuScript{CA}}
\newcommand{\sLxi}{\mathit{s}\EuScript{L}_{\xi}}

\DeclareMathOperator{\sV}{\mathit{s}\EuScript{V}}

\newcommand{\Sp}{\mathrm{Sp}}
\newcommand{\SymLin}[1]{\mathrm{SymFun}^{#1}_{\mathrm{lin}}}

\newcommand{\Catinfty}{\EuScript{C}\mathrm{at}}
\newcommand{\Endinfty}{\EuScript{E}\mathrm{nd}}
\newcommand{\Mod}{\EuScript{M}\mathrm{od}}
\newcommand{\Pre}{\EuScript{P}\mathrm{r}}

\newcommand{\susp}{\Sigma^{\infty}}
\newcommand{\deloop}{\Omega^{\infty}}
\newcommand{\free}{\mathit{L^r}}
\DeclareMathOperator{\oblv}{\mathsf{oblv}}
\newcommand{\triv}{\mathsf{triv}}

\newcommand{\trivsseq}{\mathsf{triv}}
\newcommand{\barW}{\overline{W}}
\newcommand{\Sym}{\mathrm{Sym}}

\newcommand{\Abxi}{\mathsf{Ab}_{{\xi}}}

\newcommand{\sseq}{\mathsf{SSeq}}

\newcommand{\prim}{\mathsf{Prim}}

\newcommand{\Prim}{\mathsf{Prim}}

\newcommand{\red}{\mathrm{red}}
\newcommand{\cored}{\mathrm{cored}}

\newcommand{\aug}{\mathrm{aug}}

\newcommand{\Cobar}{\mathrm{Cobar}}

\newcommand{\suspfree}{{\mathbb{L}}^{r}}

\newcommand{\suspfreexi}{{\mathbb{L}}^{r}_{\xi}}

\newcommand{\creff}{\mathrm{cr}}

\newcommand{\Fun}{\mathrm{Fun}}
\DeclareMathOperator{\nat}{\mathrm{Nat}}
\DeclareMathOperator{\map}{\mathrm{map}}
\DeclareMathOperator{\Map}{\mathrm{Map}}
\DeclareMathOperator{\Hom}{\mathrm{Hom}}

\DeclareMathOperator{\End}{\mathrm{End}}
\DeclareMathOperator{\id}{\mathsf{id}}

\newcommand{\6}{\partial}
\newcommand{\AC}{\bar{\partial}^{\star}}

\newcommand{\Qe}[1]{\overline{\mathit{Q}^{#1}}}
\newcommand{\bQe}[2]{\overline{\beta^{#1}\mathit{Q}^{#2}}}

\newcommand{\holim}{\operatorname*{holim}}
\newcommand{\colim}{\operatorname*{colim}}

\newcommand{\fib}{\operatorname*{fib}}
\newcommand{\cofib}{\operatorname*{cofib}}

\newcommand{\tfiber}{\operatorname*{tfiber}}
\newcommand{\Tot}{\operatorname*{Tot}}
\newcommand{\pTot}[1]{\operatorname*{Tot^{\mathit{#1}}}}

\newcommand{\bfLie}{\mathbf{Lie}}
\newcommand{\bfComm}{\mathbf{Comm}}

\newcommand{\F}{\mathbf F}
\newcommand{\R}{\mathbf R}
\newcommand{\Co}{\mathbf C}
\newcommand{\N}{\mathbf N}
\newcommand{\Z}{\mathbf Z}
\newcommand{\Q}{\mathbf Q}

\newcommand{\PP}{\mathbf P}
\newcommand{\bbP}{\mathbb P}
\newcommand{\D}{\mathbb D}
\newcommand{\kk}{\mathbf F}
\newcommand{\calD}{\EuScript{D}}
\newcommand{\calC}{\EuScript{C}}
\newcommand{\calE}{\EuScript{E}}

\newcommand{\calP}{\EuScript{P}}
\newcommand{\calA}{\mathcal{A}}
\newcommand{\calR}{\bar{\mathcal{R}}}
\DeclareMathOperator{\vph}{\varphi}

\DeclareMathOperator{\Ind}{\mathrm{Ind}}

\DeclareMathOperator{\Coind}{\mathrm{Coind}}

\DeclareMathOperator{\ad}{\mathrm{ad}}

\DeclareMathOperator{\rk}{\mathrm{rk}}
\newcommand{\e}{\varepsilon}

\newcommand{\lindeg}{\mathrm{sc.deg}}
\newcommand{\Gal}{\mathrm{Gal}}
\newcommand{\conn}{\mathrm{conn}}
\newcommand{\Surj}{\mathrm{Surj}}

\DeclareMathOperator{\coker}{\mathrm{coker}}

\numberwithin{equation}{section}
\numberwithin{equation}{subsection}
\renewcommand*{\theequation}{%
  \ifnum\value{subsection}=0 %
    \thesection
  \else
    \thesubsection
  \fi
  .\arabic{equation}%
}

\newtheorem{prop}[equation]{Proposition}
\newtheorem*{prop*}{Proposition}
\newtheorem{lmm}[equation]{Lemma}
\newtheorem{thm}[equation]{Theorem}
\newtheorem{varthm}{Theorem}

\newtheorem*{thm*}{Theorem}

\newtheorem*{conj*}{Conjecture}
\newtheorem{cor}[equation]{Corollary}
\theoremstyle{definition}
\newtheorem{dfn}[equation]{Definition}

\newtheorem{exmp}[equation]{Example}
\newtheorem*{ntt*}{Notation}
\theoremstyle{remark}
\newtheorem{rmk}[equation]{Remark}

\title{Algebraic Goodwillie spectral sequence}
\author{Nikolay Konovalov}
\address{Department of Mathematics, University of Chicago, Eckhart Hall, 5734 S University Ave, Chicago, IL 60637, USA}
\email{nkonovalov@uchicago.edu}
\date{}
\begin{document}
\begin{abstract}
Let $\mathit{s}\EuScript{L}$ be the $\infty$-category of simplicial restricted Lie algebras over $\mathbf{F} = \overline{\mathbf{F}}_p$, the algebraic closure of a finite field $\mathbf{F}_p$. By the work of A.~K.~Bousfield et al. on the unstable Adams spectral sequence, the category $\mathit{s}\EuScript{L}$ can be viewed as an algebraic approximation of the $\infty$-category of pointed $p$-complete spaces. We study the functor calculus in the category $\sL$. More specifically, we consider the Taylor tower for the functor $L^r\colon \EuScript{M}\mathrm{od}^{\geq 0}_{\mathbf{F}} \to \mathit{s}\EuScript{L}$ of a free simplicial restricted Lie algebra together with the associated Goodwillie spectral sequence. We show that this spectral sequence evaluated at $\Sigma^l \mathbf{F}$, $l\geq 0$ degenerates on the third page after a suitable re-indexing, which proves an algebraic version of the Whitehead conjecture. 

In our proof we compute explicitly the differentials of the Goodwillie spectral sequence in terms of the $\Lambda$-algebra of A.~K.~Bousfield et al. and the Dyer--Lashof--Lie power operations, which naturally act on the homology groups of a spectral Lie algebra. As an essential ingredient of our calculations, we establish a general Leibniz rule in functor calculus associated to the composition of mapping spaces, which conceptualizes certain formulas of W.~H.~Lin. Also, as a byproduct, we identify previously unknown Adem relations for the Dyer--Lashof--Lie operations in the odd-primary case.
\end{abstract}

\maketitle

\tableofcontents

\addtocontents{toc}{\protect\setcounter{tocdepth}{1}}
\section{Introduction}\label{section: introduction}

\subsection*{Historical overview}
Given a reduced finitary functor $F\colon \calC \to \calD$ between pointed presentable $\infty$-categories $\calC$ and $\calD$, T.~Goodwillie~\cite{GoodwillieIII} constructed a \emph{Taylor tower} $P_n(F)$ of $n$-excisive approximations:
\begin{equation}\label{equation: taylor tower, intro}
\begin{tikzcd}
&
D_3(F) \arrow[d]
&D_2(F)\arrow[d]
&D_1(F)\arrow[d, equal]\\
\ldots \arrow[r]
& P_3(F) \arrow[r]
& P_2(F) \arrow[r]
& P_1(F)
\end{tikzcd}
\end{equation}
such that the fibers have canonical deloopings $D_n(F)\simeq \Omega^{\infty} \circ \D_n(F)\circ \Sigma^\infty$ for certain functors $$\D_n(F)\colon \Sp(\calC) \to \Sp(\calD),$$ where $\Sp(\calC)$ and $\Sp(\calD)$ are the stabilizations of $\calC$ and $\calD$ respectively, see Section~\ref{section: spectrum objects}. Next, for $Y\in \calC$ and $X\in \calD$, the Taylor tower~\eqref{equation: taylor tower, intro} gives the tower of mapping spaces $\map_{\calD}(X,P_nF(Y))$. Under favorable conditions, the mapping space $\map_{\calD}(X,F(Y))$ is the homotopy limit of the latter tower $\map_{\calD}(X,P_nF(Y))$, and so the Taylor tower~\eqref{equation: taylor tower, intro} produces the \emph{Goodwillie spectral sequence}
\begin{equation}\label{equation: GSS, intro}
E^1_{t,n}(X,Y) = \pi_t\map_{\Sp(\calD)}(\Sigma^\infty X, \D_nF(\Sigma^{\infty} Y)) \Rightarrow \pi_t\map_{\calD}(X,F(Y)). 
\end{equation}
We note that the spectral sequence~\eqref{equation: GSS, intro} starts with the mapping spaces in the \emph{stable} category $\Sp(\calD)$, which are often more accessible for computations rather than mapping spaces in the ``non-stable'' category $\calD$. 

The most intriguing and well-studied instance of the Goodwillie spectral sequence is for the identity functor in the category of spaces. Namely, let $p$ be a fixed prime number, let $\calC=\calD=\spaces^{\wedge}_{*,p}$ be the category of \emph{pointed $p$-complete spaces}, and suppose that $F=\id_{\spaces}$, $X=S^0$, $Y=S^{l+1}$, $l\geq 0$ in the spectral sequence~\eqref{equation: GSS, intro}.  In this case the spectral sequence~\eqref{equation: GSS, intro} converges to the ($p$-completed) homotopy groups $\pi_*(S^{l+1})$. In addition, we will assume that $l$ is even if $p$ is odd. Then, G.~Arone and M.~Mahowald~\cite{AM99} showed that $\D_n(\id_{\spaces})(S^{l+1})\simeq *$ unless $n = p^h$. In the sequel, we will abuse the notation and write $D_n$ (resp. $\D_n$) for the functor $D_n(\id_{\calC})$ (resp. $\D_n(\id_{\calC})$), where $\id_{\calC}$ is the identity functor. Therefore the spectral sequence~\eqref{equation: GSS, intro} can be re-indexed to take the form
\begin{equation}\label{equation: RGSS, intro}
\widetilde{E}^1_{t,h}(S^{l+1}) = \pi_t\D_{p^h}(S^{l+1}) \Rightarrow \pi_t(S^{l+1}). 
\end{equation}

The Goodwillie layer spectra $\D_{p^h}(S^{l+1})$ can be identified as follows
$$\D_{p^h}(S^{l+1}) \simeq \partial_{p^h}(\id_{\spaces}) \wedge_{h\Sigma_{p^h}} (S^{l+1})^{\wedge p^h}, $$
where $\partial_{p^h}(\id_{\spaces})$ are the Spanier--Whitehead duals of certain finite CW-complexes called \emph{partition complexes}, see~\cite{Johnson95} and~\cite{AD01}. 
Thus the Goodwillie spectral sequence~\eqref{equation: RGSS, intro} computes the unstable homotopy groups of spheres from the stable homotopy groups of some well-understood spectra. The differentials in~\eqref{equation: RGSS, intro} remain mostly mysterious, although M.~Behrens~\cite{Behrens12} showed that the $d_1$-differentials are related to the classical James--Hopf invariant and their leading terms (with respect to an appropriate filtration) are given by the stable Hopf invariant.

For $l=0$, the spectral sequence~\eqref{equation: RGSS, intro} converges to the homotopy groups $\pi_*(S^1)$ of a circle, and so no non-trivial classes persist to the infinite page $\widetilde{E}^{\infty}$. Surprisingly, M.~Behrens~\cite{Behrens11} (for $p=2$) and N.~Kuhn~\cite{Kuhn15} (for any $p$) showed that the spectral sequence~\eqref{equation: RGSS, intro} also degenerates at the~\emph{second page}, so any non-trivial class on the $\widetilde{E}^1$-term either supports a non-trivial $d_1$-differential or it is killed by a $d_1$-differential. More precisely, the Goodwillie $d_1$-differentials serve as contracting chain homotopies for the Kahn--Priddy sequence in the \emph{Whitehead conjecture}~\cite{Kuhn82_Whitehead}. 

The purpose of this paper is to study a certain \emph{algebraic} analog of the Goodwillie spectral sequence~\eqref{equation: RGSS, intro} together with algebraic counterparts of the aforementioned topological results.

\subsection*{Algebraic Goodwillie spectral sequence}
The main object of this paper is the $\infty$-category $\sL$ of \emph{simplicial restricted Lie algebras} over $\kk=\overline{\F}_p$, the algebraic closure of the finite field $\F_p$, see Section~\ref{section: simplicial restricted lie algebras}. The homotopy theory of $\sL$ was developed by S.~Priddy~\cite{Priddy70short},~\cite{Priddy70long} and it was revisited by the author in~\cite{koszul}. 
A.~K.~Bousfield and E.~Curtis~\cite{BC70} showed that there is a certain relation between the category $\EuScript{S}^{\wedge}_{*,p}$ of \emph{pointed $p$-complete spaces} and the category $\sL$. In particular, they constructed the \emph{unstable Adams spectral sequence}
$$
{}^{\mathrm{UASS}}E^1_{t,*}(X)=\pi_t\free(\Sigma^{-1}\widetilde{H}_*(X;\F_p)) \Rightarrow \pi_{t-1}(X), \;\; X\in \spaces^{\wedge}_{*,p}.
$$
Here ${}^{\mathrm{UASS}}E^1$-term consists of the homotopy groups of the free simplicial restricted Lie algebra $\free(\Sigma^{-1}\widetilde{H}_*(X;\F_p))$ generated by the homology groups of a simply connected space $X$ and ${}^{\mathrm{UASS}}E^r_{*,*}$ converges to its homotopy groups $\pi_*(X)$. Therefore the category $\sL$ may serve as an algebraic approximation of the category $\EuScript{S}^{\wedge}_{*,p}$ of spaces and should be an essential ingredient to study the interaction of the unstable Adams spectral sequence and the Taylor tower. 

We write $\Mod_{\kk}^{\geq 0}$ for the $\infty$-category of simplicial vector spaces, which we identify with the category of non-negatively graded  chain complexes by the Dold--Kan correspondence. We consider the functor $$\free\colon \Mod_{\kk}^{\geq 0} \to \sL$$ of the free simplicial restricted Lie algebra. Motivated by~\cite{BC70}, we study the Goodwillie spectral sequence~\eqref{equation: GSS, intro} in the case $\calC=\calD=\sL$, $F=\id_{\sL}$ is the identity functor on $\sL$, $X=\free(\kk)$, and $Y=\free(\Sigma^l \kk)$, where $l$ is even if $p$ is odd. In this case the spectral sequence~\eqref{equation: GSS, intro} takes the form
\begin{equation}\label{equation: AGSS, intro, 1}
E^1_{t,n}(\Sigma^l \kk) = \pi_t( D_{n}(\free(\Sigma^l \kk))) \Rightarrow \pi_t(\free(\Sigma^l\kk)). 
\end{equation}
Since the functor $\free$ preserves colimits, we have an equivalence $D_n\circ \free \simeq D_n(\free)$ as functors from $\Mod^{\geq 0}_{\kk}$. The next theorem follows essentially from the Hilton--Milnor theorem.
\begin{varthm}[Corollary~\ref{corollary:derivatives of free}]\label{theorem: A,intro} There exists a natural equivalence
\begin{equation*}
D_n(\free)(V)\simeq \Omega^{\infty}\Sigma^\infty \free(\bfLie_n \otimes_{h\Sigma_n} V^{\otimes n}), \;\; V\in \Mod_{\kk}^{\geq 0},
\end{equation*}
where $\Sigma^{\infty} : \sL \rightleftarrows \Sp(\sL) : \Omega^{\infty}$ is the adjoint pair of stabilization and $\bfLie_n \in \Vect_{\kk}$ is the $n$-th space of the Lie operad.
\end{varthm}

By~\cite{AD01}, we have again $D_n(\free)(\Sigma^l\kk)\simeq *$ unless $n=p^h$, see also Proposition~\ref{proposition: vanishing of LnV}, and so the spectral sequence~\eqref{equation: AGSS, intro, 1} can be re-indexed to take the form
\begin{equation}\label{equation: RAGSS, intro}
\widetilde{E}^1_{t,2h}(\Sigma^l\kk) = \pi_tD_{p^h}(\free)(\Sigma^l \kk) \Rightarrow \pi_t(\free(\Sigma^l \kk)),
\end{equation}
where $\widetilde{E}^1_{t,2h+1}(\Sigma^l \kk) =0$, $h\geq 0$. We note that this re-indexing is slightly different to the re-indexing of~\eqref{equation: RGSS, intro}; namely, we decelerate the filtration of~\eqref{equation: RGSS, intro} by half. We do it in order to incorporate the case of odd $l$ as well, see~\eqref{equation: RAGSS}.

\begin{varthm}[Theorem \ref{theorem: whitehead conjecture}]\label{theorem: B, intro}
Suppose that $l\geq 0$ and $l$ is even if $p$ is odd. Then $\widetilde{E}^{1}_{t,m}(\Sigma^l \kk)=\widetilde{E}^2_{t,m}(\Sigma^l\kk)$, $t,m\geq 0$, and $\widetilde{E}^3_{t,m}(\Sigma^{l}\kk)=0$ if $m\geq 1$ and $t\geq 0$. In particular, the re-indexed spectral sequence~\eqref{equation: RAGSS, intro} degenerates at the third page.
\end{varthm}

Theorem~\ref{theorem: B, intro} is an algebraic counterpart to the \emph{calculus form of the Whitehead conjecture}, see~\cite{Behrens11},~\cite[Section~4.6]{Behrens12}, and~\cite{Kuhn15}. We note that the calculus form of the topological Whitehead conjecture holds for $S^1$ and does not hold for any higher dimensional sphere. By contrast, the algebraic version holds in any degree.

We prove Theorem~\ref{theorem: B, intro} by computing explicitly the $\tilde{d}_2$-differentials in the spectral sequence~\eqref{equation: RAGSS, intro}; at the moment of writing, we do not have a non-computational proof. We first describe the basis of the $\widetilde{E}^1$-term. 

\subsection*{Lambda algebra and Dyer--Lashof--Lie operations}
We write $$\mathbf{L}_k \colon \Mod_{\kk} \to \Mod_{\kk}, \;\; k\geq 0$$ for the functor $V\mapsto \bfLie_k\otimes_{h\Sigma_k} V^{\otimes k}$, see~\eqref{equation: lie powers}. By Theorem~\ref{theorem: A,intro}, we have an isomorphism $$\widetilde{E}^1_{*,2h}(\Sigma^l \kk)\cong \pi^s_*(\free(\mathbf{L}_{p^h}(\Sigma^l \kk)),$$ where $\pi^s_*(\free(\mathbf{L}_{p^h}(\Sigma^l \kk)))$ are the \emph{stable} homotopy groups of the free simplicial restricted Lie algebra generated by $\mathbf{L}_{p^h}(\Sigma^l\kk)\in \Mod^{\geq 0}_{\kk}$. Such homotopy groups were computed by~\cite{6authors} in terms of the \emph{lambda algebra} $\Lambda$, see Section~\ref{section: lambda algebra}. We recall that $\Lambda$ is a certain bigraded (differential) algebra (over $\F_p$) which is the Koszul dual algebra to the mod-$p$ Steenrod algebra $\calA_p$, see e.g.~\cite{Priddy70}. The algebra $\Lambda$ is well-understood; in particular, it is a quadratic algebra given by the explicit Adem-type relations~\eqref{equation: lambda, adem1, p is odd}, \eqref{equation: lambda, adem2, p is odd}, and~\eqref{equation: lambda, adem3, p=2}. Moreover, $\Lambda$ has a PBW-basis of admissible monomials. By~\cite{6authors}, we have natural isomorphisms
$$\widetilde{E}^1_{*,2h}(\Sigma^l \kk)\cong \pi^s_*(\free(\mathbf{L}_{p^h}(\Sigma^l \kk)) \cong \pi_*(\mathbf{L}_{p^h}(\Sigma^l\kk)) \otimes \Lambda, $$
where $\pi_*(\mathbf{L}_{p^h}(\Sigma^l\kk))$ on the right hand side are the homotopy groups of the \emph{simplicial vector space} $\mathbf{L}_{p^h}(\Sigma^l\kk)\in \Mod_{\kk}^{\geq 0}$. Here we consider the right multiplication by the elements of $\Lambda$ as the composition in homotopy groups.

The basis for the homotopy groups $$\pi_*(\mathbf{L}_{p^h}(\Sigma^l\kk)) \cong \widetilde{H}^*(\Sigma^{-1}\partial_{p^h}(\id_{\spaces})\wedge_{h\Sigma_{p^h}} (S^{l+1})^{\wedge p^h};\kk)$$ was firstly constructed by G.~Arone and M.~Mahowald~\cite{AM99} in terms of completely unadmissible sequences (see Definition~\ref{definition: cu sequence}) of the \emph{Dyer--Lashof} operations $\beta^{\e}Q^i$, see e.g.~\cite{CLM76}. We recall that the Dyer--Lashof operations act on the homology groups of an $E_\infty$-ring spectrum, and it is unclear how they are related to Lie algebras of any kind. Later, M.~Behrens~\cite{Behrens12} and J.~Kjaer~\cite{Kjaer18} defined \emph{Dyer--Lashof--Lie} operations $\bQe{\e}{i}$ which act on the homology groups of any \emph{spectral Lie algebra}. In particular, we have the maps
\begin{equation}\label{equation: dll,intro}
\bQe{\e}{i}\colon \pi_*(\mathbf{L}_{p^h}(\Sigma^l\kk)) \to \pi_*(\mathbf{L}_{p^{h+1}}(\Sigma^l\kk)), \;\; h\geq 0, 
\end{equation}
and by applying them recursively, we can reinterpret the Arone--Mahowald basis in terms of the new Dyer--Lashof--Lie operations. We review the construction and basic properties of $\bQe{\e}{i}$ in Section~\ref{section: H-Lie-algebras}.

Recall that the algebra $\Lambda$ is generated by the elements $\nu_i^{\e}\in\Lambda$, $i\geq \e$, $\e\in\{0,1\}$ (resp. $\lambda_i\in\Lambda$, $i\geq 0$ if $p=2$) which are Koszul dual to the Steenrod operations $\beta^{1-\e}P^i \in \calA_p$ (resp. $Sq^{i+1}\in\calA_2)$.

\begin{varthm}[Proposition~\ref{proposition: differential, length 1}]\label{theorem: C, intro}
Let $\iota_l \in \pi_l(\Sigma^l \kk)$ be the canonical generator. Then, we have
$$\tilde{d}_2(\iota_l\otimes \nu^{\e}_{i})  = 
\begin{cases}
\bQe{\e}{i}(\iota_l) \otimes 1
& \mbox{if $p$ is odd, $i\geq 0$, $\e\in \{0,1\}$,}\\
\Qe{i}(\iota_l) \otimes 1
& \mbox{if $p=2$, $i\geq 0$, $\e=0$.}
\end{cases}
$$
Here $\tilde{d}_2\colon \pi_*(\Sigma^l \kk)\otimes \Lambda \to \pi_*(\mathbf{L}_p(\Sigma^l\kk))\otimes \Lambda$ is the differential in the spectral sequence~\eqref{equation: RAGSS, intro}, and we set $\nu_i^{0}=\lambda_i\in\Lambda$ if $p=2$.
\end{varthm}

In other words, the differential $\tilde{d}_2$ maps the preferred generators $\nu^{\e}_i$ of the algebra~$\Lambda$ to the preferred generators of the Dyer--Lashof--Lie algebra. The proof of Theorem~\ref{theorem: C, intro} is involved; a key ingredient is the James--Hopf map. 

\subsection*{James--Hopf invariant} In Proposition~\ref{proposition: Snaith splitting}, we observe that the category $\sL$ has a variant of the Snaith splitting. This splitting allows us to construct the ($p$-th) \emph{James--Hopf map}
$$j_p\colon \Omega^{\infty} \Sigma^{\infty} \free(\Sigma^l \kk) \to \Omega^{\infty}\Sigma^{\infty} \free(\Sigma^{-1}(\Sigma^{l+1}\kk)^{\otimes p}_{h\Sigma_p}),$$
see Definition~\ref{definition:james-hopf map}. We demonstrate in Proposition~\ref{proposition:james-hopf and adjoint} that the $\tilde{d}_2$-differentials are related to the James--Hopf map $j_p$, and in Theorems~\ref{theorem: the adjoint of delta_n}, \ref{theorem: the adjoint of delta_n, p=2}, we show that the James--Hopf invariant and the $\tilde{d}_2$-differential are indeed equivalent to each other.

In Section~\ref{section: james-hopf map}, we evaluate the James--Hopf map $j_p$ at the generators $\nu^{\e}_i\in\Lambda$ of the algebra $\Lambda$. For that we relate the map $j_p$ to the \emph{unstable Hurewicz homomorphism}
$$\tilde{h}\colon \pi_*(X) \to \widetilde{H}_{*+1}(\Omega^{\infty}X), \;\;X\in \Sp(\sL), $$
see Section~\ref{section: unstable hurewicz homomorphism}. Using the idea of~\cite{KM13}, we relate the unstable Hurewicz homomorphism $\tilde{h}$ with the \emph{Tate diagonal} 
$$\Delta_p\colon X \to (X^{\otimes p})^{t\Sigma_p}$$ in the category $\Sp(\sL)$, see Section~\ref{section: tate diagonal}. Next, in Proposition~\ref{proposition:tate diagonal for spectra}, we compute the induced map $\Delta_{p*}$ on the homology groups $\widetilde{H}_*(X)$ in terms of the right action of the Steenrod operations on $\widetilde{H}_*(X)$. Since the generators $\nu_i^{\e}$ of the algebra~$\Lambda$ are analogs of the Hopf invariant one elements (see e.g.~\cite[Remark~6.4.19]{koszul}), we are able to compute the image of $\nu_i^{\e}$ under the Hurewicz homomorphism $\tilde{h}$ in terms of the Dyer--Lashof operations in Proposition~\ref{proposition: hurewicz of hopf invariant one}. This essentially finishes the proof of Theorem~\ref{theorem: C, intro}.

\subsection*{Lin's formula}
The next theorem is the combination of Theorems~\ref{theorem: lin formula, odd p} and \ref{theorem: lin formula, p = 2}, and it explains the action of the differential $$\tilde{d}_2\colon \pi_*(\Sigma^l\kk)\otimes \Lambda \to \pi_{*}(\mathbf{L}_p(\Sigma^l\kk))\otimes \Lambda$$ on the rest of the basis of admissible monomials in~$\Lambda$. 

\begin{varthm}\label{theorem: D, intro} Let $\nu_I=\nu_{i_1}^{\e_1}\cdot\ldots \cdot \nu_{i_s}^{\e_s} \in \Lambda$ be an admissible monomial. Then 
$$\tilde{d}_2(\iota_l\otimes \nu_{I}) = \bQe{\e_1}{i_1}(\iota_l)\otimes \nu_{I'}+\sum_{\alpha}c_{\alpha}\bQe{\e_1(\alpha)}{i_{1}(\alpha)}(\iota_l)\otimes \nu_{I'(\alpha)}, $$
where $\nu_{I'}= \nu_{i_2}^{\e_2}\cdot\ldots\cdot\nu_{i_s}^{\e_s} \in \Lambda$, and $I'(\alpha)=(i_2(\alpha),\e_2(\alpha),\ldots, i_s(\alpha),\e_s(\alpha))$ is an admissible sequence of length $s-1$. The constant $c_{\alpha}\in \F_p$ is the coefficient for $\nu_{i_1(\alpha)+p^N}^{\e_1(\alpha)} \cdot \nu_{I'(\alpha)} \in \Lambda$ 
in the admissible expansion of 
$$
\sum_{t=2}^{s}\nu_{i_1}^{\e_1}\cdot \ldots \cdot \nu_{i_t+p^N}^{\e_t}\cdot \ldots \cdot \nu_{i_s}^{\e_s} \in \Lambda.
$$
Here the constants $c_{\alpha}$ are independent of $N$ for $N>N_0$, where $N_0=N_0(\nu_I)$ depends only on $\nu_I$.
\end{varthm}

For $p=2$, the formula of Theorem~\ref{theorem: D, intro} for the differential $\tilde{d}_2$ recovers the formula~(3) in~\cite{Lin81}, which was used in the proof of the algebraic Kahn--Priddy theorem, see~\cite{Lin81} and~\cite{Singer81}. Unfortunately, W.~H.~Lin did not explain any motivation for the formula in his work, however, Theorem~\ref{theorem: D, intro} and its proof given in this paper can provide a conceptual explanation for Lin's formula. The key ingredient here is the Leibniz identity. 

\subsection*{Leibniz rule}
We recall the general Goodwillie spectral sequence~\eqref{equation: GSS, intro}. Assume that $\calC=\calD$ and $F=\id_{\calC}$ is the identity functor. At this moment we can assume that $\calC$ is an \emph{arbitrary} pointed compactly generated $\infty$-category. In Section~\ref{section: composition product}, for every triple $X,Y,Z\in \calC$, we construct products
\begin{equation*}
(-)\circ (-)\colon E^1_{t,1}(Y,Z) \otimes E^1_{s,1}(X,Y) \to E^1_{t+s,1}(X,Z),
\end{equation*}
\begin{equation*}
\D_n(-)\circ(-)\colon E^1_{t,1}(Y,Z) \otimes E^1_{s,n}(X,Y) \to E^1_{t+s,n}(X,Z),
\end{equation*}
and
\begin{equation*}
(-)\circ(-)\colon E^1_{t,n}(Y,Z) \otimes E^1_{s,1}(X,Y) \to E^1_{t+s,n}(X,Z)
\end{equation*}
on the $E^1$-term of the spectral sequence~\eqref{equation: GSS, intro} induced by the composition of mapping spaces in the stable category $\Sp(\calC)$, see formulas~\eqref{equation:stable composition}, \eqref{equation:left action}, and \eqref{equation:right action}. By using the recent work~\cite{Heuts21}, we observe in Theorem~\ref{theorem:leibniz rule in GSS} that the differential $d_{n-1}$ in~\eqref{equation: GSS, intro} satisfies the \emph{Leibniz rule} with respect to the constructed \emph{composition products}. Namely, let $g\in E^1_{t,1}(Y,Z)$ and $f\in E^1_{s,1}(X,Y)$ such that $d_1(g)=\ldots=d_{n-2}(g)=0$ and $d_{1}(f)=\ldots= d_{n-2}(f)=0$ for $n\geq 2$. Then 
$$d_{n-1}(g\circ f) = d_{n-1}(g)\circ f + \D_n(g)\circ d_{n-1}(f)\in E^1_{t+s-1,n}(X,Z).$$

More specifically, by unwinding the spectral sequences~\eqref{equation: AGSS, intro, 1} and~\eqref{equation: RAGSS, intro} to their original form~\eqref{equation: GSS, intro}, where $F=\id_{\sL}$, we observe in Lemma~\ref{lemma: GSS differential on hopf elements and leibniz} the following formula
\begin{equation}\label{equation: leibniz, intro}
\tilde{d}_{2}(x\nu^{\e}_{i})=\tilde{d}_{2}(x)\nu^{\e}_{i}+\D_p(\overline{x})_*(\bQe{\e}{i}(\iota_t)),
\end{equation}
where $x\in \pi_t(D_1(\free(\Sigma^{l}\kk)))=\pi^s_t\free(\Sigma^l\kk)$ is a (stable) homotopy class, a generator $\nu_i^{\e} \in \Lambda$ acts on the right by the composition, and
$$\D_p(\overline{x})_*\colon \pi_*\D_p(\free(\Sigma^t\kk)) \to \pi_*\D_p(\free(\Sigma^l\kk)) $$
is induced by the map $\overline{x}\colon \Sigma^{\infty}\free(\Sigma^t \kk)\to \Sigma^{\infty}\free(\Sigma^l\kk)$ which represents the homotopy class $x\in \pi_t(D_1(\free(\Sigma^{l}\kk)))$.

We identify the second summand in the formula~\eqref{equation: leibniz, intro} by using a \emph{James periodicity argument}. We review James periodicity and its applications to functor calculus in Section~\ref{section: james periodicity}. We finish the proof of Theorem~\ref{theorem: D, intro} by induction on the length of an admissible monomial.

\subsection*{Symmetric sequence of derivatives}
Theorem~\ref{theorem: D, intro} describes completely the differentials 
\begin{equation}\label{equation: differential, zero row, intro}
\tilde{d}_2\colon \widetilde{E}^1_{t,0}(\Sigma^l \kk) \to \widetilde{E}^1_{t-1,2}(\Sigma^l\kk),
\end{equation}
which originate from the $0$-th row. Inspired by~\cite[Corollary~1.7]{KM13} and~\cite{Behrens11}, we express in Lemmas~\ref{lemma:GSS differential, positive columns} and~\ref{lemma:GSS differential, positive columns, p =2} the other differentials 
$$\tilde{d}_2\colon \widetilde{E}^1_{t,2h}(\Sigma^l\kk) \to \widetilde{E}^1_{t-1,2h+2}(\Sigma^l\kk), \;\; h\geq 1$$
in terms of~\eqref{equation: differential, zero row, intro} and the Dyer--Lashof--Lie operations~\eqref{equation: dll,intro}.

\begin{varthm}\label{theorem: E,intro}
Let $w\in \pi_*(\mathbf{L}_{p^h}(\Sigma^l\kk))$, $\nu\in \Lambda$, so $w\otimes \nu \in \pi_*(\mathbf{L}_{p^h}(\Sigma^l\kk))\otimes \Lambda \cong \widetilde{E}^1_{*,2h}(\Sigma^l\kk)$. Suppose that
$$\tilde{d}_{2}(\iota_l\otimes \nu) = \sum_{\alpha}\bQe{\e(\alpha)}{i(\alpha)}(\iota_l)\otimes \nu_{\alpha} \in \pi_*(\mathbf{L}_p(\Sigma^l(\kk)))\otimes \Lambda \cong \widetilde{E}^1_{*,2}(\Sigma^l\kk).$$
Then 
$$\tilde{d}_{2}(w \otimes \nu)= \sum_{\alpha}\bQe{\e(\alpha)}{i(\alpha)}(w)\otimes \nu_{\alpha} \in \pi_*(\mathbf{L}_{p^{h+1}}(\Sigma^l(\kk)))\otimes \Lambda \cong \widetilde{E}^{1}_{*,2h+2}(\Sigma^l\kk).$$
\end{varthm}

Our proof of Theorem~\ref{theorem: E,intro} essentially mimics~\cite{Behrens11}. First, given an analytic (Definition~\ref{definition: analytic functor}) reduced functor $F\colon \Sp \to \sL$, we show in Section~\ref{section: derivaties and coderivatives} that the symmetric sequence of (the chain complexes of) its derivatives $$\bar{\6}_*(F) \in \sseq(\Mod_{\kk})$$ forms a left module over the \emph{$\kk$-linear spectral Lie operad} $\bfLie^{\mathrm{sp}}_*$, see Definition~\ref{definition: linear spectral lie operad}. This observation is an algebraic analog of the theorem of G.~Arone and M.~Ching \cite{AroneChing11} that the derivatives $\6_*(F)$ of a reduced finitary functor $F\in \Fun^{\omega}_*(\Sp,\spaces)$ form a left module over the spectral Lie operad, see Remark~\ref{remark: arone-ching conjecture}.

Next, we observe in Corollary~\ref{corollary: splitting} that the $\tilde{d}_2$-differentials in the spectral sequence~\eqref{equation: RAGSS, intro} are induced by the \emph{natural transformations}~\eqref{equation:Dpk to Dk}
$$\delta_n\colon D_n(\Omega^{p-2}\free) \to BD_{pn}(\Omega^{p-2}\free), \;\; n=p^h. $$
We write $\widetilde{P}_n(\Omega^{p-2}\free)$ for the fiber of $\delta_n$. We note that the functor $\widetilde{P}_n(\Omega^{p-2}\free)$ has only \emph{two} non-trivial Goodwillie layers. We study such functors and their left $\bfLie^{\mathrm{sp}}_*$-module of derivatives in Section~\ref{section: connecting map}. By applying these results to the functor $\widetilde{P}_n(\Omega^{p-2}\free)$, we observe in Theorems~\ref{theorem: the adjoint of delta_n} and~\ref{theorem: the adjoint of delta_n, p=2} that the connecting maps $\delta_n$, $n\geq p$ are encoded by the $p$-th James--Hopf map and by the only non-trivial Lie-operadic multiplication map 
$$m\colon \bfLie^{\mathrm{sp}}_{p} \otimes (\bar{\6}_n(\widetilde{P}_n(\Omega^{p-2}\free)))^{\otimes p} \to \bar{\6}_{pn}(\widetilde{P}_n(\Omega^{p-2}\free)) $$
given by the left $\bfLie^{\mathrm{sp}}_*$-module structure on $\bar{\6}_*(\widetilde{P}_{n}(\Omega^{p-2}\free))$, see~\eqref{equation: lie multiplicatons, derivatives}. The last observation finishes the proof of Theorem~\ref{theorem: E,intro}.

\subsection*{Applications}
Theorems~\ref{theorem: D, intro} and~\ref{theorem: E,intro} fully describe the $\tilde{d}_2$-differentials in the spectral sequence~\eqref{equation: RAGSS, intro} in terms of the basis of the completely unadmissible monomials in the Dyer--Lashof--Lie algebra and the basis of admissible monomials in the algebra~$\Lambda$. This description shows that the $\tilde{d}_2$-differential preserves a certain multiindex filtration on the $\widetilde{E}^1$-term (see Corollaries~\ref{cor: goodwillie differential lowers the filtation, p is odd} and~\ref{cor: goodwillie differential lowers the filtation, p is 2}), which finishes the proof of the algebraic Whitehead conjecture, see Theorem~\ref{theorem: whitehead conjecture}. 

Finally, in Section~\ref{section: adem-type relations}, we use Theorems~\ref{theorem: C, intro},~\ref{theorem: D, intro}, and~\ref{theorem: E,intro} to obtain the quadratic Adem-type relations~\eqref{equation: Lie-Adem relation, p is odd, eq1}, \eqref{equation: Lie-Adem relation, p is odd, eq2}, and~\eqref{equation: Lie-Adem relation, p is 2, eq1} between Dyer--Lashof--Lie operations $\bQe{\e}{i}$. These relations are well-known for $p=2$, see~\cite{Behrens12}, but for odd primes they have not been identified so far, see Remarks~\ref{remark: adem relations, history} and~\ref{remark: adem relations, transfers}.

\begin{varthm}\label{theorem: F, intro}
Suppose that $p$ is odd. Then the Dyer--Lashof--Lie operations~\eqref{equation: dll,intro} satisfy following relations
\begin{align*}
\bQe{\e}{j}\cdot \bQe{}{i} &=(-1)^{\e+1}\sum_{m = pi}^{i+j-1}\binom{p(m-i)-(p-1)j+\e-1}{m-pi} \bQe{}{m}\cdot\bQe{\e}{j+i-m} \\
&+(1-\e)\sum_{m=pi}^{i+j-1}\binom{p(m-i)-(p-1)j}{m-pi} \Qe{m}\cdot\bQe{}{j+i-m},
\end{align*}
where $j<pi$, $\e\in\{0,1\}$, and
\begin{align*}
\bQe{\e}{j}\cdot\Qe{i} = \sum_{m = pi+1}^{i+j-1}\binom{p(m-i)-(p-1)j-1}{m-pi-1} \bQe{\e}{m}\cdot \Qe{j+i-m},
\end{align*}
where $j\leq pi$, $\e\in\{0,1\}$.
\end{varthm}

We obtain these relations in Theorem~\ref{theorem: dll action on lie} by applying the differential $\tilde{d}_2$ twice to the quadratic admissible monomials $\iota_l\otimes \nu_{i}^{\e'}\nu_{j}^{\e} \in \pi_*(\Sigma^l \kk)\otimes \Lambda \cong \widetilde{E}^1_{*,0}(\Sigma^l \kk)$. On the one hand, since $\tilde{d}_2$ is a differential, $\tilde{d}_2^2=0$, so we must get zero. On the other hand, by applying the recipes of Theorems~\ref{theorem: D, intro} and~\ref{theorem: E,intro}, we derive the desired formulas from the Adem relations~\eqref{equation: lambda, adem1, p is odd}, \eqref{equation: lambda, adem2, p is odd} in the algebra $\Lambda$. In Theorem~\ref{theorem: homotopy of free H-Lie algebra}, we show that the list of relations in Theorem~\ref{theorem: F, intro} is exhaustive.

It is well known that the Adem relations of Theorem~\ref{theorem: F, intro} can be derived from the transfer maps in the group homology of the symmetric group $\Sigma_{p^2}$, see Remark~\ref{remark: adem relations, transfers}. We expect that similar relations hold for operations of the same nature, e.g. for the Cartan $\delta$-operations acting on the homotopy groups of a simplicial commutative $\kk$-algebra, see~\cite{Dwyer_sca},~\cite{Goerss_sca}, or the unpublished notes~\cite[Sections 8 and 12]{Bousfield_operations} by A.~K.~Bousfield.

\subsection*{Organization} The paper is organized as follows. Section~\ref{section: preliminaries} is the preliminary part of the paper. In Section~\ref{section: simplicial restricted lie algebras} we recall from~\cite{Priddy70long} the homotopy theory of the category $\sL$. We also recall the important subcategory $\sLxi\subset \sL$ of \emph{$\kk$-complete objects} (see Definition~\ref{definition: xi-complete category}) constructed in~\cite{koszul}. In Section~\ref{section: spectrum objects} we discuss the stable category $\Sp(\sL)$ (resp. $\Sp(\sLxi)$) of spectrum objects in $\sL$ (resp. $\sLxi$). In Proposition~\ref{proposition: smash, xi-complete lie algebras} and Theorem~\ref{theorem: smash, stabilization} we construct the \emph{smash products} in the categories $\sLxi$ and $\Sp(\sLxi)$ respectively. In Section~\ref{section: lambda algebra} we recall the definition of the \emph{lambda algebra $\Lambda$} of~\cite{6authors}. 
Finally, in Section~\ref{section: H-Lie-algebras}, we review the construction of \emph{Dyer--Lashof--Lie} operations (Definition~\ref{definition: dll operations}) and we introduce the useful notion of \emph{$H_\infty$-$\bfLie_*$-algebras} (see the monad~\eqref{equation: H-Lie-monad})

Sections~\ref{section: scalar degree} and~\ref{section: vanishing theorems} contain our main technical ``vanishing'' results concerning natural transformations between functors from the category $\Mod^{\geq 0}_{\kk}$ to $\sLxi$, see Proposition~\ref{proposition:maps between homogeneous functors}. We use these results in Section~\ref{section: snaith splitting} to obtain the \emph{Snaith splitting} (Proposition~\ref{proposition: Snaith splitting}) for the functor $\Sigma^{\infty}\Omega^{\infty}\suspfreexi$. In Section~\ref{section: goodwillie tower of a free lie algebra} we prove Theorem~\ref{theorem: A,intro} as Corollary~\ref{corollary:derivatives of free} and we show that the differentials in the algebraic Goodwillie spectral sequence are induced by the natural transformations $\delta_n$, see~\eqref{equation:Dpk to Dk}. Finally, we relate the maps $\delta_n$ with the $p$-th James--Hopf map, see Proposition~\ref{proposition:james-hopf and adjoint}.

In Section~\ref{section: james-hopf map} we evaluate the $p$-th James--Hopf map on the generators of the algebra $\Lambda$, see Theorem~\ref{theorem:james-hopf and dyer-lashof}. This is a significant part in the proof of Theorem~\ref{theorem: C, intro}. In Section~\ref{section: group cohomology} we recall the construction and basic properties of Dyer--Lashof classes $\beta^{\e}Q^i(v)$, $v\in V$ in the Tate cohomology groups $\pi_*(V^{\otimes p})^{t\Sigma_p}$. In Section~\ref{section: tate diagonal} we explain how to compute the map $\widetilde{H}_*(\Omega^{\infty}f)$ induced by a map $f$ of suspension spectrum objects, see Proposition~\ref{proposition:infinite loops of a map between suspension spectra}. In Proposition~\ref{proposition: hurewicz of hopf invariant one}, we apply this computation to calculate the unstable Hurewicz homomorphism.

Section~\ref{section: composition product} is devoted to the proof of the Leibniz identity for the differentials in the Goodwillie spectral sequence, see Theorem~\ref{theorem:leibniz rule in GSS}. In Section~\ref{section: category of pairs} we consider the \emph{categories of pairs} and the fiber sequences of mapping spaces associated with such categories. The main result is Corollary~\ref{corollary:leibniz rule} which show that the connecting map satisfy the Leibniz rule with respect to composition in a category of pairs. In Section~\ref{section: goodwillie differentials} we derive the desired statement about the Goodwillie spectral sequence by using $n$-excisive approximations and certain fiber squares of categories from~\cite{Heuts21}.

In Section~\ref{section: james periodicity} we explain how to use the James periodicity (Theorem~\ref{theorem:james periodicity for spaces}) to identify the ``difficult'' summand in the Leibniz identity above. The main notion here is the shift operators (Definition~\ref{definition: shift operators}) and the main result is Theorem~\ref{theorem: main theorem, james periodicity}. Section~\ref{section: james periodicity for functors} contains a general review on the James periodicity for spaces and functors. Section~\ref{section: computations with the p-th goodwillie layer} explains how to apply these general results in the particular case of the category $\sL$ of simplicial restricted Lie algebras.

In Section~\ref{section: two-stage functors} we adapt the argument of~\cite{Behrens11} to our context and prepare the proofs of Theorems~\ref{theorem: C, intro} and~\ref{theorem: E,intro}. In Section~\ref{section: coalgebras in symmetric sequences} we discuss the category of symmetric sequences $\sseq$ and we recall from~\cite{Ching05} that the category of left comodules in $\sseq$ over the cocommutative coalgebras is equivalent by the Koszul duality to the category of left modules over the $\kk$-linear spectral Lie operad $\bfLie^{\mathrm{sp}}_*$. Given an analytic functor $F\in \Fun^{\omega}_{an}(\Sp,\sL)$, we show in Section~\ref{section: derivaties and coderivatives} that its symmetric sequence of coderivatives $\bar{\6}^*(F)\in\sseq$ forms a left comodule over the cocommutative cooperad (Proposition~\ref{proposition: coderivatives are symmetric monoidal}), and moreover, the Koszul dual $\bfLie_*^{\mathrm{sp}}$-module to $\bar{\6}^*(F)$ coincides with the sequence of its derivatives $\bar{\6}_*(F)$, see Theorem~\ref{theorem:AC conjecture for analytic functors}. In Section~\ref{section: connecting map} we consider analytic functors with exactly two non-trivial Goodwillie layers; in particular, we relate in Corollary~\ref{corollary:lie action, corona, adjoint} the connecting map with the left $\bfLie^{\mathrm{sp}}_*$-module $\bar{\6}_*(F)$ of derivatives. Finally, in Section~\ref{section: main example} we apply these general results to the truncated functor $\widetilde{P}_{n}(\Omega^{p-2}L_\xi\free)$, see Theorem~\ref{theorem: the adjoint of delta_n}.

We finish the proofs of our main results in Section~\ref{section: application}. In Proposition~\ref{proposition: differential, length 1} we combine together the results of Sections~\ref{section: james-hopf map} and~\ref{section: two-stage functors} to obtain Theorem~\ref{theorem: C, intro}. We use Sections~\ref{section: composition product} and~\ref{section: james periodicity} to prove Theorem~\ref{theorem: D, intro} as Theorems~\ref{theorem: lin formula, odd p} and~\ref{theorem: lin formula, p = 2}. In Section~\ref{section: adem-type relations} we apply again the results of Section~\ref{section: two-stage functors} to finish the proof of Theorem~\ref{theorem: E,intro} in Lemmas~\ref{lemma:GSS differential, positive columns} and~\ref{lemma:GSS differential, positive columns, p =2}. We prove Theorem~\ref{theorem: F, intro} as Theorem~\ref{theorem: dll action on lie}. Section~\ref{section: whitehead conjecture} is devoted to the proof of Theorem~\ref{theorem: B, intro} which is based on Corollaries~\ref{cor: goodwillie differential lowers the filtation, p is odd} and~\ref{cor: goodwillie differential lowers the filtation, p is 2}. 

In Remark~\ref{remark: lin-nishida conjecture} we propose the \emph{Lin--Nishida conjecture} concerning the interaction of the differentials in the Adams and Goodwillie spectral sequences; the proposed conjecture extends the theorem of W.~H.~Lin~\cite{Lin81}.

\subsection*{Acknowledgments}
The author is grateful to Mark Behrens, Dmitry Kaledin, Artem Prikhodko, and Denis Tereshkin for many fruitful conversations. The author also thanks Greg Arone, William Balderrama, and Adela YiYu Zhang for their interest to the paper and helpful discussions.

\subsection*{Notation} Here we describe some notation that will be used throughout the paper. If it is not said otherwise, the word ``category'' means ``an $\infty$-category'' which we model as the quasi-categories of~\cite{HTT} and~\cite{HigherAlgebra}. The functor calculus originally was introduced by T.~Goodwillie in~\cite{GoodwillieIII} for homotopy functors between the categories of spaces or spectra, however, it was later generalized in~\cite{HigherAlgebra} to the context of functors between quasi-categories. We will refer the reader to~\cite[Chapter~6]{HigherAlgebra} for the basic machinery and the notation of the functor calculus. 

We write $\Catinfty_\infty$ for the $\infty$-category of small $\infty$-categories. Given a small category $\calC \in \Catinfty_{\infty}$ and an arbitrary $\infty$-category $\calD$, we write $\Fun(\calC, \calD)$ for the $\infty$-category of functors and natural transformations. If $\calC$ and $\calD$ are pointed, we write $\Fun_*(\calC,\calD) \subset\Fun(\calC, \calD)$ for the subcategory of reduced functors, i.e. $F\in \Fun_*(\calC,\calD)$ if and only if $F(\ast)\simeq \ast$. 

We write $\Pre^L$ for the category of presentable $\infty$-categories and colimit-preserving functors, \cite[Section~5.5]{HTT}. A presentable $\infty$-category is compactly generated if it is $\omega$-accessible, see~\cite[Section~5.5.7]{HTT}. Given two presentable $\infty$-categories $\calC, \calD \in \Pre^L$, a functor $F\colon \calC \to \calD$ is \emph{finitary} if it commutes with filtered colimits; we write $\Fun^{\omega}(\calC,\calD)$ for the $\infty$-category of finitary functors and natural transformations. If $\calC$ and $\calD$ are pointed, we write $\Fun^{\omega}_*(\calC,\calD)$ for the subcategory of reduced functors in $\Fun^{\omega}(\calC,\calD)$.

Let us denote by $\spaces$ (resp. $\spaces_*$, $\Sp$) the presentable $\infty$-category of spaces (resp. pointed spaces, spectra).

\smallskip

As was said, $p$ is a fixed prime number and $\kk=\overline{\F}_p$ is the algebraic closure of the finite field $\F_p$. We denote by $\Vect_{\kk}$ the ($1$-)category of vector spaces over $\kk$; $\Vect^{\mathrm{gr}}_{\kk}$ (resp. $\Vect^{>0}_{\kk}$) is the ($1$-)category of non-negatively (resp. positively) graded vector spaces over $\kk$. We usually denote by $V_*=\oplus_{q\geq 0} V_q$ an object of $\Vect^{\mathrm{gr}}_{\kk}$. 

We denote by $\Mod_{\kk}$ the stable $\infty$-category of unbounded chain complexes over~$\kk$. Given $V\in \Mod_{\kk}$, we write $\pi_q(V)$, $q\in\Z$ for the $q$-th homology group of $V$. We reserve the symbol $H_q$ for the homology groups of either a space or a simplicial restricted Lie algebra (Definition~\ref{definition: chain complex}). We write $\Mod_{\kk}^{+}$ (resp. $\Mod_{\kk}^{\geq 0}$) for the full subcategory of $\Mod_{\kk}$ spanned by bounded below (resp. with trivial negative homology groups) chain complexes, $V\in \Mod_{\kk}^{+} $ if and only if $\pi_q(V)=0$ for $q\ll 0$ (resp. $V\in \Mod_{\kk}^{\geq 0}$ if and only if $\pi_q(V)=0$ for $q <0$). 

Elsewhere below we identify by the Dold--Kan correspondence the $\infty$-category $\Mod_{\kk}^{\geq 0}$ with the $\infty$-category of simplicial vector spaces and use the terms ``chain complex in $\Mod_{\kk}^{\geq 0}$'' and ``simplicial vector space'' interchangeably.

Given a category $\calC$ and a group $G$, we write $\calC^{BG}$ for the functor category $\Fun(BG,\calC)$; the category $\calC^{BG}$ is the category of $G$-equivariant objects in $\calC$. 

Throughout this paper, we write $\bfLie_n \in \Vect_{\kk}^{B\Sigma_n}$ for the $n$-th vector space of the Lie operad, see Definition~\ref{definition: lie operad}.

We write $F\dashv G$ or 
\begin{equation*}
\begin{tikzcd}
F: \calC \arrow[shift left=.6ex]{r}
&\calD :G \arrow[shift left=.6ex,swap]{l}
\end{tikzcd}
\end{equation*}
if the functor $F$ is the left adjoint to the functor $G$.  We write $\oblv$ (abbrv. to obliviate) for various forgetful functors. Finally, if $\eta \colon F \to G$ is a natural transformation between functors $F,G\colon \calC \to\calD$, then we write $\eta^{X} \colon F(X) \to G(X)$, $X\in \calC$ for the induced map. We will often omit the superscript in $\eta^X$ if the object $X$ is clear from the context.

\addtocontents{toc}{\protect\setcounter{tocdepth}{2}}
\section{Preliminaries}\label{section: preliminaries}

In this section we provide background for simplicial restricted Lie algebras, the algebra $\Lambda$, and the Dyer--Lashof--Lie operations. In Section~\ref{section: simplicial restricted lie algebras} we recall from~\cite{Priddy70long} the homotopy theory of the $\infty$-category $\sL$ of simplicial restricted Lie algebras. In particular, in Definition~\ref{definition: chain complex} we recall the notions of the \emph{chain complex} $\widetilde{C}_*(L)$ and the \emph{homology groups} $\widetilde{H}_*(L)$ of $L\in \sL$. We also recall from~\cite{koszul} the notion of \emph{$\kk$-completeness} in $\sL$, the full subcategory $\sLxi \subset \sL$ of $\kk$-complete objects, and their basic properties. 

In Section~\ref{section: spectrum objects} we discuss the stable categories $\Sp(\sL)$ and $\Sp(\sLxi)$. We prove that the category $\Sp(\sLxi)$ is monadic over $\Mod_{\kk}$ in Proposition~\ref{proposition: xi-complete stable is monadic} and compute the resulting monad in Proposition~\ref{proposition: stable monad for xi-complete}.

In Section~\ref{section: monoidal structures} we construct the symmetric monoidal structure given by the \emph{smash product} in the categories $\sLxi$, $\Sp(\sLxi)$, see Proposition~\ref{proposition: smash, xi-complete lie algebras} and Theorem~\ref{theorem: smash, stabilization}. We also show that many natural functors associated with the categories $\sLxi$, $\Sp(\sLxi)$ are strong symmetric monoidal. In Remarks~\ref{remark: smash, all lie algebras} and~\ref{remark: stable infty-operad} we discuss potential problems to produce similar symmetric monoidal structures for the categories $\sL$ and $\Sp(\sL)$ of non-necessary $\kk$-complete objects.

In Section~\ref{section: lambda algebra} we define the \emph{lambda algebra $\Lambda$} of~\cite{6authors}. We point out that in this work we use the convention for $\Lambda$ from~\cite[Definition~7.1]{Wellington82} but not from the original paper~\cite{6authors}. In Theorem~\ref{theorem: stable homotopy groups of a free object}, we also recall how to compute the stable homotopy groups $\pi_*^s(\free(V))$ of the free simplicial restricted Lie algebra $\free(V)\in \sL$ in terms of the algebra~$\Lambda$.

In Section~\ref{section: H-Lie-algebras} we recall from~\cite{Behrens12} and~\cite{Kjaer18} the construction of the \emph{Dyer--Lashof--Lie classes}  $\bQe{\e}{i}$ (Definition~\ref{definition: dll-classes}). These classes produce \emph{Dyer--Lashof--Lie operations} (Definition~\ref{definition: dll operations}) which act on the homotopy groups of any $H_\infty$-$\bfLie_*$-algebra. We define \emph{$H_\infty$-$\bfLie_*$-algebras} as the Lie analog of $H_\infty$-rings from \cite[Section~1.3]{BMMS}, see~\eqref{equation: H-Lie-monad}.

\subsection{Simplicial restricted Lie algebras}\label{section: simplicial restricted lie algebras}

Let $L$ be a Lie algebra and let $x\in L$. We denote by $\ad(x)\colon L\to L$ the map given by $y\mapsto \ad(x)(y)=[y,x]$.

\begin{dfn}\label{definition:p-operation}
Let $L$ be a Lie algebra over $\kk$. A \emph{$p$-operation} on $L$ is a map $\xi\colon L \to L$ such that 

\begin{itemize}
\item $\xi(ax)=a^p\xi(x)$, $a\in \kk, x\in L$;
\item $\ad(\xi(x))=\ad(x)^{\circ p}\colon L \to L$;
\item $\xi(x+y)=\xi(x)+\xi(y)+\sum_{i=1}^{p-1}\dfrac{s_i(x,y)}{i}$, for all $x,y\in L$, where $s_{i}(x,y)$ is a coefficient of $t^{i-1}$ in the formal expression $\ad(tx+y)^{\circ(p-1)}(x)$.
\end{itemize}

\end{dfn}

\begin{dfn}[{{\cite[Definition~V.4]{Jacobson79}}}]\label{definition:lier} A \emph{restricted Lie algebra} $(L,\xi)$ is a Lie algebra $L$ (over $\kk$) equipped with a $p$-operation $\xi\colon L \to L$. We write $\rLie$ for the ($1$-)category of restricted Lie algebras.
\end{dfn}

Recall that a Lie algebra $L$ is called \emph{abelian} if $L$ is equipped with zero bracket. First, we will describe abelian restricted Lie algebras.

\begin{dfn}[{{\cite[Definition~2.1.3]{koszul}}}]\label{definition: twisted polynomial ring}
The \emph{twisted polynomial ring} $\kk\{\xi\}$ is defined as the set of polynomials in the variable $\xi$ and coefficients in $\kk$. It is endowed with a ring structure with the usual addiction and with a non-commutative multiplication that can be summarized with the relation: $$\xi a = a^{p} \xi, \; \; a\in \kk.$$
We denote by $\Mod_{\kk\{\xi\}}$ (resp. $\Mod^{\kk\{\xi\}}$) the abelian ($1$-)category of \emph{left}  (resp. \emph{right}) $\kk\{\xi\}$-modules.
\end{dfn}

The full subcategory of abelian restricted Lie algebras is equivalent to $\Mod_{\kk\{\xi\}}$ because, if $L$ is an abelian restricted Lie algebra, then the $p$-operation $\xi\colon L\to L$ is additive. 
Finally, we say that an abelian restricted Lie algebra is \emph{$p$-abelian} if the $p$-operation $\xi\colon L\to L$ is trivial, i.e. $\xi=0$. We denote by $\triv(W)$ a unique $p$-abelian restricted Lie algebra with the underlying vector space equal to $W\in \Vect_{\kk}$.


\begin{exmp}\label{example:assislie}
Given an associative $\kk$-algebra $A$, we write $A^{\circ}$ for the restricted Lie algebra whose underlying vector space is $A$ equipped with the bracket $[x,y]=xy-yx$ and the $p$-operation $\xi(x)=x^p$.
\end{exmp}


\begin{exmp}\label{example:hopfislie}
Let $H$ be a cocommutative Hopf algebra over $\kk$. Then the set of primitive elements $P(H)$ is a restricted Lie subalgebra in $H^{\circ}$, see Example~\ref{example:assislie}. 
\end{exmp}

\begin{exmp}\label{example:freelie}
Let $V\in\Vect_\kk$ and let $T(V)$ be the tensor algebra generated by~$V$. It is well-known that $T(V)$ has a unique structure of a cocommutative Hopf algebra such that generators $v\in V$ are primitive elements. Therefore $P(T(V))$ is a restricted Lie algebra, which we will denote by $\free(V)$.

We say that a restricted Lie algebra $L$ is \emph{free} if $L$ is isomorphic to $\free(V)$ for some $V\in \Vect_{\kk}$. This terminology is justified because the functor $$\free\colon \Vect_{\kk} \to \rLie$$ is the left adjoint to the forgetful functor $$\oblv\colon \rLie \to \Vect_{\kk}.$$
\end{exmp}

\begin{rmk}\label{remark: free lie algebras, fresse}
By~\cite[Theorem~1.2.5]{Fresse00}, the underlying vector space of a free restricted Lie algebra $\free(V),V\in \Vect_{\kk}$ splits as follows
$$\oblv\circ\free(V)\cong \bigoplus_{n\geq 1}L^r_n(V) = \bigoplus_{n\geq 1} (\bfLie_n \otimes V^{\otimes n})^{\Sigma_n}.$$
Here $\bfLie_n \in \Vect_{\kk}$ is the $n$-th space of the Lie operad.
\end{rmk}

Let us denote by $\sL$ the $\infty$-category of \emph{simplicial restricted Lie algebras}, i.e. $\sL$ is obtain from the (1-)category $\Fun(\Delta^{op},\rLie)$ by inverting the obvious class of weak equivalences. We extend the adjoint pair $\free \dashv \oblv$ degreewise to the adjunction
\begin{equation}\label{equation: quillen adjunction, free oblv}
\begin{tikzcd}
\free: \Mod^{\geq 0}_{\kk} \arrow[shift left=.6ex]{r}
&\sL:\oblv \arrow[shift left=.6ex,swap]{l}
\end{tikzcd}
\end{equation}
between $\infty$-categories of simplicial objects. Here we abuse notation and we denote a functor between $1$-categories and its derived functor acting between $\infty$-categories by the same symbol. As in Remark~\ref{remark: free lie algebras, fresse}, we have
\begin{equation}\label{equation: oblv free, fresse}
\oblv\circ\free(V)\simeq \bigoplus_{n\geq 1}\mathbb{L}L^r_n(V) = \bigoplus_{n\geq 1} \mathbb{L}(\bfLie_n \otimes V^{\otimes n})^{\Sigma_n},\;\; V\in \Mod_{\kk}^{\geq 0},
\end{equation}
where $\mathbb{L}L^r_n(-)\colon \Mod_{\kk}^{\geq 0} \to \Mod_{\kk}^{\geq 0}$ is the (left) nonabelian derived functor of $$L^r_n(-)=(\bfLie_n\otimes (-)^{\otimes n})^{\Sigma_n}\colon \Vect_{\kk} \to \Vect_{\kk},$$ see e.g.~\cite{DP61} and~\cite[Construction~3.3]{BM19}.

\begin{prop}\label{proposition:category property of srLie} The $\infty$-category $\sL$ is monadic over $\Mod^{\geq 0}_{\kk}$ via the adjunction $\free \dashv \oblv$. The category $\sL$ is complete and cocomplete and the forgetful functor $\oblv$ creates limits and sifted colimits. Moreover, $\sL$ is presentable and compactly generated.
\end{prop}

\begin{proof}
See~\cite[Propositions~2.1.9 and~3.2.29]{koszul}. Also, \cite[Theorem~3.2.3]{koszul} implies that the category $\sL$ is compactly generated.
\end{proof}

\begin{cor}\label{corollary: srLie is differentiable}
The category $\sL$ is differentiable in the sense of~\cite[Definition~6.1.1.6]{HigherAlgebra}; i.e. $\sL$ admits all finite limits, all sequential colimits, and sequential colimits are left exact.
\end{cor}

\begin{proof}
By~\cite[Example~6.1.1.9]{HigherAlgebra}, any compactly generated presentable category is differentiable.
\end{proof}

We observe that the functor $\triv \colon \Vect_{\kk} \to \rLie$ (see the paragraph before Example~\ref{example:assislie}) has a left adjoint
\begin{equation}\label{equation: abxi}
\Abxi\colon \rLie \to \Vect_{\kk} 
\end{equation}
given by $L\mapsto L/([L,L]+\xi(L))$, where $[L,L]\subset L$ is the (restricted) Lie ideal generated by all elements of the form $[x,y]$, $x,y\in L$.

We extend the adjoint pair $\Abxi \dashv \triv$ degreewise to the adjunction
\begin{equation}\label{equation: quillen adjunction, abxi triv}
\begin{tikzcd}
\Abxi: \sL \arrow[shift left=.6ex]{r}
&\Mod_{\kk}^{\geq 0} :\triv \arrow[shift left=.6ex,swap]{l}
\end{tikzcd}
\end{equation}
between $\infty$-categories of simplicial objects. Again, we abuse notation and we denote a functor between $1$-categories and its derived functor acting between $\infty$-categories by the same symbol. 

\begin{dfn}[{{\cite[Definition~4.2.4]{koszul}}}]\label{definition: chain complex}
Let $L\in \sL$. We define the \emph{(reduced) chain complex} $\widetilde{C}_*(L) =\widetilde{C}_*(L;\kk)\in \Mod_\kk$ of $L$ by the formula
$$\widetilde{C}_*(L) = \Sigma \Abxi(L). $$
Furthermore, we define the \emph{$q$-th homology group}    $\widetilde{H}_q(L)$ of $L$ by the next rule
$$ \widetilde{H}_q(L) = \pi_q(\widetilde{C}_*(L)), \;\; q\geq 0.$$
\end{dfn}

\begin{dfn}[{{\cite[Definition~4.2.23]{koszul}}}]\label{definition: cochain complex}
Let $L \in \sL$. We define the \emph{(reduced) cochain complex} $\widetilde{C}^*(L) = \widetilde{C}^*(L;\kk) \in \Mod_\kk$ of $L$ as follows
$$\widetilde{C}^*(L) = \Hom_{\kk} (\Sigma \Abxi(L), \kk). $$
Furthermore, we define the \emph{$q$-th cohomology group} $\widetilde{H}^q(L;\kk)$ of $L$ by the rule
$$ \widetilde{H}^q(L;\kk) = \pi_{-q}(\widetilde{C}^*(L)), \;\; q\geq 0.$$
\end{dfn}

\begin{rmk}\label{remark: non-reduced chains}
Let $L\in \sL$. We define the \emph{non-reduced chain complex} $C_*(L)\in \Mod_{\kk}$ (resp. \emph{non-reduced cochain complex} $C^*(L) \in \Mod_{\kk}$) as follows
$$C_*(L)=\kk \oplus \widetilde{C}_*(L) \;\; \text{(resp. $C^*(L)=\kk\oplus \widetilde{C}^*(L)$).}$$
\end{rmk}

\begin{rmk}\label{remark: hurewuicz homomorphism}
Note that the functor $\Abxi\colon \sL \to \Mod^{\geq 0}_{\kk}$ comes with the natural transformation 
\begin{equation*}
\id \to \triv \circ \Abxi
\end{equation*}
given by the quotient map $L \to L/([L,L]+\xi(L))=\Abxi(L)$. This natural transformation induces the \emph{Hurewicz homomorphism}
\begin{equation}\label{equation: hurewicz homomorphism}
h\colon \pi_q(L) \to \widetilde{H}_{q+1}(L), \;\; L\in \sL, \;\; q\geq 0. 
\end{equation}
\end{rmk}

\begin{dfn}[{{\cite[Definition~3.2.21]{koszul}}}]\label{definition: barW-equivalence}
A map $f\colon L' \to L$ in $\sL$ is an \emph{$\kk$-equivalence} if and only if $\widetilde{C}_*(f)$ is an equivalence in $\Mod_{\kk}$. We write $\mathcal{W}_{\kk}$ for the class of $\kk$-equivalences in $\sL$.
\end{dfn}

\begin{dfn}\label{definition: xi-complete category}
We denote by $\sLxi=\sL[\mathcal{W}^{-1}_{\kk}]$ the $\infty$-category obtained from the $\infty$-category $\sL$ by inverting the class $\mathcal{W}_{\kk}$ of $\kk$-equivalences, see \cite[Definition~1.3.4.1]{HigherAlgebra}.
\end{dfn}

\begin{thm}\label{theorem: xi-complte is an localization}
The $\infty$-category $\sLxi$ is presentable and the canonical functor $\sL \to \sLxi$ admits a fully faithful right adjoint $\sLxi \hookrightarrow \sL$. In particular, the full subcategory $\sLxi \subset \sL$ is a localization.
\end{thm}

\begin{proof}
See~\cite[Theorem~3.2.24 and Proposition~3.2.32]{koszul}.
\end{proof}

We say that a simplicial restricted Lie algebra $L\in \sL$ is \emph{$\kk$-complete} if $L\in \sLxi$. We write $L_\xi \colon \sL \to \sLxi$ for the canonical functor $\sL \to \sL[\mathcal{W}^{-1}_{\kk}]$ and we write $\iota_{\xi}\colon \sLxi \hookrightarrow \sL$ for its right adjoint. Then we have the following adjoint pair
\begin{equation}\label{equation: quillen adjunction, xifree oblv}
\begin{tikzcd}
L_\xi\free: \Mod^{\geq 0}_{\kk} \arrow[shift left=.6ex]{r}
& \sL \arrow[shift left=.6ex]{r} \arrow[shift left=.6ex,swap]{l}
&\sLxi:\oblv = \oblv \circ \iota_{\xi}. \arrow[shift left=.6ex]{l}
\end{tikzcd}
\end{equation}

\begin{cor}\label{corollary: xi-complete is monadic}
The category $\sLxi$ is complete and cocomplete and the forgetful functor $\oblv$ creates limits. \qed
\end{cor}

Next, we will describe the full subcategory $\sLxi\subset \sL$ of $\kk$-complete objects  together with the free object $L_\xi \free(V)$. Let $L\in \sL$. Note that the $p$-operation $\xi\colon L \to L$ induces a map of homotopy groups 
\begin{equation}\label{equation: xi-action}
\xi_*\colon \pi_{n}(L) \to \pi_n(L)
\end{equation}
which is \emph{additive} for $n\geq 1$. Therefore all homotopy groups $\pi_n(L), n\geq 1$ are naturally left modules over the twisted polynomial ring $\kk\{\xi\}$, see Definition~\ref{definition: twisted polynomial ring}.

\begin{rmk}[Hurewicz theorem]\label{remark: hurewicz theorem}
We recall a variant of the Hurewicz theorem from~\cite[Theorem~4.2.10]{koszul}. Let $L \in \sL$ be a simplicial restricted Lie algebra. Suppose that $\pi_0(L)=0$ and $\pi_i(L)=0$ for $0\leq i \leq n$. Then the Hurewicz homomorphism
$$h\colon \pi_{n+1}(L)/\xi \xrightarrow{\cong} \widetilde{H}_{n+2}(L) $$
is an isomorphism. 
\end{rmk}

\begin{thm}\label{theorem: connected xi-complete}
Let $L \in \sL$ be a connected simplicial restricted Lie algebra, $\pi_0(L)=0$. Then $L \in \sLxi$ if and only if all homotopy groups $\pi_n(L), n\geq 1$ are derived $\xi$-adic complete left modules over the ring $\kk\{\xi\}$. 
\end{thm}

Here a left $\kk\{\xi\}$-module $M$ is \emph{derived $\xi$-adic complete} if $M$ is derived complete with respect to the ideal $(\xi)\subset \kk\{\xi\}$.

\begin{proof}
See~\cite[Corollary~5.3.12]{koszul}.
\end{proof}

\begin{thm}\label{theorem: F-completion, homotopy groups}
Let $L \in \sL$ be a simplicial restricted Lie algebra, $\pi_0(L)=0$. Then we have a short exact sequence
$$0 \to L_1\pi_{i-1}(L) \to \pi_i(L_\xi L) \to L_0\pi_i(L) \to 0$$
for each $i\geq 0$, where $L_0, L_1$ are left derived functors of the $\xi$-completion.
\end{thm}

\begin{proof}
See~\cite[Corollary~5.3.11]{koszul}.
\end{proof}

We write $\sL_0 \subset \sL$ for the full subcategory of $\sL$ spanned by \emph{connected} simplicial restricted Lie algebras. Set $\sL_{\xi,0}= \sLxi \cap \sL_0$. Then we have $L_\xi(L)\in \sL_{\xi,0}$ provided $L\in \sL_0$. Theorem~\ref{theorem: F-completion, homotopy groups} implies that $L_\xi$ is left exact when restricted to connected objects.

\begin{cor}\label{corollary: xi-completion is left exact}
The $\kk$-completion $L_\xi\colon \sL_0 \to \sL_{\xi,0}$ is left exact, i.e. $L_\xi$ preserves finite limits. \qed
\end{cor}

\begin{cor}\label{corollary: xi-complete is differentiable}
The category $\sL_{\xi,0}$ is presentable and differentiable in the sense of~\cite[Definition~6.1.1.6]{HigherAlgebra}.
\end{cor}

\begin{proof}
Note that $\sL_{0}$ is a compactly generated presentable category, cf. Proposition~\ref{proposition:category property of srLie}. In particular, $\sL_0$ is differentiable, see \cite[Example~6.1.1.9]{HigherAlgebra}. By Theorem~\ref{theorem: xi-complte is an localization} and Theorem~\ref{theorem: F-completion, homotopy groups}, $\sL_{\xi,0}$ is a localization of $\sL_{0}$. Therefore, $\sL_{\xi,0}$ is presentable; in particular, it admits all finite limits, which are computed in $\sL_0$, and all sequential colimits, which are obtained by applying the localization $L_{\xi}$ to the colimit in $\sL_0$. By Corollary~\ref{corollary: xi-completion is left exact}, the localization $L_{\xi}\colon \sL_0 \to \sL_{\xi,0}$ is left exact. Hence sequential colimits in $\sL_{\xi,0}$ are left exact as well.
\end{proof}

\begin{prop}\label{proposition: free xi-complete}
Let $V \in \Mod_{\kk}^{\geq 1}$ be a connected simplicial vector space, $\pi_0(V)=0$. Then 
\begin{enumerate}
\item $\pi_i(\free(V))$ is a free $\kk\{\xi\}$-module for each $i\geq 0$;
\item the underlying simplicial vector space of an $\kk$-complete free simplicial restricted Lie algebra $L_\xi \free(V)$ splits as follows
$$\oblv \circ L_\xi \free(V) \simeq \prod_{n\geq 1}\mathbb{L}L^r_n(V) = \prod_{n\geq 1} \mathbb{L}(\bfLie_n \otimes V^{\otimes n})^{\Sigma_n},
$$
where $\mathbb{L}L^r_n(-)$ are the nonabelian derived functors from~\eqref{equation: oblv free, fresse}. 
\end{enumerate}
\end{prop}

\begin{proof}
The statement essentially follows by the Curtis theorem~\cite{Curtis_lower} and~\cite{IRS20}. See for details~\cite[Proposition~6.4.4 and Corollary~6.4.7]{koszul}.
\end{proof}

\subsection{Spectrum objects}\label{section: spectrum objects}
Let $\calC$ be a pointed $\infty$-category which admits finite limits and colimits. We write $\Omega_{\calC}\colon \calC \to \calC$ (resp. $\Sigma_{\calC}\colon \calC \to \calC$) for the \emph{loop} (resp. the \emph{suspension}) functor. Recall from~\cite[Definition~1.4.2.8]{HigherAlgebra} that the \emph{stabilization} $\Sp(\calC)$ of $\calC$ is defined as the category $$\Sp(\calC)=\mathrm{Exc}_*(\spaces_*^{\mathrm{fin}},\calC)$$ of reduced, $1$-excisive functors from the category $\spaces_*^{\mathrm{fin}}$ of pointed, finite spaces to the category $\calC$. Alternatively, by~\cite[Proposition~1.4.2.24]{HigherAlgebra}, the stabilization $\Sp(\calC)$ can be identified with the limit of the tower
$$\Sp(\calC) \simeq \lim(\ldots\to \calC \xrightarrow{\Omega_{\calC}} \calC \xrightarrow{\Omega_{\calC}} \calC). $$
Here the limit is taken in the category of $\infty$-categories. We denote by $$\Omega^{\infty}_{\calC}\colon \Sp(\calC) \to \calC$$ the projection onto the first component.

Moreover, if $\calC$ is a presentable $\infty$-category, then the stabilization $\Sp(\calC)$ is also presentable and it is equivalent to the colimit 
$$\Sp(\calC) \simeq \colim(\calC \xrightarrow{\Sigma_{\calC}} \calC \xrightarrow{\Sigma_{\calC}} \calC \to \ldots)$$
computed in the category $\Pre^L$ of presentable $\infty$-categories. We denote by $$\Sigma^{\infty}_{\calC}\colon \calC \to \Sp(\calC)$$ the functor induced by the embedding of the first component. By~\cite[Proposition~1.4.4.4]{HigherAlgebra}, $\Sigma^{\infty}_{\calC}$ is the left adjoint to $\Omega^{\infty}_{\calC}$. We will often omit the subscripts in  $\Omega_{\calC}, \Sigma_{\calC}, \Omega^{\infty}_{\calC}, \Sigma^{\infty}_{\calC}$ if the category $\calC$ is clear from the context.

Note that $\Sp(\Mod_{\kk}^{\geq 0})\simeq \Mod_{\kk}$ and the canonical functor $$\Omega^\infty\colon \Sp(\Mod_{\kk}^{\geq 0}) \simeq \Mod_{\kk} \to \Mod_{\kk}^{\geq 0}$$ coincides with the Postnikov truncation $\tau_{\geq 0}$. Consider the adjoint pair~\eqref{equation: quillen adjunction, free oblv}. Since the forgetful functor $\oblv \colon \sL \to \Mod_{\kk}^{\geq 0}$
is left exact, it induces a functor
\begin{equation}\label{equation: forgetful, stable}
\oblv \colon \Sp(\sL) = \mathrm{Exc}_*(\spaces_*^{\mathrm{fin}},\sL) \to \mathrm{Exc}_*(\spaces_*^{\mathrm{fin}},\Mod_{\kk}^{\geq 0})=\Mod_{\kk}.
\end{equation}

\begin{prop}\label{proposition: stable is monadic}
The forgetful functor~\eqref{equation: forgetful, stable} admits a left adjoint 
\begin{equation}\label{equation: free lie, stable}
\suspfree\colon \Mod_{\kk} \to \Sp(\sL)
\end{equation}
such that
\begin{enumerate}
\item there is a natural equivalence $\suspfree(V) \simeq \susp \free(V) \in \Sp(\sL)$ if $V\in \Mod_{\kk}^{\geq 0}$;
\item the adjunction $\suspfree \dashv \oblv$ is monadic;
\item the composite $\oblv \circ \suspfree$ splits as follows
$$
\oblv\circ\suspfree(V)\simeq \bigoplus_{n\geq 1}\colim_m \bigg(\Sigma^{-m}\mathbb{L}L^r_n(\tau_{\geq 0} \Sigma^m V)\bigg),
$$
where $V\in \Mod_{\kk}$ and $\mathbb{L}L^r_n \colon \Mod_{\kk}^{\geq 0} \to \Mod_{\kk}^{\geq 0}$ are the nonabelian derived functors from~\eqref{equation: oblv free, fresse}.
\end{enumerate}
\end{prop}

\begin{proof}
The proposition follows from~\cite[Corollary~6.2.2.16]{HigherAlgebra} by applying it to the monadic adjunction of Proposition~\ref{proposition:category property of srLie}.
\end{proof}

Consider the full subcategory $\iota_{\xi}\colon \sL_{\xi,0} \hookrightarrow \sL_0$ of connected $\kk$-complete objects. By Corollary~\ref{corollary: xi-completion is left exact}, the embedding $\iota_{\xi}$ admits a left exact  left adjoint $L_\xi$. Therefore the induced functor
$$\iota^{st}_{\xi}\colon \Sp(\sL_{\xi,0}) \to \Sp(\sL_0)$$
between the categories of spectrum objects is fully faithful and it admits the left adjoint $$L_{\xi}^{st} \colon \Sp(\sL_0) \to \Sp(\sL_{\xi,0})$$ which is induced by the $\kk$-completion functor $L_\xi$. Finally, since the categories $\sL$, $\sLxi$ are presentable, the embeddings $\sL_0 \subset \sL$,  $\sL_{\xi,0} \subset \sLxi$ induce equivalences
$$\Sp(\sL_0) \simeq \Sp(\sL) \;\; \text{and} \;\; \Sp(\sL_{\xi,0}) \simeq \Sp(\sLxi). $$
We summarize the observations of this paragraph in the next proposition.
\begin{prop}\label{proposition: stabilization of xi-complete}
The stabilization $\Sp(\sLxi)$ is a full subcategory of $\Sp(\sL)$ and the fully faithful embedding $\iota^{st}_{\xi}\colon \Sp(\sLxi) \hookrightarrow \Sp(\sL)$ admits a left adjoint $$L^{st}_{\xi}\colon \Sp(\sL) \to \Sp(\sLxi).$$
In particular, $\Sp(\sLxi) \subset \Sp(\sL)$ is a localization. \qed
\end{prop}

\begin{dfn}\label{definition: stable homotopy groups}
Let $X\in \Sp(\sL)$. We define the \emph{$q$-th (stable) homotopy group}    $\pi_q(X)$ of $X$ by the next rule
$$ \pi_q(X) = \pi_0(\deloop \Omega^q X) \cong \pi_1(\deloop \Omega^{q-1} X), \;\; q\in \Z.$$
\end{dfn}

\begin{rmk}\label{remark: stable homotopy groups of xi-complete}
Let $X\in \Sp(\sL)$. Then, as in the unstable case, the homotopy groups $\pi_q(X), q\in \Z$ are naturally left modules over the twisted polynomial ring $\kk\{\xi\}$, see Definition~\ref{definition: twisted polynomial ring}. By applying Theorem~\ref{theorem: connected xi-complete}, one can see that $X\in \Sp(\sLxi)$ if and only if all homotopy groups $\pi_q(X),q\in\Z$ are derived $\xi$-adic complete left $\kk\{\xi\}$-modules.
\end{rmk}

Consider the adjoint pair~\eqref{equation: quillen adjunction, xifree oblv}. Since the forgetful functor $\oblv \colon \sLxi \to \Mod_{\kk}^{\geq 0}$
is left exact, it induces a functor
\begin{equation}\label{equation: forgetful, stable, xi-complete}
\oblv \colon \Sp(\sLxi) = \mathrm{Exc}_*(\spaces_*^{\mathrm{fin}},\sLxi) \to \mathrm{Exc}_*(\spaces_*^{\mathrm{fin}},\Mod_{\kk}^{\geq 0})=\Mod_{\kk}.
\end{equation}
Similar to Proposition~\ref{proposition: stable is monadic}, we obtain the monadicity statement for the stable category $\Sp(\sLxi)$ of $\xi$-complete objects.

\begin{prop}\label{proposition: xi-complete stable is monadic}
The forgetful functor~\eqref{equation: forgetful, stable, xi-complete} admits a left adjoint 
\begin{equation}\label{equation: free stable xi-complete}
\suspfreexi\colon \Mod_{\kk} \to \Sp(\sLxi)
\end{equation}
such that
\begin{enumerate}
\item there is a natural equivalence $\suspfreexi(V) \simeq L^{st}_{\xi}(\susp \free(V)) \in \Sp(\sLxi)$ if $V\in \Mod_{\kk}^{\geq 0}$;
\item the adjunction $\suspfreexi \dashv \oblv$ is monadic.
\end{enumerate}
\end{prop}

\begin{proof}
Note that the functor~\eqref{equation: forgetful, stable, xi-complete} is also induced by the forgetful functor
$$\oblv\colon \sL_{\xi,0} \to \Mod_{\kk}^{\geq 1}.$$
By Corollary~\ref{corollary: xi-complete is differentiable}, the category $\sL_{\xi,0}$ is differentiable. The assertion follows by~\cite[Corollary~6.2.2.16]{HigherAlgebra}.
\end{proof}

Since the forgetful functor~\eqref{equation: forgetful, stable, xi-complete} does not commute with sequential colimits, it might be difficult to compute the composite $\oblv \circ \suspfreexi$ in a full generality. We will compute this composite by restricting it to bounded below objects. 

Recall that an object $X\in \Mod_{\kk}$ is called \emph{$c$-connected} if $\pi_i(X)=0$ for all $i\leq c$. Also, $\conn(X) \in \N$ is the largest number $c$ such that $X$ is $c$-connected. 

\begin{lmm}\label{lemma: nonabelian derived, freudenthal, pre}
Let $F\colon \Mod_{\kk}^{\geq 0} \to \Mod_{\kk}^{\geq 0}$ be a functor which preserves sifted colimits and $F(0)=0$. Suppose that $U,W \in \Mod_{\kk}^{\geq 0}$, $\conn(U)=c_1$, $\conn(W)=c_2$. Then the map
$$F(U)\oplus F(W) \to F(U\oplus W) $$
is $(c_1+c_2+1)$-connected, i.e. the cofiber of the map is $(c_1+c_2+1)$-connected.
\end{lmm}

\begin{proof}
Let $G_U \colon \Mod^{\geq 0}_{\kk} \to \Mod^{\geq 0}_{\kk}$ be a functor given by the formula
$$G_U(-) = \cofib(F(-) \to F(U\oplus -)).$$
Here the map is induced by the direct summand inclusion $i\colon (-) \hookrightarrow U\oplus (-)$. By the assumption $U$ is $c_1$-connected, therefore $i$ is $c_1$-connected as well. Then, by~\cite[Proposition~3.43]{BM19}, $G_U(-)\in \Mod_{\kk}^{\geq c_1+1}$. In particular, $G_U\simeq \Sigma^{c_1+1} G'_U$ for some functor
$$G'_U\colon \Mod_{\kk}^{\geq 0} \to \Mod_{\kk}^{\geq 0} $$
which preserves sifted colimits.
Note that 
\begin{align*}
\cofib(F(U)\oplus F(W) \to F(U\oplus W)) &\simeq \cofib(G_U(0) \to G_U(W)) \\
&\simeq \Sigma^{c_1+1}\cofib(G'_U(0) \to G'_U(W)). 
\end{align*}
Again, by~\cite[Proposition~3.43]{BM19}, the cofiber $\cofib(G'_U(0) \to G'_U(W))$ is $c_2$-connected. This implies the statement.
\end{proof}

\begin{lmm}\label{lemma: nonabelian derived, freudenthal}
Let $F\colon \Mod_{\kk}^{\geq 0} \to \Mod_{\kk}^{\geq 0}$ be a functor which preserves sifted colimits and $F(0)=0$. Suppose that $V\in \Mod_{\kk}^{\geq 0}$, $\conn(V)=c$. Then the suspension map
$$\sigma \colon \Sigma F(V) \to F(\Sigma V) $$
is $(2c+1)$-connected.
\end{lmm}

\begin{proof}
We write the suspension functor $\Sigma \colon \Mod_{\kk}^{\geq 0} \to \Mod_{\kk}^{\geq 0}$ as the geometric realization 
$$\Sigma V \simeq |V^{\oplus (\bullet-1)}|, \;\; V\in \Mod_{\kk}^{\geq 0}. $$
Therefore it suffices to show that the map
$$F(V)^{\oplus m} \to F(V^{\oplus m}) $$
is $(2c+1)$-connected for each $m\geq 0$. The latter follows by Lemma~\ref{lemma: nonabelian derived, freudenthal, pre}.
\end{proof}

Recall from~\eqref{equation: oblv free, fresse} the nonabelian derived functors $\mathbb{L}L^r_n \colon \Mod_{\kk}^{\geq 0} \to \Mod_{\kk}^{\geq 0}$, $n\geq 1$. Note that the functors $\mathbb{L}L^r_n$, $n\geq 1$ preserve sifted colimits.

\begin{cor}\label{corollary: P1 commutes with inf product}
Let $V \in \Mod_{\kk}^{\geq 0}$ be a simplicial vector space. Then the natural map
$$\colim_{m} \bigg(\Sigma^{-m}\prod_{n\geq 1} \mathbb{L}L^r_n(\Sigma^m V) \bigg)\xrightarrow{\simeq} \prod_{n\geq 1} \colim_{m} \bigg( \Sigma^{-m} \mathbb{L}L^r_n(\Sigma^m V) \bigg) $$
is an equivalence.
\end{cor}
\begin{proof}
Follows from Lemma~\ref{lemma: nonabelian derived, freudenthal} by applying it to the functors $\mathbb{L}L^r_n \colon \Mod_{\kk}^{\geq 0} \to \Mod_{\kk}^{\geq 0}$, $n\geq 1$.
\end{proof}

\begin{cor}\label{corollary: deloop is xi-complete}
Let $V \in \Mod_{\kk}^{\geq 1}$ be a connected simplicial vector space, $\pi_0(V)=0$. Then the following sequential colimit
$$X(V)=\colim_m \left( \iota_{\xi} \Omega^m L_\xi \free (\Sigma^m V) \right) \simeq \colim_m \left( \Omega^m \iota_{\xi} L_\xi \free (\Sigma^m V) \right)\in \sL$$
is $\kk$-complete as a colimit computed in the category $\sL$.
\end{cor}

\begin{proof}
By Theorem~\ref{theorem: connected xi-complete}, it is enough to show that $X(V)\in \sL$ is connected and that $\pi_q(X(V))$ is a derived $\xi$-adic complete left $\kk\{\xi\}$-module for each $q\geq 1$. Fix $q\geq 0$, then 
$$\pi_q(X(V)) \cong \colim_m \pi_{q+m}(L_{\xi}\free(\Sigma^m V)). $$
By Proposition~\ref{proposition: free xi-complete} and Lemma~\ref{lemma: nonabelian derived, freudenthal}, there exists $m_0=m_0(q)\geq 0$ such that the natural map
$$\pi_{q+m}(L_{\xi}\free(\Sigma^m V)) \xrightarrow{\cong}\pi_{q+m+1}(L_{\xi}\free(\Sigma^{m+1} V)) $$
is an isomorphism for $m\geq m_0$. Therefore
$$\pi_q(X(V)) \cong \pi_{q+m_0}(L_{\xi}\free(\Sigma^{m_0} V)), $$
which proves the corollary.
\end{proof}

\begin{prop}\label{proposition: stable monad for xi-complete}
Let $V\in \Mod_{\kk}^{+}$ be a bounded below chain complex, $\pi_i(V)=0$ for $i\ll 0$. Then there exists a natural equivalence
$$\oblv \circ \suspfreexi(V) \simeq \prod_{n\geq 1}\colim_m \bigg(\Sigma^{-m}\mathbb{L}L^r_n(\tau_{\geq 0} \Sigma^m V)\bigg),
$$
where $\mathbb{L}L^r_n(-)$ are the nonabelian derived functors from~\eqref{equation: oblv free, fresse}. 
\end{prop}

\begin{proof}
By Proposition~\ref{proposition: xi-complete stable is monadic}, the monadic adjunction $\suspfreexi \dashv \oblv$ is induced from the adjoint pair
$$
\begin{tikzcd}
\free: \Mod^{\geq 1}_{\kk} \arrow[shift left=.6ex]{r}
&\sL_{\xi,0}:\oblv. \arrow[shift left=.6ex,swap]{l}
\end{tikzcd}
$$
By Corollary~\ref{corollary: xi-complete is differentiable}, the category $\sL_{\xi,0}$ is differentiable. Therefore, by~\cite[Proposition~6.2.2.15]{HigherAlgebra}, we have
$$\Omega^\infty \circ \suspfreexi(U) \simeq L_\xi \left(\colim_m ( \iota_{\xi} \Omega^m L_\xi \free (\tau_{\geq 1}\Sigma^m U) )\right)$$
for $U\in \Mod_{\kk}$. 

Without loss of generality, we can assume that $V \in \Mod_{\kk}^{\geq 1}$ is a \emph{connected} chain complex. Then, we have
$$\oblv \circ \suspfreexi(V) \simeq \oblv \circ L_\xi \left(\colim_m ( \iota_{\xi} \Omega^m L_\xi \free (\Sigma^m V) )\right). $$
By Corollary~\ref{corollary: deloop is xi-complete}, we can drop the first $\kk$-completion, and so
\begin{align*}
\oblv \circ \suspfreexi(V) &\simeq \oblv \left(\colim_m (\iota_{\xi} \Omega^m L_\xi \free (\Sigma^m V) )\right)\\
& \simeq \colim_m \left(\oblv\circ \Omega^m \iota_{\xi} L_{\xi} \free(\Sigma^m V)\right) \\
& \simeq\colim_m \left(\Sigma^{-m} \oblv \circ L_{\xi} \free(\Sigma^m V)\right) \\
& \simeq\colim_{m} \bigg(\Sigma^{-m}\prod_{n\geq 1} \mathbb{L}L^r_n(\Sigma^m V)\bigg),
\end{align*}
where the last equivalence follows by Proposition~\ref{proposition: free xi-complete}. Finally, Corollary~\ref{corollary: P1 commutes with inf product} implies the proposition.
\end{proof}

Consider the chain complex functor $\widetilde{C}_*\colon \sL \to \Mod_{\kk}$, see Definition~\ref{definition: chain complex}. By construction, the functor $\widetilde{C}_*$ is left adjoint. Therefore, by~\cite[Corollary~1.4.4.5]{HigherAlgebra}, there is a unique functor
\begin{equation}\label{equation: chain complex, stable}
\widetilde{C}_* \colon \Sp(\sL) \to \Mod_{\kk} 
\end{equation}
such that $\widetilde{C}_*(L) \simeq \widetilde{C}_*(\Sigma^\infty_{\sL} L)$ for $L\in\sL$. Here we abuse notation and we denote the chain complex functor from the category $\sL$ and the functor~\eqref{equation: chain complex, stable} from its stabilization $\Sp(\sL)$ by the same symbol. Finally, we recall the construction of the functor~\eqref{equation: chain complex, stable}. Let $X\in \Sp(\sL)$. Then 
$$\widetilde{C}_*(X) = \colim_{m} \Sigma^{-m} \widetilde{C}_*(\deloop \Sigma^{m} X) \in \Mod_{\kk}. $$

\begin{dfn}\label{definition: cochain complex, stable}
The \emph{cochain complex} $\widetilde{C}^*(X) \in \Mod_\kk$ of $X \in \Sp(\sL)$ is given by
$$\widetilde{C}^*(X) = \Hom_{\kk}(\widetilde{C}_*(X),\kk). $$
The \emph{$q$-th homology group} $\widetilde{H}_q(X)$ of $X \in \Sp(\sL)$ is given by
$$ \widetilde{H}_q(X) = \pi_q(\widetilde{C}_*(X)), \;\; q\in \Z.$$
Dually, the \emph{$q$-th cohomology group} $\widetilde{H}^q(X)$ of $X \in \Sp(\sL)$ is given by the rule
$$ \widetilde{H}^q(X) = \pi_{-q}(\widetilde{C}^*(X)), \;\; q\in \Z.$$
\end{dfn}

\begin{rmk}\label{remark: homology of a free algebra}
By Definition~\ref{definition: chain complex} of the chain complex functor, we have equivalences
$$\widetilde{C}_*\circ \free\simeq \widetilde{C}_*\circ L_\xi \free\simeq \Sigma(-)\colon \Mod_{\kk}^{\geq 0} \to \Mod_{\kk}^{\geq 0}, $$
$$\widetilde{C}_*\circ \suspfree\simeq \widetilde{C}_*\circ \suspfreexi\simeq \Sigma(-)\colon \Mod_{\kk} \to \Mod_{\kk}. $$
\end{rmk}

\begin{rmk}\label{remark: stable hurewicz homomorphism}
Let $X\in \Sp(\sL)$. Then, as in Remark~\ref{remark: hurewuicz homomorphism}, we obtain the \emph{stable Hurewicz homomorphism}
\begin{equation}\label{equation: stable hurewicz homomorphism}
h\colon \pi_q(X) \to \widetilde{H}_{q+1}(X), \;\; q\in \Z. 
\end{equation}
Moreover, if $X$ is $n$-connected for some $n\in\Z$, then the stable Hurewicz homomorphism induces the isomorphism
$$\pi_{n+1}(X)/\xi \xrightarrow{\cong} \widetilde{H}_{n+2}(X), $$
see Remark~\ref{remark: hurewicz theorem}. This implies that if $X\in \Sp(\sL)$ is a bounded below object (i.e. $\pi_i(X)=0$ for $i\ll 0$), then its $\kk$-completion $L^{st}_{\xi}X$ is trivial if and only if its homology groups $\widetilde{H}_q(X)$, $q\in \Z$ vanish.
\end{rmk}

\begin{rmk}\label{remark: chain complex is conservative}
We write $\Sp(\sL)^{+} \subset \Sp(\sL)$ for the full subcategory of $\Sp(\sL)$ spanned by bounded below spectrum objects. Set $\Sp(\sLxi)^{+}= \Sp(\sLxi) \cap \Sp(\sL)^{+}$. Then, by Remark~\ref{remark: stable hurewicz homomorphism}, the restricted functor
$$\widetilde{C}_* \colon \Sp(\sLxi)^{+} \hookrightarrow \Sp(\sL) \xrightarrow{\widetilde{C}_*} \Mod_{\kk}$$
is conservative.
\end{rmk}

\subsection{Monoidal structures}\label{section: monoidal structures}
The Cartesian product endows the category $\sL$ with the symmetric monoidal structure $\sL^{\times}$. However, since $\sL$ is a pointed category, the symmetric monoidal structure $\sL^{\times}$ is \emph{not} closed, and moreover, it is not even clear that the Cartesian product commutes with pushouts, see e.g.~\cite[Remark~4.5.13]{koszul}. As a result, it is difficult to work with the symmetric monoidal category $\sL^{\times}$. Nevertheless, we will demonstrate in this section that the Cartesian product behaves nicely when restricted to the full subcategory $\sL_{\xi}\subset \sL$. In order to do that we need to introduce the auxiliary $\infty$-category $\sCA_0$ of reduced simplicial truncated coalgebras.

\begin{dfn}\label{definition: truncated coalgebra}
Let $C=(C,\eta\colon \kk \to C)\in \CoAlg^{\mathrm{aug}}$ be a \emph{finite-dimensional} coaugmented cocommutative coalgebra and let $C^*=(C^*,\eta^*\colon C^* \to \kk)$ be its dual augmented algebra. The coalgebra $C$ is called \emph{truncated} if, for every $x\in \ker(\eta^*)$, we have $x^p=0$. Moreover, an infinite-dimensional coalgebra $C\in \CoAlg^{\mathrm{aug}}$ is called \emph{truncated} if $C$ is an union of finite-dimensional truncated sub-coalgebras. 
\end{dfn}

We write $\trcoalg$ for the full ($1$-)subcategory of $\CoAlg^{\mathrm{aug}}$ spanned by truncated coalgebras. By~\cite[Proposition~2.2.14]{koszul}, the forgetful functor
$$\oblv \colon \trcoalg \hookrightarrow \CoAlg^{\aug} \to \Vect_{\kk}, \;\; (C,\eta) \mapsto \coker(\eta) $$
admits a right adjoint $\Sym^{\mathrm{tr}}\colon \Vect_{\kk} \to \trcoalg$.

\begin{prop}\label{proposition:category property of trcoalg} The $1$-category $\trcoalg$ is comonadic over $\Vect_{\kk}$ via the adjunction $\oblv \dashv \Sym^{\mathrm{tr}}$. The $1$-category $\trcoalg$ is complete and cocomplete, (co)-limits can be computed in the $1$-category $\CoAlg^{\mathrm{aug}}$, and the forgetful functor $\oblv$ creates colimits. Moreover, $\trcoalg$ is presentable. 
\end{prop}

\begin{proof}
See~\cite[Proposition~2.2.10 and Proposition~2.2.15]{koszul}.
\end{proof}

Let $C,D\in \trcoalg$ be truncated coalgebras. Then the Cartesian product $C\times D$ is the tensor product $C\otimes_{\kk} D$ equipped with the obvious comultiplication and the coCartesian coproduct $C\sqcup D$ is the wedge sum of augmented vector spaces $(C\oplus D)/\kk$ also equipped with the obvious comultiplication.

\begin{dfn}\label{definition: smash product coalgebras}
Let $C,D\in \trcoalg$ be truncated coalgebras. We define the \emph{smash product} $C\wedge D\in \trcoalg$ as the following coequalizer
$$
\begin{tikzcd}
C\wedge D = \mathrm{coeq}(C \sqcup D \arrow[shift left=.75ex]{r}{\mathrm{can}}
  \arrow[shift right=.75ex,swap]{r}
&
C \times D),
\end{tikzcd}
$$
where the upper arrow is the canonical map from coproduct to product and the lower arrow is the composition the counit map and the coaugmentation map. In particular, $\oblv(C\wedge D)\cong \oblv(C)\otimes \oblv(D)$
\end{dfn}

\begin{prop}\label{proposition: smash product, coalgebras}
The smash product $-\wedge -$ makes $\trcoalg$ into a closed (non-unital) symmetric monoidal $1$-category such that the forgetful functor 
$$\oblv\colon (\trcoalg, \wedge) \to (\Vect_{\kk},\otimes) $$
is a strong symmetric monoidal functor.
\end{prop}

\begin{proof}
The proposition follows directly from the definitions. By Proposition~\ref{proposition:category property of trcoalg}, the forgetful functor $\oblv$ creates colimits. Therefore the smash product $-\wedge -$ commutes with colimits separately at each variable, and so the monoidal ($1$-)category $(\trcoalg, \wedge)$ is closed.
\end{proof}

We write $\sotrcoalg$ for the ($1$-)category of reduced (i.e. $C_\bullet \in \sotrcoalg \Leftrightarrow C_0 \cong \kk$) simplicial truncated coalgebras. Recall that the ($1$-)category $\sotrcoalg$ is endowed with a well-behaved model structure.

\begin{thm}[{{\cite[Theorem~3.2.10]{koszul}}}]\label{theorem:modelsotrcoalg} There exists a simplicial combinatorial left proper model structure on $\sotrcoalg$ such that $f\colon C_\bullet \to D_\bullet$ is

\begin{itemize}
\item a \emph{weak equivalence} if and only if $\pi_*(f)$ is an isomorphism;
\item a \emph{cofibration} if and only if $f\colon C_\bullet \to D_\bullet$ is degreewise injective;
\item a \emph{fibration} if and only if $f$ has the right extension property with respect to all acyclic cofibrations. \qed
\end{itemize}
\end{thm}

We write $\sCA_0$ for the \emph{underlying $\infty$-category} of the simplicial model ($1$-)category $\sotrcoalg$, see~\cite[Section~A.2]{HigherAlgebra}. We note that the $1$-categorical (Quillen) adjunction
\begin{equation*}
\begin{tikzcd}
\oblv: \sotrcoalg \arrow[shift left=.6ex]{r}
&\soVect_{\kk} :\Sym^{\mathrm{tr}} \arrow[shift left=.6ex,swap]{l}
\end{tikzcd}
\end{equation*}
induces the adjoint pair
\begin{equation}\label{equation: truncated coalgebras, infty}
\begin{tikzcd}
\oblv: \sCA_0 \arrow[shift left=.6ex]{r}
&\Mod^{\geq 1}_{\kk} :\Sym^{\mathrm{tr}} \arrow[shift left=.6ex,swap]{l}
\end{tikzcd}
\end{equation}
between underlying $\infty$-categories. Theorem~\ref{theorem:modelsotrcoalg}  yields an $\infty$-categorical analog of Proposition~\ref{proposition:category property of trcoalg}.

\begin{cor}\label{corollary: infty-category property of trcoalg} The category $\sCA_0$ is comonadic over $\Mod^{\geq 1}_{\kk}$ via the adjunction $\oblv \dashv \Sym^{\mathrm{tr}}$. The category $\sCA_0$ is complete and cocomplete and the forgetful functor $\oblv$ creates colimits. Moreover, $\sCA_0$ is presentable. \qed 
\end{cor}

We refer the reader to~\cite[Chapter~4]{Hovey99} for the definition of a (non-unital) \emph{monoidal model $1$-category} and to~\cite[Definition~4.3.11]{DAGIII} for the definition of a \emph{simplicial symmetric (non-unital) monoidal model $1$-category}. The next proposition follows immediately from the definitions.

\begin{prop}\label{proposition: smash, coalgebras, monoidal model category}
The (degreewise) smash product $-\wedge -$ makes the model $1$-category $\sotrcoalg$ from Theorem~\ref{theorem:modelsotrcoalg} into a simplicial (non-unital) symmetric monoidal model $1$-category. Moreover, the forgetful functor
$$\oblv\colon (\sotrcoalg, \wedge) \to (\soVect_{\kk},\otimes) $$
turns into a strong symmetric monoidal left Quillen functor. \qed
\end{prop}

By applying~\cite[Proposition~4.3.13]{DAGIII} to Proposition~\ref{proposition: smash, coalgebras, monoidal model category}, we observe that the smash product $-\wedge -$ on the model ($1$-)category $\sotrcoalg$ determines the (non-unital) symmetric monoidal structure on the underlying $\infty$-category $\sCA_0$. We will denote the induced symmetric monoidal $\infty$-category by $\sCA_0^{\wedge}$. We summarize some properties of $\sCA_0^{\wedge}$ in the next proposition.

\begin{prop}\label{proposition: smash, coalgebras, infty monoidal}
The symmetric monoidal category $\sCA_0^{\wedge}$ is presentably symmetric monoidal~\cite[Definition~2.1]{NikolausSagave17}, i.e. the smash product
$$- \wedge - \colon \sCA_0 \times \sCA_0 \to \sCA_0 $$
preserves colimits separately in each variable. Moreover, the forgetful functor
$$\oblv \colon \sCA_0^{\wedge} \to \Mod_{\kk}^{\geq 1, \otimes} $$
is strong symmetric monoidal and there are cofiber sequences
$$C\sqcup D \xrightarrow{\mathrm{can}} C\times D \to C\wedge D  $$
for any pair $C,D\in \sCA_0$. \qed
\end{prop}

\begin{rmk}\label{remark: cartesian and smash, coalgebras}
Let $\sCA_0^{\times}$ be the Cartesian symmetric monoidal category. Then the identity functor $\id\colon \sCA_0^{\times} \to \sCA_0^{\wedge}$ (resp. $\id\colon \sCA_0^{\wedge} \to \sCA_0^{\times}$) is oplax (resp. lax) symmetric monoidal.
\end{rmk}

We recall that the category $\sLxi$ of $\kk$-complete simplicial restricted Lie algebras is equivalent to the category $\sCA_0$ of reduced simplicial truncated coalgebras.

\begin{thm}\label{theorem: coalgebras and lie algebras}
There is an equivalence of $\infty$-categories
\begin{equation}\label{equation: equivalence of infty-categories}
PG: \sCA_0 \simeq \sLxi : \barW U^r
\end{equation}
such that the chain complex functor $\widetilde{C}_* \colon \sLxi \to \Mod_{\kk}^{\geq 1}$ is equivalent to the composite $\oblv \circ \barW U^r$.
\end{thm}

\begin{proof}
See~\cite[Theorem~3.2.26 and Proposition~3.2.18]{koszul}.
\end{proof}

By applying Proposition~\ref{proposition: smash, coalgebras, infty monoidal} to the equivalence~\eqref{equation: equivalence of infty-categories}, we produce a symmetric monoidal structure on the $\infty$-category $\sLxi$.

\begin{prop}\label{proposition: smash, xi-complete lie algebras}
There exists a (non-unital) presentably symmetric monoidal structure $\sLxi^{\otimes}$ on the category $\sLxi$ such that the chain complex functor
$$\widetilde{C}_* \colon \sLxi^{\otimes} \to \Mod_{\kk}^{\geq 1, \otimes} $$
is strong symmetric monoidal and there are cofiber sequences
\begin{equation}\label{equation: smash product is smash}
L\sqcup L' \xrightarrow{\mathrm{can}} L\times L' \to L\otimes L'  
\end{equation}
for any pair $L,L'\in \sLxi$. \qed
\end{prop}

\begin{rmk}\label{remark: cartesian and smash, lie algebras}
Let $\sLxi^{\times}$ be the Cartesian symmetric monoidal category. Then the identity functor $\id\colon \sLxi^{\times} \to \sLxi^{\otimes}$ (resp. $\id\colon \sLxi^{\otimes} \to \sLxi^{\times}$) is oplax (resp. lax) symmetric monoidal.
\end{rmk}

\begin{rmk}\label{remark: smash, all lie algebras}
We wish to have a similar smash product on the entire category $\sL$ of \emph{all} simplicial restricted Lie algebras. However, since we are not aware whether or not Cartesian products in $\sL$ commute with pushouts (see~\cite[Remark~4.5.13]{koszul}), we can not say that the smash product defined on the category $\sL$ by the formula~\eqref{equation: smash product is smash} is even \emph{associative}.
\end{rmk}

\begin{thm}\label{theorem: smash, stabilization}
There is a unique (non-unital) symmetric monoidal structure $\Sp(\sLxi)^{\otimes}$ on the stabilization $\Sp(\sLxi)$ which makes the suspension functor
$$\Sigma^{\infty}_{\sLxi}\colon \sLxi^{\otimes} \to \Sp(\sLxi)^{\otimes}$$
into (strong) symmetric monoidal. Furthermore, this symmetric monoidal structure also makes the chain complex functor
$$\widetilde{C}_*\colon \Sp(\sLxi)^{\otimes} \to \Mod_{\kk}^{\otimes}$$
into a (strong) symmetric monoidal functor.
\end{thm}

\begin{proof}
The first part follows directly by applying~\cite[Theorem~5-1]{GGN15} to the symmetric monoidal category $\sL_{\xi}^{\otimes}$. The second part follows by applying~\cite[Proposition~5-4]{GGN15} to the symmetric monoidal functor
$\widetilde{C}_*\colon \sLxi^{\otimes} \to \Mod_{\kk}^{\geq 1,\otimes}$.
\end{proof}

Consider a functor $\triv\colon \Vect_{\kk} \to \trcoalg$ which sends a vector space $W$ to the trivial coalgebra $\triv(W)=W\oplus \kk$. The functor $\triv$ admits a right adjoint $P\colon \trcoalg \to \Vect_{\kk}$ which sends a coalgebra $C$ to the subset $P(C)\subset C$ of primitive elements. We extend the adjoint pair $\triv \dashv P$ degreewise to the Quillen adjunction
\begin{equation*}
\begin{tikzcd}
\triv: \soVect_{\kk} \arrow[shift left=.6ex]{r}
&\sotrcoalg :P \arrow[shift left=.6ex,swap]{l}
\end{tikzcd}
\end{equation*}
between the model ($1$-)categories of reduced simplicial objects. We observe that the (degreewise) smash product $-\wedge -$ makes the trivial coalgebra functor
$$\triv \colon (\soVect_{\kk},\otimes) \to (\sotrcoalg,\wedge)$$
into a strong symmetric monoidal left Quillen functor. Taking the underlying $\infty$-categories, we obtain a symmetric monoidal functor
$$\triv\colon \Mod_{\kk}^{\geq 1,\otimes} \to \sCA_0^{\wedge}. $$ 
Finally, by~\cite[Example~3.2.19]{koszul}, the left adjoint functor $L_\xi\free\colon \Mod_{\kk}^{\geq 0} \to \sLxi$ from~\eqref{equation: quillen adjunction, xifree oblv} is equivalent to the composite
$$L_\xi\free(-) \simeq PG\circ \triv(\Sigma(-))\colon \Mod_{\kk}^{\geq 0} \xrightarrow{\triv(\Sigma (-))} \sCA_0 \xrightarrow{PG} \sLxi,  $$
where $PG$ is the equivalence~\eqref{equation: equivalence of infty-categories}. As a consequence, we obtain the following proposition.

\begin{prop}\label{proposition: shifted free is symmetric monoidal}
The shifted $\kk$-complete free simplicial restricted Lie algebra functor
$$L_\xi \free(\Sigma^{-1} (-))\colon \Mod_{\kk}^{\geq 1,\otimes} \to \sLxi^{\otimes} $$
is (strong) symmetric monoidal. \qed
\end{prop}

Since the suspension functor $\Sigma^{\infty}\colon \sLxi \to \Sp(\sLxi)$ is symmetric monoidal, we observe the stable variant of Proposition~\ref{proposition: shifted free is symmetric monoidal}.

\begin{cor}\label{proposition: shifted stable free is symmetric monoidal}
The shifted left adjoint functor~\eqref{equation: free stable xi-complete}:
$$\suspfreexi \circ \Sigma^{-1}\colon \Mod_{\kk}^{\otimes} \to \Sp(\sLxi)^{\otimes} $$
is (strong) symmetric monoidal.
\end{cor}

\begin{proof}
We apply~\cite[Proposition~5-4]{GGN15} to the colimit-preserving symmetric monoidal composite
$$\Mod_{\kk}^{\geq 1,\otimes} \xrightarrow{L_\xi \free(\Sigma^{-1})} \sLxi^{\otimes} \xrightarrow{\Sigma^{\infty}} \Sp(\sLxi)^{\otimes}.$$
\end{proof}

Let $\calC$ be a pointed presentable $\infty$-category. The Cartesian product on $\calC$ gives it the structure of a symmetric monoidal $\infty$-category, which is encoded as a (non-unital) $\infty$-operad $$\calC^{\times}_{\mathrm{nu}} \to N\mathrm{Surj}$$ whose structure map is a coCartesian fibration, see~\cite[Section~2.1]{HigherAlgebra}. For any (non-unital) differentiable $\infty$-operad $\calC^{\otimes}_{\mathrm{nu}} \to N\mathrm{Surj}$, which encodes a symmetric monoidal structure on $\calC$, J.~Lurie defines the stabilization $\Sp(\calC^{\otimes}) \to N\mathrm{Surj}$ of the $\infty$-operad $\calC^{\otimes}$ and proves its existence, see~\cite[Section~6.2.4]{HigherAlgebra}. We note that the resulting object $\Sp(\calC^{\otimes})$ is only a stable (non-unital) $\infty$-operad (see~\cite[Definition~6.2.4.10]{HigherAlgebra} and~\cite[Definition~2.11]{Heuts21}), but not necessary a (non-unital) symmetric monoidal category. 

\begin{prop}\label{proposition: monoidal structures coincide}
The lax symmetric monoidal functor
$$\Sp(\sLxi)^{\otimes} \xrightarrow{\Omega^{\infty}_0} \sL_{\xi,0}^{\otimes} \xrightarrow{\id} \sL_{\xi,0}^{\times} $$
exhibits $\Sp(\sLxi)^{\otimes}$ as the stabilization of the Cartesian monoidal structure $\sL_{\xi,0}^{\times}$.
Here $\Omega^{\infty}_0$ is the connective cover of $\Omega^{\infty}$, i.e. the right adjoint to the strong symmetric monoidal functor
$$\sL_{\xi,0}^{\otimes} \hookrightarrow \sLxi^{\otimes} \xrightarrow{\susp} \Sp(\sLxi)^{\otimes}.  $$
In particular, there are equivalences of stable $\infty$-operads
$$\Sp(\sLxi^{\times}) \simeq \Sp(\sLxi^{\otimes}) \simeq \Sp(\sLxi)^{\otimes}.$$
\end{prop}

\begin{proof}
By Corollary~\ref{corollary: xi-complete is differentiable}, the $\infty$-operad $\sL^{\times}_{\xi,0}$ is differentiable, see~\cite[Definition~6.2.4.11]{HigherAlgebra}. Therefore, the stabilization $\Sp(\sL^{\times}_{\xi,0})$ exists and unique, see~\cite[Propositions~6.2.4.14 and~6.2.4.15]{HigherAlgebra}. By~\cite[Definition~6.2.4.10]{HigherAlgebra}, the $\infty$-operad $\Sp(\sL^{\times}_{\xi,0})$ is corepresentable and it is defined by $k$-th symmetric tensor products $$\otimes^k\colon \Sp(\sL_{\xi})^{\times k}\simeq \Sp(\sL_{\xi,0})^{\times k} \to \Sp(\sL_{\xi,0}) \simeq \Sp(\sL_{\xi}), \;\; k\geq 1,$$ see~\cite[Remark~6.2.4.4]{HigherAlgebra}, which are (multi-)derivatives of the functor
\begin{equation}\label{equation: monoidal coincide, eq1}
(X_1,\ldots,X_k)\in \Sp(\sL_{\xi})^{\times k} \mapsto \Omega^{\infty}_0(X_1)\times \ldots \times \Omega^{\infty}_0(X_k) \in \sL_{\xi,0}, 
\end{equation}
see~\cite[Definition~6.2.1.1]{HigherAlgebra}. Note that the derivative of the functor~\eqref{equation: monoidal coincide, eq1} coincides with the derivative of its composite with $\susp$. In other words, $\otimes^k$, $k\geq 1$ is the derivative of the functor
\begin{equation}\label{equation: monoidal coincide, eq2}
(X_1,\ldots,X_k)\in \Sp(\sL_{\xi})^{\times k} \mapsto \susp\left(\Omega^{\infty}_0(X_1)\times \ldots \times \Omega^{\infty}_0(X_k)\right) \in \Sp(\sL_{\xi}).
\end{equation}
By Proposition~\ref{proposition: smash, xi-complete lie algebras} and Theorem~\ref{theorem: smash, stabilization}, the reduction of the functor~\eqref{equation: monoidal coincide, eq2} is the functor
\begin{align*}
(X_1,\ldots,X_k)\in \Sp(\sL_{\xi})^{\times k} &\mapsto \susp\left(\Omega^{\infty}_0(X_1)\otimes \ldots \otimes \Omega^{\infty}_0(X_k)\right) \\
&\simeq \susp\Omega^{\infty}_0(X_1)\otimes \ldots \otimes \susp \Omega^{\infty}_0(X_k) \in \Sp(\sL_{\xi}).
\end{align*}
Since the monoidal product for the category $\Sp(\sLxi)^{\otimes}$ commutes with colimits, the multi-linearization of the last functor sends a tuple $(X_1,\ldots,X_k)\in \Sp(\sL_{\xi})^{\times k}$ to the following product of linearizations
\begin{equation*}
P_1(\susp\Omega^{\infty}_0)(X_1)\otimes \ldots \otimes P_1(\susp\Omega^{\infty}_0)(X_k) \in \Sp(\sL_{\xi}).
\end{equation*}
Finally, $P_1(\susp\Omega^{\infty}_0) \simeq P_1(\susp\Omega^{\infty}) \simeq \id_{\sLxi}$ which implies the assertion.
\end{proof}

\begin{rmk}\label{remark: stable infty-operad}
The abstract stabilization $\Sp(\sLxi^{\times})$ of the Cartesian symmetric monoidal category $\sLxi^{\times}$ is a symmetric monoidal category. At the moment of writing, we are not aware whether or not the stabilization $\Sp(\sL^{\times})$ of the 
Cartesian symmetric monoidal category $\sL^{\times}$ is still symmetric monoidal, cf. Remark~\ref{remark: smash, all lie algebras}.
\end{rmk}

\subsection{\texorpdfstring{$\Lambda$}{Lambda}-algebra}\label{section: lambda algebra}
Recall that the \emph{lambda algebra} $\Lambda=\Lambda_{*,*}$ is the associative bigraded algebra over $\F_p$ generated by the elements $\lambda_a, a\geq 1$ of bidegree $|\lambda_a|=(2a(p-1)-1,1)$ and $\mu_a, a\geq 0$ of bidegree $|\mu_a|=(2a(p-1),1)$ (resp. $\lambda_a, a\geq 0$ of bidegree $|\lambda_a|=(a,1)$ if $p=2$) subject to the following Adem-type relations:
\begin{enumerate}
\item If $p$ is odd, $b\geq pa$, and $\e\in\{0,1\}$, then
\begin{align}\label{equation: lambda, adem1, p is odd}
\lambda_a\nu^{\e}_{b} &= \sum_{i=0}^{a+b}(-1)^{i+a+\e}\binom{(p-1)(b-i)-\e}{i-pa}\nu^{\e}_{a+b-i}\lambda_i \\
&+(1-\e)\sum_{i= 0}^{a+b}(-1)^{i+a+1}\binom{(p-1)(b-i)-1}{i-pa}\lambda_{a+b-i}\mu_i. \nonumber 
\end{align}
\item If $p$ is odd, $b>pa$, and $\e\in\{0,1\}$, then
\begin{equation}\label{equation: lambda, adem2, p is odd}
\mu_a\nu^{\e}_{b} = \sum_{i=1}^{a+b}(-1)^{i+a}\binom{(p-1)(b-i)-1}{i-pa-1}\mu_{a+b-i}\nu^{\e}_i.
\end{equation}
\item If $p=2$ and $b>2a$, then 
\begin{equation}\label{equation: lambda, adem3, p=2}
\lambda_a\lambda_{b} = \sum_{i=1}^{a+b}\binom{b-i-1}{i-2a-1}\lambda_{a+b-i}\lambda_i. 
\end{equation}
\end{enumerate}
Here we set $\nu_a^{0}=\mu_a, a\geq 0$ and $\nu_a^1=\lambda_a,a>0$. Notice that we use the definition of the lambda algebra from~\cite[Definition~7.1]{Wellington82}, but not from the original paper~\cite{6authors}. 

A (possibly void) sequence $I=(i_1,\e_1,\ldots, i_s,\e_s)$ (resp. $I=(i_1,\ldots, i_s)$), $i_t\geq \e_t$, $\e_t\in\{0,1\}$, $1 \leq t\leq s$ (resp. $i_t\geq 0$, $1\leq t\leq s$) determines the element $\nu_I = \nu_{i_1}^{\e_1}\cdot \ldots \cdot \nu_{i_s}^{\e_s}$ (resp. $\lambda_I=\lambda_{i_1}\cdot \ldots \cdot \lambda_{i_s}$) of $\Lambda$.

\begin{dfn}\label{definition: admissible monomials}
A sequence $I=(i_1,\e_1,\ldots,i_s,\e_s)$ (resp. $I=(i_1,\ldots,i_s)$) is called \emph{admissible} if $pi_j-\e_j \geq i_{j+1}$, $1 \leq j\leq s-1$ (resp. $2i_j \geq i_{j+1}$, $1\leq j \leq s-1$). The monomial $\nu_I \in \Lambda$ (resp. $\lambda_I \in \Lambda$) is called \emph{admissible} if the sequence $I$ is admissible.
\end{dfn}

Let $L\in\sL$ be a simplicial restricted Lie algebra. Then a homotopy class $x\in \pi_q(L)$ is represented by a map $x\colon \Sigma^q \kk \to \oblv(L) \in \Mod_{\kk}^{\geq 0}$, and we will write $$\overline{x}\colon \free(\Sigma^{q}\kk) \to L$$ for its adjoint in the category $\sL$.  If $x\in \pi_q(L)$ and $\alpha\in \pi_{q'}(\free(\Sigma^q \kk))$, then we write $x\alpha \in \pi_{q'}(L)$ for the map $x\alpha \colon \Sigma^{q'}\kk \to \oblv(L)$ whose adjoint $\overline{x\alpha}$ is equivalent to the composite 
$$\overline{x\alpha}\simeq \overline{x}\circ \overline{\alpha} \colon \free(\Sigma^{q'}\kk) \to \free(\Sigma^q \kk) \to L.$$
Therefore any homotopy class $\alpha\in\pi_{q'}(\free (\Sigma^q \kk))$ defines a natural homotopy operation $x\mapsto x\alpha$, $x\in \pi_q(L)$.

If $L$ is a simplicial restricted Lie algebra, then the homotopy groups $\pi_*(L)$ inherit a graded restricted Lie algebra structure, see e.g.~\cite[Section~6]{MilnorMoore65} for the definition. We write $x^{[p]}\in \pi_{pq}(L)$, $x\in \pi_q(L)$ for the induced $p$-operation. We note that $x^{[p]}=0$ for $x\in \pi_q(L)$ if $p$ and $q$ are both odd.

\begin{prop}\label{proposition: lambda, mu, adem}
For any $q\geq 0$ there are $\nu^{\e}_a \in \pi_{q+2a(p-1)-\e}(\free(\Sigma^q \kk))$, $a\geq \e$, $\e\in\{0,1\}$ (resp. $\lambda_a \in \pi_{q+a}(\free(\Sigma^q \kk))$ if $p=2$) such that $\nu^{\e}_a = 0$ if and only if $2a>q$ (resp. $\lambda_a=0$ if and only if $a>q$) and the induced natural homotopy operations $x\mapsto x\nu^{\e}_a$ (resp. $x\mapsto x\lambda_a$), $x\in \pi_q(L)$, $L\in\sL$
\begin{enumerate}
\item are semi-linear, i.e. $(\alpha x)\nu_a^{\e}=\alpha^p(x\nu_a^{\e})$, $\alpha\in\kk$ (resp. $(\alpha x)\lambda_a=\alpha^2(x\lambda_a)$, $\alpha\in\kk$);
\item are additive if $2a<q$ or $2a=q$ and $\e=1$ (resp. $a<q$);
\item are stable, i.e. $(\sigma x)\nu^{\e}_a = \sigma(x\nu^{\e}_a)$, where $\sigma\colon \pi_q(L) \to \pi_{q+1}(\Sigma L)$ is the suspension homomorphism;
\item satisfy Adem relations~\eqref{equation: lambda, adem1, p is odd} and~\eqref{equation: lambda, adem2, p is odd} (resp.~\eqref{equation: lambda, adem3, p=2}).
\item satisfy $[x,y\nu^{\e}_a]=0$ (resp. $[x,y\lambda_a]=0$), $x\in \pi_{q'}(L)$, $y\in \pi_q(L)$, and $2a<q$ or $2a=q$ and $\e=1$ (resp. $a<q$).
\end{enumerate}
Finally, the top operation $x\mapsto x\nu^0_{a}$, $x\in \pi_{2a}(L)$ (resp. $x\mapsto x\lambda_a$, $x\in \pi_a(L)$) coincides with $p$-operation $x\mapsto x^{[p]}$.
\end{prop}

\begin{proof}
See~\cite{6authors} and~\cite[Proposition~8.3]{BC70}. See also \cite[Propostion 13.1]{Wellington82} for a better exposition of the odd prime case.
\end{proof}

\begin{rmk}\label{remark: xi and nu_0 actions}
Let $L\in\sL$ be a simplicial restricted Lie algebra. Then the map $\xi\colon \pi_n(L) \to \pi_n(L)$, $n\geq 0$ from~\eqref{equation: xi-action} coincides (up to a multiplication on a non-zero constant) with the homotopy operation $x\mapsto x\mu_0=x\nu_0^0$ (resp. $x\mapsto x\lambda_0$), $x\in \pi_n(L)$; see Remark~\cite[Remark~6.4.18]{koszul}.
\end{rmk}

Let $X\in\Sp(\sL)$ be a spectrum object in the category $\sL$. Using the third part of Proposition~\ref{proposition: lambda, mu, adem}, we extend the unstable homotopy operations given there to the stable homotopy groups $\pi_*(X)$. In this way, the stable homotopy groups $\pi_*(X)$ naturally form a \emph{right (semi-linear) $\Lambda$-module}.

Recall that $\mathrm{char}(\kk)=p$. Let $V\in \Vect_{\kk}$ be a vector space over $\kk$. We define the \emph{$(-1)$-th Frobenius twist} $V^{(-1)}\in \Vect_{\kk}$ of $V$ as follows. As an abelian group, $V^{(-1)}=V$ and we endow it with a new $\kk$-action 
$$-\cdot-\colon \kk \times V^{(-1)} \to V^{(-1)} $$
given by $a\cdot v = a^p v$, $a\in \kk$, $v\in V^{(-1)}=V$. Since the field $\kk$ is perfect, there also exists the inverse operation; namely, we define the \emph{Frobenius twist} of $V$ as a unique $\kk$-vector space $V^{(1)}$ such that $(V^{(1)})^{(-1)}= V$.

We write $\Lambda_s, s\geq 0$ for the subspace $\Lambda_s =\Lambda_{*,s}\subset \Lambda$; in other words, $\Lambda_s, s\geq 0$ is the subspace of $\Lambda$ spanned by the monomials of length precisely $s$.

\begin{thm}\label{theorem: stable homotopy groups of a free object}
Let $V\in \Mod_{\kk}$ and let $\suspfree\colon \Mod_{\kk}\to \Sp(\sL)$ be the left adjoint~\eqref{equation: free lie, stable}. Then the natural map
$$\bigoplus_{s\geq 0}\pi_*(V)^{(s)}\otimes \Lambda_s \xrightarrow{\cong} \pi_*(\suspfree(V)) \cong \pi_*(\oblv \circ \suspfree(V)) $$
is an isomorphism, where $v \otimes \nu_{I} \mapsto v\nu_{I}$ (resp. $v \otimes \lambda_{I}\mapsto v\lambda_I$), $v\in \pi_*(V)$, $\nu_{I} \in \Lambda_s$ (resp. $\lambda_{I} \in \Lambda_s$). Moreover, under the decomposition of Proposition~\ref{proposition: stable is monadic}, we have
$$
\colim_{m}\pi_*\bigg(\Sigma^{-m}\mathbb{L}L^r_n(\tau_{\geq 0} \Sigma^m V)\bigg)\cong
\begin{cases}
\pi_*(V)^{(s)}\otimes \Lambda_s & \mbox{if $n=p^s$, $s\geq 0$,}\\
0 & \mbox{otherwise,}
\end{cases}
$$
where $\mathbb{L}L^r_n \colon \Mod_{\kk}^{\geq 0} \to \Mod_{\kk}^{\geq 0}$ are the nonabelian derived functors from~\eqref{equation: oblv free, fresse}.
\end{thm}

\begin{proof}
This theorem is a one of the main results of~\cite{6authors}. Namely, see Propositions~5.3 and~5.4' in ibid.
\end{proof}

\begin{rmk}\label{remark: second grading}
Let $V\in \Mod_{\kk}$. We write $\pi_{*,s}(\suspfree(V))$ for the colimit
$$\pi_{*,s}(\suspfree(V)) = \colim_{m}\pi_*\bigg(\Sigma^{-m}\mathbb{L}L^r_{p^s}(\tau_{\geq 0} \Sigma^m V)\bigg), \;\; s\geq 0. $$
Then, by Theorem~\ref{theorem: stable homotopy groups of a free object}, $\pi_*(\suspfree(V))=\bigoplus_{s\geq 0} \pi_{*,s}(\suspfree(V))$. In particular, the homotopy groups of a free object are \emph{bigraded}.
\end{rmk}

\begin{cor}\label{corollary: stable homotopy groups of a xi-complete free object}
Let $V\in \Mod_{\kk}^{+}$ be a bounded below chain complex. Then there is a natural isomorphism of right $\Lambda$-modules
$$\pi_*(\suspfreexi(V))= \pi_*(\oblv \circ \suspfreexi(V))\cong \prod_{s\geq 0}\pi_*(V)^{(s)}\otimes \Lambda_s \cong \prod_{s\geq 0}\pi_{*,s}(\suspfree(V)), $$
where $\suspfreexi\colon \Mod^{+}_{\kk}\to \Sp(\sLxi)$ is the left adjoint~\eqref{equation: free stable xi-complete}.
\end{cor}

\begin{proof}
Follows by Proposition~\ref{proposition: stable monad for xi-complete} and Theorem~\ref{theorem: stable homotopy groups of a free object}.
\end{proof}

\begin{rmk}\label{remark: second grading, xi-complete}
We write $\pi_{*,s}(\suspfreexi(V))\subset \pi_*(\suspfreexi(V)), s\geq 0$ for the subspace $$\pi_{*,s}(\suspfreexi(V))=\pi_{*,s}(\suspfree(V)) \subset \prod_{s\geq 0}\pi_{*,s}(\suspfree(V)) \cong \pi_*(\suspfreexi(V)).$$
Then, by Corollary~\ref{corollary: stable homotopy groups of a xi-complete free object}, $\pi_*(\suspfreexi(V))=\prod_{s\geq 0} \pi_{*,s}(\suspfreexi(V))$. In particular, the homotopy groups of a free object are both graded and \emph{pro-graded}, i.e. an inverse limit of graded vector spaces.
\end{rmk}

\begin{rmk}\label{remark: hurewicz, lambda}
Under the isomorphisms of Theorem~\ref{theorem: stable homotopy groups of a free object} and Corollary~\ref{corollary: stable homotopy groups of a xi-complete free object}, the Hurewicz homomorphisms (see Remark~\ref{remark: stable hurewicz homomorphism}
$$h\colon \pi_q(\suspfree(V)) \cong \bigoplus_{s\geq 0} \pi_{q,s}(\suspfree(V)) \to \widetilde{H}_{q+1}(\suspfree(V))\cong \pi_q(V), \;q\in\Z$$
$$h\colon \pi_q(\suspfreexi(V)) \cong \prod_{s\geq 0} \pi_{q,s}(\suspfreexi(V)) \to \widetilde{H}_{q+1}(\suspfreexi(V))\cong \pi_q(V), \; q\in\Z$$
are the projections on the first component. 
\end{rmk}

\begin{rmk}\label{remark: unstable homotopy groups}
Let $W\in \Vect^{\mathrm{gr}}_{\kk}$ be a graded vector space. Following~\cite[Definition~9.3]{Wellington82}, we denote by  $W\widehat{\otimes}\Lambda$ the subspace of $W\otimes \Lambda$ spanned by $w\otimes y\in W\otimes \Lambda$ such that $y\in \Lambda$ is an admissible monomial beginning with $\nu_a^\e$ (resp. $\lambda_a$) such that $2a-\e <|w|$ (resp. $a<|w|$).

Let $V\in \Mod^{\geq 0}_{\kk}$ be a simplicial vector space. Then by~\cite[Theorem~8.5]{BC70} and~\cite[Proposition~13.2]{Wellington82}, the natural map
$$\theta\colon \bigoplus_{s\geq 0}\free(\pi_*V)^{(s)}\widehat{\otimes} \Lambda_s \xrightarrow{\cong} \pi_*(\free(V)) = \pi_*(\oblv \circ \free(V)) $$
is an isomorphism (compatible with right $\Lambda$-actions), where $\free(\pi_*V)$ is the free \emph{graded} restricted Lie algebra generated by $\pi_*(V)$ and  $\theta(w \otimes \nu_{I})  = w\nu_{I}$ (resp. $\theta(w \otimes \lambda_{I})= v\lambda_I$), $w\in \free(\pi_*V)$, $\nu_{I} \in \Lambda_s$ (resp. $\lambda_{I} \in \Lambda_s$).
\end{rmk}

\begin{rmk}\label{remark: xi-complete unstable homotopy groups}
Let $V\in \Mod_{\kk}^{\geq 0}$ be a \emph{connected} simplicial vector space, $\pi_0(V)=0$. Then by combining Theorem~\ref{theorem: F-completion, homotopy groups} with Proposition~\ref{proposition: free xi-complete} and Remarks~\ref{remark: xi and nu_0 actions}, \ref{remark: unstable homotopy groups}, we obtain a natural isomorphism (compatible with right $\Lambda$-actions)
$$\theta\colon \prod_{s\geq 0}\free(\pi_*V)^{(s)}\widehat{\otimes} \Lambda_s \xrightarrow{\cong} \pi_*(L_\xi\free(V)) = \pi_*(\oblv \circ L_\xi\free(V)), $$
where $L_\xi \free(V)$ is the $\kk$-complete free simplicial restricted Lie algebra.
\end{rmk}

\begin{rmk}\label{remark: omitting frobenius twist}
Since the field $\kk$ is perfect, the vector spaces $V^{(s)}$, $s\in\Z$ are \emph{non-canonically} isomorphic to $V$. Therefore, in the cases when we are not interested in functoriality, we will omit Frobenius twists; e.g. we will write $\pi_*(\suspfree(V)) \cong \pi_*(V)\otimes \Lambda$ and $\pi_{*,s}(\suspfree(V))\cong \pi_*(V)\otimes \Lambda_s$ instead of $\pi_*(\suspfree(V)) \cong \bigoplus_{s\geq 0}\pi_*(V)^{(s)}\otimes \Lambda_s$ and $\pi_{*,s}(\suspfree(V))\cong \pi_*(V)^{(s)}\otimes \Lambda_s$ respectively.
\end{rmk}

\subsection{Dyer--Lashof--Lie operations}\label{section: H-Lie-algebras}
Let us denote by $\FB$ the $1$-category of finite sets (including the empty set) and bijections. We define the \emph{category of symmetric sequences} $\sseq(\calC)$ in an $\infty$-category $\calC$ as follows $$\sseq(\calC)=\Fun(\FB,\calC).$$ We note that the category $\sseq(\calC)$ is equivalent to the product category $\prod\limits_{n\geq 0}\calC^{B\Sigma_n}$. If $\calC=\Mod_{\kk}$ is the category of chain complexes, then we shorten $\sseq(\Mod_{\kk})$ as $\sseq$.

Suppose that the category $\calC$ is presentably symmetric monoidal. Then we equip the category of symmetric sequences $\sseq(\calC)$ with the (non-symmetric) monoidal structure $\sseq(\calC)^{\circ}$ given by the \emph{composition product}, see e.g.~\cite[Section~3.1.1]{GF12} and~\cite[Section~4.1.2]{Brantner_Thesis}. Recall that an \emph{operad} is an (associative) algebra in $\sseq(\calC)^{\circ}$ and a \emph{cooperad} is a (coassociative) coalgebra in $\sseq(\calC)^{\circ}$. In this paper, we will be mainly interested in three examples of operads.

\begin{dfn}\label{definition: lie operad}
The \emph{Lie operad} $\bfLie_*\in \sseq(\Vect_{\kk})$ is the operad in the category of vector spaces $\Vect_{\kk}$ whose algebras are Lie algebras.
\end{dfn}

\begin{rmk}\label{remark: lie algebras in char 2}
In characteristic $2$, a Lie bracket is just assumed to be antisymmetric. We do not assume necessarily that $[x,x]=0$, see~\cite[Section~1.1.13]{Fresse00}.
\end{rmk}

There is a symmetric monoidal fully faithful embedding $\Vect_{\kk} \subset \Mod_{\kk}$, so we also consider the Lie operad $\bfLie_*$ as an operad in $\sseq=\sseq(\Mod_{\kk})$.

\begin{dfn}\label{definition: linear spectral lie operad}
The \emph{$\F$-linear spectral Lie operad} $\bfLie^{\mathrm{sp}}_*\in \sseq(\Mod_\kk)$ is the Koszul dual operad (in the sense of~\cite[Section~5.2.2]{HigherAlgebra}) for the cocommutative cooperad $\bfComm_* \in \sseq$.
\end{dfn}

The \emph{shifted $\F$-linear spectral Lie operad} $s\bfLie^{\mathrm{sp}}_* \in \sseq(\Mod_{\kk})$ is defined as a symmetric sequence by the formula $$s\bfLie^{\mathrm{sp}}_k = \Sigma^{k-1}\bfLie^{\mathrm{sp}}_k \otimes \mathbf{sgn}^{\otimes k} \in \Mod_{\kk}^{B\Sigma_k}, \;\; k\geq 0,$$
where $\mathbf{sgn}\in \Vect^{B\Sigma_k}_{\kk}$ is the sign representation of the permutation group $\Sigma_k$. The operad multiplication maps are suspensions of those for the $\kk$-linear spectral Lie operad.

\begin{rmk}\label{remark: shifted linear spectral lie operad}
If $L$ is an algebra (in $\Mod_{\kk}$) over the shifted $\F$-linear spectral Lie operad $s\bfLie^{\mathrm{sp}}_*$, then $\Sigma L$ is canonically an algebra over $\bfLie^{\mathrm{sp}}_*$. Vice versa, if $L$ is a $\bfLie^{\mathrm{sp}}_*$-algebra, then $\Sigma^{-1}L$ is a $s\bfLie^{\mathrm{sp}}_*$-algebra.
\end{rmk}

\begin{prop}\label{proposition: homology of spectral lie operad}
There are isomorphisms 
$$\pi_q(s\bfLie^{\mathrm{sp}}_k)\cong
\begin{cases}
\bfLie_k \in \Vect_{\kk}^{B\Sigma_k}& \mbox{if $q=0$, $k\geq 0$,}\\
0 & \mbox{otherwise.}
\end{cases}
$$
More precisely, there is an isomorphism of operads 
$$\pi_0(s\bfLie^{\mathrm{sp}}_*)\cong \bfLie_* \in \sseq(\Vect_\kk) $$
and an operadic equivalence 
$$s\bfLie^{\mathrm{sp}}_* \simeq \bfLie_* \in \sseq(\Mod_{\kk}). $$
\end{prop}

\begin{proof}
See~\cite[Section~6]{Fresse04} for the first part. The second part follows by taking a zig-zag of (co)connective covers.
\end{proof}

\begin{rmk}\label{remark: Lie identifications}
By Proposition~\ref{proposition: homology of spectral lie operad}, the category $s\bfLie^{\mathrm{sp}}_*(\Mod_{\kk})$ is canonically equivalent to the category $\bfLie_*(\Mod_\kk)$ of \emph{(homotopy) Lie algebras} in the category $\Mod_{\kk}$. In the sequel, we will identify the operads $s\bfLie^{\mathrm{sp}}_*$ and $\bfLie_*$ as well as the category of algebras over them. Moreover, we will use the notation $\bfLie_*$ instead of $s\bfLie^{\mathrm{sp}}_*$. However, we will keep using $\bfLie_*^{\mathrm{sp}}$ instead of $s^{-1}\bfLie_*$ to highlight its relevance to the spectral Lie operad and the Koszul duality.
\end{rmk}

It is well-known that the chain complex $\bfLie^{\mathrm{sp}}_k\in \Mod^{B\Sigma_k}_{\kk}$, $k\geq 2$ can be identified with the reduced cochain complex of the so-called \emph{partition complex}, see e.g.~\cite{Fresse04}. We briefly recall the construction. Let $\mathcal{T}_k$, $k\geq 2$ be the set of all rooted (non-planar) trees (the root and all leaves have valence one) labelled by the finite set $\underline{k}=\{1,\ldots,k\}$. Let $\mathcal{T}_{k,s}\subset \mathcal{T}_k$ be the subset of all trees with exactly $s$ internal vertices. Then the group $\Sigma_k$ acts on the subset $\mathcal{T}_{k,s}$, and we write $\Sigma_T\subset \Sigma_k$  for the stabilizer of a tree $T\in \mathcal{T}_{k,s}$. For example, there is a unique tree $T_{cor}\in \mathcal{T}_{k,1}$ (the corona tree) and its stabilizer $\Sigma_{T_{cor}}$ is $\Sigma_k$.  

\begin{prop}[\cite{Fresse04}]\label{proposition: resolution for Lie}
Let $k\geq 2$. Then there is a chain complex $(C_{\bullet},d)$ of $\Sigma_k$-vector spaces such that 
$$C_{-s} \cong
\begin{cases}
\bigoplus_{T\in\mathcal{T}_{k,s}}\mathrm{Ind}_{\Sigma_T}^{\Sigma_k} \kk & \mbox{if $1 \leq s \leq k-1$}\\
0 & \mbox{otherwise}
\end{cases}
$$
and the complex $(C_{\bullet},d)$ is equivalent to $\bfLie^{\mathrm{sp}}_k\simeq \Sigma^{1-k}\bfLie_k\otimes \mathbf{sgn}^{\otimes k}$ in the category $\Mod_{\kk}^{B\Sigma_k}$. \qed
\end{prop}

We write $\mathbf{L}_k \colon \Mod_{\kk} \to \Mod_{\kk}$, $k\geq 0$ for the functor given by 
\begin{equation}\label{equation: lie powers}
\mathbf{L}_k(V)=\bfLie_k\otimes_{h\Sigma_k} V^{\otimes k}, \;\; V\in \Mod_{\kk}.
\end{equation}
Proposition~\ref{proposition: homology of spectral lie operad} implies that there are natural equivalences
$$\mathbf{L}_k(V) \simeq \Sigma^{-1}\bfLie^{\mathrm{sp}}_{k} \otimes_{h\Sigma_k}(\Sigma V)^{\otimes k}, \;\; k\geq 0,\; V\in \Mod_{\kk}.$$
By Proposition~\ref{proposition: resolution for Lie}, there is a natural map
\begin{equation}\label{equation: corona cell map, kjaer}
c^{k;1}\colon \bfLie^{\mathrm{sp}}_k \simeq (C_\bullet,d) \to C_{-1}\simeq \Sigma^{-1}\kk \in \Mod_{\kk}^{\Sigma_{k}}, \;\; k\geq 2,
\end{equation}
where the middle arrow is the canonical projection. The map~\eqref{equation: corona cell map, kjaer} induces the natural transformation
\begin{equation}\label{equation: corona cell map, nat trans}
c^{k;1}\colon \mathbf{L}_k(V) \to \Sigma^{-2}(\Sigma V)^{\otimes k}_{h\Sigma_k}, \;\; k\geq 2, \; V\in \Mod_{\kk}.
\end{equation}

\begin{prop}[{{\cite{Kjaer18}}}]\label{proposition: corona cell is almost equivalence}
Suppose that $V\simeq \Sigma^l \kk$, $l\in \Z$. Then 
\begin{enumerate}
\item the map~\eqref{equation: corona cell map, nat trans} $$c^{p;1}\colon \mathbf{L}_p(V) \to \Sigma^{-2}(\Sigma V)^{\otimes p}_{h\Sigma_{p}}$$ is an equivalence if $p=2$ or $p$ is odd and $l$ is even;
\item the suspension map 
\begin{equation}\label{equation: suspention map, lie powers}
E\colon \Sigma \mathbf{L}_p(V) \to \mathbf{L}_p(\Sigma V)
\end{equation}
is an equivalence if $p$ is odd and $l$ is even;
\item the induced map
$$c^{p;1}_*\colon \pi_i(\mathbf{L}_p(V)) \to \pi_i(\Sigma^{-2}(\Sigma V)^{\otimes p}_{h\Sigma_p})$$
is an isomorphism if $p$ is odd, $l=2m-1$, $m\in\Z$, and $i\neq 2mp-3$;
\item $\dim \pi_{2mp-3}(\mathbf{L}_p(V))=1$ if $p$ is odd and $l=2m-1$, $m\in\Z$.
\end{enumerate}
\end{prop}

\begin{proof}[Sketch of the proof]
The proposition was essentially proven in~\cite[Lemma~3.1]{Kjaer18}. Here we sketch the argument for the reader's convenience.

The last two parts follows directly by the first two parts and~\cite[Section~I.1]{CLM76}, which describes the basis for $\pi_*(\Sigma^q\kk)^{\otimes p}_{h\Sigma_p}$, $q\in\Z$. The second part follows from the fiber sequence on p.~86 in~\cite[Section~4.1.3]{Brantner_Thesis}. We will prove the first part.

It is clear that $c_{2;1}$ is equivalence, so we will assume that $p$ is odd and $V\simeq \Sigma^l \kk$ and $l$ is even. By Proposition~\ref{proposition: resolution for Lie}, there exists a spectral sequence
\begin{equation}\label{equation:ss for lie}
E^1_{s,t}=H_{t+1}(\Sigma_p, C_s\otimes (\pi_{l+1}(\Sigma V)^{\otimes p})) \Rightarrow \pi_{t+s}(\mathbf{L}_{p}(V)),
\end{equation}
where $t$ stands for the total homological degree.

We note that the edge map $$\pi_{i}(\mathbf{L}_p(V)) \to E^\infty_{-1,i+1} \to E^1_{-1,i+1} = \pi_{i+2} ((\Sigma V)^{\otimes p}_{h\Sigma_p})\cong \pi_{i}(\Sigma^{-2}(\Sigma V)^{\otimes p}_{h\Sigma_p}), \;\; i\in \Z$$
coincides with $c^{p;1}_*$. Since the order $|\Sigma_T|$ of the stabilizer group is coprime to $p$ for $T\in \mathcal{T}_{p,s}$, $s\geq 2$, we have
$$E^1_{s,*} = \bigoplus_{T\in\mathcal{T}_{p,|s|}} H_{*-p(l+1)}(\Sigma_p,\mathrm{Ind}^{\Sigma_p}_{\Sigma_T}(\mathbf{sgn})) \cong \bigoplus_{T\in\mathcal{T}_{p,|s|}}\Sigma^{p(l+1)}(\mathbf{sgn})_{\Sigma_{T}}, \;\; s\leq -2.$$
Since the group $\Sigma_T$, $T\in \mathcal{T}_{p,s}$, $s\geq 2$ always contains a $2$-cycle, we obtain that the coinvariants $(\mathbf{sgn})_{\Sigma_{T}}=0$ vanish. Therefore $E^1_{s,t}=0$ if $s\leq -2$, which implies the proposition.
\end{proof}

\begin{prop}\label{proposition: vanishing of LnV}
Suppose that $V\simeq \Sigma^{l}\kk$, $l\in\Z$. Then
\begin{enumerate}
\item $\mathbf{L}_n(V)\simeq 0$ if $n\neq p^h$ for some $h\geq 0$ and $p=2$ or $p$ is odd and $l$ is even;
\item $\mathbf{L}_n(V)\simeq 0$ if $n\neq p^h, 2p^h$ for some $h\geq 0$ and both $p$ and $l$ are odd.
\end{enumerate}
\end{prop}

\begin{proof}
For the first part, we note that the homotopy groups $\pi_*(\mathbf{L}_n(V))$ are the (shifted) group homology $H_*(\Sigma_n, \bfLie_n)$. These group homology are trivial (if $n$ is not a power of $p$) by the last paragraph of~\cite[Remark~6.3]{AD01}. For the second part, we use the EHP-sequence-type argument to reduce it to the case of even $l$, see e.g.~\cite[Theorem~8.5]{AB21} and~\cite[Section~4.1.3]{Brantner_Thesis}.
\end{proof}

Let $V=\Sigma^l\kk\in \Mod_{\kk}$, $l\in\Z$ be a shift of the $1$-dimensional vector space, $\iota_l\in \pi_l(\Sigma^l\kk)$ is the canonical generator. Recall from~\cite[Section~I.1]{CLM76} that the homotopy groups $\pi_*(V^{\otimes p}_{h\Sigma_p})$ have a basis given by the \emph{Dyer--Lashof classes} $\beta^{\e}Q^i(\iota_l)$, see also Section~\ref{section: group cohomology} for a short review. We use Proposition~\ref{proposition: corona cell is almost equivalence} to construct a basis in the homotopy groups $\pi_*(\mathbf{L}_p(V))$ out of the classical Dyer--Lashof classes. 

If $p=2$, we define $\Qe{i}(\iota_l) \in \pi_{l+i-1}(\mathbf{L}_2(\Sigma^l \kk))$, $i\in \Z$ by the formula
$$c^{2;1}_*(\Qe{i}(\iota_l)) = \sigma^{-2}Q^{i}(\iota_{l+1}) \in \pi_{l+i-1}(\Sigma^{-2}(\Sigma^{l+1}\kk)^{\otimes 2}_{h\Sigma_2}).$$ 
If $p$ is odd and $l$ is even, we define $\bQe{\e}{i}(\iota_l) \in \pi_{l+2i(p-1)-\e-1}(\mathbf{L}_p(\Sigma^l \kk))$, $i\in \Z$, $\e\in \{0,1\}$ by the formula
$$c^{p;1}_*(\bQe{\e}{i}(\iota_l)) = \sigma^{-2}\beta^{\e}Q^{i}(\iota_{l+1}) \in \pi_{l+2i(p-1) - \e -1}(\Sigma^{-2}(\Sigma^{l+1}\kk)^{\otimes p}_{h\Sigma_p}).$$ 
If both $p$ and $l$ are odd, we define $\bQe{\e}{i}(\iota_l) \in \pi_{l+2i(p-1)-\e-1}(\mathbf{L}_p(\Sigma^l \kk))$, $i\in \Z$, $\e\in \{0,1\}$ by the formula
$$\bQe{\e}{i}(\iota_l) = (-1)^{\e}E_*(\sigma \bQe{\e}{i}(\iota_{l-1})), $$
where $E$ is the suspension map~\eqref{equation: suspention map, lie powers}.
By Proposition~\ref{proposition: corona cell is almost equivalence}, all homotopy classes $\bQe{\e}{i}(\iota_l)$ above are \emph{well-defined}.

\begin{dfn}\label{definition: dll-classes}
Let $V\in \Mod_{\kk}$ and let $v \in \pi_{l}V$ be a homotopy class represented by a map $\overline{v}\colon \Sigma^l \kk \to V$. Then we define the \emph{Dyer--Lashof--Lie class} $\bQe{\e}{i}(v) \in \pi_{l+2i(p-1)-\e-1}(\mathbf{L}_p(V))$, $i\in \Z$, $\e\in \{0,1\}$ (resp. $\Qe{i}(v) \in \pi_{l+i-1}(\mathbf{L}_2(V))$, $i\in \Z$ if $p=2$) by the formula
$$\bQe{\e}{i}(v) = \overline{v}_*(\bQe{\e}{i}(\iota_l)) \;\; \text{(resp. $\Qe{i}(v)=\overline{v}_*(\Qe{i}(\iota_l))$)}.$$
\end{dfn}

\begin{cor}\label{corollary: kjaer basis of free lie algebra} Let $V\in \Mod_\kk$ be a chain complex. Then the Dyer--Lashof--Lie classes $\bQe{\e}{i}(v)\in\pi_*(\mathbf{L}_p(V))$, $v\in \pi_*(V)$, $i\in\Z$, $\e\in \{0,1\}$ (resp. $\Qe{i}(v) \in\pi_*(\mathbf{L}_2(V))$, $v\in \pi_*(V)$, $i\in\Z$ if $p=2$) satisfy following properties
\begin{enumerate}
\item semi-linearity, i.e. $\bQe{\e}{i}(\alpha v) = \alpha^p\bQe{\e}{i}(v)$, $\alpha \in \kk$ (resp. $\Qe{i}(\alpha v) = \alpha^2\Qe{i}(v)$, $\alpha \in \kk$);
\item additivity, i.e. $\bQe{\e}{i}(v+w) = \bQe{\e}{i}(v) + \bQe{\e}{i}(w)$, $v,w\in \pi_*(V)$;
\item stability, i.e. $\bQe{\e}{i}(\sigma v) = (-1)^{\e} \sigma\bQe{\e}{i}(v)$.
\item $\bQe{\e}{i}(v)=0$ if and only if $2i\leq |v|$ (resp. $\Qe{i}(v)=0$ if and only if $i\leq |v|$). \qed
\end{enumerate}
\end{cor}

The Lie operad $\bfLie_*$ 
produces the monad
$$\bbP_{\bfLie_*}(-)= \bigoplus_{k\geq 0} \bfLie_k \otimes_{h\Sigma_k} (-)^{\otimes k} $$ 
acting on the $\infty$-category $\Mod_{\kk}$. By Proposition~\ref{proposition: homology of spectral lie operad}, the monad $\bbP_{\bfLie_*}$ 
descends to the monad 
\begin{equation}\label{equation: H-Lie-monad}
F_{\bfLie}\colon \mathrm{Ho}(\Mod_{\kk}) \to \mathrm{Ho}(\Mod_{\kk})
\end{equation}
acting on the \emph{homotopy} ($1$-)category $\mathrm{Ho}(\Mod_{\kk})$. By analogy with $E_\infty$- and $H_\infty$-algebras, see e.g.~\cite[Section~1.3]{BMMS}, we define the ($1$-)category of \emph{$H_\infty$-$\bfLie_*$-algebras} as the ($1$-)category of algebras over the monad $F_{\bfLie}$. By abuse of notation we will also write $F_{\bfLie}(V)$ for the free $H_\infty$-$\bfLie_*$-algebra generated by $V\in \mathrm{Ho}(\Mod_{\kk})$. Finally, we note that $\bfLie_*$-algebras 
in the $\infty$-category $\Mod_{\kk}$ descends to $H_\infty$-$\bfLie_*$-algebras in the homotopy ($1$-)category $\mathrm{Ho}(\Mod_{\kk})$.

\begin{rmk}\label{remark: H-Lie-algebra explicit}
Unwinding the definition, an $H_\infty$-$\bfLie_*$-algebra is a chain complex $L\in \mathrm{Ho}(\Mod_{\kk})$ equipped with a sequence of maps
\begin{equation}\label{equation: lie multiplications}
\gamma_k\colon \mathbf{L}_k(L) \to L, \;\; k\geq 2
\end{equation}
such that the following associativity diagrams 
$$
\begin{tikzcd}[column sep=huge]
\mathbf{L}_{k}(\mathbf{L}_n(L)) \arrow{r}{m_{k;n,\ldots, n}} \arrow{d}[swap]{\mathbf{L}_{k}(\gamma_n)}
&\mathbf{L}_{kn}(L)\arrow{d}{\gamma_{kn}} \\
\mathbf{L}_{k}(L) \arrow{r}{\gamma_{k}}
&L
\end{tikzcd}
$$
commute in the homotopy ($1$-)category $\mathrm{Ho}(\Mod_{\kk})$ for all $k,n\geq 2$. Here the map $m_{k;n,\ldots,n}$ is induced by the operadic multiplication map.
\end{rmk}

\begin{dfn}\label{definition: dll operations}
Let $L$ be an $H_\infty$-$\bfLie_*$-algebra. If $p=2$, we define the \emph{Dyer--Lashof--Lie operation} 
$$\Qe{i}\colon \pi_l(L)^{(1)} \to \pi_{l+i-1}(L), \;\; l,i\in\Z $$
by the formula $\Qe{i}(x)=\gamma_{2*}(\Qe{i}(x))$, $x\in \pi_l(L)$, where $\gamma_2 \colon \mathbf{L}_2(L) \to L$ is the structure map~\eqref{equation: lie multiplications}. 
Similarly, if $p$ is odd, we define the \emph{Dyer--Lashof--Lie operation} 
$$\bQe{\e}{i}\colon \pi_l(L)^{(1)} \to \pi_{l+2i(p-1)-\e-1}(L), \;\; l,i\in\Z, \e\in\{0,1\}$$
by the formula $\bQe{\e}{i}(x)=\gamma_{p*}(\bQe{\e}{i}(x))$, $x\in \pi_l(L)$, where $\gamma_p \colon \mathbf{L}_p(L) \to L$ is the structure map~\eqref{equation: lie multiplications}. 
\end{dfn}

\begin{rmk}\label{remark: dll, free lie algebra}
Suppose that $L=F_{\bfLie}(V)$ is a free $H_\infty$-$\bfLie_*$-algebra. Then there is a natural decomposition
$$\pi_*(F_{\bfLie}(V)) \cong \bigoplus_{k\geq 0}\pi_*(\mathbf{L}_k(V)).$$
Let $x\in \pi_*(\mathbf{L}_k(V))\subset \pi_*(F_{\bfLie}(V))$, $k\geq 0$ be a homotopy class. Under the decomposition above, the Dyer--Lashof--Lie operation $\bQe{\e}{i}$, $i\in\Z$, $\e\in\{0,1\}$ (resp. $\Qe{i}$, $i\in\Z$ if $p=2$) maps $x$ into the component $\pi_*(\mathbf{L}_{pk}(V))\subset \pi_*(F_{\bfLie}(V))$ (resp. $\pi_*(\mathbf{L}_{2k}(V))\subset \pi_*(F_{\bfLie}(V))$). By abuse of notation we refer to the class $\bQe{\e}{i}(x) \in \pi_*(\mathbf{L}_{pk}(V))$ (resp. $\bQe{\e}{i}(x) \in \pi_*(\mathbf{L}_{2k}(V))$) again as a Dyer--Lashof--Lie class.
\end{rmk}

\section{Algebraic functor calculus}\label{section: algebraic functor calculus}
In this section we prove our starting results on the Taylor tower for the functor $$\free\colon \Mod^{\geq 0}_{\kk} \to \sL $$
of the free simplicial restricted Lie algebra. The main results are Corollary~\ref{corollary:derivatives of free}, which identifies the Goodwillie layers $D_n(\free)$, Theorem~\ref{theorem:splitting}, which shows that the Taylor tower of the truncated functor $P_{[n;pn-1]}(\free)$ (see~\eqref{equation: thick layer}) splits after the $\kk$-completion and the $(p-2)$-th fold looping, and Proposition~\ref{proposition:james-hopf and adjoint}.

In Section~\ref{section: scalar degree} we discuss the \emph{endofunctor category} $\Endinfty_{\kk}=\Fun^{\omega}(\Mod_{\kk},\Mod_{\kk})$ of the $\kk$-linear category $\Mod_{\kk}$. Since we consider \emph{all} functors, not necessary $\kk$-linear ones, the category $\Endinfty_{\kk}$ is rich. Given a character $\chi \in \mathfrak{X}(\kk) = \Hom(\kk^{\times}, \kk^{\times})$, we define the full subcategory $\Endinfty_{\kk}^{hB\Gamma_{\chi}} \subset \Endinfty_{\kk}$ of functors with \emph{scalar degree} $\chi$, see Definition~\ref{definition:linear degree}. We show in Proposition~\ref{proposition: scalar degree is fully faithful} that the subcategories $\Endinfty_{\kk}^{hB\Gamma_{\chi}}$, $\Endinfty_{\kk}^{hB\Gamma_{\chi'}}$ are orthogonal to each other if $\chi\neq \chi' \in \mathfrak{\kk}$. Unlike the case of endofunctors of an \emph{abelian} ($1$-)category, where the base field may be arbitrary, see e.g.~\cite{Kuhn_generic}, our results for the $\infty$-category $\Mod_{\kk}$ put certain restrictions on the field~$\kk$, see Lemma~\ref{lemma: cohomology of multiplicative group}. 

In Section~\ref{section: vanishing theorems} we show that there are no non-trivial natural transformations (resp. natural transformations admitting delooping) between homogeneous functors from the category $\Mod_{\kk}^{+}$ to $\Sp(\sLxi)$ (resp. $\sLxi$) of certain form and different excisive degrees, see Proposition~\ref{proposition:suspention and natural transformations} (resp. Proposition~\ref{proposition:maps between homogeneous functors}). We apply these vanishing results in Section~\ref{section: snaith splitting} to obtain the \emph{Snaith splitting} (Proposition~\ref{proposition: Snaith splitting}) and to define the important \emph{James--Hopf map} (Definition~\ref{definition:james-hopf map}).

In Section~\ref{section: goodwillie tower of a free lie algebra} we prove Theorem~\ref{theorem: A,intro} as Corollary~\ref{corollary:derivatives of free}. Using this computation and the vanishing results of Section~\ref{section: vanishing theorems}, we demonstrate the splitting of the Taylor tower for the functor $P_{[n;pn-1]}(\Omega^{p-2}L_\xi\free)$ in Corollary~\ref{corollary: splitting}. In Proposition~\ref{proposition:james-hopf and adjoint} we use this splitting to show that the differentials in the Goodwillie spectral sequence are induced by certain natural transformations $\delta_n$. Moreover, we factorize $\delta_n$ as the composite of the $p$-th James--Hopf map $j_p$ and a certain map $\deloop \D_{pn}(\tilde{\delta}_n)$ of infinite loop objects.

\subsection{Scalar degree}\label{section: scalar degree}

Recall that $\kk$ is the algebraic closure of the finite field $\F_p$. We write $\mathfrak{X}(\kk)=\Hom(\kk^{\times},\kk^{\times})$ for the \emph{character group} of the multiplicative group~$\kk^{\times}$. Note that there is an inclusion $\Z\subset \Hom(\kk^{\times},\kk^{\times})$ given by $n\mapsto \chi_n$, $n\in\Z$, where $\chi_n(a)=a^n$, $a\in \kk^{\times}$.

The tensor product $-\otimes-\colon \Mod_{\kk}\times \Mod_{\kk} \to \Mod_{\kk}$ induces the functor
$$\Mod_{\kk} \to \Fun^\omega(\Mod_{\kk},\Mod_{\kk}), \;\; V \mapsto (- \mapsto V\otimes - ).$$
Let $B\kk^{\times} \subset \Mod_{\kk}$ be the full subcategory spanned by the one-dimensional vector space $\kk \in \Mod_{\kk}$. By restricting, we obtain a functor 
$$B\kk^{\times}  \hookrightarrow \Mod_{\kk} \to \Fun^{\omega}(\Mod_{\kk},\Mod_{\kk})$$
which maps the only object of $B\kk^{\times}$ to the identity functor and its automorphism given by $a\in \kk^{\times}$ to the natural transformation $$m_a\colon \id \to \id$$ given by the multiplication by $a\in \kk^{\times}$.
In other words, we obtain a $B\kk^{\times}$-action on the $\infty$-category $\Mod_{\kk}\in\Pre^L$. We prolong this $B\kk^\times$-action to the natural $B(\kk^{\times}\times \kk^{\times})$-action on the endofunctor category $\Endinfty_{\kk}=\Fun^{\omega}(\Mod_{\kk},\Mod_{\kk})$:
\begin{equation}\label{equation:BGm-action}
B(\kk^{\times}\times\kk^{\times}) \to \Fun^{\omega}(\Endinfty_{\kk}, \Endinfty_{\kk}).
\end{equation}

More informally, the action~\eqref{equation:BGm-action} means that any endofunctor $F\in \Endinfty_{\kk}$ is endowed with the natural transformation $\eta_{a,b}\colon F \to F$, $(a,b)\in \kk^{\times}\times \kk^{\times}$ given by the rule 
\begin{equation}\label{equation:BGmxG_m-action}
\eta^V_{a,b}=m^{F(V)}_{b}\circ F(m^V_{a^{-1}})\colon F(V) \to F(V).
\end{equation}
Moreover, these natural transformations induce the $\kk^{\times}\times \kk^{\times}$-action on $F$.

Let $\chi\in \mathfrak{X}(\kk)$ be a homomorphism $\chi\colon \kk^{\times} \to \kk^{\times}$ and let $\Gamma_{\chi}\subset \kk^{\times}\times \kk^{\times}$ be the corresponding graph subgroup. We denote by $L_\chi$ the natural forgetful functor $$L_{\chi}\colon \Endinfty_{\kk}^{hB\Gamma_{\chi}} \to \Endinfty_{\kk}.$$

\begin{rmk}\label{remark: meaning of the scalar degree}
Suppose that $F\in \Endinfty_{\kk}$ is in the essential image of the functor $L_\chi$, $\chi\in \mathfrak{X}(F)$. Then all natural transformations~\eqref{equation:BGmxG_m-action} are equivalent to the identity. In particular, for any $V\in \Mod_{\kk}$ there is a homotopy between the maps $F(m_a^V)$ and $m_{\chi(a)}^{F(V)}$.
\end{rmk}

\begin{prop}\label{proposition: scalar degree, right adjoint}
The functor $L_\chi$, $\chi\in\mathfrak{X}(\kk)$ is conservative and admits the right adjoint $$R_{\chi}\colon \Endinfty_{\kk} \to \Endinfty_{\kk}^{hB\Gamma_{\chi}}.$$ Moreover, the counit map $L_{\chi}R_{\chi}F \to F$ coincides with $F^{h\Gamma_{\chi}} \to F$, where the $\Gamma_{\chi}$-action on $F$ is given by the natural transformations~\eqref{equation:BGmxG_m-action}.
\end{prop}

\begin{proof}
The proof is similar to the proof of~\cite[Lemma~7.2]{Yuan21}. Indeed, let us consider $B^2\Gamma_{\chi}$ as a pointed space, and denote the inclusion of the basepoint by $i\colon\{ * \} \to B^2\Gamma_{\chi}$.  The $B\Gamma_{\chi}$-action~\eqref{equation:BGm-action}
determines a pullback square
\begin{equation*}
\begin{tikzcd}
\Endinfty_{\kk} \arrow[d,"q"]\arrow[r,"\bar{i}"] & \calD \arrow[d,"f"] \\
\{ * \} \arrow[r,"i"] & B^2\Gamma_{\chi},
\end{tikzcd}
\end{equation*}
where $f$ is a coCartesian (and Cartesian) fibration.  Let $\mathrm{sect}^{\circ}(f) \subset \mathrm{sect}(f)$ be the full subcategory of sections which send any morphism in $B^2\Gamma_{\chi}$ to a coCartesian morphism in $\calD$.  We have an equivalence $\mathrm{sect}^{\circ}(f) \simeq \Endinfty_{\kk}^{hB\Gamma_{\chi}},$ and the functor $L_{\chi}$ is the restriction of a section of $f$ along $i$.  

The right adjoint $R_{\chi}$ can be given by the $f$-right Kan extension along $i$, see~\cite[Definition 2.8]{QS19}. 
By \cite[Corollary 2.11]{QS19}, such Kan extension exists as long as for any $F\in \Endinfty_{\kk}$, considered as a section of $q$, the induced diagram
$$B\Gamma_{\chi}\simeq *\times_{B^2\Gamma_{\chi}} (B^2\Gamma_{\chi})_{*/} \to \calD_{*}\simeq \Endinfty_{\kk}$$ 
has a limit. However, such diagram always has a limit because the category $\Endinfty_{\kk}$ is complete.  Moreover, by the same corollary, $L_\chi R_\chi F \simeq F^{h\Gamma_{\chi}}, F\in \Endinfty_{\kk}$ and the resulting section of $f$ is in the full subcategory $\mathrm{sect}^{\circ}(f)\subset \mathrm{sect}(f)$.  

The functor $L_{\chi}$ is an exact functor between stable categories. Therefore, $L_{\chi}$ is conservative if and only if $L_{\chi}F\simeq 0$ implies $F\simeq 0$ for $F\in \Endinfty_{\kk}^{hB\Gamma_{\chi}}$. Since $B^2\Gamma_{\chi}$ is a connected Kan complex, the latter is clear.
\end{proof}

The next lemma is crucial for the proof that the functor $L_\chi$, $\chi\in \mathfrak{X}(\kk)$ is fully faithful. 

\begin{lmm}\label{lemma: cohomology of multiplicative group}
Let $\chi \in \mathfrak{X}(\kk)$ be a character and let $V\in \Vect_{\kk}$ be a vector space. Then 
$$ 
H^i(\kk^\times, \chi\otimes V)=
\begin{cases}
V & \mbox{if $i=0$ and $\chi=\chi_0$ is trivial,}\\
0 & \mbox{otherwise.}
\end{cases}
$$
\end{lmm}

\begin{proof}
Since $\kk=\overline{\F}_p=\cup \F_{p^n}$, the multiplicative group $\kk^\times$ is the direct limit of finite groups $\F_{p^n}^{\times}$. Therefore we have the Milnor exact sequence
$$0 \to {\lim_n}^1 H^{i-1}(\F_{p^n}^{\times},\chi\otimes V) \to H^i(\kk^\times,\chi\otimes V) \to \lim_n H^i(\F_{p^n}^{\times}, \chi\otimes V)\to 0, \;\; i\geq 0. $$

Since the order $|\F_{p^n}^{\times}|$ is coprime to $p$, the group cohomology $H^i(\F_{p^n}^{\times},\chi\otimes V)=0$ if $i\geq 1$. Moreover, if $\chi\in \mathfrak{X}(\F)$ is non-trivial, then for each $n\geq 1$ there exists a multiple $m=kn$, $k\geq 1$ such that the restricted character $\chi|_{\F^{\times}_{p^m}}$ is also non-trivial. Therefore for $\chi\neq \chi_0$ we have
$${\lim_n}^1 H^0(\F_{p^n}^{\times},\chi\otimes V) = \lim_n H^0(\F_{p^n}^{\times},\chi\otimes V) =0,$$
and so $H^*(\kk^{\times},\chi\otimes V)=0$. If $\chi=\chi_0$, then $H^0(\F_{p^n}^{\times},V)=V$, $n\geq 1$, and so $H^0(\kk^{\times},V)=V$ and $H^i(\kk^{\times},V)=0$ if $i\geq 0$.
\end{proof}

\begin{prop}\label{proposition: scalar degree is fully faithful}
Let $\chi,\chi'\in \mathfrak{X}(\kk)$ be characters. Then the composition $$R_{\chi'}\circ L_\chi \colon \Endinfty_{\kk}^{hB\Gamma_{\chi}} \to \Endinfty_{\kk}^{hB\Gamma_{\chi'}}$$ is trivial if $\chi\neq \chi'$ and the identity if $\chi=\chi'$. In particular, the functors $L_\chi$, $\chi\in \mathfrak{X}(\kk)$ are fully faithful. 
\end{prop}

\begin{proof}
By Proposition~\ref{proposition: scalar degree, right adjoint}, the functor $L_{\chi'}$ is conservative, and so it suffices to compute the composition $L_{\chi'}\circ R_{\chi'}\circ L_{\chi}$. By applying the same proposition, we have
$$L_{\chi'}\circ R_{\chi'}\circ L_{\chi}(F)\simeq (L_\chi F)^{h\Gamma_{\chi'}}, \;\; F\in \Endinfty_{\kk}. $$
By Remark~\ref{remark: meaning of the scalar degree}, we have $$\left((L_\chi F)^{h\Gamma_{\chi'}}\right)(V)\simeq \left(\chi'\otimes \chi^{-1} \otimes F(V)\right)^{h\kk^{\times}}$$
for a perfect $V \in \Mod_{\kk}^{\omega}$. By Lemma~\ref{lemma: cohomology of multiplicative group}, the latter is equivalent to $F$ if $\chi=\chi'$ and trivial otherwise.
\end{proof}

\begin{dfn}\label{definition:linear degree}
A non-zero functor $F\colon \Mod_{\kk} \to \Mod_{\kk}$ is called \emph{scalable} if $F\in \Endinfty_{\kk}$ is in the essential image of $L_\chi$ for some $\chi \in\mathfrak{X}(\kk)$. If $F\simeq L_\chi(\bar{F})$ for some $\bar{F}\in \Endinfty_{\kk}^{hB\Gamma_{\chi}}$, then $\chi\in \mathfrak{X}(\kk)$ is the \emph{scalar degree} of $F$. 
\end{dfn}

By Proposition~\ref{proposition: scalar degree is fully faithful}, the scalar degree of a scalable functor $F\in\Endinfty_{\kk}$ is well-defined. We denote the scalar degree of $F$ by $\lindeg(F)$.

\begin{prop}\label{proposition: scalable functors are closed colimits}
The full subcategory $\Endinfty_{\kk}^{hB\Gamma_{\chi}} \subset \Endinfty_{\kk}$ of scalable functors of scalar degree $\chi \in \mathfrak{X}(\kk)$ is closed under arbitrary small colimits and finite limits. \qed
\end{prop}

\begin{exmp}\label{example:composition of linear degrees}
Recall that the set of characters $(\mathfrak{X}(\kk),\circ)$ is a monoid with respect to composition. Let $F,G\colon \Mod_{\kk} \to \Mod_{\kk}$ be scalable functors such that $\lindeg(F)=\chi$ and $\lindeg(G)=\chi'$. Then $\lindeg(F\circ G)=\chi\circ \chi'$.
\end{exmp}

For $A\in \Mod_{\kk}^{B\Sigma_n}$, we define an $n$-homogeneous functor $H^n_A\in \Endinfty_{\kk}$ by the rule
\begin{equation}\label{equation:diagonal homogeneous functor}
H_A^n(V)=A\otimes_{h\Sigma_n}V^{\otimes n}\in \Fun^{\omega}(\Mod_{\kk},\Mod_{\kk}), \;\; V\in \Mod_{\kk}.
\end{equation}
We note that $\lindeg(H_A^n) = \chi_n$, where $\chi_n\in \Z \subset \mathfrak{X}(\kk)$.

\begin{cor}\label{corollary:vanishing of natural transformations}
For modules $A\in \Mod_{\kk}^{B\Sigma_n}$ and $B\in \Mod_{\kk}^{B\Sigma_m}$ the space of natural transformations $\nat(H_A^n,H_B^m)$ is contractible if $n\neq m$. \qed
\end{cor}

\begin{exmp}\label{example:galois group} Let $\Gal(\kk) = \Gal(\kk/\F_p)$ be the Galois group of $\kk$ over $\F_p$.  We define an $\kk$-$\kk$-bimodule $\kk_{\chi}, \chi\in\Gal(\kk)$ as follows: $\kk_{\chi}$ is equal to $\kk$ as \emph{left} $\kk$-modules and the \emph{right} $\kk$-action on $\F_\chi$ is twisted by $\chi$: 
$$x\cdot a = \chi(a)x, \;\; x\in \F_{\chi}, \;\; a\in \F. $$

The bimodule $\F_{\chi}$ induces the scalable functor $\F_{\chi}\otimes (-)\colon \Mod_{\kk}\to \Mod_{\kk}$ and the scalar degree of $\kk_{\chi}\otimes (-)$ is $\chi$.
\end{exmp}

\begin{rmk}\label{remark: unbounded vs bounded}
The results of this section stay true if we replace the endofunctors category $\Endinfty_{\kk}=\Fun^{\omega}(\Mod_{\kk},\Mod_{\kk})$ with any of following functor categories: $\Fun^{\omega}(\Mod_{\kk}^{+},\Mod_{\kk})$, $\Fun^{\omega}(\Mod_{\kk}^{\geq 0},\Mod_{\kk})$.
\end{rmk}

\begin{exmp}\label{example: scalar degree of dold-puppe}
Let $F\in \Fun(\Vect^{\omega}_{\kk},\Vect_{\kk})$ be a functor of the form
$$F(V)= (A\otimes V^{\otimes n})_{\Sigma_n} \;\; \text{or} \;\; F(V)= (A\otimes V^{\otimes n})^{\Sigma_n} $$
for a finite-dimensional $V\in \Vect^{\omega}_{\kk}$, some $\Sigma_n$-module $A\in \Vect^{B\Sigma_n}_{\kk}$, and $n\geq 0$. Then the non-abelian derived functor (see~\cite[Construction~3.3]{BM19}):
$$\mathbb{L}F \colon \Mod^{\geq 0}_{\kk} \to \Mod_{\kk} $$ 
is scalable and it has scalar degree $\lindeg(\mathbb{L}F)=\chi_n$, where $\chi_n \in \Z \subset \mathfrak{X}(\kk)$.
\end{exmp}

\subsection{Vanishing theorems}\label{section: vanishing theorems} 

\begin{prop}\label{proposition:suspention and natural transformations}
Let $A\in \Mod_{\kk}^{B\Sigma_n}$ be a $\Sigma_n$-module and let $B\in \Mod_{\kk}^{+,B\Sigma_m}$ be a bounded below $\Sigma_m$-module, $n,m\geq 0$. We consider the induced functors $H^n_A, H^m_B\in \Fun^{\omega}(\Mod^{+}_{\kk},\Mod_{\kk})$ as functors from the category $\Mod_{\kk}^{+}$ of bounded below objects. Then the natural map
$$\nat(H_A^n,H_B^m) \xrightarrow{\sim} \nat(\suspfreexi H_A^n,\suspfreexi H_B^m) $$
induced by the postcomposition with the left adjoint functor $\suspfreexi\colon \Mod_{\kk} \to \Sp(\sLxi)$ of~\eqref{equation: free stable xi-complete} is an equivalence.
\end{prop}

\begin{proof}
The functors $H^n_A, \suspfreexi H^n_A$ are $n$-homogeneous and the functors $H^m_B,\suspfreexi H^m_B$ are $m$-homogeneous, so both sides are trivial for $m<n$. Therefore we can assume that $m\geq n$ for the rest of the proof. By Proposition~\ref{proposition: xi-complete stable is monadic}, we have
\begin{align*}
\nat(\suspfreexi H_A^n,\suspfreexi H_B^m) \simeq \nat(H_A^n, \oblv \circ \suspfreexi H_B^m).
\end{align*}
By the assumptions, we can use Proposition~\ref{proposition: stable monad for xi-complete}, and so, we unwind the composition $\oblv\circ \suspfreexi H^m_B$ as to the following direct product
$$\oblv\circ \suspfreexi H^m_B(V) \simeq \prod_{k\geq 1}\colim_j \bigg(\Sigma^{-j}\mathbb{L}L^r_k(\tau_{\geq 0} \Sigma^j H^m_{B}(V))\bigg), \; \; V\in \Mod_{\kk}^{+},
$$
where $\mathbb{L}L^r_k(-)$ are the nonabelian derived functors from~\eqref{equation: oblv free, fresse}. By Proposition~\ref{proposition: scalable functors are closed colimits} and Examples~\ref{example:composition of linear degrees}, \ref{example: scalar degree of dold-puppe}, the scalar degree of the $k$-th component 
$$\lindeg\left( \colim_j \bigg(\Sigma^{-j}\mathbb{L}L^r_k(\tau_{\geq 0} \Sigma^j H^m_{B}(-))\bigg) \right)$$
is $\chi_{km} \in \mathfrak{X}(\kk)$. By Proposition~\ref{proposition: scalar degree is fully faithful}, this implies that the space of natural transformations $$\nat\left(H_A^n, \colim_j \bigg(\Sigma^{-j}\mathbb{L}L^r_k(\tau_{\geq 0} \Sigma^j H^m_{B}(-))\bigg)\right) \simeq \ast$$ is contractible for $k\geq 2$. Indeed, the scalar degree of the target functor is $\chi_{km}$ and $\chi_{n}\neq \chi_{km}\in \mathfrak{X}(\kk)$ if $m\geq n$ and $k\geq 2$. This implies the proposition.
\end{proof}

\begin{rmk}\label{remark:inverse to the equivalence}
By Remark~\ref{remark: homology of a free algebra}, the composite $\Sigma^{-1}(\widetilde{C}_*\circ \suspfree)\simeq \id$ is the identity functor. Therefore the inverse to the equivalence of Proposition~\ref{proposition:suspention and natural transformations} is given by the postcomposition with the chain functor $\Sigma^{-1}\widetilde{C}_*$.
\end{rmk}

By Corollary~\ref{corollary:vanishing of natural transformations}, we observe immediately the following statement.

\begin{cor}\label{corollary:vanishing of natural transformations, spectral}
We keep the notation and assumptions of Proposition~\ref{proposition:suspention and natural transformations}. Suppose that $m\neq n$. Then the space of natural transformations: 
$$\nat(\suspfreexi H_A^n,\suspfreexi H_B^m) \simeq \ast $$
is contractible. \qed
\end{cor}


For a $\Sigma_n$-module $A\in \Mod_{\kk}^{B\Sigma_n}$, we define an $n$-homogeneous functor 
\begin{equation}\label{equation: typical layer}
\mathbb{H}_A^n\colon \Mod_{\kk}\to \sLxi
\end{equation}
by the rule
$$
\mathbb{H}_A^n(V)=\deloop\suspfreexi H_A^n(V)= \deloop \suspfreexi(A\otimes_{h\Sigma_n} V^{\otimes n}), \;\; V\in \Mod_{\kk},
$$
where $H_A^n$ is the homogeneous functor~\eqref{equation:diagonal homogeneous functor}. We will examine the space of natural transformations between functors of the form $\mathbb{H}^n_A$ in Proposition~\ref{proposition:maps between homogeneous functors}, the proof of which requires the following series of lemmas.

\begin{lmm}\label{lemma:composition of homogeneous functors, stable}
Let $F\colon \calC \to \calD$ and $G\colon \calD \to \calE$ be reduced functors between stable cocomplete categories. Suppose that $F$ is $n$-homogeneous, $G$ is $m$-homogeneous, and $G$~preserves filtered colimits. Then the composite $G\circ F\colon \calC \to \calE$ is $mn$-homogeneous.
\end{lmm}

\begin{proof}
By~\cite[Proposition~6.1.4.14]{HigherAlgebra}, we have natural equivalences $$F(X) \simeq L_1(X,\ldots,X)_{h\Sigma_n}, X\in \calC \;\; \text{and} \;\; G(Y) \simeq L_2(Y,\ldots,Y)_{h\Sigma_m}, Y\in \calD,$$
where $L_1\colon \calC^{\times n} \to \calD$ and $L_2\colon \calD^{\times m} \to \calE$ are symmetric multilinear functors. Since $G$ preserves filtered colimits, the functor $L_2$ preserves arbitrary colimits at each variable. Therefore, the composite $G\circ F$ has the following form
$$G\circ F(X) \simeq L(X,\ldots,X)_{h\Sigma_m \wr \Sigma_n}, $$
where $\Sigma_m \wr \Sigma_n = \Sigma_m \ltimes (\Sigma_n)^{\times m}$ is the wreath product and $L\colon \colon \calC^{\times nm} \to \calE$ is a ($\Sigma_m \wr \Sigma_n$)-symmetric multilinear functor
\begin{align*}
L(X_{11}, \ldots, X_{1n}, &X_{21}, \ldots, X_{2n}, \ldots X_{m1}, \ldots, X_{mn}) \simeq \\
&\simeq L_2(L_1(X_{11}, \ldots, X_{1n}), \ldots, L_1(X_{m1}, \ldots, X_{mn})) \in \calE
\end{align*}
for $X_{ij} \in \calC$, $1 \leq i\leq m$, and $1\leq j \leq n$.
\end{proof}

\begin{lmm}\label{lemma: precomposition with homogeneous, stable}
Let $F\colon \calC \to \calD$ and $G\colon \calD \to \calE$ be reduced functors between stable cocomplete categories. Suppose that $F$ is $n$-homogeneous and $G$~preserves filtered colimits. Then the natural morphisms $$D_m(G\circ F) \xrightarrow{\simeq} D_m(P_k(G)\circ F) \xleftarrow{\simeq} D_m(D_k(G) \circ F) \simeq D_k(G) \circ F$$ are equivalences if $m=nk$, and $D_m(G\circ F)$ is contractible if $n\nmid m$. 
\end{lmm}

\begin{proof}
By~\cite[Proposition~1.3.1]{AroneChing11} (see also~\cite[\S6]{Ching10}), the natural map
$$D_m(G\circ F) \xrightarrow{\simeq} D_m(P_m(G)\circ F) $$
is an equivalence. Therefore, we can assume that $G$ is $m$-excisive. 

If $n\nmid m$, then we have $D_{m}(D_i(G)\circ F)\simeq 0$ for all $i\leq m$, see Lemma~\ref{lemma:composition of homogeneous functors, stable}. Otherwise, if $m=nk$, then $D_{m}(D_i(G)\circ F)\simeq 0$ for all $i\leq m$, except $i=k$, see again Lemma~\ref{lemma:composition of homogeneous functors, stable}. We finish the proof by using the fiber sequences
$$D_i(G) \to P_i(G) \to P_{i-1}(G), \; i\leq m$$
and the induction on the Taylor tower of $G$.
\end{proof}

\begin{lmm}\label{lemma:derivatives of suspension of bbH}
Let $A\in \Mod^{+,B\Sigma_n}_{\kk}$ be a bounded below $\Sigma_n$-module. Then the $m$-th Goodwillie layer $D_m(\susp \mathbb{H}^n_A)$ of the functor $\susp \mathbb{H}^n_A\colon \Mod_{\kk}\to \Sp(\sLxi)$ is zero if $n\nmid m$ and $$\Sigma D_m(\susp \mathbb{H}^n_A)\simeq \suspfreexi\left((\Sigma H^n_A)^{\otimes k}_{h\Sigma_k}\right)$$ if $m=nk$.
\end{lmm}

\begin{proof}
By~\cite[Lemma~B.3]{Heuts21}, the $k$-th Goodwillie layer $D_k(\susp\deloop)$ of the functor $\susp\deloop \colon \Sp(\sLxi) \to \Sp(\sLxi)$ is equivalent to $D_k(\susp\deloop)(-)\simeq (-)^{\otimes k}_{h\Sigma_k}$, where $-\otimes-$ is the monoidal structure of Theorem~\ref{theorem: smash, stabilization}. Since the functor $\suspfreexi H^n_A$ is homogeneous, we observe that 
$$D_m(\susp\deloop\suspfreexi H^n_A)\simeq
\begin{cases}
(\suspfreexi H^n_A)^{\otimes k}_{h\Sigma_k} & \mbox{if $m=nk$,}\\
0 & \mbox{otherwise,}
\end{cases}
$$
see Lemma~\ref{lemma: precomposition with homogeneous, stable}.
Finally, by Corollary~\ref{proposition: shifted stable free is symmetric monoidal}, $(\suspfreexi H^n_A)^{\otimes k}_{h\Sigma_k} \simeq \Sigma^{-1}\suspfreexi\left((\Sigma H^n_A)^{\otimes k}_{h\Sigma_k}\right)$.
\end{proof}

\begin{lmm}\label{lemma:suspension is trivial}
Suppose that $1<k<p$. Then the natural transformation \begin{equation}\label{equation:suspension map}
\sigma\colon \Sigma V^{\otimes k}_{h\Sigma_k}\to (\Sigma V)^{\otimes k}_{h\Sigma_k},\;\; V\in \Mod_{\kk}
\end{equation} induced by the suspension is trivial. \qed
\end{lmm}

\begin{prop}\label{proposition:maps between homogeneous functors} Let $A\in \Mod_{\kk}^{B\Sigma_n}$ be a $\Sigma_n$-module and let $B\in \Mod_{\kk}^{+,B\Sigma_m}$ be a bounded below $\Sigma_m$-module, $m>n\geq 0$. We consider the induced functors $\mathbb{H}^n_A, \mathbb{H}^m_B\in \Fun^{\omega}(\Mod^{+}_{\kk},\sLxi)$ as functors from the category $\Mod_{\kk}^{+}$ of bounded below objects. Then the space of natural transformations $\nat(\mathbb{H}^n_A,\mathbb{H}^m_B)$ is contractible if $n\nmid m$.
Moreover, if $m=kn$, $1 <k< p$, then the map 
\begin{equation}\label{equation:postcomposition with Omega}\nat(\mathbb{H}^n_A,\mathbb{H}^m_B) \xrightarrow{\Omega\circ} \nat(\Omega\mathbb{H}^n_A,\Omega\mathbb{H}^m_B)\simeq \nat(\mathbb{H}^n_{\Sigma^{-1}A},\mathbb{H}^m_{\Sigma^{-1}B})\end{equation} induced by the postcomposition with the loop functor $\Omega$ is trivial. In particular, $\Omega\eta$ is trivial for any $\eta\in\pi_0\nat(\mathbb{H}^n_A,\mathbb{H}^m_B)$.
\end{prop}
\begin{proof}
We have
\begin{align*}
\nat(\mathbb{H}_A^n,\mathbb{H}_B^m) \simeq \nat(\susp\deloop\suspfreexi H^n_A,\suspfreexi H^m_B)
\simeq\nat(P_m(\susp\deloop\suspfreexi H^n_A),\suspfreexi H^m_B).
\end{align*}
By Lemma~\ref{lemma:derivatives of suspension of bbH} and Corollary~\ref{corollary:vanishing of natural transformations, spectral}, we obtain that
$$\nat(D_i(\susp\deloop\suspfreexi H^n_A), \suspfreexi H^m_B) \simeq \ast $$
is contractible for $i\leq m-1$. Since $P_{m-1}(\susp\deloop\suspfreexi H^n_A)$ is a finite extension of the homogeneous functors $D_i(\susp\deloop\suspfreexi H^n_A)$ for $i\leq m-1$, we observe that
$$\nat(P_{m-1}(\susp\deloop\suspfreexi H^n_A), \suspfreexi H^m_B) \simeq \ast $$
is contractible and
$$\nat(P_m(\susp\deloop\suspfreexi H^n_A),\suspfreexi H^m_B) \simeq\nat(D_m(\susp\deloop\suspfreexi H^n_A),\suspfreexi H^m_B). $$
Again, by Lemma~\ref{lemma:derivatives of suspension of bbH}, the latter is trivial if $n\nmid m$. If $m=nk$, we use Proposition~\ref{proposition:suspention and natural transformations} to continue and we obtain
$$\nat(\mathbb{H}_A^n,\mathbb{H}_B^m) \simeq \nat\left(\Sigma^{-1}(\Sigma H_A^n)^{\otimes k}_{h\Sigma_k},H_B^m\right)\simeq \nat((\Sigma H_A^n)^{\otimes k}_{h\Sigma_k}, \Sigma H_B^m). $$

Under the obtained equivalence, the map~\eqref{equation:postcomposition with Omega} is induced by the precomposition with the suspension map $\sigma$ of~\eqref{equation:suspension map}:
\begin{align*}\nat((\Sigma H_A^n)^{\otimes k}_{h\Sigma_k},\Sigma H_B^m) \xrightarrow{\circ \sigma} \nat(\Sigma(H_A^n)^{\otimes k}_{h\Sigma_k},\Sigma H_B^m)  \simeq \nat((\Sigma H_{\Sigma^{-1}A}^n)^{\otimes k}_{h\Sigma_k},\Sigma H_{\Sigma^{-1}B}^m).
\end{align*}
By Lemma~\ref{lemma:suspension is trivial}, the map $\sigma$ is trivial if $k<p$, and so this implies the proposition.
\end{proof}

We finish the section with the following observation. Note that the homotopy groups $\pi_*(\mathbb{H}^n_A(V))$ are the (non-negative) homotopy groups $\pi_*(\suspfreexi H_A^n(V))$ of the spectrum object $\suspfreexi H_A^n(V)$. In particular, if $A\in \Mod_{\kk}^{+,B\Sigma_n}$ and $V \in \Mod_{\kk}^{+}$ are bounded below, then the homotopy groups of $\mathbb{H}^n_A(V)$ are bigraded by Remark~\ref{remark: second grading, xi-complete}. 

\begin{prop}\label{proposition:differential lowers the second grading}
Let $A\in \Mod_{\kk}^{+,B\Sigma_n}$ and $B\in \Mod_{\kk}^{+,B\Sigma_{pn}}$ be bounded below modules and let $\eta\in \pi_0\nat(\mathbb{H}_A^n,\mathbb{H}_B^{pn})$ be a natural transformation.  Then the induced map $\eta^V_*$ lowers the second grading by one, i.e.
$$\eta^V_*\left(\pi_{*,s}(\mathbb{H}^n_A(V))\right) \subseteq \pi_{*,s-1}(\mathbb{H}^{pn}_B(V)), \;\; V\in \Mod_{\kk}^{+}, \;\; s\geq 0.$$
\end{prop}
\begin{proof}
Consider $\pi_*(\mathbb{H}_A^n(V))$ as a representation of the multiplicative group $\kk^{\times}$ induced by the $\kk^\times$-action on $V$. Since the functor $H^n_A\in \Fun^{\omega}(\Mod_{\kk}^{+},\Mod_{\kk})$ has scalar degree $\chi_n\in \mathfrak{X}(\kk)$, we have an isomorphism  of $\kk^\times$-representations
$$\pi_*(\mathbb{H}_A^n(V)) \cong \prod_{s\geq 0} \triv(\pi_{*,s}\suspfreexi H_A^n(V))\otimes \chi_{np^s}, \;\; V\in \Mod_{\kk}^{+}.$$ 
By functoriality, the map $\eta^V_*$ is a map of $\kk^\times$-representations, which implies the proposition.
\end{proof}

\subsection{Snaith splitting}\label{section: snaith splitting}

\begin{prop}[Snaith splitting]\label{proposition: Snaith splitting} The Taylor tower for the composite $$F=\susp\deloop\suspfreexi\colon 
\Mod^{+}_{\kk} \to \Sp(\sLxi)$$ splits naturally as follows
$$P_n F \simeq D_nF \oplus P_{n-1} F, \;\; n\geq 1. $$
In particular, for every $n\geq 1$ and bounded below $V\in\Mod^{+}_{\kk}$, there is a natural equivalence:
$$ P_n F(V) \simeq \bigoplus_{j=0}^n D_j F(V)\simeq \bigoplus_{j=0}^{n}(\suspfreexi(V))^{\otimes j}_{h\Sigma_j}\simeq \bigoplus_{j=0}^{n}\suspfreexi(\Sigma^{-1}(\Sigma V)^{\otimes j}_{h\Sigma_j}).$$
Moreover, if $V$ is connective, i.e. $\pi_i(V)=0$ for $i<0$, then $F(V)\simeq \holim\limits_{n} P_nF(V)$. 
\end{prop}

\begin{proof}
By Lemma~\ref{lemma:derivatives of suspension of bbH}, we have a natural equivalence $$\Sigma D_nF(V)\simeq \suspfreexi((\Sigma V)^{\otimes n}_{h\Sigma_n}), \;\; V\in\Mod^{+}_{\kk}, \;\; n\geq 1.$$
By induction on $n$ and Corollary~\ref{corollary:vanishing of natural transformations, spectral}, we observe the desired splitting of the Taylor tower. 

Finally, if $\pi_i(V)=0$ for $i\leq -1$, then $\pi_i(D_nF(V))=0$ for $n\geq 1$ and $i\leq n-2$. This implies the last statement.
\end{proof}

For every $n\geq 1$ and $V\in\Mod^{+}_{\kk}$, Proposition~\ref{proposition: Snaith splitting} gives natural projections
\begin{align}\label{equation: snaith projections}
\Pi^V_n\colon \susp\deloop\suspfreexi(\Sigma^{-1}V) &\to P_n(\susp\deloop\suspfreexi)(\Sigma^{-1}V) \\ 
&\to D_n(\susp\deloop\suspfreexi)(\Sigma^{-1}V) \simeq \suspfreexi(\Sigma^{-1}V^{\otimes n}_{h\Sigma_n}). \nonumber
\end{align}

\begin{dfn}\label{definition:james-hopf map}
The \emph{James--Hopf map} 
$$j^V_n\colon \deloop\suspfreexi(\Sigma^{-1}V)\to \deloop\suspfreexi(\Sigma^{-1}V^{\otimes n}_{h\Sigma_n}), \;\; V\in \Mod^{+}_{\kk}, \;\; n\geq 1$$
is the natural transformation adjoint to the projection $\Pi^V_n$ of~\eqref{equation: snaith projections}.
\end{dfn}

Recall that the homotopy groups $\pi_*(\deloop\suspfreexi(\Sigma^{-1}V)), V\in \Mod^+_{\kk}$ are bigraded 
$$\pi_q(\deloop\suspfreexi(\Sigma^{-1}V)) \cong \prod_{s\geq 0} \pi_{q,s}(\suspfreexi(\Sigma^{-1}V)),\;\; q\geq 0,$$
see Remark~\ref{remark: second grading, xi-complete}. By Proposition~\ref{proposition:differential lowers the second grading}, we obtain the following statement.

\begin{prop}\label{proposition:james-hopf and primitives}
Let $V\in\Mod^{+}_{\kk}$ be a bounded below chain complex. Then
$$j^V_{p*}\left(\pi_{q,s}(\suspfreexi(\Sigma^{-1}V))\right) \subseteq \pi_{q,s-1}(\suspfreexi(\Sigma^{-1}V^{\otimes p}_{h\Sigma_p})), \;\; q,s\geq 0.  $$
\end{prop}


\subsection{Taylor tower of a free Lie algebra}\label{section: goodwillie tower of a free lie algebra}

In this section we compute the Goodwillie layers $D_n(\free)$ of the free simplicial restricted Lie functor $\free$ and we show that $\free$ is an analytic functor. This would follow essentially from the Hilton--Milnor theorem. 

Let $L_n=L(x_1,\ldots, x_n)$ be a free (isotropic) Lie algebra (i.e. $[x,x]=0$ for $x\in L$) over the integers generated by symbols $x_1,\ldots, x_n$ and let $B_n\subset L_n$ be the Hall basis for $L_n$, see e.g.~\cite[p.~512]{Whitehead_elements}. Then a word $w\in B_n$ is an iterated Lie bracket 
$$w=[x_{j_1}[x_{j_2},\ldots,x_{j_s}]], $$
where $j_t\in \{1,\ldots,n\}$, $1 \leq t\leq s$. Given an $n$-tuple of simplicial vector spaces $V_{1,\bullet}, V_{2,\bullet},\ldots, V_{n,\bullet}$, we associate to $w\in B_n$ the iterated tensor product
$$w(V_{1,\bullet},\ldots,V_{n,\bullet})=V_{j_1,\bullet}\otimes\ldots \otimes V_{j_s,\bullet}$$
and the canonical inclusion
\begin{equation}\label{equation: hilton-milnor, inclusions}
l_w\colon w(V_{1,\bullet},\ldots, V_{n,\bullet})\to \oblv \circ \free(V_{1,\bullet}\oplus \ldots \oplus V_{n,\bullet}) 
\end{equation}
such that $l_w(v^{j_1}_{1}\otimes \ldots \otimes v^{j_s}_s) = [v^{j_1}_{1}[v^{j_2}_2,\ldots,v^{j_s}_s]]$, where $v^{j_t}_t \in V_{j_t,\bullet}$. These maps determine the map
\begin{equation}\label{equation: hilton-milnor theorem}
\bigoplus_{w\in B_n} \oblv \circ \free(w(V_{1,\bullet},\ldots, V_{n,\bullet})) \xrightarrow{\cong} \oblv \circ \free(V_{1,\bullet}\oplus \ldots \oplus V_{n,\bullet}), 
\end{equation}
which is an isomorphism of simplicial vector spaces, see e.g.~\cite[Example~8.7.4]{Neisendorfer10}.

We write $|w|, w\in B_n$ for the cardinality of the minimal subset $A\subset \{x_1,\ldots, x_n\}$ such that $w\in L(A)\subset L(x_1,\ldots,x_n)$. We note that $1\leq |w|\leq n$ and $|w|=1$ if and only if $w=x_i$ for some $1 \leq i\leq n$. The next corollary follows immediately from the definition of the cross-effect and the isomorphism~\eqref{equation: hilton-milnor theorem}.

\begin{cor}\label{corollary:cross-effect} Let $V_{1,\bullet},\ldots,V_{n,\bullet}\in \sVect_{\kk}$ be simplicial vector spaces. Then the map 
$$\oblv \circ\creff_n(\free)(V_{1,\bullet},\ldots, V_{n,\bullet})\to \oblv \circ \free(V_{1,\bullet}\oplus\ldots V_{n,\bullet}) \in \Mod_{\kk}$$
splits naturally as a direct summand and there exists a natural equivalence
$$\oblv \circ \creff_n(\free)(V_{1,\bullet},\ldots, V_{n,\bullet}) \simeq \bigoplus_{\substack{w\in B_n, \\ |w| = n}}\oblv\circ \free(w(V_{1,\bullet}, \ldots, V_{n,\bullet})) $$
compatible with the isomorphism~\eqref{equation: hilton-milnor theorem}.
Here $\creff_n(\free)\colon (\Mod_{\kk}^{\geq 0})^{\times n}\to \sL$ is the $n$-th cross-effect of the functor $\free\colon \Mod_{\kk}^{\geq 0} \to \sL$. \qed
\end{cor}

Recall that the $n$-th space of the Lie operad $\bfLie_n\in \Vect_{\kk}$ is the $\kk$-vector space spanned by all $w\in B_n$ of length $n$. Therefore $\bfLie_n$ is a linear subspace of $L_n\otimes \kk$. The symmetric group $\Sigma_n$ acts on $L_n\otimes \kk = L(x_1,\ldots,x_n)\otimes \kk$ by permuting $x_1,\ldots, x_n$ and this action induces the $\Sigma_n$-action on $\bfLie_n$.

Let us denote by $l(w)$ the length of $w\in B_n$. For any $w\in B_n$ such that $|w|=l(w)=n$ and any simplicial vector spaces $V_{1,\bullet},\ldots, V_{n,\bullet}$ we have natural isomorphisms:
$$w(V_{1,\bullet},\ldots, V_{n,\bullet})\cong V_{1,\bullet} \otimes \ldots \otimes V_{n,\bullet}. $$
Using these isomorphisms and the canonical inclusions~\eqref{equation: hilton-milnor, inclusions}, we obtain the map:
$$l\colon \free(\bfLie_n \otimes V_{1,\bullet}\otimes \ldots \otimes V_{n,\bullet}) \to \free(V_{1,\bullet}\oplus \ldots\oplus V_{n,\bullet}) $$
of simplicial restricted Lie algebras. We note that the map $l$ is $\Sigma_n$-equivariant, where the group $\Sigma_n$ acts by permuting the spaces $V_{1,\bullet},\ldots, V_{n,\bullet}$ on both sides and it acts on $\bfLie_n$ by the canonical action above.

\begin{prop}\label{proposition: computation of Dn} Let $V_{1,\bullet},\ldots,V_{n,\bullet}\in \sVect_{\kk}$ be simplicial vector spaces. Then the map $l$ factorizes through the map 
$$\bar{l}\colon \free(\bfLie_n \otimes V_{1,\bullet}\otimes \ldots \otimes V_{n,\bullet}) \to \creff_n(\free)(V_{1,\bullet},\ldots, V_{n,\bullet}), $$
and if the vector spaces $V_{1,\bullet},\ldots, V_{n,\bullet}$ are $k$-connected, then the map $\bar{l}$ is $(n+1)k$-connected.
\end{prop}

\begin{proof}
By the construction, the map $l$ is adjoint to the canonical inclusion 
$$\bfLie_n\otimes V_{1,\bullet}\otimes \ldots \otimes V_{n,\bullet} \to \oblv \circ \free(V_{1,\bullet}\oplus \ldots \oplus V_{n,\bullet}),$$
and by Corollary~\ref{corollary:cross-effect}, its image lies in the direct summand $\oblv \circ \creff_n(\free)$. Moreover, by the same corollary, the cofiber of the map $\bar{l}$ is equivalent to
$$\bigoplus_{\substack{w\in B_n, |w| = n, \\ l(w)\geq n+1}}\oblv\circ \free(w(V_{1,\bullet}, \ldots, V_{n,\bullet})), $$
which implies the connectivity statement.
\end{proof}

Now, we are able to identify the Goodwillie layers $D_n(\free)$ of the free simplicial restricted Lie algebra functor $$\free\colon \Mod_{\kk}^{\geq 0} \to \sL.$$ We recall that $$D_n(\free)(-)\simeq \deloop \D_n(\free)(\tau_{\geq 0}-),$$ where $\D_n(\free)$ is the functor from $\Mod_{\kk}$ to the stabilization $\Sp(\sL)$, see e.g.~\cite[Corollary~6.1.2.9]{HigherAlgebra}. By~\cite[Proposition~6.1.3.6]{HigherAlgebra}, the multi-linearization $P_{\vec{1}}\creff_n(\free)$ of the $n$-th cross-effect $\creff_n(\free)\colon (\Mod_{\kk}^{\geq 0})^{\times n} \to \sL$ is also equipped with a natural delooping:
$$P_{\vec{1}}\creff_n(\free) \simeq \deloop \Delta_n(\free)(\tau_{\geq 0}-), \;\; \Delta_n(\free)\in \SymLin{n}(\Mod_{\kk},\Sp(\sL)). $$
By~\cite[Proposition~6.1.4.14]{HigherAlgebra}, there is a natural equivalence $$\D_n(\free)(V) \simeq \Delta_n(\free)(V,\ldots,V)_{h\Sigma_n}, \;\; V\in \Mod_{\kk}$$ 
Finally, Proposition~\ref{proposition: computation of Dn} implies immediately the following corollary.
\begin{cor}\label{corollary:derivatives of free} There exists a natural equivalence:
\begin{equation*}
\D_n(\free)(V)\simeq \suspfree(\bfLie_n \otimes_{h\Sigma_n} V^{\otimes n}), \;\; V\in \Mod_{\kk}. 
\end{equation*}
\end{cor}

The Taylor tower of the functor $\free\colon \Mod_{\kk}^{\geq 0} \to \sL$ is the following tower of fibrations
\begin{equation}\label{equation: algebraic tower}
\begin{tikzcd}
&
D_3(\free) \arrow[d]
&D_2(\free)\arrow[d]
&D_1(\free)\arrow[d, equal]\\
\ldots \arrow[r]
& P_3(\free) \arrow[r]
& P_2(\free) \arrow[r]
& P_1(\free).
\end{tikzcd}
\end{equation}
We consider the induced spectral sequence
\begin{equation}\label{equation: AGSS}
E^1_{t,n}(V)=\pi_t(D_n(\free)(V)) \Rightarrow \pi_t\free(V), \;\; V\in \Mod_{\kk}^{\geq 0}
\end{equation}
with the differentials $d_r\colon E^r_{t,n}(V)\to E^r_{t-1,n+r}(V)$, $r\geq 1$.

\begin{cor}\label{corollary: exponential vanishing}
Let $V\in \Mod_{\kk}^{\geq 0}$. Then for any $x\in E^1_{t,n}(V)$, we have 
$$d_1(x)=d_2(x)=\ldots = d_{(p-1)n-1}(x)=0. $$
\end{cor}

In other words, the first differential which might act non-trivially on $E^1_{*,n}(V)$ lands on $E^1_{*,pn}(V)$.

\begin{proof}
Consider $E^1_{*,n}(V)$, $n\geq 1$ as a representation of the multiplicative group $\kk^{\times}$ induced by the $\kk^{\times}$-action on $V$. Then, by Corollary~\ref{corollary:derivatives of free} and Remark~\ref{remark: second grading}, we have isomorphisms of $\kk^{\times}$-representations
$$E^1_{*,n}(V)\cong \bigoplus_{s\geq 0}\triv(\pi_{*,s}\suspfree(\mathbf{L}_n(V))) \otimes \chi_{np^s}, \;\; n\geq 1,  $$
where $\mathbf{L}_n(V)=\bfLie_n\otimes_{h\Sigma_n}V^{\otimes n}$. By functoriality, the differentials are maps of $\kk^\times$-representations, which implies the corollary. 
\end{proof}

Note that Corollary~\ref{corollary:derivatives of free} and Remark~\ref{remark: second grading} imply that the spectral sequence \eqref{equation: AGSS} is \emph{trigraded}
$$E^1_{t,n}(V) = \bigoplus_{s\geq 0} E^1_{t,n,s}(V) \cong \bigoplus_{s\geq 0}\pi_{t,s}(\suspfree(\mathbf{L}_n(V)), \;\; t\geq 0,\; n\geq 1. $$
\begin{cor}\label{corollary: agss is trigraded}
Let $V\in \Mod_{\kk}^{\geq 0}$. Then the first non-trivial differential $$d_{(p-1)n}\colon E^1_{t,n}(V) \to E^1_{t-1,pn}(V)$$ lowers the third degree by one, i.e. $$d_{(p-1)n}(E^1_{t,n,s}(V))\subseteq E^1_{t-1,pn,s-1}(V), \;\; t,s\geq 0, \;n\geq 1.$$
\end{cor}

\begin{proof}
See the proof of Corollary~\ref{corollary: exponential vanishing}.
\end{proof}

Next, we use the description of the Goodwillie layers above and the vanishing results of Section~\ref{section: vanishing theorems} to show that the Taylor tower for the $\kk$-complete functor $L_\xi\free$ has certain splitting properties, which rigidify the vanishing of Corollary~\ref{corollary: exponential vanishing}.

\begin{lmm}\label{lemma: completion and derivatives}
The natural transformations
$$P_n(L_\xi \free)(V) \xrightarrow{\simeq} L_\xi (P_n(\free)(V)), \;\; n\geq 1, $$
$$D_n(L_\xi \free)(V) \xrightarrow{\simeq} L_\xi (D_n(\free)(V)), \;\; n\geq 1 $$
are equivalences for any connected $V\in \Mod^{\geq 1}_{\kk}$. Moreover, the natural transformation
$$\D_n(L_\xi \free)(V) \xrightarrow{\simeq} L^{st}_\xi \D_n(\free)(V) $$
is an equivalence for any bounded below $V\in \Mod_{\kk}^{+}$.
\end{lmm}

\begin{proof}
Follows directly by Corollary~\ref{corollary: xi-completion is left exact} and Proposition~\ref{proposition: free xi-complete}.
\end{proof}

For the rest of the paper, we consider $L_\xi \free\colon \Mod^{\geq 1}_{\kk} \to \sLxi$ as a functor from the category of \emph{connected} chain complexes.

Let $F\colon \Mod_{\kk}^{\geq 1} \to \sLxi$ be any functor. We set $P_{[n,i]}(F)$ to be the fiber of the projection map $P_i(F) \to P_{n-1}(F)$:
\begin{equation}\label{equation: thick layer}
P_{[n;i]}(F)=\fib(P_iF \to P_{n-1}F),
\end{equation}
where $i\geq n\geq 1$.

\begin{thm}\label{theorem:splitting}
Assume that $n\leq i \leq pn-1$. Then the projection map $$p^i_n\colon P_{[n;i]}(L_\xi\free) \to D_n(L_\xi\free)$$ admits a splitting after applying $\Omega^{l-1}$ where $l=\left\lfloor \frac{i}{n} \right\rfloor$, i.e. there is a map
$$s_{n}^{i}\colon D_{n}(\Omega^{l-1}L_\xi\free)\simeq \Omega^{l-1}D_n(L_\xi\free) \to \Omega^{l-1}P_{[n;i]}(L_\xi\free)\simeq P_{[n;i]}(\Omega^{l-1}L_\xi\free)$$
such that $\Omega^{l-1}p^i_n\circ s_n^i \simeq \id$.
\end{thm}

\begin{proof}
We proceed by induction on $i$. If $i=n$, then $p^n_n=\id$ and we are done. Assume that the theorem is true for $i=ln+r$ such that $0\leq r<n$ and $l<p$, we will prove it for $i+1$ as well. Consider the diagram
$$
\begin{tikzcd}
&P_{[n;i+1]}(\Omega^{l-1}L_\xi\free) \arrow{d}{\Omega^{l-1}p_{i+1}} \\
D_n(\Omega^{l-1}L_\xi\free) \arrow{r}{s^i_n} 
&P_{[n;i]}(\Omega^{l-1}L_\xi\free) \arrow{r}{\Omega^{l-1}\delta} 
&BD_{i+1}(\Omega^{l-1}L_\xi\free),
\end{tikzcd}
$$
where $s^i_n$ is a section of $\Omega^{l-1}p_n^i$, which exists by the assumption, and $\delta$ is the classifying map for the principal fibration
$$D_{i+1}(L_\xi\free) \to P_{[n;i+1]}(L_\xi\free) \xrightarrow{p_{i+1}} P_{[n;i]}(L_\xi\free). $$
It is enough to construct a lift $s^{i+1}_n\colon D_n(\Omega^{l-1}L_\xi\free) \to P_{[n;i+1]}(\Omega^{l-1}L_\xi\free)$ of $s^i_n$ along $\Omega^{l-1}p_{i+1}$. We note that the composition $\Omega^{l-1}\delta\circ s^i_n$ is the obstruction for existing a lift $s^{i+1}_n$. 

Consider two cases. Assume first that $r<n-1$, so $n\nmid i+1$ and $l=\left\lfloor\frac{i}{n}\right\rfloor=\left\lfloor\frac{i+1}{n}\right\rfloor$. Then 
$$\Omega^{l-1}\delta\circ s^i_n \in \pi_0\nat(D_n(\Omega^{l-1}L_\xi\free),BD_{i+1}(\Omega^{l-1}L_\xi\free)).$$
Since $n\nmid i+1$, the latter space is contractible by Corollary~\ref{corollary:derivatives of free} and Proposition~\ref{proposition:maps between homogeneous functors}.

The second case is $r=n-1$. Then $n\mid i+1$, $i+1=(l+1)n$. In this case, a lift $s^{i+1}_n$ of $s^i_n$ could not exist, however there exists a lift $$s^{i+1}_n\colon D_n(\Omega^{l}L_\xi\free) \to P_{[n;i+1]}(\Omega^l L_\xi\free)$$ of the map $\Omega s^i_n$. Indeed, the obstruction for lifting $\Omega s^i_n$ is 
$$\Omega\circ(\Omega^{l-1}\delta \circ s^i_n)\simeq\Omega^l\delta\circ\Omega s^i_n \in \pi_0\nat(D_n(\Omega^{l}L_\xi\free),BD_{i+1}(\Omega^{l}L_\xi\free)). $$
By Proposition~\ref{proposition:maps between homogeneous functors}, this natural transformation is trivial, which proves the theorem.
\end{proof}

\begin{cor}\label{corollary: splitting}
The Taylor tower of the functor $$P_{[n;pn-1]}(L_\xi\free)\colon \Mod_{\kk}^{\geq 0} \to \sLxi$$ splits after applying the $(p-2)$-th fold loop functor $\Omega^{p-2}$. \qed
\end{cor}

\begin{rmk}\label{remark:loops and rational picture}

Theorem~\ref{theorem:splitting} is inspired by the following observation made in the characteristic zero case. Indeed, let $dg\EuScript{L}_0$ be the $\infty$-category of connected differential graded Lie algebras over the field of rational numbers $\Q$. Consider the free Lie algebra functor:
$$L\colon \Mod_{\Q}^{\geq 1} \to dg\EuScript{L}_0. $$
Then, by~\cite[Proposition~1.7.2, Chapter~6]{GR_volumeII}, the composite $\Omega\circ L$ is naturally equivalent to the trivial dg-Lie algebra $\Sigma^{-1}\oblv \circ L$. This implies that the Taylor tower of the functor $\Omega L$ splits. 

In the case of positive characteristic, we expect that the Taylor tower of the functor $$\Omega P_{[n;pn-1]}(\free)\colon \Mod_{\kk}^{\geq 0} \to \sL$$
splits, which is an improvement of Theorem~\ref{theorem:splitting}. Indeed, this expected splitting will support the vanishing of the intermediate differentials in Corollary~\ref{corollary: exponential vanishing} and it will improve the error terms in Theorem~\ref{theorem: the adjoint of delta_n} and Proposition~\ref{proposition: james-hopf and differential, reserve}.
\end{rmk}

\begin{rmk}\label{remark: completed agss}
The Taylor tower of the functor $L_\xi\free\colon\Mod_{\kk}^{\geq 1} \to \sLxi$ induces the spectral sequence
\begin{equation}\label{equation: completed agss}
\widehat{E}^1_{t,n}(V) = \pi_t(D_n(L_\xi\free)(V))\Rightarrow \pi_t(L_\xi\free(V)), \;\; V\in \Mod_{\kk}^{\geq 1}
\end{equation}
with the differentials $\hat{d}_r\colon \widehat{E}^r_{t,n}(V) \to \widehat{E}^r_{t-1,n+r}(V)$, $r\geq 1$. By Lemma~\ref{lemma: completion and derivatives}, Corollary~\ref{corollary:derivatives of free}, and Theorem~\ref{theorem: stable homotopy groups of a free object}, we obtain
$$\widehat{E}^1_{t,n}(V) \cong {E}^1_{t,n}(V)^{\wedge}_{\xi}, \;\; t,n\geq 1. $$
Here ${E}^1_{t,n}(V)^{\wedge}_{\xi}$ is the $\xi$-adically completed first page of the spectral sequence~\eqref{equation: AGSS}. As in Corollary~\ref{corollary: agss is trigraded}, we have
$$\widehat{E}^1_{t,n}(V) = \prod_{s\geq 0} E^1_{t,n,s}(V) \cong \prod_{s\geq 0}\pi_{t,s}(\suspfree(\mathbf{L}_n(V)), \;\; t,n\geq 1. $$
Moreover, arguing as in Corollary~\ref{corollary: exponential vanishing}, we obtain
$$\hat{d}_1(x)=\hat{d}_2(x)=\ldots = \hat{d}_{(p-1)n-1}(x)=0 $$
for any $x\in \widehat{E}^1_{t,n}(V)$. We note that the first non-trivial differential $\hat{d}_{(p-1)n}$ in the spectral sequence~\eqref{equation: completed agss} contains the same amount of information as the first non-trivial differential $d_{(p-1)n}$ in the spectral sequence~\eqref{equation: AGSS}. Namely, we have
$$\hat{d}_{(p-1)n} = \prod_{s\geq 0} d_{(p-1)n}\big|_{E^1_{t,n,s}(V)}\colon \prod_{s\geq 0} E^1_{t,n,s}(V) \to \prod_{s\geq 1} E^1_{t-1,pn,s-1}(V),\;\; t,n\geq 1.$$
Here the restricted differential $d_{(p-1)n}\big|_{E^1_{t,n,s}(V)}$ is considered as a map from the component $E^1_{t,n,s}(V)$ to $E^1_{t-1,pn,s-1}$, see Corollary~\ref{corollary: agss is trigraded}. 

Due to a little difference between the spectral sequences~\eqref{equation: AGSS} and~\eqref{equation: completed agss}, we will use them interchangeably for the rest of the paper.
\end{rmk}

Finally, we relate the first non-trivial differential $d_{(p-1)n}\colon E^1_{t,n}(V) \to E^1_{t-1,pn}(V)$ with the James--Hopf map, see Definition~\ref{definition:james-hopf map}. Fix a section $$s_n^{pn-1}\colon D_n(\Omega^{p-2} L_\xi\free) \to P_{[n,pn-1]}(\Omega^{p-2} L_\xi\free)$$
of $p^{pn-1}_n$ provided by Theorem~\ref{theorem:splitting}. Let us denote by $\widetilde{P}_n(\Omega^{p-2}L_\xi\free)$ the pullback of the following diagram
$$
\begin{tikzcd}
\widetilde{P}_n(\Omega^{p-2}L_\xi\free) \arrow{r} \arrow{d} 
& P_{[pn,n]}(\Omega^{p-2}L_\xi\free) \arrow{d} \\
D_n(\Omega^{p-2}L_\xi\free) \arrow{r}{s_n^{pn-1}}  
& P_{[pn-1,n]}(\Omega^{p-2}L_\xi\free). 
\end{tikzcd}
$$
In particular, we have the following fiber sequence
\begin{equation}\label{equation:Dpk to Dk}
\widetilde{P}_{n}(\Omega^{p-2}L_\xi\free) \to D_{n}(\Omega^{p-2}L_\xi\free) \xrightarrow{\delta_n} BD_{pn}(\Omega^{p-2}L_\xi\free).
\end{equation}
Note that the differential $d_{(p-1)n}$ is induced by the classifying map $\delta_n^V$. The natural transformation 
$$\delta_n\colon \deloop \D_n(\Omega^{p-2}L_\xi\free) \to B \deloop \D_{pn}(\Omega^{p-2}L_\xi\free)$$ 
has the adjoint map 
$$\tilde{\delta}_n\colon \Sigma^\infty\Omega^\infty\D_n(\Omega^{p-2}L_\xi\free) \to \Sigma\D_{pn}(\Omega^{p-2}L_\xi\free).$$

Recall that $\mathbf{L}_n(V)=\bfLie_n\otimes_{h\Sigma_n}V^{\otimes n}\in \Mod_{\kk}$, $V\in \Mod_{\kk}$, see~\eqref{equation: lie powers}. Then, by Lemma~\ref{lemma:derivatives of suspension of bbH} and Corollary~\ref{corollary:derivatives of free}, we have the induced map $$\D_{pn}(\tilde{\delta}_n) \colon \suspfreexi(\Sigma^{-1}(\Sigma^{3-p}\mathbf{L}_n(V))^{\otimes p}_{h\Sigma_p}) \to \suspfreexi(\Sigma^{3-p}\mathbf{L}_{pn}(V)), \;\; V\in\Mod_{\kk}^{+}.$$

\begin{prop}\label{proposition:james-hopf and adjoint}
There is a natural equivalence $$\delta^{V}_n\simeq \deloop\D_{pn}(\tilde\delta^{V}_n) \circ j_p^{\Sigma^{3-p}\mathbf{L}_n(V)},\;\; V\in \Mod^{\geq 1}_{\kk},\;\; n\geq 1,$$
where $j_p^{\Sigma^{3-p}\mathbf{L}_n(V)}$ is the James--Hopf map 
$$j_p^{\Sigma^{3-p}\mathbf{L}_n(V)}\colon \deloop\suspfreexi(\Sigma^{2-p}\mathbf{L}_n(V)) \to \deloop\suspfreexi(\Sigma^{-1}(\Sigma^{3-p}\mathbf{L}_n(V))^{\otimes p}_{h\Sigma_p}),$$
see Definition~\ref{definition:james-hopf map}.
\end{prop}

\begin{proof}
By Proposition~\ref{proposition: Snaith splitting} and Corollary~\ref{corollary:derivatives of free}, we can rewrite the adjoint map $\tilde{\delta}^V_n$ as follows
$$\tilde{\delta}^V_n\colon\bigoplus_{j\geq 1} \suspfreexi(\Sigma^{-1}(\Sigma^{3-p}\mathbf{L}_n(V))^{\otimes j}_{h\Sigma_j}) \to \suspfreexi(\Sigma^{3-p}\mathbf{L}_{pn}V). $$
By Corollary~\ref{corollary:vanishing of natural transformations, spectral}, $\tilde{\delta}^V_n$ factorizes through the projection $\Pi_{p}$, see~\eqref{equation: snaith projections}, which implies the proposition.
\end{proof}

\begin{rmk}\label{remark: no lower components in spaces}
There is a partial counterpart of Proposition~\ref{proposition:james-hopf and adjoint} in the category of spaces. We recall that the Goodwillie layers $D_n(\id_{\spaces})(S^1)$ are $p$-complete contractible if $n\neq p^h$ for some $h\geq 0$. According to~\cite[Theorem 6.1]{ADL08}, the connecting maps 
$$d_h\colon D_{p^h}(\id_{\spaces})(S^1)\to BD_{p^{h+1}}(\id_{\spaces})(S^1),\;\; h\geq 0$$ in the Taylor tower for $\id_{\spaces}$ evaluated at $S^1$ can be delooped $h$ times, i.e. $d_h\simeq\Omega^h\psi_h$ for certain maps $\psi_h$. Moreover, there are equivalences $$B^hD_{p^h}(\id_{\spaces})(S^1)\simeq \deloop\susp Z_h, \;\; h\geq 0$$ for certain path connected spaces $Z_h\in\spaces $. N.~Kuhn~\cite[Section 3.2]{Kuhn15} and M.~Behrens~\cite[Lemma~4.2]{Behrens11} showed that the maps $\psi_h,h\geq 0$ factorize as the composite of the $p$-th James Hopf map 
$$\deloop\susp Z_h\to \deloop (\susp Z_h)^{\otimes p}_{h\Sigma_p}$$ and a certain map $$\deloop (\susp Z_{h})^{\otimes p}_{h\Sigma_p} \to \deloop \susp Z_{h+1}, $$
which is induced by a map of suspension spectra. We note that this factorization resembles the factorization in Proposition~\ref{proposition:james-hopf and adjoint}.
\end{rmk}

\section{James--Hopf map}\label{section: james-hopf map}

In this section we evaluate the James--Hopf map on the generators of the algebra~$\Lambda$, see Theorem~\ref{definition:james-hopf map}. We explain the relevant topological result about the elements of Hopf invariant one in Remark~\ref{remark: james-hopf, informal}.

In Section~\ref{section: group cohomology} we recall the basis of \emph{Dyer--Lashof operations} in the \emph{Tate cohomology groups} $\widehat{H}^*(C_p,V^{\otimes p})$. The main reference for this section is~\cite{CLM76}.

In Section~\ref{section: tate diagonal} we define the \emph{Tate diagonal} $\Delta_p$ in the stable category $\Sp(\sLxi)$, see~\eqref{equation: tate diagonal}. In Proposition~\ref{proposition:tate diagonal for spectra} we compute the induced map $\Delta_{p*}$ on the homology groups in terms of the Steenrod and Dyer--Lashof operations. In Proposition~\ref{proposition:infinite loops of a map between suspension spectra}, we explain how to compute the map $\widetilde{H}_*(\deloop f)$ induced by a map $f\colon \suspfreexi W \to \suspfreexi V$ of suspension spectrum objects.

In Section~\ref{section: unstable hurewicz homomorphism} we relate the James--Hopf map with the \emph{unstable Hurewicz homomorphism}. We evaluate the latter on the generators of the algebra~$\Lambda$ in Proposition~\ref{proposition: hurewicz of hopf invariant one} by using Proposition~\ref{proposition:infinite loops of a map between suspension spectra}.

\subsection{Group cohomology}\label{section: group cohomology} Let $C_p$ be a cyclic group of order $p$ generated by an element $\zeta_p\in C_p$. We recall that $H^1(C_p,\kk)\cong \Hom(C_p,\kk)$ and we will denote by $u\in H^1(C_p,\kk)$ the cohomology class corresponding to the group homomorphism which sends $\zeta_p\in C_p$ to $1\in \kk$. Then the cohomology ring $H^*(C_p,\kk)$ of the group~$C_p$ is
$$ 
H^*(C_p, \kk)= \pi_{-*}\kk^{hC_p}\cong 
\begin{cases}
\kk[t]\otimes \Lambda(u) & \mbox{if $p$ is odd,}\\
\kk[u] & \mbox{if $p=2$.}
\end{cases}
$$
Here $t=\beta u\in H^2(C_p,\kk)$, $\beta$ is the Bockstein homomorphism, and $\Lambda(u)$ is the free exterior algebra generated by $u$. If $p=2$, then we set $t=u^2$.


Recall that there is the \emph{norm map}
$$N\colon M_{hC_p} \to M^{hC_p},\; M\in \Mod^{BC_p}_{\kk},$$
see e.g.~\cite[Example~6.1.6.22]{HigherAlgebra}, and let us denote its cofiber by $M^{tC_p}$. The \emph{Tate cohomology} $\widehat{H}^*(C_p,M)$ of $C_p$ with coefficients in $M$ are defined as follows
$$\widehat{H}^*(C_p,M)=\pi_{-*}M^{tC_p}. $$ 
Note that there are canonical isomorphisms 
\begin{equation}\label{equation:tate isomorphism1}
\widehat{H}^q(C_p,\kk)\cong {H}^q(C_p,\kk), \; q\geq 0,
\end{equation}
and 
\begin{equation}\label{equation:tate isomorphism2}
\widehat{H}^q(C_p,\kk)\cong H_{-q-1}(C_p,\kk)=\pi_{-q-1}\kk_{hC_p},\; q<0.
\end{equation}

We recall that the functors $(-)^{hC_p}, (-)^{tC_p}\colon \Mod_{\kk}^{BC_p} \to \Mod_{\kk}$ are lax symmetric monoidal and the canonical map
$$\mathrm{can}\colon (-)^{hC_p} \to (-)^{tC_p} $$
is a natural transformation of lax symmetric monoidal functors, see~e.g.~\cite[Section~I.3]{NS18}. In particular, the Tate cohomology $\widehat{H}^*(C_p,\kk)$ is a ring and the canonical map $H^*(C_p,\kk) \to \widehat{H}^*(C_p,\kk)$ is a ring homomorphism. We have
$$ 
\widehat{H}^*(C_p, \kk)\cong
\begin{cases}
\kk[t,t^{-1}]\otimes \Lambda(u) & \mbox{if $p$ is odd,}\\
\kk[u, u^{-1}] & \mbox{if $p=2$.}
\end{cases}
$$
In particular, for every $q\in \Z$, the cup-product in the Tate cohomology $\widehat{H}^*(C_p,\kk)$ defines the pairing
\begin{equation}\label{equation:pairing in tate cohomology}
\widehat{H}^q(C_p,\kk)\otimes \widehat{H}^{-q-1}(C_p,\kk) \xrightarrow{-\smile -} \widehat{H}^{-1}(C_p,\kk)\cong H_0(C_p,\kk)=\kk
\end{equation}
between vector spaces $\widehat{H}^q(C_p,\kk)$ and $\widehat{H}^{-q-1}(C_p,\kk)$.

\begin{prop}\label{proposition:tate and kronecker pairings}
The pairing~\eqref{equation:pairing in tate cohomology} is non-degenerate and coincides with the Kronecker pairing between cohomology and homology classes under the isomorphisms~\eqref{equation:tate isomorphism1} and~\eqref{equation:tate isomorphism2}. \qed
\end{prop}


\begin{rmk}\label{remark:notation for tate duality}
Let $a\in \widehat{H}^{-q-1}(C_p,\kk)$, $q\in \Z$. We write $a^{\vee}\in \widehat{H}^q(C_p,\kk)$ for the dual cohomology class with respect to the non-degenerate pairing~\eqref{equation:pairing in tate cohomology}. By Proposition~\ref{proposition:tate and kronecker pairings}, this notation is consistent with the standard notation for the dual vector under the isomorphism
$$H^q(C_p,\kk) \cong H_q(C_p,\kk)^*. $$
Moreover, for a non-zero $a\in \widehat{H}^{-q-1}(C_p,\kk)$, we have
$$ 
a^{\vee} =
\begin{cases}
a^{-1}ut^{-1} & \mbox{if $p$ is odd,}\\
a^{-1}u^{-1} & \mbox{if $p=2$.}
\end{cases}
$$
\end{rmk}

More generally, 
for each $V\in \Mod_{\kk}$ and $q\in \Z$, we define natural Dyer--Lashof homomorphisms 
\begin{align}\label{equation:tate-dyer-lashof}
\beta^{\e }Q^i&\colon \pi_q(V)^{(1)} \to \pi_{q+2(p-1)i+1-\e}(V^{\otimes p})^{tC_p} \;\; \text{if $p$ is odd and} \\
Q^i&\colon \pi_q(V)^{(1)} \to \pi_{q+i+1}(V^{\otimes 2})^{tC_2} \;\; \text{if $p=2$}, \nonumber
\end{align}
where $i\in \Z$ and $\e\in\{0,1\}$ which generalize the ordinary Dyer--Lashof operations~\cite[Section~I.1]{CLM76}. 
Namely, any element $v\in \pi_q(V)$ defines the map $v\colon \Sigma^q\kk \to V$ in the category $\Mod_{\kk}$, and so it induces the map
$$v^{\otimes p}\colon \kk^{tC_p} \otimes \Sigma^{pq}\kk \to \kk^{tC_p} \otimes ((\Sigma^q\kk)^{\otimes p})^{tC_p} \to ((\Sigma^q\kk)^{\otimes p})^{tC_p} \xrightarrow{v} (V^{\otimes p})^{tC_p}. $$
Then we define the operations $Q_i\colon \pi_q(V)^{(1)} \to \pi_{pq+i+1}(V^{\otimes p})^{tC_p}$, $i\in \Z$ as follows
$$
Q_{i}(v) = 
\begin{cases}
v^{\otimes p}_*(ut^{-j-1}\otimes \iota_{pq}) & \mbox{if $p$ is odd and $i=2j$,}\\
v^{\otimes p}_*(t^{-j-1}\otimes \iota_{pq}) & \mbox{if $p$ is odd and $i=2j+1$,}\\
v^{\otimes 2}_*(u^{-i-1}\otimes \iota_{2q}) & \mbox{if $p=2$,}
\end{cases}
$$
where $\iota_{pq}\in \pi_{pq}(\Sigma^{pq}\kk)$ is the canonical generator.
Finally, for $v\in \pi_q(V)$, $i\in \Z$, and $\e\in\{0,1\}$, we set 
$$
\beta^{\e}Q^{i}(v) = 
\begin{cases}
(-1)^i \nu(q) Q_{(2i-q)(p-1)-\e}(v) & \mbox{if $p$ is odd,}\\
Q_{i-q}(v) & \mbox{if $p=2$,}
\end{cases}
$$
where $\nu(q)=(-1)^{q(q-1)m/2}(m!)^q$, $m=(p-1)/2$. 

\begin{prop}\label{proposition:tate-dyer-lashof operations}
The operations~\eqref{equation:tate-dyer-lashof} satisfy following properties:
\begin{enumerate}
\item $\beta^\e Q^i(\alpha v)=\alpha^p \beta^\e Q^i(v)$, $\alpha\in \kk$, $v\in \pi_q(V)$;

\item the class $\beta^{\e}Q^i(v), v\in \pi_q(V)$ lies in the image of $$\pi_q(V^{\otimes p})^{t\Sigma_p}\cong (\pi_q(V^{\otimes p})^{tC_p})^{C_{p-1}}\hookrightarrow \pi_q(V^{\otimes p})^{tC_p};$$

\item $\beta^{\e}Q^i(\sigma v)=(-1)^{\e}\sigma \beta^{\e}Q^i(v), v\in \pi_q(V)$, where $\sigma\colon \pi_{*}(V) \to \pi_{*+1}(\Sigma V)$ is the suspension isomorphism.

\item the image of $\beta^{\e}Q^i(v)\in \pi_{*}(V^{\otimes p})^{tC_p}, v\in \pi_q(V)$ in $\pi_{q+2i(p-1)-\e}(V^{\otimes p})_{hC_p}$ (in $\pi_{q+i}(V^{\otimes 2})_{hC_2}$ if $p=2$) under the canonical map
$$\mathrm{can}\colon (V^{\otimes p})^{tC_p} \to \Sigma V^{\otimes p}_{hC_p} $$
coincides with the Dyer--Lashof operation $\beta^{\e}Q^{i}(v)$ of~\cite[Section~I.1]{CLM76}. In particular, the class $\beta^{\e}Q^i(v) \in \pi_*(V^{\otimes p})^{tC_p}$ maps to zero if $2i-\e<q$ (if $i<q$ and $p=2$).
\end{enumerate}
\end{prop}

\begin{proof}
The proof is identical to the proof of the corresponding properties for the standard Dyer--Lashof operations given in~\cite[Section~I.1]{CLM76}.
\end{proof}

\begin{prop}\label{propostion: tate-dyer-lashof is a basis}
Suppose that $V\in \Mod_{\kk}$ is perfect, i.e. $\dim\pi_*(V)<\infty$. Then the Dyer--Lashof classes $\beta^{\e}Q^i(v)\in \pi_*(V^{\otimes p})^{tC_p}$, $v\in \pi_*(V)$, $i\in\Z$, $\e\in\{0,1\}$ form a basis of the subspace $\pi_*(V^{\otimes p})^{t\Sigma_p}\subset\pi_*(V^{\otimes p})^{tC_p}$.
\end{prop}

\begin{proof}
By the construction, the statement holds if $V=\kk$. For an arbitrary perfect $V\in \Mod_{\kk}$, the statement is true because the functor $V\mapsto (V^{\otimes p})^{tC_p}$ is exact.
\end{proof}

\begin{rmk}\label{remark:basis, not finite type}
Since the functor $V\mapsto (V^{\otimes p})^{tC_p}, V\in \Mod_{\kk}$ does not commute with filtered colimits,  the last proposition does not hold for an arbitrary $V\in \Mod_{\kk}$.
\end{rmk}

Since the functor $V\mapsto (V^{\otimes p})^{tC_p}$ is lax symmetric monoidal, there exists a functorial map
\begin{equation}\label{equation:tate is lax symmetric monoidal}
\mu\colon (V^{\otimes p})^{tC_p} \otimes (W^{\otimes p})^{tC_p} \to ((V\otimes W)^{\otimes p})^{tC_p}, \;\; V,W\in\Mod_{\kk}.
\end{equation}

\begin{prop}[Cartan formula]\label{proposition: tate-cartan formula}
Let $v\in\pi_q(V)$, $w\in \pi_r(W)$, $i,j\in \Z$. If $p$ is odd, we have
\begin{enumerate}
\item $\mu_*(Q^i(v)\otimes Q^j(w)) =0 $;
\item $\mu_*(\beta Q^i(v) \otimes Q^j(w)) = Q^{i+j}(v\otimes w)$;
\item $\mu_*(Q^i(v) \otimes \beta Q^j(w)) = (-1)^q Q^{i+j}(v\otimes w)$;
\item $\mu_*(\beta Q^i(v) \otimes \beta Q^j(w)) = \beta Q^{i+j}(v\otimes w)$.
\end{enumerate}
If $p=2$, we have $\mu_*(Q^i(v)\otimes Q^j(w))=Q^{i+j-1}(v\otimes w)$.
\end{prop}

\begin{proof} The proof is identical to the proof of the Cartan formula for the standard Dyer--Lashof operations, see e.g.~\cite[Section~I.1]{CLM76}.
\end{proof}

Let $W=V^{\vee}=\Hom(V,\kk)\in \Mod_{\kk}$ be the weak dual object to $V\in \Mod_{\kk}$. Then we have the pairing
\begin{equation}\label{equation:tate pairing, infty}
(W^{\otimes p})^{tC_p} \otimes (V^{\otimes p})^{tC_p} \xrightarrow{\mu} ((W\otimes V)^{\otimes p})^{tC_p} \xrightarrow{ev_V} \kk^{tC_p} \xrightarrow{\mathrm{can}} \Sigma \kk_{hC_p} \to \Sigma \kk, 
\end{equation}
which induces the pairing on the homotopy groups
\begin{equation}\label{equation:tate pairing}
\langle -,-\rangle\colon \pi_{-q}(W^{\otimes p})^{tC_p} \otimes \pi_{q+1}(V^{\otimes p})^{tC_p} \to \pi_0(\kk)=\kk, \;\; q\in\Z.
\end{equation}
We note that if $V=\kk$, then this pairing coincides with the pairing~\eqref{equation:pairing in tate cohomology} on the Tate cohomology.

The next corollary follows directly from Proposition~\ref{proposition: tate-cartan formula} and the last part of Proposition~\ref{proposition:tate-dyer-lashof operations}.

\begin{cor}\label{corollary:pairing on tate-dyer-lashof}
Let $V\in \Mod_{\kk}$ and set $W=\Hom(V,\kk)$. Suppose that $w\in\pi_{q}(W)$, $v\in\pi_{r}(V)$, $i,j\in\Z$, and $\e,\e'\in\{0,1\}$. Then, if $p$ is odd, we have 
$$
\langle\beta^{\e} Q^i(w), \beta^{\e'} Q^j(v)\rangle =
\begin{cases}
(-1)^{q\e'}\langle w, v\rangle^{p} & \mbox{if $q+r=i+j=0$ and $\e+\e'=1$,}\\
0 & \mbox{otherwise.}
\end{cases}
$$
If $p=2$, we have
$$
\langle Q^i(w), Q^j(v)\rangle =
\begin{cases}
\langle w, v\rangle^{2} & \mbox{if $q+r=0$ and $i+j+1=0$,}\\
0 & \mbox{otherwise.}
\end{cases}
$$ \qed
\end{cor}

\subsection{Tate diagonal}\label{section: tate diagonal}
We recall that, by Remark~\ref{remark: cartesian and smash, lie algebras},  the composite 
$$\sLxi^{\times} \xrightarrow{\susp} \Sp(\sLxi)^{\otimes} \xrightarrow{\widetilde{C}_*} \Mod_{\kk}^{\otimes} $$
is oplax symmetric monoidal, where $\widetilde{C}_*$ is the chain complex functor (Definition~\ref{definition: cochain complex, stable}). In particular, for every $X\in \sLxi$, the diagonal map $X\to X\times X$ induces the natural transformation
\begin{equation}\label{equation:comultiplication of chains}
\mu^X\colon \widetilde{C}_*(\susp X) \to (\widetilde{C}_*(\susp X)^{\otimes p})^{h\Sigma_p},
\end{equation}
which is given by the~\emph{comultiplication}.

We write $\widetilde{\mu}^X$, $X\in \sLxi$ for the following composite
\begin{equation}\label{equation:comultiplication of chains, tate}
\widetilde{\mu}^X\colon \widetilde{C}_*(\susp X) \xrightarrow{\mu^X} (\widetilde{C}_*(\susp X)^{\otimes p})^{h\Sigma_p} \xrightarrow{\mathrm{can}}(\widetilde{C}_*(\susp X)^{\otimes p})^{t\Sigma_p}.
\end{equation} 
By dualizing the map $\widetilde{\mu}$ and by using the pairing~\eqref{equation:tate pairing, infty}, we obtain the natural transformation $\widetilde{\mu}^{\vee}$, which also coincides with the following composite
$$\widetilde{\mu}^{\vee}\colon\Sigma^{-1}(\widetilde{C}^*(\Sigma^{\infty}X)^{\otimes p})^{t\Sigma_p} \xrightarrow{\mathrm{can}} (\widetilde{C}^*(\Sigma^{\infty}X)^{\otimes p})_{h\Sigma_p} \xrightarrow{m} \widetilde{C}^{*}(\susp X),\;\; X\in \sL,$$
where $m$ is given by the \emph{multiplication} of cochains. 

We recall that the cohomology groups $\widetilde{H}^q(X) = \pi_{-q}\widetilde{C}^{*}(\susp X)$, $X\in \sL$ are equipped with the action of the Steenrod operations, see~\cite[Section~5]{Priddy70long}, \cite[Theorem~8.5]{May70operations}, and~\cite[Section~6.1]{koszul}. By~\cite[Definition~2.2]{May70operations}, this action is defined by the next formula (where $p$ is odd)
$$\beta^{\e}P^i(y)=\widetilde{\mu}^{\vee}_*(\sigma^{-1}\beta^{\e}Q^{-i}(y)) \in \widetilde{H}^{q+2(p-1)i+\e}(X), \;\; y\in \widetilde{H}^q(X), \;\; i\geq 0, \;\; \e\in\{0,1\}. $$
If $p=2$, then $Sq^i(y)=\widetilde{\mu}^{\vee}_*(\sigma^{-1}Q^{-i}(y)) \in \widetilde{H}^{q+i}(X)$, $i\geq 1$.

Finally, we recall that together with the \emph{left} action of the Steenrod operations on the cohomology groups there is also the adjoint \emph{right} action of the Steenrod operations on the homology groups. In other words, for every $x\in \widetilde{H}_q(X)$, $X\in \sLxi$, we define $x\beta^{\e}P^i \in \widetilde{H}_{q-2i(p-1)-\e}(X)$, $i\geq 0$, $\e\in\{0,1\}$ (resp. $xSq^i\in \widetilde{H}_{q-i}(X)$, $i\geq 1$) by the following rule
$$\langle y, x\beta^{\e}P^i\rangle^p = \langle \beta^{\e}P^i(y), x\rangle\;\; \text{for all $y\in \widetilde{H}^{q-2i(p-1)-\e}(X)$} $$
(resp. $\langle y, xSq^i\rangle^2 = \langle Sq^i(y), x\rangle$ for all $y\in \widetilde{H}^{q-i}(X)$ if $p=2$).

\begin{lmm}\label{lemma:tate diagonal for spaces}
Let $X\in\sLxi$ be an $\kk$-complete simplicial restricted Lie algebra and let $x\in \widetilde{H}_q(X), q\geq 1$ be a homology class. Then, in $\pi_q(\widetilde{C}_*(\susp X)^{\otimes p})^{t\Sigma_p}$, we have
$$ \widetilde{\mu}_*(x) =
\begin{cases}
\sum_{i\geq 0}Q^{i}(x\beta P^{i}) + (-1)^q\sum_{i>0}\beta Q^{i}(xP^i) & \mbox{if $p$ is odd,}\\
\sum_{i>0}Q^{i-1}(xSq^i) & \mbox{if $p=2$.}
\end{cases}
$$
\end{lmm}

\begin{proof}
We prove the statement only for an odd prime $p$. The argument for $p=2$ is identical and we leave it to the reader to complete the details. First, assume that $X\in\sLxi$ is of finite type, i.e. $\dim\widetilde{H}_*(X)<\infty$. Then, by Proposition~\ref{propostion: tate-dyer-lashof is a basis}, we have 
$$\widetilde{\mu}_*(x)=\sum_{\substack{i\in \Z\\ \e\in\{0,1\}}}\beta^{\e}Q^{i}(u_{i,\e})$$ 
for some homology classes $u_{i,\e}\in \widetilde{H}_*(X)$; here the class $u_{i,\e}$ has degree $r_{i,\e}=q-2i(p-1)+\e-1$. By Corollary~\ref{corollary:pairing on tate-dyer-lashof}, if $p$ is odd, for an arbitrary $y\in \widetilde{H}^{r_{i,\e}}(X)$, $i\in\Z$, $\e\in\{0,1\}$ we have
\begin{align*}
\langle y, u_{i,\e}\rangle^p &= (-1)^{r_{i,\e}\e}\langle \beta^{\e'}Q^j(y),\widetilde{\mu}_*(x) \rangle =(-1)^{r_{i,\e}\e}\langle \widetilde{\mu}^{\vee}_*(\sigma^{-1}\beta^{\e'}Q^j(y)), x \rangle\\
&=(-1)^{r_{i,\e}\e}\langle \beta^{\e'}P^{-j}(y), x \rangle =(-1)^{r_{i,\e}\e}\langle y, x\beta^{\e'}P^{-j} \rangle^p,
\end{align*}
where $i+j=0$ and $\e+\e'=1$. It is clear that $(-1)^{r_{i,e}\e} = (-1)^ {q\e}$, and so $u_{i,e}=(-1)^{q\e}x\beta^{\e'}P^{i}$ for all $i\in \Z$ and $\e\in\{0,1\}$.

Finally, if $X\in \sLxi$ is arbitrary, then $X$ can be presented as a filtered colimit $X\simeq \colim X_i$, where each $X_i \in \sLxi^{\mathrm{ft}}$ is of finite type. Therefore the natural transformation $\widetilde{\mu}^X$ can be factorized as the composite
\begin{align*}
\widetilde{\mu}^X\colon \widetilde{C}_*(\susp X) \simeq \colim \widetilde{C}_*(\susp X_i) &\xrightarrow{\colim \widetilde{\mu}^{X_i}}\colim (\widetilde{C}_*(\susp X_i)^{\otimes p})^{t\Sigma_p}\\ 
&\to (\widetilde{C}_*(\susp X)^{\otimes p})^{t\Sigma_p},
\end{align*}
which implies the statement.
\end{proof}

Since both functors $$X\mapsto \widetilde{C}_*(\susp X), (\widetilde{C}_*(\susp X) ^{\otimes p})^{t\Sigma_p} \in \Mod_{\kk}, \;\; X\in \sLxi$$ are $1$-excisive and the first functor preserves filtered colimits, there is a unique up to an equivalence natural transformation
\begin{equation}\label{equation: tate diagonal}
\Delta^X_p\colon \widetilde{C}_*(X) \to (\widetilde{C}_*(X)^{\otimes p})^{t\Sigma_p}, \;\; X\in \Sp(\sLxi)
\end{equation}
such that $\widetilde{\mu}\simeq \Delta_p\circ \susp$, see e.g.~\cite[Theorem~6.2.3.21]{HigherAlgebra}. Moreover, by Remark~6.2.3.23 and the proof of Proposition~6.2.1.9 from ibid., the natural transformation $\Delta_p$ can be given as the composite
\begin{align*}
\Delta_p^X \colon \widetilde{C}_*(X) \simeq P_1\widetilde{C}_*(\susp\deloop X)&\xrightarrow{P_1\widetilde{\mu}} P_1(\widetilde{C}_*(\susp\deloop X)^{\otimes p})^{t\Sigma_p} \\
&\xrightarrow{\alpha} (P_1\widetilde{C}_*(\susp\deloop X)^{\otimes p})^{t\Sigma_p}\simeq (\widetilde{C}_*(X)^{\otimes p})^{t\Sigma_p},
\end{align*}
where $X\in \Sp(\sLxi)$ is a spectrum object, $P_1$ is the Goodwillie linearization, and $\alpha$ is the canonical map. We note that $\alpha$ is usually \emph{not} an equivalence, since the functor $V\mapsto (V^{\otimes p})^{t\Sigma_p}, V\in\Mod_{\kk}$ does not commute with sequential colimits.

The next proposition follows directly from Lemma~\ref{lemma:tate diagonal for spaces} and the description of the map $\Delta_p$ above.

\begin{prop}\label{proposition:tate diagonal for spectra}
Let $X\in\Sp(\sLxi)$ be a spectrum object and let $x\in \widetilde{H}_q(X)$, $q\in \Z$ be a homology class. Then, in $\pi_q(\widetilde{C}_*(X)^{\otimes p})^{t\Sigma_p}$, we have
$$ 
\begin{gathered}[b]
\Delta_{p*}(x) =
\begin{cases}
\sum_{i\geq 0}Q^{i}(x\beta P^{i}) + (-1)^q\sum_{i>0}\beta Q^{i}(xP^i) & \mbox{if $p$ is odd,}\\
\sum_{i>0}Q^{i-1}(xSq^i) & \mbox{if $p=2$.}
\end{cases}\\[-\dp\strutbox]
\end{gathered}
$$
\end{prop}

Let us denote by $G\colon \Sp(\sLxi) \to \Mod_{\kk}$ the following composite functor
$$G\colon \Sp(\sLxi) \xrightarrow{\susp_{\sLxi} \deloop_{\sLxi}} \Sp(\sLxi) \xrightarrow{\widetilde{C}_*} \Mod_{\kk}. $$
The comultiplication map~\eqref{equation:comultiplication of chains} and the counit map $\susp \deloop \to \id$ induce the natural transformation:
$$\nu^X\colon G(X)=\widetilde{C}_*(\susp\deloop X) \xrightarrow{\mu^{\deloop X}} (\widetilde{C}_*(\susp\deloop X)^{\otimes p})^{h\Sigma_p} \to (\widetilde{C}_*(X)^{\otimes p})^{h\Sigma_p},$$
where $X\in \Sp(\sLxi)$. Since the chain complex functor $\widetilde{C}_*\colon \Sp(\sLxi) \to \Mod_{\kk}$ is exact and symmetric monoidal (see Theorem~\ref{theorem: smash, stabilization}) and by~\cite[Lemma~B.3]{Heuts21}, the natural transformation $\nu$ induces an equivalence on the $p$-th Goodwillie layers. In other words, the induced map
$$D_p(\nu)\colon D_pG(X) \to D_p\left((\widetilde{C}_*(X)^{\otimes p})^{h\Sigma_p}\right) \simeq (\widetilde{C}_*(X)^{\otimes p})_{h\Sigma_p}, \;\; X\in \Sp(\sLxi) $$
is an equivalence. This observation and the construction of the natural transformation~\eqref{equation: tate diagonal} implies the next proposition.


\begin{prop}\label{proposition:chain of omega-infinity}
The following diagram is Cartesian:
$$
\begin{tikzcd}
P_pG \arrow{rr}{P_p\nu} \arrow{d} 
& 
& (\widetilde{C}_*(-)^{\otimes p})^{h\Sigma_p} \arrow{d}{\mathrm{can}}\\
P_{p-1}G \arrow{r}
& P_1G\simeq \widetilde{C}_*(-) \arrow{r}{\Delta_p}
& (\widetilde{C}_*(-)^{\otimes p})^{t\Sigma_p},
\end{tikzcd}
$$
where $\Delta_p$ is the natural transformation~\eqref{equation: tate diagonal}.
\qed
\end{prop}

The combination of Proposition~\ref{proposition:chain of omega-infinity} and Proposition~\ref{proposition: Snaith splitting} implies the following corollary.

\begin{cor}\label{corollary: tate on free is trivial}
The natural transformation 
$$\Delta_p\circ \suspfreexi\colon \widetilde{C}_*(\suspfreexi V) \to (\widetilde{C}_*(\suspfreexi V)^{\otimes p})^{t\Sigma_p}$$
is trivial for any bounded below $V\in\Mod_{\kk}^{+}$. \qed
\end{cor}

Let $W,V\in \Vect^{\mathrm{gr}}_{\kk}$ be graded vector spaces and let $$f\colon \suspfreexi W \to \suspfreexi V \in \Sp(\sLxi)$$ be a map of suspension spectrum objects such that the induced map $$f_*=0\colon \Sigma W\cong \widetilde{H}_*(\suspfreexi W) \to \widetilde{H}_*(\suspfreexi V)\cong \Sigma V$$ of the homology groups is trivial. We write $\cofib(f)$ for the homotopy cofiber of the map $f$. We have natural (by $W$ and $V$) split exact sequences:
$$0 \to \Sigma V \to \widetilde{H}_*(\cofib(f)) \to \Sigma^{2}W \to 0, $$
$$0 \to \pi_*((\Sigma V)^{\otimes p})^{t\Sigma_p} \to \pi_*\left((\widetilde{C}_*(\cofib(f))^{\otimes p})^{t\Sigma_p}\right) \to \pi_*\Sigma ((\Sigma W)^{\otimes p})^{t\Sigma_p} \to 0.$$
Since the Steenrod operations act trivially on the homology groups of a free simplicial restricted Lie algebra, we have $x\beta^{\e}P^i\in \Sigma V$ (resp. $xSq^i \in \Sigma V$) for a homology class $x\in \widetilde{H}_*(\cofib(f))$ and a Steenrod operation $\beta^{e}P^i, i\geq 0, \e\in\{0,1\}$ (resp. $Sq^i, i\geq 1$). Moreover, if $x\in \Sigma V$, then $x\beta^{\e}P^i =0, i\geq 0, \e\in\{0,1\}$ (resp. $xSq^i=0, i\geq 1$). Therefore Proposition~\ref{proposition:tate diagonal for spectra} implies that there is a unique map 
\begin{equation}\label{equation:connection map, suspension}
d_f\colon \Sigma^{2}W\to \pi_*((\Sigma V)^{\otimes p})^{t\Sigma_p}
\end{equation}
such that the following diagram commutes
\begin{equation}\label{equation: df definition}
\begin{tikzcd}
\widetilde{H}_*(\cofib(f)) \arrow{r}{\Delta_{p*}} \arrow[two heads]{d} 
& \pi_*\left((\widetilde{C}_*(\cofib(f))^{\otimes p})^{t\Sigma_p}\right) \\
\Sigma^2 W \arrow{r}{d_f}
& \pi_*((\Sigma V)^{\otimes p})^{t\Sigma_p}. \arrow[hook]{u}
\end{tikzcd}
\end{equation}

We write $F=\fib(\Delta_p)\colon \Sp(\sLxi) \to \Mod_{\kk}$ for the homotopy fiber of the natural transformation~\eqref{equation: tate diagonal}; there is the natural fiber sequence
$$F(X) \to \widetilde{C}_*(X)  \xrightarrow{\Delta_p} (\widetilde{C}_*(X)^{\otimes p})^{t\Sigma_p},\;\; X\in \Sp(\sLxi). $$
By Corollary~\ref{corollary: tate on free is trivial}, the natural transformation 
$$\Delta_p\circ \suspfreexi\colon \widetilde{C}_*(\suspfreexi V) \to (\widetilde{C}_*(\suspfreexi V)^{\otimes p})^{t\Sigma_p}, \;\; V\in\Mod_{\kk}^{+}$$
is trivial. Thus there is a natural (by $V$) splitting 
\begin{align}\label{equation: splitting, fiber of tate}
F(\suspfreexi V) &\simeq \widetilde{C}_*(\suspfreexi V) \oplus \Sigma^{-1}(\widetilde{C}_*(\suspfreexi V)^{\otimes p})^{t\Sigma_p}\\
&\simeq \Sigma V \oplus \Sigma^{-1}((\Sigma V)^{\otimes p})^{t\Sigma_p}, \;\; V\in \Mod_{\kk}^+. \nonumber
\end{align}
We note that this splitting is natural with respect to maps in the category $\Mod_{\kk}^{+}$, but not with respect to maps of suspension spectra in the category $\Sp(\sLxi)$.

\begin{lmm}\label{lemma:fiber of tate diagonal}
Let $V,W\in \Vect^{\mathrm{gr}}_{\kk}$ be graded vector spaces and let $$f\colon \suspfreexi W \to \suspfreexi V\in \Sp(\sLxi)$$ be a map of suspension spectra such that the induced map $f_*\colon \widetilde{H}_*(\suspfreexi W) \to \widetilde{H}_*(\suspfreexi V)$ is trivial. Then the following diagram commutes
\begin{equation}\label{equation: fiber of tate, commsquare}
\begin{tikzcd}[column sep=large]
\pi_*F(\suspfreexi W) \arrow{r}{F(f)_*} \arrow[two heads]{d} 
& \pi_*F(\suspfreexi V) \\
\Sigma W \arrow{r}{\Sigma^{-1}d_f}
& \pi_*\Sigma^{-1}((\Sigma V)^{\otimes p})^{t\Sigma_p}, \arrow[hook]{u}
\end{tikzcd}
\end{equation}
where $d_f$ is the map~\eqref{equation:connection map, suspension}.
\end{lmm}

\begin{proof}
Consider the commutative diagram in the $\infty$-category $\Mod_{\kk}$:
\begin{equation}\label{equation: fiber of tate eq1}
\begin{tikzcd}[column sep=large]
\widetilde{C}_*(\suspfreexi W) \arrow{r}{\widetilde{C}_*(f)} \arrow{d}{\Delta_p} 
& \widetilde{C}_*(\suspfreexi V) \arrow{d}{\Delta_p}\\
(\widetilde{C}_*(\suspfreexi W)^{\otimes p})^{t\Sigma_p} \arrow{r}{\widetilde{C}_*(f)}
&(\widetilde{C}_*(\suspfreexi V)^{\otimes p})^{t\Sigma_p}.
\end{tikzcd}
\end{equation}
By the assumption, all maps in the diagram~\eqref{equation: fiber of tate eq1} are trivial and the induced map between horizontal cofibers is the Tate diagonal
$$\Delta_p\colon \widetilde{C}_*(\cofib(f)) \to (\widetilde{C}_*(\cofib(f))^{\otimes p})^{t\Sigma_p} $$
Therefore, the diagram~\eqref{equation: fiber of tate eq1} maps to the commutative diagram
\begin{equation}\label{equation: fiber of tate eq2}
\begin{tikzcd}
\widetilde{C}_*(\suspfreexi W) \arrow{r} \arrow{d} 
& 0 \arrow{d}\\
0 \arrow{r}
&(\widetilde{C}_*(\suspfreexi V)^{\otimes p})^{t\Sigma_p}.
\end{tikzcd}
\end{equation}
The induced map between horizontal cofibers in~\eqref{equation: fiber of tate eq2}
\begin{equation}\label{equation: lemma, tate diagonal}
\Sigma^2 W \simeq \Sigma \widetilde{C}_*(\suspfreexi W) \to (\widetilde{C}_*(\suspfreexi V)^{\otimes p})^{t\Sigma_p} \simeq ((\Sigma V)^{\otimes p})^{t\Sigma_p}
\end{equation}
induces the map $d_f$ on the homotopy groups, see the commutative diagram~\eqref{equation: df definition}. Note that, in the diagram~\eqref{equation: fiber of tate eq2}, the induced map between vertical cofibers is equivalent to the induced map between horizontal cofibers. However, by the definition of the splitting~\eqref{equation: splitting, fiber of tate}, the induced map between vertical fibers in~\eqref{equation: fiber of tate eq2} makes the square~\eqref{equation: fiber of tate, commsquare} commutative.
\end{proof}

Recall that $G\colon \Sp(\sLxi) \to \Mod_{\kk}$ is defined as $G(X)= \widetilde{C}_*(\susp\deloop X)$ for a spectrum object $X\in \Sp(\sLxi)$. Since the category $\Mod_{\kk}$ is $p$-local, the Taylor tower of the functor $P_{p-1}G$ splits naturally. Therefore, by Corollary~\ref{proposition:chain of omega-infinity}, there exists a natural transformation $$\eta\colon F(X)=\fib(\Delta^X_p)\to P_pG(X)=P_p(\widetilde{C}_*(\susp\deloop X)), \;\; X\in \Sp(\sLxi)$$
such that the following diagram commutes (where $X\in \Sp(\sLxi)$):
\begin{equation}\label{equation:tate diagonal, commdiag}
\begin{tikzcd}[column sep=large]
\Sigma^{-1}(\widetilde{C}_*(X)^{\otimes p})^{t\Sigma_p} \arrow{r} \arrow{d}{\mathrm{can}}
& F(X) \arrow{r} \arrow{d}{\eta}
&\widetilde{C}_*(X) \arrow{d}{\simeq}\\
\widetilde{C}_*(X)^{\otimes p}_{h\Sigma_p} \simeq D_pG(X) \arrow{r} 
&P_pG(X) \arrow{r}
& P_1G(X).
\end{tikzcd}
\end{equation}

We recall that the Taylor tower of the composite
$$G \circ \suspfreexi \colon \Mod^{+}_{\kk} \xrightarrow{\suspfreexi} \Sp(\sLxi) \xrightarrow{G} \Mod_{\kk}$$ also splits, see Proposition~\ref{proposition: Snaith splitting}. In particular, for any graded vector space $V\in \Vect_{\kk}^{\mathrm{gr}}$, we have the following natural (by $V$) split inclusions
\begin{align}
\pi_*((\Sigma V)^{\otimes p}_{h\Sigma_p})\cong\pi_*D_pG(\suspfreexi V)\subset \pi_*P_pG(\suspfreexi V), \label{equation: inclusion, snaith, p} \\
\Sigma V\cong \pi_*D_1G(\suspfreexi V)\subset \pi_*P_pG(\suspfreexi V). \label{equation: inclusion, snaith, 1}
\end{align}
Here the inclusions~\eqref{equation: inclusion, snaith, p} and~\eqref{equation: inclusion, snaith, 1} split by the maps $\widetilde{C}_{*}(\Pi_p)$ and $\widetilde{C}_*(\Pi_1)$) respectively, where $\Pi_p$ and $\Pi_1$ are canonical projections~\eqref{equation: snaith projections}.

The commutative diagram~\eqref{equation:tate diagonal, commdiag} and Lemma~\ref{lemma:fiber of tate diagonal} imply immediately the following proposition.

\begin{prop}\label{proposition:infinite loops of a map between suspension spectra}
Let $V,W \in \Vect_{\kk}^{\mathrm{gr}}$ be graded vector spaces and let $$f\colon \suspfreexi W \to \suspfreexi V \in \Sp(\sLxi)$$ be a map of suspension spectra such that  the induced map $\widetilde{H}_*(f)=0$ is trivial. Then the following diagram commutes
$$
\begin{tikzcd}[column sep=large]
\pi_*P_pG(\suspfreexi W) \arrow{rr}{P_pG(f)_*}
&
&\pi_*P_pG(\suspfreexi V) \\
\Sigma W \arrow[hook]{u} \arrow{r}{\Sigma^{-1}d_f}
& \pi_*\Sigma^{-1}((\Sigma V)^{\otimes p})^{t\Sigma_p} \arrow{r}{\mathrm{can}}
&\pi_*(\Sigma V)^{\otimes p}_{h\Sigma_p}, \arrow[hook]{u}
\end{tikzcd}
$$
where $d_f$ is the map~\eqref{equation:connection map, suspension}. \qed
\end{prop}

\subsection{Unstable Hurewicz homomorphism}\label{section: unstable hurewicz homomorphism}

Let $X\in \sLxi$ be an $\kk$-complete simplicial restricted Lie algebra. Recall the Hurewicz homomorphism~\eqref{equation: hurewicz homomorphism}
$$h\colon \pi_q(X) \to \widetilde{H}_{q+1}(X)=\pi_{q+1}(\widetilde{C}_*(\susp X)), \;\; q\geq 0. $$
We introduce an \emph{unstable} variant of the Hurewicz homomorphism. Now, let $X\in \Sp(\sLxi)$ be a spectrum object, for every $q\geq 0$, we define a homomorphism $\tilde{h}$ by the following composition
\begin{align*}
\tilde{h}\colon \pi_q(X) \cong \pi_q(\deloop X) &\xrightarrow{h} \pi_{q+1}\widetilde{C}_*(\susp \deloop X)\simeq \pi_{q+1}G(X) \to \pi_{q+1}P_pG( X),
\end{align*}
where the last map is induced by the natural transformation $G\to P_pG$. We note that the map $\tilde{h}$ is functorial with respect to all maps of spectrum objects. 

Let $V=\bigoplus_{q>0} V_q \in \Vect_{\kk}^{>0}$ be a positively graded vector space and let $$X=\suspfreexi(\Sigma^{-1} V)\cong L^{st}_{\xi}\susp \free(\Sigma^{-1}V)$$ be the $\kk$-complete suspension spectrum of the free simplicial restricted Lie algebra $\free(\Sigma^{-1}V)$. Then, by Corollary~\ref{corollary: stable homotopy groups of a xi-complete free object}, the homotopy groups $$\pi_*(X)\cong \prod_{s\geq 0}\Sigma^{-1}V^{(s)}\otimes \Lambda_s,$$ where $\Lambda=\bigoplus_{s\geq 0}\Lambda_s$ is the lambda algebra. We compute the map $\tilde{h}$ restricted to the subspace $\pi_{*,1}(X)\subset \pi_*(X)$. Recall that the vector space $$\pi_{q+1}V^{\otimes p}_{h\Sigma_p} \cong \pi_{q+1}D_pG(\suspfreexi(\Sigma^{-1}V)), \;\; q\geq 0$$ is a subspace of $\pi_{q+1}P_pG(\suspfreexi(\Sigma^{-1}V))$.

\begin{prop}\label{proposition: hurewicz of hopf invariant one}
Let $v \in V_q$, then we have
$$\tilde{h}(\sigma^{-1}v\otimes \nu^{\e}_{i})  = 
\begin{cases}
Q^i(v) \in \pi_{q+2i(p-1)}(V^{\otimes p}_{h\Sigma_p}) & \mbox{if $p$ is odd, $i\geq 0, \e=0$,}\\
(-1)^q\beta Q^i(v) \in \pi_{q+2i(p-1)-1}(V^{\otimes p}_{h\Sigma_p}) & \mbox{if $p$ is odd, $i\geq 1, \e=1$,}\\
Q^i(v) \in \pi_{q+i}(V^{\otimes 2}_{h\Sigma_2}) & \mbox{if $p=2$ and $i\geq 0$.}
\end{cases}
$$
\end{prop}

\begin{proof} 
We set $x=\sigma^{-1}v\otimes \nu^{\e}_{i}$, $l=|x|=q-1+2i(p-1)-\e$ if $p$ is odd, and $l=|x|=q+i-1$ if $p=2$. Let $W=\Sigma^{l}\kk$ be a one-dimensional graded vector space spanned on the canonical basis element $\iota_{l}$ and let 
$$f\colon \suspfreexi W \to \suspfreexi (\Sigma^{-1}V) \in \Sp(\sLxi)$$ be the map of suspension spectra such that $$f_*(\iota_{l})=x\in \pi_{l}\suspfreexi(\Sigma^{-1} V)\cong \Sigma^{-1}V \otimes \Lambda.$$
There is a commutative diagram
$$
\begin{tikzcd}[column sep=large]
\pi_{l}\suspfreexi W \arrow{r}{f_*} \arrow{d}{\tilde{h}} 
& \pi_{l}\suspfreexi(\Sigma^{-1}V) \arrow{d}{\tilde{h}}\\
\pi_{l+1}P_pG(\suspfreexi W) \arrow{r}{P_pG(f)_*}
&\pi_{l+1}P_pG(\suspfreexi (\Sigma^{-1}V)).
\end{tikzcd}
$$
By the splitting~\eqref{equation: inclusion, snaith, 1}, we have $$\tilde{h}(\iota_{l})=\sigma\iota_{l}\in \Sigma W \cong \pi_{l+1}P_1G(\suspfreexi W)\subset \pi_{l+1}P_pG(\suspfreexi W).$$ Therefore, it suffices to compute the map $P_pG(f)_*$. The induced map $\widetilde{H}_*(f)$ vanishes, and so, by Proposition~\ref{proposition:infinite loops of a map between suspension spectra}, it is enough to compute the associated map $d_f$ (see~\eqref{equation:connection map, suspension}).

The homology groups $\widetilde{H}_*(\cofib(f))$ of the cofiber of the map $f$ are isomorphic to the direct sum $V\oplus \Sigma^2 W$. By~\cite[Remark~6.4.19]{koszul}, if $p$ is odd, the Steenrod operations $\beta^{\e'}P^j, j\geq 0, \e'\in\{0,1\}$ act on $\sigma^2 \iota_{l}$ as follows 
$$(\sigma^2 \iota_{l})\beta^{\e'}P^j = 
\begin{cases}
v \in V\subset \widetilde{H}_*(\cofib(f)) & \mbox{if $i=j$ and $\e+\e'=1$,}\\
0 & \mbox{otherwise.}
\end{cases}
$$
If $p=2$, then we have
$$(\sigma^2 \iota_{l})Sq^{j} =
\begin{cases}
v \in V\subset \widetilde{H}_*(\cofib(f)) & \mbox{if $j=i+1$,}\\
0 & \mbox{otherwise.}
\end{cases}
$$

By Proposition~\ref{proposition:tate diagonal for spectra}, we calculate the map $d_f\colon \Sigma^2 W \to \pi_{l+2} (V^{\otimes p})^{t\Sigma_p}$. Namely, we obtain
$$d_f(\sigma^2 \iota_{l})  = 
\begin{cases}
Q^i(v) \in \pi_{q+2i(p-1)+1}(V^{\otimes p})^{t\Sigma_p} & \mbox{if $p$ is odd, $i\geq 0, \e=0$,}\\
(-1)^{l+2}\beta Q^i(v) \in \pi_{q+2i(p-1)}(V^{\otimes p})^{t\Sigma_p} & \mbox{if $p$ is odd, $i\geq 1, \e=1$,}\\
Q^i(v) \in \pi_{q+i+1}(V^{\otimes 2})^{t\Sigma_2} & \mbox{if $p=2$, $i\geq 0$.}
\end{cases}
$$
Finally, the statement follows from the last part of Proposition~\ref{proposition:tate-dyer-lashof operations}.
\end{proof}

Recall that $V\in \Vect^{>0}_{\kk}$ is a positively graded vector space. Consider the $p$-th James--Hopf map (see Definition~\ref{definition:james-hopf map}):
\begin{equation}\label{equation:james-hopf section}
j_p\colon \deloop\suspfreexi(\Sigma^{-1}V) \to \deloop\suspfreexi(\Sigma^{-1}V^{\otimes p}_{h\Sigma_p}).
\end{equation}
By Proposition~\ref{proposition:differential lowers the second grading}, we have
$$j_{p*}(\sigma^{-1}v\otimes \nu^{\e}_i) \in \pi_{*,0}(\suspfreexi(\Sigma^{-1}V^{\otimes p}_{h\Sigma_p})) \subset \pi_*(\suspfreexi(\Sigma^{-1}V^{\otimes p}_{h\Sigma_p})). $$
We note that $\pi_{*,0}(\suspfreexi(\Sigma^{-1}V^{\otimes p}_{h\Sigma_p}))\cong \pi_*(\Sigma^{-1}V^{\otimes p}_{h\Sigma_p}))$, see Remark~\ref{remark: second grading, xi-complete}. In the next theorem we compute $j_{p*}(\sigma^{-1}v\otimes \nu^{\e}_i)\in \pi_*(\Sigma^{-1} V^{\otimes p}_{h\Sigma_p})$ in terms of the Dyer--Lashof operations.
\begin{thm}\label{theorem:james-hopf and dyer-lashof}
Let $v \in V_q$, then 
$$j_{p*}(\sigma^{-1}v\otimes \nu^{\e}_{i})  = 
\begin{cases}
\sigma^{-1}Q^i(v) 
& \mbox{if $p$ is odd, $i\geq 0, \e=0$,}\\
(-1)^q\sigma^{-1}\beta Q^i(v) 
& \mbox{if $p$ is odd, $i\geq 1, \e=1$,}\\
\sigma^{-1}Q^i(v) 
& \mbox{if $p=2$ and $i\geq 0$.}
\end{cases}
$$
\end{thm}

\begin{proof}

By the definition of the James--Hopf map, the following diagram commutes
$$
\begin{tikzcd}[column sep = large]
\pi_{*}\suspfreexi(\Sigma^{-1}V) \arrow{r}{j_{p*}} \arrow{d}{\tilde{h}} 
& \pi_{*}\suspfreexi(\Sigma^{-1}V^{\otimes p}_{h\Sigma_p}) \arrow{d}{h}\\
\pi_{*+1}P_{p}G(\suspfreexi(\Sigma^{-1}V)) \arrow{r}{\widetilde{C}_*(\Pi_p)}
& \pi_{*+1}\widetilde{C}_*(\suspfreexi(\Sigma^{-1}V^{\otimes p}_{h\Sigma_p})).
\end{tikzcd}
$$
Here $h$ is the stable Hurewicz homomorphism (Remark~\ref{remark: stable hurewicz homomorphism}), $\Pi_p$ is the projection~\eqref{equation: snaith projections}, and $\tilde{h}$ is the unstable Hurewicz homomorphism.
By Remark~\ref{remark: hurewicz, lambda}, the right vertical arrow $h$ is the projection onto the component $\pi_{*,0}\suspfreexi(\Sigma^{-1}V^{\otimes p}_{h\Sigma_p})\subset \pi_{*}\suspfreexi(\Sigma^{-1}V^{\otimes p}_{h\Sigma_p})$.  Since $\widetilde{C}_*(\Pi_p)$ is the projection onto the component $\pi_{*+1}(V^{\otimes p}_{h\Sigma_p}) \subset \pi_{*+1}P_{p}G(\suspfreexi(\Sigma^{-1}V))$, see~\eqref{equation: inclusion, snaith, p}, we obtain that $$j_{p*}(x)=\sigma^{-1}\tilde{h}(x)$$ for any $x\in \pi_{*,1}\suspfreexi(\Sigma^{-1}V)$. Finally, the theorem follows from Proposition~\ref{proposition: hurewicz of hopf invariant one}.
\end{proof}

\begin{rmk}\label{remark: james-hopf, informal}
Informally, Theorem~\ref{theorem:james-hopf and dyer-lashof} means that the James--Hopf invariant $$j_{p*}\colon \pi_*(\susp X) \to \pi_*((\susp X)^{\otimes p}_{h\Sigma_p})\to \widetilde{H}_*((\susp X)^{\otimes p}_{h\Sigma_p})$$ maps a homotopy class detected by a primary cohomology operation $P$ to the dual (in a certain sense) Dyer--Lashof class of $P$. The corresponding result in the category of spaces $\spaces$ could be overlooked due to a sparsity of Hopf invariant one elements, but still can be observed e.g. by using~\cite{Kuhn82}, which relates the James--Hopf invariant with the ordinary Hopf invariant map $H_p\colon \pi_*(\Sigma X) \to \pi_*((\Sigma X)^{\otimes p})$, $X\in\spaces$.
\end{rmk}

\section{Composition product}\label{section: composition product}

This section is devoted to the proof of the \emph{Leibniz identity} for the differentials in the Goodwillie spectral sequence~\eqref{equation:GSS1} with respect to the composition products~\eqref{equation:stable composition}, \eqref{equation:left action}, and~\eqref{equation:right action} on its $E^1$-term, see Theorem~\ref{theorem:leibniz rule in GSS}. Although the composition products can be defined on the entire $E^1$-term (see~\cite{Ching21} and~\cite{AroneChing11}), we are able to prove the Leibniz rule only for the differentials which originate from the first row, see Remark~\ref{remark: composition product on the first page}.

Our proof relies heavily on the work~\cite{Heuts21}. In Section~\ref{section: category of pairs} we review the \emph{category of pairs}. Given an endofunctor $H\colon \calC \to \calC$, we construct the category $\calC_H$ whose objects are pairs $(X,\eta\colon X\to H(X))$, $X\in \calC$. Given $X,Y\in \calC$, we observe in Corollary~\ref{corollary:splitting of the mapping space} that the mapping space $\map_{\calC_H}((X,0),(Y,0))$ splits as the product
$$\map_{\calC_H}((X,0),(Y,0))\simeq \map_{\calC}(X,Y) \times \map_{\calC}(X,\Sigma^{-1}HY). $$
We show in Corollary~\ref{corollary:leibniz rule} that this splitting behaves as a \emph{square-zero extension} with respect to the composition products. This observation proves a variant of the Leibniz rule for the categories of pairs.

In Section~\ref{section: goodwillie differentials} we recall the definition of the \emph{strong $n$-excisive approximation} $\calP_n(\calC)$ of the compactly generated category $\calC$. The category $\calP_n(\calC)$ encodes mapping spaces $\map_{\calC}(X,P_n(Y))$. By the results of~\cite{Heuts21}, the natural functor $\calP_n(\calC) \to \calP_{n-1}(\calC)$ is a base change of the functor $\Sp(\calC) \to \Sp(\calC)_H$ for a certain endofunctor $H$, see Proposition~\ref{proposition:leibniz rule for Goodwillie}. In this way, the compatibility of the Taylor tower of the mapping spaces $\map_{\calC}(X,P_n(Y))$ with the composition products is the same as for the category of pairs. The last observation proves the main theorem~\ref{theorem:leibniz rule in GSS}.

\subsection{Category of pairs}\label{section: category of pairs}

Recall from~\cite[Section~5.5.7]{HTT} that an $\infty$-category is called compactly generated if it is presentable and $\omega$-accessible. Alternatively, an $\infty$-category is compactly generated if and only if it is equivalent to $\Ind(\calD)=\Ind_{\omega}(\calD)$, where~$\calD$ is a (small) $\infty$-category which admits finite colimits. Here $\Ind_{\omega}(\calD)$ denotes the free completion of $\calD$ with respect to ($\omega$-)filtered colimits.

Let $\calC$ be a stable compactly generated $\infty$-category. Let denote by $\calC^\omega$ the full subcategory of $\calC$ spanned on compact objects. We denote the embedding $\calC^\omega \hookrightarrow \calC$ by $\iota$. 

Let $H\colon \calC\to \calC$ be an endofunctor. We write $\calC^\omega_H$ for \emph{the $\infty$-category of pairs} of objects $X \in\calC^\omega$ equipped with a map $X\to H(X)$. More precisely, $\calC^\omega_H$ can be defined as the pullback of the following span diagram
$$\calC^\omega \xrightarrow{(\iota,H)} \calC \times \calC \xleftarrow{(\mathrm{ev}_0,\mathrm{ev}_1)} \calC^{\Delta^1}. $$
We write $U\colon \calC^{\omega}_H \to \calC$ for the forgetful functor, $U(X,\eta)=X$. Let $X,Y\in \calC^{\omega}_H$, then there is an equalizer diagram of mapping spaces
\begin{equation}\label{equation: maps in category of pairs}
\begin{tikzcd}
\map_{\calC^{\omega}_H}(X,Y) \arrow{r}
&
\map_{\calC}(UX,UY) \arrow[shift left=.6ex]{r}
  \arrow[shift right=.6ex,swap]{r}
&
\map_{\calC}(UX,HUY).
\end{tikzcd}
\end{equation}

%
%

\begin{lmm}\label{lemma: pairs forgetful}
The forgetful functor $U\colon \calC^{\omega}_H \to \calC^{\omega}$ preserves colimits and it is conservative.
\end{lmm}

\begin{proof}
The functor $U$ is the pullback of the conservative functor $$(\mathrm{ev}_0,\mathrm{ev}_1)\colon \calC^{\Delta^1} \to \calC\times \calC.$$ Therefore, $U$ is conservative as well. Finally, $U$ preserves colimits by the formula~\eqref{equation: maps in category of pairs}.
\end{proof}

We note that this construction is functorial, i.e. any natural transformation $\eta\colon H_1 \to H_2$ defines a functor 
$$U_{\eta}\colon \calC^\omega_{H_1} \to \calC^\omega_{H_2},$$
which commutes with the projections to $\calC^\omega$. More precisely, we have a functor:
$$\calC^\omega_{(-)}\colon \mathrm{End}(\calC) \to \Catinfty_{\infty}, $$
where $\Catinfty_{\infty}$ is the $\infty$-category of (small) $\infty$-categories, and by abuse of notation, $\mathrm{End}(\calC) = \Fun(\calC^\omega,\calC)$. Since $\calC$ is stable, the category $\mathrm{End}(\calC)$ is pointed with zero object is the zero endofunctor. We have $\calC^\omega_0\simeq \calC^\omega$, thus the functor $\calC^\omega_{(-)}$ factorizes via
$$\Catinfty_{\infty,\calC^\omega\!/\!\calC^\omega} \to \Catinfty_{\infty},$$
where $\Catinfty_{\infty,\calC^\omega\!/\!\calC^\omega} $ is the category of $\infty$-categories over $\calC^\omega$ with a fixed section, i.e. $\Catinfty_{\infty,\calC^\omega\!/\!\calC^\omega}= \left(\Catinfty_{\infty\, /\!\calC^\omega}\right)_*$, the category of pointed objects \cite[Definition~7.2.2.1]{HTT} in the overcategory $\Catinfty_{\infty\, /\!\calC^\omega}$, see~\cite[Section~1.2.9]{HTT}. Informally, the last sentence means that we have the projection functor $U\colon \calC^\omega_H \to \calC^\omega$ given by $$U(X,X\to H(X))=X,$$ and its section $s\colon \calC^\omega \to \calC^\omega_H$ given by $$s(X)=(X, X\xrightarrow{0} H(X)).$$ Moreover, given a natural transformation $\eta\colon H_1 \to H_2$, the functor $U_{\eta}\colon \calC^\omega_{H_1} \to \calC^\omega_{H_2}$ commutes with $U$ and $s$.

Next, we show that the functor  $U_{\eta}\colon \calC^\omega_{H_1} \to \calC^\omega_{H_2}$ commutes with colimits. 
\begin{prop}\label{proposition:forgetful functor preserves colimits}
The functor $\calC^\omega_{(-)}\colon \mathrm{End}(\calC)\to \Catinfty_{\infty}$ factorizes via the subcategory $\Catinfty_{\infty}^{\mathrm{Rex}}$ of $\Catinfty_{\infty}$ spanned on categories which admit finite colimits and functors preserving finite colimits.
\end{prop}

\begin{proof}
By Lemma~\ref{lemma: pairs forgetful}, the forgetful functor $U\colon \calC^\omega_{H}\to \calC^\omega$ preserves colimits and is conservative. This implies the assertion.
\end{proof}

\begin{prop}\label{proposition:pairs respects fiber sequences}
The functor $\calC^\omega_{(-)}\colon \End(\calC) \to \Catinfty_{\infty, \calC^\omega\!/\!\calC^\omega}$ preserves fiber sequences. In other words, if  $F \to G \to H$ is a fiber sequence in $\End(\calC)$, then
$$
\begin{tikzcd}
\calC^\omega_{F} \arrow{r} \arrow{d}{U} 
& \calC^\omega_{G} \arrow{d} \\
\calC^\omega \arrow{r}{s} 
& \calC^\omega_{H}
\end{tikzcd}
$$
is a pullback diagram in $\Catinfty_{\infty}$.
\end{prop}

\begin{proof}
Let $(\calD,f,s')$ be an object of $\Catinfty_{\infty,\calC^\omega\!/\!\calC^\omega}$, where $f\colon \calD \to \calC^\omega$ and $s'\colon \calC^\omega \to \calD$ are functors and $s'$ is a section of $f$. Then, for every endofunctor $K\in \End(\calC)$, there is an equivalence of pointed spaces
$$\Fun((\calD,f,s'),(\calC^\omega_K,U,s))\simeq \nat(f,K\circ f), $$
which implies the proposition.
\end{proof}

\begin{cor}\label{corollary:pairs respects fiber sequences}
Let $F \to G \to H$ be a fiber sequence in $\End(\calC)$. Then
$$
\begin{tikzcd}
\Ind(\calC^\omega_{F}) \arrow{r} \arrow{d}{U} 
& \Ind(\calC^\omega_{G}) \arrow{d} \\
\calC \arrow{r}{s} 
& \Ind(\calC^\omega_{H})
\end{tikzcd}
$$
is a pullback diagram in $\Pre^L_{\omega}$.
\end{cor}

\begin{proof}
Follows by Proposition~\ref{proposition:forgetful functor preserves colimits} and~\cite[Appendix~A]{Heuts21}.
\end{proof}

Let us denote by $\calC_H$ the ind-completion $\Ind(\calC^\omega_H)$ of the category $\calC^\omega_H$, $H\in\End(\calC)$. By the definition, the category $\calC_H$ is compactly generated. We warn the reader that the category $\Ind(\calC^\omega_H)$ is \emph{not} equivalent to the pullback in $\Pre^L$ of the span diagram
$$\calC \xrightarrow{(\id,H)} \calC \times \calC \xleftarrow{(\mathrm{ev}_0,\mathrm{ev}_1)} \calC^{\Delta^1}, $$
except the case when the endofunctor $H$ commutes with colimits.

By Proposition~\ref{proposition:forgetful functor preserves colimits}
and~\cite[Corollary~5.5.2.9]{HTT}, the forgetful functor $$U\colon \Ind(\calC^\omega_H) \to \calC, \;\; H\in \End(\calC)$$ admits a right adjoint $C\colon \calC \to \Ind(\calC^\omega_H)$. Moreover, by the same argument, the section $s\colon \calC \to \Ind(\calC^\omega_H)$ also admits a right adjoint $r\colon \Ind(\calC^\omega_H) \to \calC$. Since $s$ preserves compact objects, the right adjoint $r$ is the left Kan extension of the functor $$r|_{\calC^\omega_H} \colon \calC^\omega_H \to \calC, \;\; \big(X\xrightarrow{\eta} H(X)\big) \mapsto \fib(\eta).$$

\begin{prop}\label{proposition:split fiber sequence}
Let $H\in \End(\calC)$ be an endofunctor and let $s,r$ be as above. Then there is a split fiber sequence in the category $\End(\calC)$:
$$\Sigma^{-1}H \to rs \to \id_{\calC}, $$
where a splitting of the map $rs \to \id_{\calC}$ is given by the unit map $\id_{\calC} \to rs$ for the adjunction $s\dashv r$.
\end{prop}

\begin{proof}
Let $X,Y\in \Ind(\calC^\omega_H)$. There is an equalizer diagram of mapping spaces
$$
\begin{tikzcd}
\map_{\calC_H}(X,Y) \arrow{r}
&
\map_{\calC}(UX,UY) \arrow[shift left=.6ex]{r}
  \arrow[shift right=.6ex,swap]{r}
&
\map_{\calC}(UX,HUY),
\end{tikzcd}
$$
which induces the equalizer diagram in $\End(\calC_H)$
$$
\begin{tikzcd}
\id_{\calC_H} \arrow{r}
&
CU \arrow[shift left=.6ex]{r}
  \arrow[shift right=.6ex,swap]{r}
&
CHU.
\end{tikzcd}
$$
We postcompose this diagram with the right adjoint $r$ and we precompose it  with the section $s$. Since $Us\simeq \id$ and $rC\simeq \id$, we obtain the following equalizer diagram
$$
\begin{tikzcd}
rs \arrow{r}
&
\id_{\calC} \arrow[shift left=.6ex]{r}
  \arrow[shift right=.6ex,swap]{r}
&
H.
\end{tikzcd}
$$
Finally, since both natural transformations $\id_{\calC} \to H$ on the right hand side are trivial, the proposition follows. 
\end{proof}


Since the space $\map_{\calC_H}(sX,sY), X,Y\in \calC$ is equivalent to the mapping space $\map_{\calC}(X,rsY)$, we observe that $\map_{\calC_H}(sX,sY)$ is naturally an infinite loop space. Another way to see it is to observe that the functor $s\colon \calC \to \calC_H$ preserves direct sums.

\begin{cor}\label{corollary:splitting of the mapping space}
Let $X,Y\in\calC$ be a pair of objects and $H\in \End(\calC)$ be an endofunctor. Then the map
$$U_{X,Y}\colon \map_{\calC_{H}}(sX,sY) \to \map_{\calC}(X,Y) \in \spaces $$
is a map of infinite loop spaces and it is a principal fibration. Moreover, the fiber of the map $U_{X,Y}$ over the zero map $X\xrightarrow{0}Y$ is equivalent to the mapping space $\map_{\calC}(X,\Sigma^{-1}HY)$. \qed
\end{cor}

The principal fibration $U_{X,Y}$ has the splitting given by the section $s$. Therefore there is an equivalence of infinite loop spaces
\begin{equation}\label{equation:splitting of the mapping space}
(-)+(-)\colon\map_{\calC}(X,Y) \times \map_{\calC}(X,\Sigma^{-1}HY) \xrightarrow{\sim} \map_{\calC_H}(sX,sY).
\end{equation}
We write $i_H\colon \map_{\calC}(X,\Sigma^{-1}HY) \to \map_{\calC_H}(sX,sY)$ for the inclusion of the fiber and we write
$$d_H\colon \map_{\calC_H}(sX,sY) \to \map_{\calC}(X,\Sigma^{-1}HY)$$ for the projection map induced by the splitting above. We note that the splitting~\eqref{equation:splitting of the mapping space} and the map $d_H$ are functorial with respect to natural transformations. For instance, if $\eta\colon H_1 \to H_2$ is a natural transformation of endofunctors, then the following diagram commutes
\begin{equation}\label{equation:functoriality of the differential}
\begin{tikzcd}
\map_{\calC_{H_1}}(sX,sY) \arrow{r}{U_{\eta}} \arrow{d}{d_{H_1}}
& \map_{\calC_{H_2}}(sX,sY) \arrow{d}{d_{H_2}} \\
\map_{\calC}(X,\Sigma^{-1}H_1Y) \arrow{r}{\eta} 
& \map_{\calC}(X,\Sigma^{-1}H_2Y).
\end{tikzcd}
\end{equation}

Next, we show that the composition in the category $\Ind(\calC^\omega_H)$ respects the splitting~\eqref{equation:splitting of the mapping space}. Namely, let $X,Y,Z \in \calC$ be a triple of objects, then there is the composition map \begin{align}\label{equation:composition}
-\circ -\colon \map_{\calC_H}(sY,sZ) \times \map_{\calC_H}(sX,sY) &\to \map_{\calC_H}(sX,sZ), \\
(g,f) &\mapsto g\circ f \nonumber
\end{align}
of mapping spaces in the category $\calC_H=\Ind(\calC^\omega_H)$, which is unique up to a suitable equivalence. Since the section $s\colon \calC \to \calC_H$ preserves direct sums, we obtain the following proposition.

\begin{prop}\label{proposition: composition is linear on second argument}
Let $X,Y,Z\in \calC$ be a triple of objects and let $H\in \End(\calC)$ be an endofunctor. Then the composition map~\eqref{equation:composition} is linear on the second argument. More precisely, the adjoint map
$$\map_{\calC_H}(sY,sZ) \to \map_{\spaces}(\map_{\calC_H}(sX,sY), \map_{\calC_H}(sX,sZ)) $$
factors through the mapping space $\map_{\Omega^{\infty}}(\map_{\calC_H}(sX,sY), \map_{\calC_H}(sX,sZ))$ of infinite loop spaces. \qed
\end{prop}

We note that the map~\eqref{equation:composition} is not linear on the first argument, except the case when the endofunctor $H$ preserves finite coproducts. However, since the following diagram commutes
\begin{equation}\label{equation: composition preserves decomposition}
\begin{tikzcd}[column sep=large]
\map_{\calC_H}(sY,sZ) \times \map_{\calC}(X,Y) \arrow{r}{(-)\circ s(-)} \arrow{d}[swap]{\simeq}
& \map_{\calC_H}(sX,sZ) \arrow{d}{\simeq} \\
\map_{\calC}(Y,rs Z) \times \map_{\calC}(X,Y) \arrow{r}{(-)\circ (-)}
&\map_{\calC}(Y,rsZ),
\end{tikzcd}
\end{equation}
the composition map $(-)\circ s(-)$ preserves the decomposition~\eqref{equation:splitting of the mapping space}. Moreover, by using the adjunction $s\dashv r$ and Proposition~\ref{proposition:split fiber sequence}, we observe the following compatibilities between the decomposition~\eqref{equation:splitting of the mapping space} and the composition product~\eqref{equation:composition}.
\begin{prop}\label{proposition:bimodule structure}
Let $X,Y,Z\in \calC$ be a triple of objects and let $H\colon \calC \to \calC$ be an endofunctor. Then the following diagrams commute
$$
\begin{tikzcd}[column sep=huge, row sep=large]
\map_{\calC}(Y,\Sigma^{-1}HZ) \times \map_{\calC}(X,Y) \arrow{r}{i_H\times s} \arrow{d}[swap]{\circ}
& \map_{\calC_H}(sY,sZ) \times \map_{\calC_H}(sX,sY) \arrow{d}{\circ} \\
\map_{\calC}(X,\Sigma^{-1}HZ) \arrow{r}{i_H}
&\map_{\calC_H}(sY,sZ),
\end{tikzcd}
$$
and
$$
\begin{gathered}[b]
\begin{tikzcd}[column sep=huge, row sep=large]
\map_{\calC}(Y,Z) \times \map_{\calC}(X,\Sigma^{-1}HY) \arrow{r}{s\times i_H} \arrow{d}[swap]{\Sigma^{-1}H(-)\circ-}
& \map_{\calC_H}(sY,sZ) \times \map_{\calC_H}(sX,sY) \arrow{d}{\circ} \\
\map_{\calC}(X,\Sigma^{-1}HZ) \arrow{r}{i_H}
&\map_{\calC_H}(sY,sZ).
\end{tikzcd}\\[-\dp\strutbox]
\end{gathered}
$$
\end{prop}

Consider the endofunctor $H\oplus H\colon \calC \to \calC$. The projection on the $i$-th direct summand induces the functor
$$U_i\colon \calC_{H\oplus H} \to \calC_H, \;\; i=1,2. $$
Similarly, the inclusion of the $i$-th direct summand induces the functor
$$s_i\colon \calC_{H} \to \calC_{H\oplus H}, \;\; i=1,2. $$
We have equivalences $U_1s_1\simeq U_2s_2\simeq \id_{\calC_H}$, $s_1 s\simeq s_2 s$ is the zero section $\calC \to \calC_{H\oplus H}$, and $U_1s_2\simeq U_2 s_1\simeq sU$, where $U\colon \calC_H \to \calC$ is the forgetful functor and $s\colon \calC \to \calC_H$ is the section.

\begin{cor}\label{corollary:almost leibniz rule}
Let $X,Y,Z\in \calC$ be a triple of objects and let $H\colon \calC \to \calC$ be an endofunctor. Then the composite
\begin{align*}
\map_{\calC_H}(sY,sZ)\times \map_{\calC_H}(sX,sY) &\xrightarrow{s_1(-)\circ s_2(-)}\map_{\calC_{H\oplus H}}(sX,sZ) \\
&\xrightarrow{d_{H\oplus H}} \map_{\calC}(X,\Sigma^{-1}(H\oplus H)Z) \\
&\xrightarrow{\sim} \map_{\calC}(X,\Sigma^{-1}HZ)\times \map_{\calC}(X,\Sigma^{-1}HZ) 
\end{align*}
has the first component to be equivalent $d_H(-)\circ U(-)$ and the second component to be equivalent $\Sigma^{-1}HU(-)\circ d_H(-)$. 
\end{cor}

\begin{proof}
We compute only the first component here; the computation of the second component is completely parallel and we leave it to the reader to complete the details. Consider the projection
$$\pi_1(d_{H\oplus H}(s_1(-)\circ s_2(-)))\colon \map_{\calC_H}(sY,sZ)\times \map_{\calC_H}(sX,sY) \to \map_{\calC}(X,\Sigma^{-1}HZ)$$
on the first component. By the diagram~\eqref{equation:functoriality of the differential}, we have 
$$\pi_1(d_{H\oplus H}(s_1(-)\circ s_2(-)))\simeq d_H(U_1(s_1(-)\circ s_2(-)))\simeq d_H((-)\circ sU(-)).$$
Using the decomposition~\eqref{equation:splitting of the mapping space} and the observation with the diagram~\eqref{equation: composition preserves decomposition}, we obtain
\begin{align*}
d_H((-)\circ sU(-))&\simeq d_H((i_Hd_H(-)+sU(-))\circ sU(-)) \\
&\simeq d_H(i_Hd_H(-)\circ sU(-)+ sU(-)\circ sU(-))\\
&\simeq d_H(i_Hd_H(-)\circ sU(-)). 
\end{align*}
Finally, we apply the first diagram of Proposition~\ref{proposition:bimodule structure} to rewrite the map $$i_Hd_H(-)\circ sU(-)$$ as $i_H(d_H(-)\circ U(-))$.
\end{proof}

Since the category $\calC$ is stable, the category $\End(\calC)$ of endofunctors is also stable. In particular, there is the \emph{fold map} $$\nabla\colon H\oplus H \to H.$$ We have $\nabla \circ s_i\simeq \id$ for $i=1,2$. The next statement follows from Corollary~\ref{corollary:almost leibniz rule} and the diagram~\eqref{equation:functoriality of the differential} with $\eta=\nabla$.
\begin{cor}\label{corollary:leibniz rule}
The composition
\begin{align*}
\map_{\calC_H}(sY,sZ)\times \map_{\calC_H}(sX,sY) &\xrightarrow{-\circ -}\map_{\calC_{H}}(sX,sZ) \xrightarrow{d_{H}} \map_{\calC}(X,\Sigma^{-1}HZ)
\end{align*}
is equivalent to the sum 
$$d_H(-)\circ U(-)+ \Sigma^{-1}HU(-)\circ d_H(-). $$
\end{cor}

\subsection{Goodwillie differentials}\label{section: goodwillie differentials}
Let $\calC$ be a pointed compactly generated (not necessary stable) $\infty$-category. We recall a few fundamental definitions of~\cite{Heuts21}. 

\begin{dfn}[Definition~1.2 of~\cite{Heuts21}]
\label{def:nexcapprox}
An adjunction $F: \calC \rightleftarrows \calD :G$ is a \emph{weak $n$-excisive approximation} to $\calC$ if the following two conditions are satisfied:
\begin{itemize}
\item The $\infty$-category $\calD$ is pointed and compactly generated. Moreover, its identity functor $\id_{\calD}$ is $n$-excisive.
\item The map $P_n(\id_{\calC}) \to GF$ induced by the unit is an equivalence. Also, the natural transformation $P_n(FG) \to \id_{\calD}$ induced by the counit is an equivalence.
\end{itemize}
\end{dfn}

\begin{dfn}[Definition~1.3 of~\cite{Heuts21}]
\label{def:strongapprox}
A pointed compactly generated $\infty$-category $\calD$ is \emph{$n$-excisive} if every weak $n$-excisive approximation $F': \calD \rightleftarrows \EuScript{E}: G'$ is an equivalence. A weak $n$-excisive approximation $F: \calC \rightleftarrows \calD: G$ is said to be a \emph{strong $n$-excisive approximation} if the $\infty$-category $\calD$ is $n$-excisive.
\end{dfn} 

Recall that an $\infty$-category is $1$-excisive if and only if it is stable~\cite[Corollary~2.17]{Heuts21}. We write $\calP_n(\calC)$ for the strong $n$-excisive approximation of $\calC$, $\calP_1(\calC)\simeq \Sp(\calC)$. We also recall that there is a pair of adjoint functors
$$
\begin{tikzcd}
\susp_n: \calC \arrow[shift left=.6ex]{r}
&\calP_{n}(\calC) :\deloop_n \arrow[shift left=.6ex,swap]{l}
\end{tikzcd}
$$
such that $\deloop_n\susp_n\simeq P_n(\id_{\calC})$. 

Let $X,Y\in \calC$ be a pair of objects, we will denote by $\map_n(X,Y)$ the mapping space $\map_{\calP_n(\calC)}(\susp_n X,\susp_n Y)$ in the category $\calP_n(\calC)$. The fiber sequence of functors:
\begin{equation}\label{equation: googwillie fiber sequence}
P_n(\id_{\calC}) \to P_{n-1}(\id_{\calC}) \to BD_{n}(\id_{\calC}), \;\; n\geq 1
\end{equation}
induces a fiber sequence of (pointed) mapping spaces:
\begin{equation}\label{equation:fiber sequences, taylor tower}
\map_n(X,Y) \to \map_{n-1}(X,Y) \xrightarrow{\delta_n} \map_{1}(\susp X,\Sigma \D_n\susp Y), \;\; n\geq 1,
\end{equation}
where we abbreviate $P_{n}=P_n(\id_{\calC})$, $D_n=D_n(\id_{\calC})$, and $\D_n=\D_n(\id_{\calC})$.

\begin{prop}\label{proposition:leibniz rule for Goodwillie}
Let $X,Y,Z\in \calC$ be a triple of objects, then the composite
$$\map_{n-1}(Y,Z)\times \map_{n-1}(X,Y) \xrightarrow{-\circ-} \map_{n-1}(X,Z) \xrightarrow{\delta_n} \map_1(\susp X,\Sigma\D_n\susp Z) $$
is equivalent to the sum
$$\delta_n(-)\circ \susp(-)+\Sigma\D_n(\susp(-))\circ \delta_n(-). $$
\end{prop}

\begin{proof}
By~\cite[Proposition~5.12 and Lemma~4.31]{Heuts21}, there is a pullback diagram of compactly generated $\infty$-categories
$$
\begin{tikzcd}
\calP_n(\calC) \arrow{r} \arrow{d} 
& \Ind(\Sp(\calC)^\omega_{F}) \arrow{d} \\
\calP_{n-1}(\calC) \arrow{r}
& \Ind(\Sp(\calC)^\omega_{G})
\end{tikzcd}
$$
for certain endofunctors $F,G\in \End(\Sp(\calC))$. Furthermore, by~\cite[Proposition~C.24]{Heuts21}, there is a fiber sequence
$$
F \to G \to \Sigma^{2}\D_n,
$$
which combined with Corollary~\ref{corollary:pairs respects fiber sequences} gives a pullback square
$$
\begin{tikzcd}
\calP_n(\calC) \arrow{r}{\susp} \arrow{d} 
& \Sp(\calC) \arrow{d}{s} \\
\calP_{n-1}(\calC) \arrow{r}
& \Ind(\Sp(\calC)^\omega_{\Sigma^{2}\D_n})
\end{tikzcd}
$$
in the category $\Pre^L_{\omega}$. Therefore there is the following commutative diagram of fiber sequences
$$
\begin{tikzcd}
\map_n(X,Y) \arrow{r} \arrow{d}{\susp} 
& \map_{n-1}(X,Y) \arrow{d} \arrow{r}{\delta_n}
& \map_1(\susp X, \Sigma\D_n\susp Y) \arrow[equal]{d}\\
\map_1(\susp X,\susp Y) \arrow{r}{s}
& \map_{\Sp(\calC)_{\Sigma^2\D_n}}(sX,sY) \arrow{r}{d_{\Sigma^2\D_n}} 
& \map_1(\susp X, \Sigma\D_n\susp Y).
\end{tikzcd}
$$
Since the two left vertical morphisms preserve the composition product, the proposition follows by Corollary~\ref{corollary:leibniz rule}.
\end{proof}

Let $X,Y\in \calC$ be a pair of objects. The fiber sequences~\eqref{equation: googwillie fiber sequence} also induce the following tower of fibrations of mapping spaces
\begin{equation*}
\begin{tikzcd}[column sep=small]
&
\map_{\calC}(X,D_nY) \arrow[d]
&\map_{\calC}(X,D_{n-1}Y)\arrow[d]
&\map_{\calC}(X,D_{n-2}Y)\arrow[d]
&\\
\ldots \arrow[r]
& \map_{\calC}(X,P_n Y) \arrow[r]
& \map_{\calC}(X,P_{n-1} Y) \arrow[r]
& \map_{\calC}(X,P_{n-2} Y) \arrow[r]
& \ldots
\end{tikzcd}
\end{equation*}
We consider the associated spectral sequence
\begin{equation}\label{equation:GSS1}
E^1_{t,n}(X,Y)=\pi_t\map_{\calC}(X,D_nY) \Rightarrow \pi_t\map_{\calC}(X,Y)
\end{equation} 
with the differential $d_r\colon E^r_{t,n}(X,Y) \to E^r_{t-1,n+r}(X,Y)$ and the first page being isomorphic to $$E^1_{t,n}(X,Y)\cong \pi_t\map_{\Sp(\calC)}(\Sigma^{\infty}X,\D_n \Sigma^{\infty} Y).$$

Next, let $X,Y,Z\in \calC$ be a triple of objects. Then we consider three maps induced by the composition product:
\begin{equation}\label{equation:stable composition}
(-)\circ (-)\colon E^1_{t,1}(Y,Z) \otimes E^1_{s,1}(X,Y) \to E^1_{t+s,1}(X,Z),
\end{equation}
\begin{equation}\label{equation:left action}
\D_n(-)\circ(-)\colon E^1_{t,1}(Y,Z) \otimes E^1_{s,n}(X,Y) \to E^1_{t+s,n}(X,Z),
\end{equation}
and
\begin{equation}\label{equation:right action}
(-)\circ(-)\colon E^1_{t,n}(Y,Z) \otimes E^1_{s,1}(X,Y) \to E^1_{t+s,n}(X,Z).
\end{equation}
For instance, the map~\eqref{equation:left action} is given as the composite of the map
$$\pi_t\map(\susp Y,\susp Z) \xrightarrow{\D_n(-)} \pi_{t}\map(\D_n\susp Y,\D_n\susp Z)$$
with the map 
\begin{align*}
\pi_{t}\map(\D_n\susp Y,\D_n\susp Z) \otimes \pi_s\map(\susp X,\D_n\susp Y) \to \pi_{t+s}\map(\susp X,\D_n\susp Z)
\end{align*}
induced by the composition product in the stable category $\Sp(\calC)$. The two other maps~\eqref{equation:stable composition} and~\eqref{equation:right action} are defined similarly.

\begin{thm}\label{theorem:leibniz rule in GSS}
Let $X,Y,Z\in \calC$ be a triple of objects and let $g\in E^1_{t,1}(Y,Z)$ and $f\in E^1_{s,1}(X,Y)$ be elements on the $E^1$-terms of the spectral sequences~\eqref{equation:GSS1}. Suppose that all differentials
$$d_1(g)=d_2(g)=\ldots=d_{n-2}(g)=0\;\; \text{and} \;\; d_{1}(f)=\ldots= d_{n-2}(f)=0, \;\; n\geq 2$$
vanish. Then the differential $$d_{n-1}(g\circ f)\in E^n_{t+s-1,n}(X,Z)$$
of the composite $g\circ f \in E^1_{t+s,1}(X,Z)$ is represented by the sum
$$d_{n-1}(g)\circ f + \D_n(g)\circ d_{n-1}(f)\in E^1_{t+s-1,n}(X,Z).$$
\end{thm}

\begin{proof}
By the assumption, the elements $g\in E^1_{t,1}(Y,Z)$ and $f\in E^1_{s,1}(X,Y)$ persist until the $E^{n-1}$-terms of the spectral sequences~\eqref{equation:GSS1}. Hence, there are lifts
$$\tilde{g}\in \pi_t\map(Y,P_{n-1}Z)\simeq \pi_t\map_{n-1}(Y,Z)$$
$$\tilde{f}\in \pi_s\map(X,P_{n-1}Y) \simeq \pi_s\map_{n-1}(X,Y) $$
of the corresponding homotopy classes $g \in \pi_t\map_1(Y,Z)$ and $f\in \pi_s\map_1(X,Y)$ respectively.
Furthermore, the composite $\tilde{g}\circ \tilde{f}\in \map_{n-1}(X,Z)$ in the category $\calP_{n-1}(\calC)$ is a lift of the composite $g\circ f\in \map_1(X,Z)$ and the differential $d_{n-1}(g\circ f)$ is defined by the formula  
$$d_{n-1}(g\circ f)=\delta_{n*}(\tilde{g}\circ\tilde{f})\in \pi_{t+s-1}\map(\susp X,\D_n\susp Z), $$
where the map $$\delta_{n*}\colon \pi_*\map_n(X,Z) \to \pi_{*-1}\map(\susp X,\D_n\susp Z)$$ is induced by the right map in the fiber sequence~\eqref{equation:fiber sequences, taylor tower}. By Proposition~\ref{proposition:leibniz rule for Goodwillie}, the theorem follows.
\end{proof}

\begin{rmk}\label{remark: composition product on the first page}
Using the operad structure on the derivatives of the identity functor $\id_{\calC}\colon \calC \to \calC$ (see e.g.~\cite{Ching21} or~\cite{AroneChing11}), one can obtain the natural transformations
$$\D_n(\D_m(X)) \to \D_{nm}(X)\in \Sp(\calC), \;\; X\in\Sp(\calC), \; n,m\geq 1. $$
As in the construction of the map~\eqref{equation:left action}, these natural transformations induce the composition product on the $E^1$-term of the spectral sequence~\eqref{equation:GSS1}
\begin{equation}\label{equation: general composition product}
E^1_{t,n}(Y,Z) \otimes E^1_{s,m}(X,Y) \to E^1_{t+s,nm}(X,Z), \;\; X,Y,Z\in\calC. 
\end{equation}
Theorem~\ref{theorem:leibniz rule in GSS} shows that the differential in~\eqref{equation:GSS1} satisfies the Leibniz identity with respect to the product~\eqref{equation: general composition product} if $n=1$ or $m=1$. At the moment of writing, we are not aware if the differential satisfies the Leibniz identity for arbitrary $n,m\geq 1$. The biggest issue here is that the differential changes the gradings additively, but the composition product~\eqref{equation: general composition product} acts on the second grading multiplicatively.
\end{rmk}

\begin{rmk}\label{remark: composition product, hopf}
Let $\calC=\spaces$ be the $\infty$-category of spaces. Then it seems that the closest results to Theorem~\ref{theorem:leibniz rule in GSS} in the literature are~\cite[Theorem~1]{Hilton54} and~\cite[Theorem~6.3]{Ando68}. For example, the former states that the Hopf invariant $H\colon \pi_{r}(S^n) \to \pi_{r+1}(S^{2n})$, $r,n\geq 1$ has a variant of the Leibniz identity with respect to the composition product for maps of spheres.
\end{rmk}


\section{James periodicity}\label{section: james periodicity}
Let $\sL$ be the category of simplicial restricted Lie algebras and let $D_p\colon \sL \to \sL$ be the $p$-th Goodwillie layer of the identity functor $\id_{\sL}$. We recall that there is a unique $p$-homogeneous functor $$\D_p\colon \Sp(\sL) \to \Sp(\sL)$$ such that $D_p \simeq \deloop \circ \D_p \circ \susp$. Let $W,V \in \Mod^{\geq 0}_{\kk}$ be simplicial vector spaces, $\dim\pi_*(W)=\dim\pi_*(V)=1$, and let $$f\colon \suspfree W\to \suspfree V \in \Sp(\sL)$$ be a map of suspension spectrum objects. The main result of this section is Theorem~\ref{theorem: main theorem, james periodicity}, which we use to compute the induced $\Lambda$-linear map
$$\D_p(f)_*\colon \pi_*(\mathbf{L}_p(W))\otimes \Lambda \cong \pi_*\D_p(\suspfree W) \to \pi_*\D_p(\suspfree V) \cong \pi_*(\mathbf{L}_p(V))\otimes \Lambda. $$

Using the \emph{James periodicity} for spaces (Theorem~\ref{theorem:james periodicity for spaces}), we construct a natural filtration $\widetilde{\D}^{(\bullet)}_p$ for the functor $\D_{p}$ such that the thick quotients $\widetilde{\D}^{(k+n)}_p/\widetilde{\D}^{(k-1)}_p$ of length $n$ are \emph{periodic} with a period \emph{depending on $n$}, see Proposition~\ref{proposition: james flitration for functors}. 

In Section~\ref{section: computations with the p-th goodwillie layer} we show that the homotopy groups of the composite $\widetilde{\D}_p^{(n)}\circ \suspfree$ are given by the Dyer--Lashof--Lie classes $\bQe{\e}{i}(w)$ (Definition~\ref{definition: dll-classes}) with a certain upper bound on $i$, see Remark~\ref{remark: image of filtration}. Finally, in Theorem~\ref{theorem: main theorem, james periodicity} we show that the classes $$\D_p(f)_*(\bQe{\e}{i}(w)), \D_p(f)_*(\bQe{\e}{i+k}(w)) \in \pi_*(\D_p(\suspfree V))$$ coincide up to a \emph{shift} $\psi_k$ (Definition~\ref{definition: shift operators}) for a suitable $k$.

\subsection{James periodicity for functors}\label{section: james periodicity for functors} We recall the James periodicity for spaces. Let $G=C_p$ be a cyclic group of order $p$ generated by a non-trivial element $\zeta_p\in C_p$. We write $\theta_p$ for the following $G$-representation
$$ 
\theta_p =
\begin{cases}
\Co, \; \zeta_p(v)=e^{\frac{2\pi i}{p}}v,\; v\in \Co & \mbox{if $p$ is odd,}\\
\R, \; \zeta_2(v)=-v, \; v\in \R & \mbox{if $p=2$.}\\
\end{cases}
$$
We equip $\theta_p$ with a non-degenerate $G$-invariant real quadratic form.
Finally, suppose that $V$ is a $G$-representation equipped with a non-degenerate real quadratic form $q\colon V\to \R$. Then we write $$S(V)=\{v\in V\;|\; q(v)=1\}\in\spaces^{BG}$$ for the unit sphere in $V$ equipped with the induced $G$-action.

\begin{dfn}\label{definition:skeleta for EG}
The $G$-space $EG^{(n)}\in \spaces^{BG}$ is the unit sphere $S\big(\theta_p^{\oplus (n+1)}\big)$, $n\geq 0$ The pointed $G$-space $EG^{(n)}_+ \in \spaces^{BG}_*$, $n\geq 0$ is the $G$-space $EG^{(n)}$ with the added disjoint basepoint.
\end{dfn}

\begin{rmk}\label{remark:EG^n are complete filtration}
The $G$-spaces $EG^{(n)}\in\spaces^{BG}$, $n\geq 0$ provide an exhaustive filtration for the contractible $G$-space $EG\simeq \ast \in \spaces^{BG}$. In other words, the homotopy colimit of the following sequence of inclusions
$$EG^{(0)} \hookrightarrow EG^{(1)} \hookrightarrow EG^{(2)} \hookrightarrow \ldots $$
is contractible.
\end{rmk}

Since the group $G$ acts on the unit spheres $S\big(\theta_p^{\oplus (n+1)}\big)$ freely, the homotopy quotient $EG^{(n)}_{hG}$ is equivalent to the set-theoretical quotient $S\big(\theta_p^{\oplus (n+1)}\big)/G$, and 
$$ 
EG^{(n)}_{hG}\simeq S\big(\theta_p^{\oplus (n+1)}\big)/G = 
\begin{cases}
L^n(p) & \mbox{if $p$ is odd,}\\
\R\PP^n& \mbox{if $p=2$,}\\
\end{cases}
$$
where $L^n(p)$ is the lens space and $\R\PP^n$ is the $n$-dimensional real projective space.

If $p$ is odd, then there is a projection map 
$$S\big(\theta_p^{\oplus (n+1)}\big)/G=L^n(p) \to \Co\PP^n, $$
and we write $\pi_n\colon EG^{(n)}\to  \Co\PP^{n}$ for the composite $$EG^{(n)}  \to S(\theta_p^{\oplus(n+1)})/G=L^n(p) \to \Co\PP^n.$$ If $p=2$, we write $\pi_n\colon EG^{(n)} \to S\big(\theta_p^{\oplus (n+1)}\big)/G=\R\PP^n$ for the projection map.

\begin{dfn}\label{definition: tautological line bundles}
Let $\gamma^1_{\Co}$ (resp. $\gamma^1_{\R}$) be the tautological line bundle over the projective space $\Co\PP^n$ (resp. $\R\PP^n$). We denote by $\xi_n$, $n\geq 0$ the real $G$-equivariant vector bundle over the $G$-space $EG^{(n)}$ defined as follows
$$ 
\xi_n = 
\begin{cases}
r(\pi_n^*(\gamma^1_{\Co})) & \mbox{if $p$ is odd,}\\
\pi_n^*(\gamma^1_{\R})& \mbox{if $p=2$,}\\
\end{cases}
$$
where the functor $r$ is the realification of a vector bundle.
\end{dfn}

\begin{rmk}\label{remark: tautological line bundle}
Although the bundles $\xi_n$, $n\geq 0$ are trivial as vector bundles, they are non-trivial as \emph{$G$-equivariant} bundles.
\end{rmk}

We note that the Thom space $\mathrm{Th}(\xi_n^{\oplus m})\in \spaces^{BG}$ of the bundle sum $\xi_n^{\oplus m}$, $n,m\geq 0$ is equivalent to the quotient $G$-space $EG^{(n+m)}/EG^{(m-1)}$.

\begin{thm}[James periodicity]\label{theorem:james periodicity for spaces}
For every $n\in \N$ there exists a constant $N=N(n)$ such that for any $k\in \N$, $\nu_p(k)>N$ there is an equivalence
$$\Phi_k\colon \susp(EG^{(k+n)}/EG^{(k-1)}) \xrightarrow{\sim} \Sigma^{k'}\susp(EG^{(n)}_+) \in \Sp^{BG},$$
where $k'=2k$ if $p$ is odd, and $k'=k$ if $p=2$.
\end{thm}

\begin{proof}
Recall that if $\zeta$, $\xi$ are real $G$-equivariant vector bundles over a (connected) $G$-space $X$ such that their classes $[\zeta]$, $[\xi]$ are equal in the reduced $K$-group $\widetilde{KO}_G(X)$, then there is an equivalence 
$$\Sigma^{m}\susp \mathrm{Th}(\zeta) \simeq \Sigma^{m'}\susp \mathrm{Th}(\xi) \in \Sp^{BG} $$ 
of suspension spectra. Here $m=\rk(\xi)$, $m'=\rk(\zeta)$ are the real ranks of the bundles $\xi$ and $\zeta$ respectively.

Therefore it suffices to compute the order of the class $[\xi_n]\in \widetilde{KO}_G(EG^{(n)})$. Since the group $G$ acts freely on the space $EG^{(n)}$, we have
$$\widetilde{KO}_G(EG^{(n)}) \cong \widetilde{KO}(EG^{(n)}/G)\cong
\begin{cases}
\widetilde{KO}(L^n(p)) & \mbox{if $p$ is odd,}\\
\widetilde{KO}(\R\PP^n) & \mbox{if $p=2$.}\\
\end{cases}
$$

By~\cite[Theorem~7.4]{Adams62}, the $K$-group $\widetilde{KO}(\R\PP^n)$ is isomorphic to $\Z/2^{\phi(n)}\Z$, where $\phi(n)$ is the number of integers $s\leq n$ such that $s\equiv 0,1,2,4 \mod 8$. So, the theorem follows for $p=2$.

Finally, by~\cite[Theorem~2]{Kambe66}, the order of the class $[\xi_n]$ in $\widetilde{KO}(L^n(p))$ is equal to $p^{\left\lceil\frac{n}{p-1}\right\rceil}$. This implies the theorem for an odd prime $p$ as well.
\end{proof}

\begin{rmk}\label{remark:exact value of N}
The proof above shows that one can take $N=N(n)\geq \left\lceil\frac{n}{p-1}\right\rceil$ in Theorem~\ref{theorem:james periodicity for spaces}.
\end{rmk}


\begin{rmk}\label{remark:james equivalence choice}
In the notation of Section~\ref{section: group cohomology}, the $G$-equivariant cohomology groups $H^*_G(EG^{(n)},\F_p)=H^*(EG^{(n)}_{hG},\F_p)$ are isomorphic to the following module over the group cohomology ring $H^*(G,\F_p)=H^*(BG,\F_p)$
$$H^*_G(EG^{(n)},\F_p)\cong
\begin{cases}
\F_p[t]/t^{n+1} \otimes \Lambda(u) & \mbox{if $p$ is odd,}\\
\F_p[u]/u^{n+1} & \mbox{if $p=2$.}\\
\end{cases}
$$

We note that the equivalence $\Phi_k$ from Theorem~\ref{theorem:james periodicity for spaces} is not unique. In this paper, we choose $\Phi_k$ such that the induced homomorphism 
$$\Phi^*_k\colon\widetilde{H}^{*-k'}_G(EG^{(n)}_+,\F_p) \to \widetilde{H}^{*}_G(EG^{(k+n)}_+,\F_p)$$
is the multiplication by $t^{k}$ if $p$ is odd, and it is the multiplication by $u^k$ if $p=2$.  
\end{rmk}

Let $\calC,\calD$ be stable presentable $\infty$-categories. Recall that $G=C_p$ is a cyclic group. Similar to the category of symmetric functor (see~\cite[Notation~6.1.4.1]{HigherAlgebra}), we define the $\infty$-category $\mathrm{CycFun}(\calC,\calD)$ of \emph{cyclic functors} as follows
$$\mathrm{CycFun}(\calC,\calD) = \Fun^{\omega}((\calC^{p}\times EG)/G,\calD). $$
Furthermore, we denote by $\mathrm{CycFun}_{\mathrm{lin}}(\calC,\calD)$ the full subcategory of cyclic \emph{multilinear} functors in the category $\mathrm{CycFun}(\calC,\calD)$ spanned by those cyclic functors $F\colon (\calC^{p}\times EG)/G \to\calD$ whose underlying functor $\overline{F}\colon \calC^{p} \to \calD $ is multilinear, see~\cite[Definition~6.1.3.7]{HigherAlgebra}. We note that there is the canonical forgetful functor
$$\mathrm{SymFun}(\calC,\calD) \to \mathrm{CycFun}(\calC,\calD) $$
from the category of symmetric functors to the category of cyclic functors induced by the group inclusion $C_p\subset \Sigma_p$. Moreover, this forgetful functor maps the full subcategory $\mathrm{SymFun}_{\mathrm{lin}}(\calC,\calD)$ onto $\mathrm{CycFun}_{\mathrm{lin}}(\calC,\calD)$

We recall that $\mathrm{Homog}^p(\calC,\calD)$ is the category of $p$-homogeneous functors from $\calC$ to $\calD$ and we also recall that the $p$-th cross-effect $\creff_{(p)}$ gives an equivalence (see e.g.~\cite[Theorem~6.1.4.7]{HigherAlgebra})
$$\creff_{(p)}\colon \mathrm{Homog}^p(\calC,\calD) \xrightarrow{\sim} \mathrm{SymFun}_{\mathrm{lin}}(\calC,\calD) $$
between the category $\mathrm{Homog}^p(\calC,\calD)$ and the category of multilinear symmetric functors. Moreover, the inverse to the equivalence above is given as follows $$F \mapsto \Xi(F),\;\; F\in \mathrm{SymFun}_{\mathrm{lin}}(\calC,\calD),$$ where $\Xi(F)(X)=\overline{F}(X,\ldots,X)_{h\Sigma_p}$, $X\in\calC$, and $\overline{F}\colon \calC^p \to \calD$ is the underlying functor for $F$.

Suppose that $H\in\mathrm{Homog}^p(\calC,\calD)$ is a $p$-homogeneous functor. For the rest of the section, we consider its $p$-th cross-effect $\creff_{(p)}H\in \mathrm{CycFun}_{\mathrm{lin}}(\calC,\calD)$ as a multilinear \emph{cyclic} functor by forgetting the symmetric structure on $\creff_{(p)}H$. 

We note that the stable presentable category $\mathrm{CycFun}_{\mathrm{lin}}(\calC,\calD) \in \Pre^L_{\mathrm{st}}$ is equivalent to 
$$\Fun_{\mathrm{lin}}\left(\calC^{\otimes p}_{hG}, \calD\right) \simeq \left(\Fun_{\mathrm{lin}}(\calC^{\otimes p}, \calD)\right)^{hG} \in \Pre^L_{\mathrm{st}}.$$
Since the stable presentable category $\Fun_{\mathrm{lin}}(\calC^{\otimes p}, \calD)$ is a module over the symmetric monoidal category $\Sp\in \CAlg(\Pre^L_{\mathrm{st}})$ equipped with the trivial $G$-action, we observe that $\mathrm{CycFun}_{\mathrm{lin}}(\calC,\calD) \in \Pre^L_{\mathrm{st}}$ is a module over the symmetric monoidal category $\Sp^{BG} \in \CAlg(\Pre^L_{\mathrm{st}})$. Therefore, the tensor product $K\otimes H \in \mathrm{CycFun}_{\mathrm{lin}}(\calC,\calD)$ is defined for a functor $H\in\mathrm{CycFun}_{\mathrm{lin}}(\calC,\calD)$ and a spectrum $K \in \Sp^{BG}$.

\begin{dfn}\label{definition:filtration on homogeneous functors}
Let $H\in \mathrm{Homog}^p(\calC,\calD)$ be a $p$-homogeneous functor. We define the functor $\widetilde{H}^{(n)}\in \mathrm{Homog}^{p}(\calC,\calD), n\geq 0$ by the rule
$$\widetilde{H}^{(n)}(X)=\left(\susp EG^{(n)}_+\otimes (\creff_{(p)}H)\right)(X,\ldots,X)_{hC_p}, \;\; X\in \calC.$$ 
Similarly, we define the functor $\widetilde{H}\in \mathrm{Homog}^p(\calC,\calD)$ by the rule 
$$\widetilde{H}(X)=(\creff_{(p)}H)(X,\ldots,X)_{hC_p}, \;\; X\in \calC.$$
\end{dfn}

We note that there are the compatible natural transformations $$i^m_n\colon \widetilde{H}^{(n)} \to \widetilde{H}^{(m)}, \; i_n\colon \widetilde{H}^{(n)} \to \widetilde{H}, \;  p_n \colon \widetilde{H}^{(n)} \to H, \; p\colon \widetilde{H} \to H,$$
where $m\geq n\geq 0$ and $H\in \mathrm{Homog}^p(\calC,\calD)$, induced by various natural maps between $C_p$-spaces $EC_p^{(n)}$, $n\geq 0$.

\begin{prop}\label{proposition: james flitration for functors}
Let $H\in\mathrm{Homog}^p(\calC,\calD)$ be a $p$-homogeneous functor. Then
\begin{enumerate}
\item the colimit $\colim\limits_{n} \widetilde{H}^{(n)}$ in $\mathrm{Homog}^p(\calC,\calD)$ is equivalent to the functor $\widetilde{H}$;
\item
for every $n\in \N$ there exists a constant $N=N(n)$ such that for any $k\in \N$, $\nu_p(k)>N$ there is a natural equivalence
$$\Phi_k\colon \widetilde{H}^{(n+k)}/\widetilde{H}^{(k-1)} \xrightarrow{\sim} \Sigma^{k'} \widetilde{H}^{(n)} \in \mathrm{Homog}^p(\calC,\calD),$$
induced by  the equivalence of Theorem~\ref{theorem:james periodicity for spaces}, where $k'=2k$ if $p$ is odd, and $k'=k$ if $p=2$.
\end{enumerate}
\end{prop}

\begin{proof}
The proposition follows directly from the definitions and Theorem~\ref{theorem:james periodicity for spaces}. 
\end{proof}

In the next proposition we list a few simple properties of the filtration above. The proof is straightforward and we leave it to the reader.
\begin{prop}\label{proposition: james filtration, properties}
Let $H\in \mathrm{Homog}^p(\calC,\calD)$ be a $p$-homogeneous functor.
\begin{enumerate}
\item Suppose that $\calC_1,\calD_1$ are two other stable presentable $\infty$-categories and let $L_1\in \Fun^L(\calC_1,\calC)$, $L_2\in \Fun^L(\calD,\calD_1)$ be exact functors. Then we have
$$\widetilde{HL_1}^{(n)}\simeq \widetilde{H}^{(n)} L_1 \in \Fun^{\omega}(\calC_1,\calD), \;\; \widetilde{L_2 H}^{(n)}\simeq L_2 \widetilde{H}^{(n)} \in \Fun^{\omega}(\calC,\calD_1). $$
\item Suppose that the category $\calD$ is $p$-local, then the projection map $p\colon \widetilde{H} \to H$ has a natural splitting given by the transfer map
$$H(X)\simeq (\creff_{(p)}H)(X,\ldots,X)_{h\Sigma_p} \to (\creff_{(p)}H)(X,\ldots,X)_{hC_p} \simeq \widetilde{H}(X), \;\; X\in \calC. $$
\end{enumerate}
\end{prop}

\begin{exmp}\label{example: james filtration vector spaces}
Suppose that the categories $\calC, \calD$ are the category of chain complexes, i.e. $\calC=\calD=\Mod_{\kk}$. Let us consider a $p$-homogeneous functor $H\in \mathrm{Homog}^p(\Mod_{\kk},\Mod_{\kk})$ given by the rule
$$H(V)=A\otimes_{h\Sigma_p} V^{\otimes p},\;\; V\in \Mod_{\kk}, \;\; A\in \Mod_{\kk[\Sigma_p]}.$$
Here we describe the functors $\widetilde{H}^{(n)}\in \mathrm{Homog}^p(\Mod_{\kk}, \Mod_{\kk})$, $n\geq 0$.

Let $P_\bullet\in \Mod_{\kk[C_p]}$ be the following periodic chain complex of $\kk[C_p]$-modules
$$P_{\bullet}=\left(P_0=\kk[C_p] \xleftarrow{1-g} P_1=\kk[C_p] \xleftarrow{1+g+\ldots+g^{p-1}} P_2 = \kk[C_p] \xleftarrow{1-g} \ldots \right). $$
The complex $P_\bullet$ is a free resolution of the trivial $\kk[C_p]$-module $\kk$. We write $P_{\bullet}^{(m)}$, $m\geq 0$ for the \emph{$m$-th stupid truncation} of $P_\bullet$, i.e. $P^{(m)}_i=P_i$ if $i\leq m$, and $P^{(m)}_i=0$ otherwise. Then we observe that the functor $\widetilde{H}^{(n)}\in \mathrm{Homog}^p(\Mod_{\kk},\Mod_{\kk})$, $n\geq 0$ can be given as follows
$$\widetilde{H}^{(n)}(V) \simeq 
\begin{cases}
\big(P^{(2n)}_{\bullet}\otimes A\big)\otimes_{hC_p} V^{\otimes p} & \mbox{if $p$ is odd, $V\in \Mod_{\kk}$,}\\
\big(P^{(n)}_{\bullet}\otimes A\big)\otimes_{hC_2} V^{\otimes 2} & \mbox{if $p=2$, $V\in \Mod_{\kk}$.}
\end{cases}
$$

Moreover, we observe that the homotopy groups $\pi_*\widetilde{H}^{(n)}(V)$, $V\in \Mod_{\kk}$ form a right module over the cohomology ring $H^*(C_p,\kk)$ by the cap-product. By Remark~\ref{remark:james equivalence choice}, we obtain that the induced map
$$\Phi_{k*}\colon  \pi_*\widetilde{H}^{(k+n)}(V) \to \pi_{*-k'}\widetilde{H}^{(n)}(V), \;\; V\in \Mod_{\kk}$$
is given by the cap-product with $t^{k}$ if $p$ is odd, and with $u^k$ if $p=2$.
\end{exmp}

\subsection{\texorpdfstring{$p$}{p}-th Goodwillie layer}\label{section: computations with the p-th goodwillie layer}

We recall that $\D_p\colon \Sp(\sL) \to \Sp(\sL)$ is the $p$-th derivative of the identity functor $\id_{\sL}$. By Corollary~\ref{corollary:derivatives of free}, the composite 
$$ \Mod_{\kk} \xrightarrow{\suspfree} \Sp(\sL) \xrightarrow{\D_p} \Sp(\sL)$$
is equivalent to the functor $\suspfree(\bfLie_p\otimes_{h\Sigma_p}V^{\otimes p})$, $V\in \Mod_{\kk}$. We consider the filtration $\widetilde{\D}^{(n)}_p\colon \Sp(\sL) \to \Sp(\sL)$, $n\geq 0$ on $\D_p$ by the functors from Definition~\ref{definition:filtration on homogeneous functors}. By Example~\ref{example: james filtration vector spaces}, the composite $\widetilde{\D}^{(n)}_p\circ \suspfree$ is equivalent to the functor
$$\widetilde{\D}^{(n)}_p(\suspfree V) \simeq 
\begin{cases}
\suspfree\Big(\big(P^{(2n)}_{\bullet}\otimes \bfLie_p \big)\otimes_{hC_p} V^{\otimes p}\Big) & \mbox{if $p$ is odd, $V\in \Mod_{\kk}$,}\\
\suspfree\Big(\big(P^{(n)}_{\bullet}\otimes \bfLie_2\big)\otimes_{hC_2} V^{\otimes 2}\Big) & \mbox{if $p=2$, $V\in \Mod_{\kk}$,}
\end{cases}
$$

Let $V\in \Mod_{\kk}$ be a chain complex such that $\pi_l(V)\cong\kk, l\geq 0$ and $\pi_*(V)=0, \ast\neq l$. By Corollary~\ref{corollary: kjaer basis of free lie algebra}, the homotopy groups $\pi_*(\bfLie_p \otimes_{h\Sigma_p}V^{\otimes p})$ are spanned by Dyer--Lashof--Lie classes $\bQe{\e}{i}(v)$, $v\in \pi_l(V)$, $i\geq 0$, $\e\in\{0,1\}$. Namely, let $v\in \pi_l(V)$ be a non-zero element, then $\pi_q(\bfLie_p\otimes_{h\Sigma_p}V^{\otimes p})$ is generated by
\begin{enumerate}
\item $\Qe{i}(v)$ if $p$ is odd and $q=l+2i(p-1)-1$, $2i \geq l+1$;
\item $\bQe{}{i}(v)$ if $p$ is odd and $q=l+2i(p-1)-2$, $2i\geq l+1$;
\item $\Qe{i}(v)$ if $p=2$ and $q=l+i-1$, $i\geq l+1$; 
\end{enumerate}
and $\pi_q(\bfLie_p\otimes_{h\Sigma_p}V^{\otimes p})=0$ in all other cases. 

\begin{rmk}\label{remark: image of filtration}
Let $V\in \Mod_{\kk}$ be such that $\pi_l(V)\cong\kk, l\geq 0$ and $\pi_*(V)=0, \ast\neq l$. Then the projection map $\widetilde{\D}_p^{(n)} \to \D_p$, $n\geq 0$ induces the $\Lambda$-linear map
$$p_{n*}\colon \pi_*(\widetilde{\D}_p^{(n)}(\suspfree V)) \to \pi_*(\D_p(\suspfree V)).$$
By Proposition~\ref{proposition: james flitration for functors}, the homotopy groups $\pi_*(\D_p(\suspfree V))$ are the union of the images of all maps $p_{n*}, n\geq 0$. Moreover, suppose that $\bQe{\e}{i}(v)\in \pi_*(V)$ is a Dyer--Lashof--Lie class, then $\bQe{\e}{i}(v)$ belongs to the image of the map $p_{n*}$ if and only if $$2n>(p-1)(2i-l) \;\; \text{(resp. $n> i-l$)}. $$
\end{rmk}

Using the description above, we define the following series of endomorphisms of the homotopy groups $\pi_*(\bfLie_p\otimes_{h\Sigma_p}V^{\otimes p})$.
\begin{dfn}\label{definition: shift operators} 
Let $V\in \Mod_{\kk}$ be a chain complex such that $\pi_l(V)=\kk, l\geq 0$ and $\pi_*(V)=0, \ast\neq l$. We define the \emph{shift operators} $$\vph_k\colon \pi_*(\bfLie_p\otimes_{h\Sigma_p}V^{\otimes p}) \to\pi_{*-\bar{k}}(\bfLie_p\otimes_{h\Sigma_p}V^{\otimes p}), \;\; k\in\Z, $$
where $\bar{k}=2k(p-1)$ if $p$ is odd, and $\bar{k}=k$ if $p=2$,
as follows
$$\vph_k(\bQe{\e}{i}(v)) = \bQe{\e}{i-k}(v), \;\; v\in \pi_l(V), \;\; i\geq 1, \;\; \e\in\{0,1\}.$$
\end{dfn}

Since $\pi_*(\D_p(\suspfree V)) \cong \pi_*(\bfLie_p \otimes_{h\Sigma_p}V^{\otimes p})\otimes \Lambda$, $V\in \Mod_{\kk}$, the map $\vph_k, k\in\Z$ induces the $\Lambda$-linear map 
\begin{equation}\label{equation: Lambda-shift operators}
\psi_k\colon \pi_*(\D_p(\suspfree V)) \to \pi_{*-\bar{k}}(\D_p(\suspfree V)), \;\;  V\in \Mod_{\kk},\;\; k\in \Z,
\end{equation}
and where $\bar{k}$ as in Definition~\eqref{definition: shift operators}.
The next proposition follows directly from Proposition~\ref{proposition: james flitration for functors}, the last paragraph of Example~\ref{example: james filtration vector spaces}, and Definition~\ref{definition: dll-classes} of Dyer--Lashof--Lie classes.

\begin{prop}\label{proposition: phi_k are locally functorial}
For every $n\in \N$ there exists a constant $N=N(n)$ such that for any $k\in \N$, $\nu_p(k)>N$ the following diagram commutes
$$
\begin{tikzcd}[column sep=large]
\pi_*(\widetilde{\D}^{((p-1)k+n)}_p(\suspfree V)) \arrow{r}{\Phi_{(p-1)k*}} \arrow{d}[swap]{p_{((p-1)k+n)*}}
&\pi_{*-\bar{k}}(\widetilde{\D}^{(n)}_p(\suspfree V)) \arrow{d}{p_{n*}} \\
\pi_*({\D}_p(\suspfree V)) \arrow{r}{\psi_k}
&\pi_{*-\bar{k}}({\D}_p(\suspfree V)),
\end{tikzcd}
$$
where $\bar{k}$ as in Definition~\ref{definition: shift operators}. \qed
\end{prop}

We use Proposition~\ref{proposition: phi_k are locally functorial} to show that the maps~\eqref{equation: Lambda-shift operators} are~\emph{locally} functorial.

\begin{thm}\label{theorem: main theorem, james periodicity}
Let $V,W\in \Mod^{\geq 0}_{\kk}$ be chain complexes such that $\dim \pi_*(V)=\dim \pi_*(W)=1$ and let
$$f\colon \suspfree W \to \suspfree V\in \Sp(\sL)$$ be a map of suspension spectra. Then for any $x\in \pi_q\D_p(\suspfree W), q\geq 0$ there exists a constant $N=N(x)$ such that for any $k\in \N$, $\nu_p(k)>N$ the following identity holds
$$\D_p(f)_*(x)=\psi_{k}(\D_p(f)_*(\psi_{-k}(x))),$$
where $\D_p(f)_*\colon \D_p(\suspfree W) \to \D_p(\suspfree V)$ is the map induced by $f$.
\end{thm}

\begin{proof}
Let $l\geq 0$ be such that $\pi_l(W)=\kk$. Since all maps in the identity above are $\Lambda$-linear, we can assume that $x=\bQe{\e}{i}(w)$ (resp. $x=\Qe{i}(w)$ and $p=2$), where $w\in \pi_l(W)$, $2i\geq l+1$ and $\e\in\{0,1\}$ (resp. $i\geq l+1$). 

By Proposition~\ref{proposition: james flitration for functors}, there exist $n\in\N$ and $y\in \pi_q(\D_p^{(n)}(\suspfree W))$ such that $p_{n*}y=x$. Moreover, by the same proposition and Example~\ref{example: james filtration vector spaces}, there exists a constant $N=N(n)$ such that for all $k\in\N$, $\nu_p(k)>N$ the map
$$\Phi_{(p-1)k*}\colon  \pi_{q+\bar{k}}(\widetilde{\D}^{((p-1)k+n)}_p(\suspfree W)) \to \pi_{q}(\widetilde{\D}^{(n)}_p(\suspfree W))$$
is surjective. Here $\bar{k}$ as in Definition~\ref{definition: shift operators}. Therefore there exists a homotopy class $z\in \pi_{q+\bar{k}}(\widetilde{\D}^{((p-1)k+n)}_p(\suspfree W))$ such that $\Phi_{(p-1)k*}(z)=y$. By Proposition~\ref{proposition: phi_k are locally functorial}, we obtain $\psi_k(p_{((p-1)k+n)*}(z))=x$, and so $p_{((p-1)k+n)*}(z) = \psi_{-k}(x)$. Finally, we have
\begin{align*}
\D_p(f)_*(x)&=\D_p(f)_*(p_{n*}\Phi_{(p-1)k*}z)\\
&=p_{n*}\Phi_{(p-1)k*}(\widetilde{\D}^{((p-1)k+n)}_p(f)_*(z))\\
&=\psi_k p_{((p-1)k+n)*}(\widetilde{\D}^{((p-1)k+n)}_p(f)_*(z))\\
&= \psi_k(\D_p(f)_*(p_{((p-1)k+n)*}z))\\
&=\psi_k(\D_p(f)_*(\psi_{-k}(x))). \qedhere
\end{align*}
\end{proof}

\begin{rmk}\label{remark: james periodicity, bound}
In Theorem~\ref{theorem: main theorem, james periodicity}, suppose that $x=\bQe{\e}{i}(w)\in \pi_{q} \D_p(\suspfree W)$ (resp. $x=\Qe{i}(w)$ and $p=2$) is a Dyer--Lashof--Lie class, $w\in \pi_l(W)\cong \kk$. By Remark~\ref{remark: image of filtration} and Remark~\ref{remark:exact value of N}, one can take the constant $N=N(x)$ to be any greater than $2i-l$ (resp. $i-l$).
\end{rmk}

We finish the section by describing the image of the map $\D_p(f)_*$.
\begin{rmk}\label{remark: dyer-lashof-lie, excess}
Let $W\simeq \Sigma^{l}\kk$, $V\simeq \Sigma^{l'}\kk$, $l>l'$, and let $f\colon \suspfree W \to \suspfree V$ be a map of suspension spectra. Suppose that $\bQe{\e}{i}(\iota_l)\in \pi_{q}\D_p(\suspfree W)$ (resp. $\Qe{i}(\iota_l)$ if $p=2$) is a Dyer--Lashof--Lie class, $i\geq 0$, $\e\in\{0,1\}$. Then, by using Remark~\ref{remark: image of filtration}, we observe that
$$\D_p(f)_*(\bQe{\e}{i}(\iota_l))=\sum_{\alpha}\bQe{\e(\alpha)}{i(\alpha)}(\iota_{l'}) \otimes \nu_{\alpha} \in \pi_*(\suspfree V), $$
resp. $\D_2(f)_*(\Qe{i}(\iota_l))=\sum_{\alpha}\Qe{i(\alpha)}(\iota_{l'}) \otimes \lambda_{\alpha}$, where $\nu_{\alpha}\in \Lambda$ (resp. $\lambda_\alpha \in\Lambda$) is an arbitrary element of the $\Lambda$-algebra, $\e(\alpha)\in\{0,1\}$, $i(\alpha)\geq 0$, and $$2i(\alpha)\leq 2i- (l-l'),\;\; \text{resp. $i(\alpha)\leq i -(l-l')$}.$$
\end{rmk}

\section{Two-stage functors}\label{section: two-stage functors}

In this section  we adapt the argument of~\cite{Behrens11} to our context and prepare the proofs of Theorems~\ref{theorem: C, intro} and~\ref{theorem: E,intro}. More precisely, we showed in Proposition~\ref{proposition:james-hopf and adjoint} that the differentials in the algebraic Goodwillie spectral sequence are induced by the natural transformations~\eqref{equation:Dpk to Dk}
$$\delta_n\colon D_n(\Omega^{p-2}L_\xi \free) \to BD_{pn}(\Omega^{p-2}L_\xi \free) $$
which factorize as the composites $\delta_n\simeq \deloop\D_{pn}(\tilde{\delta}_n) \circ j_p$, where $j_p$ is the James--Hopf map and $\D_{pn}(\tilde{\delta}_n)$ are certain natural transformations~\eqref{equation: stable piece, differential} of functors from $\Mod_{\kk}^{+}$ to the stable category $\Sp(\sLxi)$. In our main theorems~\ref{theorem: the adjoint of delta_n}, \ref{theorem: the adjoint of delta_n, p=2} we relate the induced map $\D_{pn}(\tilde{\delta}_n)_*$ with the multiplication maps in the (spectral) Lie operad~\eqref{equation: lie multiplicatons, derivatives}.

In Section~\ref{section: coalgebras in pronilpotent}, we identify the abstractly defined category of cocommutative coalgebras $\CoAlg(\calC^\otimes)$ (Definition~\ref{definition: abstract coalgebras}) provided the symmetric monoidal category $\calC^\otimes$ is \emph{pro-nilpotent} (Definition~\ref{definition: pro-nilpotent category}) and \emph{$\Sigma$-free} (Definition~\ref{definition: Sigma-free}). The main result of the section is Corollary~\ref{corollary: pronil dp sigma-free same} which shows that $\CoAlg(\calC^\otimes)$ is co\-monadic and the comonad is defined by the spectral cocommutative cooperad (Theorem~\ref{theorem: comonad for nilpotent divided power coalgebras}).

In Section~\ref{section: coalgebras in symmetric sequences} we endow the category $\sseq_0=\sseq_0(\Mod_{\kk})$ with the (non-unital) symmetric monoidal product $-\circledast -$ given by the Day convolution. By the results of Section~\ref{section: coalgebras in pronilpotent}, the category $\CoAlg(\sseq^\circledast)$ of cocommutative coalgebras in $\sseq^{\circledast}$ models the category of left comodules over the cocommutative cooperad $\bfComm_*$ in symmetric sequences, Remark~\ref{remark:coalg=left comodule}. We show in Proposition~\ref{proposition:coalg is presentable} that the category $\CoAlg(\sseq_0^\circledast)$ is presentable and comonadic over $\sseq_0$; the resulting comonad is~\eqref{equation: cofree coalgebra in sseq}. 
Thereafter we apply the Koszul duality machine from~\cite[Section~3]{GF12} to obtain an explicit adjoint pair~\eqref{equation: koszul adjunction}
$$
\begin{tikzcd}
Q^{\mathrm{enh}}: \bfLie^{\mathrm{sp}}(\sseq_0^\circledast) \arrow[shift left=.6ex]{r}
& \CoAlg(\sseq_0^\circledast) :\prim^{\mathrm{enh}}, \arrow[shift left=.6ex,swap]{l}
\end{tikzcd}
$$
where the category $\bfLie^{\mathrm{sp}}(\sseq_0^\circledast)$ models the category of left modules over the spectral Lie operad $\bfLie_{*}^{\mathrm{sp}}$ in $\sseq_0$, see the formula~\eqref{equation: prim triv Lie} and the discussion before. We end Section~\ref{section: coalgebras in symmetric sequences} with the definition of the \emph{corona cell map}~\eqref{equation: corona cell map, general}, which approximates the symmetric sequence $\prim^{\mathrm{enh}}(A)$ of primitive elements for a given $A\in \CoAlg(\sseq_0^\circledast)$.

Given a reduced finitary functor $F\in \Fun^{\omega}_*(\Sp,\sL)$, we define in~\eqref{equation: linear coderivatives} its symmetric sequence $\bar{\6}^*(F)\in \sseq$ of \emph{coderivatives}. By Proposition~\ref{proposition: coderivatives are symmetric monoidal}, the sequence $\bar{\6}^*(F)$ has a natural structure of a commutative coalgebra (with respect to the Day convolution), which we denote by $\AC(F) \in \CoAlg(\sseq_0^\circledast)$, see~\eqref{equation: Arone-Chong coderivatves}. In Theorem~\ref{theorem:AC conjecture for analytic functors} we show that $\prim(\AC(F))$ is the sequence of \emph{derivatives} $\bar{\6}_*(F)$, see~\eqref{equation: linear derivatives}, provided the functor $F$ is \emph{analytic} (Definition~\ref{definition: analytic functor}). The last theorem provides the sequence $\bar{\6}_*(F)$ with a canonical \emph{left $\bfLie_*^{\mathrm{sp}}$-module} structure, see~\eqref{equation: lie multiplicatons, derivatives}.

In Section~\ref{section: connecting map} we consider analytic functors with exactly two non-trivial Goodwillie layers, see Definition~\ref{defininition: two stage}. By reasoning as in~\cite[Section~2]{Behrens11}, we relate in Corollary~\ref{corollary:lie action, corona, adjoint} the only non-trivial Lie-operadic multiplication map 
$$m\colon \bfLie_{k}^{\mathrm{sp}}\otimes \bar{\6}_n(F)^{\otimes k} \to \bar{\6}_{nk}(F)$$
with the corona cell map and the connecting map $\delta\colon D_n(F) \to BD_{nk}(F)$. In Section~\ref{section: main example} we apply Corollary~\ref{corollary:lie action, corona, adjoint} to the truncated functor $\widetilde{P}_n(\Omega^{p-2}L_\xi \free)$ (see Section~\ref{section: goodwillie tower of a free lie algebra}) and identify the induced map $\D_{pn}(\tilde{\delta}_n)_*$ in Theorems~\ref{theorem: the adjoint of delta_n} and~\ref{theorem: the adjoint of delta_n, p=2}.

\subsection{Coalgebras in pro-nilpotent symmetric monoidal categories}\label{section: coalgebras in pronilpotent}
Let us denote by $\Pre^{L,\otimes}_{\mathrm{st}}=\CAlg(\Pre^L_{\mathrm{st}})$ the category of stable presentably (non-unital) symmetric monoidal categories (see~\cite[Definition~2.1]{NikolausSagave17}) and colimit-preserving strong symmetric monoidal functors. 
In Corollary~\ref{corollary: pronil dp sigma-free same}, we compare several notions of \emph{cocommutative coalgebras} in the category $\calC^{\otimes} \in \Pre^{L,\otimes}_{\mathrm{st}}$ if $\calC^{\otimes}$ is \emph{pro-nilpotent} (see Definition~\ref{definition: pro-nilpotent category}) and \emph{$\Sigma$-free} (see Definition~\ref{definition: Sigma-free}).

\begin{dfn}[Definition~4.1.1 in \cite{GF12}]\label{definition: pro-nilpotent category}
A stable presentably (non-unital) symmetric monoidal category $\calC^{\otimes}\in \Pre^{L,\otimes}_{\mathrm{st}}$ is called \emph{pro-nilpotent} if it can be exhibited as the limit
$$\calC^{\otimes} \xrightarrow{\simeq} \lim_{j\in \N^{op}} \calC_j^{\otimes} \in \Pre^{L,\otimes}_{\mathrm{st}} $$
such that
\begin{enumerate}
\item $\calC_0=0$;
\item for every $i\geq j$, the transition functor $f_{i,j}\colon \calC_i \to \calC_j$ additionally preserves limits;
\item for every $i \in \N$, the tensor product $\calC_i\otimes\calC_i \to \calC_i$ restricted to $\ker(f_{i,i-1})\otimes \calC_i$ is null-homotopic.
\end{enumerate}
Finally, a pro-nilpotent category $\calC^{\otimes}$ is called \emph{nilpotent of order $n$} if the transition functors $f_{i, j}$ are equivalences for $i\geq j\geq n$.
\end{dfn}

\begin{rmk}\label{remark: pronilpotent projections}
If $\calC^{\otimes} \in \Pre^{L,\otimes}_{\mathrm{st}}$ is pro-nilpotent, we write $f_j\colon \calC^{\otimes} \to \calC_j^{\otimes}$, $j\geq 0$ for the projections. Note that $f_j$, $j\geq 0$ are jointly conservative and preserve both limits and colimits.
\end{rmk}

\begin{lmm}\label{lemma: basic properties of pro-nilpotent}
$\quad$
\begin{enumerate}
\item Suppose that $\calC^{\otimes} \in \Pre^{L,\otimes}_{\mathrm{st}}$ is nilpotent of order $n$, then  $X^{\otimes n}\simeq 0$ for any $X\in\calC$.
\item Suppose that $\calC^{\otimes} \in \Pre^{L,\otimes}_{\mathrm{st}}$ is pro-nilpotent, then the natural map
\begin{equation}\label{equation: symmetric comonad pro-nilpotent}
\eta \colon \bigoplus_{n\geq 0} X^{\otimes n}_{h\Sigma_n} \to \prod_{n\geq 0} X^{\otimes n}_{h\Sigma_n}
\end{equation}
is an equivalence for any $X\in\calC$.
\end{enumerate}
\end{lmm}

\begin{proof}
We proof the first part by induction on the nilpotency order. If $n=0$, then the statement is clear. Next, assume that there exists a symmetric monoidal functor $F\colon \calC^{\otimes} \to \calD^{\otimes}$ such that 
\begin{enumerate}
\item $\calD^\otimes$ is nilpotent of order $n-1$;
\item the tensor product $\calC\otimes\calC \to \calC$ restricted to $\ker(F)\otimes \calC$ is null-homotopic;
\end{enumerate}
We claim that $X^{\otimes n}\simeq 0$ for any $X\in \calC$. Indeed, $X^{\otimes n}\simeq X^{\otimes (n-1)}\otimes X$ and $X^{\otimes (n-1)} \in \ker(F)$ by the inductive assumption. The second property implies the claim.

By Remark~\ref{remark: pronilpotent projections}, the map $\eta$ is an equivalence if and only if $$f_j(\eta)\colon f_j\Big(\bigoplus_{n\geq 0} X^{\otimes n}_{h\Sigma_n}\Big) \simeq \bigoplus_{n\geq 0} (f_jX)^{\otimes n}_{h\Sigma_n} \to \prod_{n\geq 0} (f_jX)^{\otimes n}_{h\Sigma_n}\simeq f_j\Big(\prod_{n\geq 0} X^{\otimes n}_{h\Sigma_n}\Big)$$ is an equivalence for all $j\geq 0$. Since the category $\calC_j$ is nilpotent of order $j$ and stable, the lemma follows.
\end{proof}

Recall that a \emph{comonad} $M\colon \calC \to \calC$ is a (coassociative) coalgebra object in the monoidal endomorphism category $\mathrm{End}(\calC) = \Fun(\calC,\calC)$ equipped with the composition product. We refer the reader to~\cite[Section~4.7]{HigherAlgebra} for a general account for (co)monads and (co)algebras over them in $\infty$-categories.

\begin{thm}\label{theorem: comonad for nilpotent divided power coalgebras}
Let $\calC^{\otimes} \in \Pre^{L,\otimes}_{\mathrm{st}}$ be a stable presentably symmetric monoidal category. Then the spectral cocommutative cooperad induces the comonad $\Sym_{\calC} \colon \calC \to \calC$ of the following form
$$\Sym_{\calC}(X)\simeq \bigoplus_{n\geq 1} X^{\otimes n}_{h\Sigma_n}, \;\; X\in \calC. $$
\end{thm}

\begin{proof}
See e.g.~\cite[Section~3.5]{GF12} or~\cite[Section~4.1.2]{Brantner_Thesis}.
\end{proof}

\begin{dfn}\label{definition: nilpotent divided power coalgebras}
Let $\calC^{\otimes} \in \Pre^{L,\otimes}_{\mathrm{st}}$ be a stable presentably symmetric monoidal category. We define the \emph{category of nilpotent divided power (cocommutative and non-unital) coalgebras} $\cCAlg^{\mathrm{nil}}_{\mathrm{d.p}}(\calC^{\otimes})$ in the category $\calC^{\otimes}$ as the category of coalgebras over the comonad $\Sym_{\calC}$.
\end{dfn}

We point out that there is an adjoint pair
\begin{equation}\label{equation: nil dp free oblv}
\begin{tikzcd}
\oblv: \cCAlg^{\mathrm{nil}}_{\mathrm{d.p}}(\calC^{\otimes}) \arrow[shift left=.6ex]{r}
&\calC:\mathfrak{C}_{\mathrm{d.p}} \arrow[shift left=.6ex,swap]{l}
\end{tikzcd}
\end{equation}
such that the category $\cCAlg^{\mathrm{nil}}_{\mathrm{d.p}}(\calC^{\otimes})$ is comonadic over $\calC$ and $\oblv \circ \mathfrak{C}_{\mathrm{d.p}} \simeq \Sym_{\calC}$.

\begin{dfn}\label{definition: finite totalizations}
Let $\calC$ be a category which admits small limits. A cosimplicial object $X^\bullet\in \Fun(\Delta,\calC)$ is \emph{$m$-coskeletal} if it is the right Kan extension of its restriction to $\Delta_{\leq m}$, that is the full subcategory of $\Delta$ spanned by $[0], [1], \ldots, [m]$. Note that the category $\Delta_{\leq m}$ is \emph{finite}.
\end{dfn}

\begin{dfn}\label{definition: preserve finite totalizations}
Let $\calC, \calD$ be categories which admit small limits. A functor $F\colon \calC \to \calD$ \emph{preserves finite totalizations} if, for every $m\geq 0$ and every $m$-coskeletal cosimplicial object $X^\bullet\in \Fun(\Delta,\calC)$, the natural map
$$F(\Tot(X^\bullet)) \to \Tot(F(X^{\bullet}))$$
is an equivalence.
\end{dfn}

\begin{rmk}\label{remark: fiber of particial total}
Let $\calC$ be a pointed category which admits small limits. For a cosimplicial object $X^\bullet\in \Fun(\Delta, \calC)$ and $q\geq 1$, there are the following fiber sequences
$$\Omega^qN^qX^{\bullet} \to \mathrm{Tot}^{q}(X^\bullet) \to \mathrm{Tot}^{q-1}(X^{\bullet})$$
$$N^qX^{\bullet} \to X^q \xrightarrow{\rho_q} \lim_{\substack{[q] \twoheadrightarrow [k]\\k<q}} X^{k},$$
see~\cite[Proposition~X.6.3]{BK72} and~\cite{MS16}. We recall that $X^\bullet$ is $m$-coskeletal if and only if the projection map $\rho_q$ is an equivalence for each $q\geq m+1$. We also note that $N^qX^{\bullet}$ is equivalent to the total fiber of a $q$-cubical diagram
\begin{equation}\label{equation:total fiber}
N^qX^{\bullet} \simeq \tfiber_{S\subseteq \underline{q}} X^{q-|S|},
\end{equation}
where all maps are given by the codegeneracy maps in the cosimplicial object $X^{\bullet}$. Finally, if $\calC$ is stable and $N^qX^\bullet$ is contractible for all $q\geq m+1$, then $X^\bullet$ is $m$-coskeletal.
\end{rmk}

\begin{thm}\label{theorem: nil dp basic properties}
Let $\calC^{\otimes} \in \Pre^{L,\otimes}_{\mathrm{st}}$ be a nilpotent symmetric monoidal category. Then
\begin{enumerate}
\item the category $\cCAlg^{\mathrm{nil}}_{\mathrm{d.p}}(\calC^{\otimes})$ is presentable;
\item the forgetful functor
$$\oblv \colon \cCAlg^{\mathrm{nil}}_{\mathrm{d.p}}(\calC^{\otimes}) \to \calC $$
commutes with finite totalizations;
\item for any object $X \in \cCAlg^{\mathrm{nil}}_{\mathrm{d.p}}(\calC^{\otimes})$, there exists $m\geq 0$ such that the cobar resolution $$C^{\bullet}(X) \colon \Delta \to \cCAlg^{\mathrm{nil}}_{\mathrm{d.p}}(\calC^{\otimes}), \;\; [n]\in \Delta \mapsto \mathfrak{C}_{\mathrm{d.p}}(\Sym_{\calC}^{\circ n}(\oblv(X)))$$
is $m$-coskeletal.
\end{enumerate}
\end{thm}

\begin{proof}
The first assertion follows by~\cite[Propositions~7.3 and~7.5]{Heuts24} together with Lemma~\ref{lemma: basic properties of pro-nilpotent}.

By~\cite[Proposition~3.37]{BM19}, the comonad $\Sym_{\calC} \colon \calC \to \calC$ preserves finite totalizations. Therefore, by~\cite[Corollary~4.2.3.5]{HigherAlgebra}, the forgetful functor $\oblv$ also preserves finite totalizations.

By Remark~\ref{remark: fiber of particial total}, we compute $$N^qC^\bullet(X) \simeq \mathfrak{C}_{\mathrm{d.p}}\left(\widetilde{\Sym}_{\calC}^{\circ q}(\oblv(X))\right), \;\; X\in \cCAlg^{\mathrm{nil}}_{\mathrm{d.p}}(\calC^{\otimes}), \;q\geq 1, $$
where $\widetilde{\Sym}_{\calC}(-)\simeq \bigoplus_{n\geq 2} (-)^{\otimes n}_{h\Sigma_n}$ is an endofunctor of $\calC$. By Lemma~\ref{lemma: basic properties of pro-nilpotent}, we observe that $N^qC^\bullet(X) \simeq\ast$ for $q\gg 0$, which implies the last assertion.
\end{proof}

Let $A, B \in \calC$ and $\calC^{\otimes} \in \Pre^{L,\otimes}_{\mathrm{st}}$ be a stable presentably symmetric monoidal category. Consider the natural map
$$\Sym_{\calC}(A\oplus B) = \bigoplus_{n\geq 1}(A\oplus B)^{\otimes n}_{h\Sigma_n} \to (A\oplus B)^{\otimes 2}_{h\Sigma_2} \to A\otimes B.$$
Then the adjoint pair~\eqref{equation: nil dp free oblv} induces the natural projection map
\begin{equation}\label{equation: pre lax free nil dp}
\mathfrak{C}_{\mathrm{d.p}}(A) \times \mathfrak{C}_{\mathrm{d.p}}(B) \simeq  \mathfrak{C}_{\mathrm{d.p}}(A\oplus B) \to \mathfrak{C}_{\mathrm{d.p}}(A\otimes B)
\end{equation}
in $\cCAlg^{\mathrm{nil}}_{\mathrm{d.p}}(\calC^{\otimes})$. Next, let $X, Y \in \cCAlg^{\mathrm{nil}}_{\mathrm{d.p}}(\calC^{\otimes})$ be nilpotent divided power coalgebras in $\calC^{\otimes}$. Then the map~\eqref{equation: pre lax free nil dp} and the unit map of the adjoint pair~\eqref{equation: nil dp free oblv} induce the map 
$$
X\times Y \to \mathfrak{C}_{\mathrm{d.p}}(\oblv(X)) \times \mathfrak{C}_{\mathrm{d.p}}(\oblv(Y)) \xrightarrow{\eqref{equation: pre lax free nil dp}} \mathfrak{C}_{\mathrm{d.p}}(\oblv(X)\otimes \oblv(Y))
$$
in $\cCAlg^{\mathrm{nil}}_{\mathrm{d.p}}(\calC^{\otimes})$. By applying again the adjoint pair~\eqref{equation: nil dp free oblv}, we obtain the natural map
\begin{equation}\label{equation: pre oplax oblv nil dp}
\oblv(X\times Y) \to \oblv(X)\otimes \oblv(Y) \in \calC.
\end{equation}

\begin{prop}\label{proposition: cartesian product in nil dp}
Let $\calC^{\otimes} \in \Pre^{L,\otimes}_{\mathrm{st}}$ be a nilpotent symmetric monoidal category. Then the sequence 
$$\oblv(X)\oplus \oblv(Y) \simeq \oblv(X\sqcup Y) \to \oblv(X\times Y) \to \oblv(X)\otimes \oblv(Y)
$$
is a (split) fiber sequence in $\calC$ for every $X,Y \in \cCAlg^{\mathrm{nil}}_{\mathrm{d.p}}(\calC^{\otimes})$.
\end{prop}

\begin{proof}
By Theorem~\ref{theorem: nil dp basic properties}, the forgetful functor $\oblv$ preserves finite totalizations and any object $C \in \cCAlg^{\mathrm{nil}}_{\mathrm{d.p}}(\calC^{\otimes})$ is a finite totalization of a cosimplicial diagram made of cofree objects. Therefore, we can assume that $X\simeq \mathfrak{C}_{\mathrm{d.p}}(A)$ and $Y\simeq \mathfrak{C}_{\mathrm{d.p}}(B)$ are cofree coalgebras, $A,B\in\calC$. Finally, the proposition follows since
$$\Sym_{\calC}(A\oplus B) \simeq \Sym_{\calC}(A) \oplus \Sym_{\calC}(B) \oplus \Sym_{\calC}(A) \otimes \Sym_{\calC}(B). $$
\end{proof}

\begin{thm}\label{theorem: nil dp is differentiable}
Let $\calC^{\otimes} \in \Pre^{L,\otimes}_{\mathrm{st}}$ be a nilpotent symmetric monoidal category. Then the category $\cCAlg^{\mathrm{nil}}_{\mathrm{d.p}}(\calC^{\otimes})$ is differentiable in the sense of~\cite[Definition~6.1.1.6]{HigherAlgebra}.
\end{thm}

\begin{proof}
By Theorem~\ref{theorem: nil dp basic properties}, the category $\cCAlg^{\mathrm{nil}}_{\mathrm{d.p}}(\calC^{\otimes})$ is presentable. So, in particular, it admits all finite limits and all sequential colimits. By Proposition~\ref{proposition: cartesian product in nil dp}, sequential colimits commute with finite products. By \cite[Proposition~4.4.3.2]{HTT} (see also \cite[\href{https://kerodon.net/tag/03HF}{Tag 03HF}]{kerodon}), it is enough to show that sequential colimits commute with equalizers.

By Theorem~\ref{theorem: nil dp basic properties}, sequential colimits commute with finite totalizations. For any equalizer diagram $A \rightrightarrows B$, we have an equivalence
\begin{equation}
\begin{tikzcd}\label{equation: equalizer is totalization}
\lim\Big(A \arrow[shift left=.6ex]{r}
  \arrow[shift right=.6ex,swap]{r}
&
B\Big) \simeq \Tot\Big([n] \mapsto (B^{\times n} \times A)\Big).
\end{tikzcd}
\end{equation}
Note that the cosimplicial diagram on the right hand side is $1$-coskeletal. Since sequential colimits commute with finite products, sequential colimits map the cosimplicial diagram in~\eqref{equation: equalizer is totalization} to a cosimplicial diagram of the same form. Therefore, sequential colimits commute with arbitrary equalizers.
\end{proof}

\begin{rmk}\label{remark: coalgebras are not compactly generated}
Typically, the category $\cCAlg^{\mathrm{nil}}_{\mathrm{d.p}}(\calC^{\otimes})$ is \emph{not} compactly generated even if the presentable category $\calC$ is compactly generated.
\end{rmk}

\begin{dfn}\label{definition: Sigma-free}
A symmetric monoidal category $\calC^{\otimes}$ is called \emph{$\Sigma$-free} if $\calC$ admits arbitrary small limits and  there exists a functor
$$\widetilde{T}^n\colon \calC \to \calC $$
for every $n\geq 1$ such that the composite 
$$\calC \xrightarrow{\widetilde{T}^n} \calC \xrightarrow{\Coind^{\Sigma_n}} \calC^{B\Sigma_n} $$
is equivalent to the functor $X \in \calC \mapsto X^{\otimes n} \in \calC^{B\Sigma_n}$. Here $\Coind^{\Sigma_n}$ is the right adjoint to the forgetful functor $\calC^{B\Sigma_n}=\Fun(B\Sigma_n,\calC) \to \calC$.
\end{dfn}

\begin{rmk}\label{remark: Sigma-free middle}
Let $\calC^{\otimes} \in \Pre^{L,\otimes}$ be a $\Sigma$-free presentably symmetric monoidal category. Then the functors $\widetilde{T}^n$, $n\geq 1$ are defined uniquely since $\widetilde{T}^n(X)\simeq (X^{\otimes n})^{h\Sigma_n}$.
\end{rmk}

\begin{rmk}\label{remark: Sigma-free and Tate}
Let $\calC^{\otimes} \in \Pre^{L,\otimes}_{\mathrm{st}}$ be a \emph{stable} presentably symmetric monoidal category. Suppose that $\calC^{\otimes}$ is $\Sigma$-free. Then the norm map (see \cite[Example~6.1.6.22]{HigherAlgebra})
$$N\colon (X^{\otimes n})_{h\Sigma_n} \xrightarrow{\simeq} (X^{\otimes n})^{h\Sigma_n} $$
is an equivalence for any $X\in \calC$ and $n\geq 1$. Moreover, $(X^{\otimes n})_{h\Sigma_n} \simeq \widetilde{T}^n(X)$ for $n\geq 1$.
\end{rmk}

\begin{lmm}\label{lemma: sigma-free suspention}
Let $\calC^{\otimes} \in \Pre^{L,\otimes}_{\mathrm{st}}$ be a stable presentably symmetric monoidal category. Suppose that $\calC^{\otimes}$ is $\Sigma$-free. Then the natural transformation
$$\sigma \colon \Sigma X^{\otimes n}_{h\Sigma_n} \to (\Sigma X)^{\otimes n}_{h\Sigma_n} $$
induced by the suspension is trivial for every $n\geq 2$. \qed
\end{lmm}

\begin{prop}\label{proposition: nil dp stabilization nil}
Let $\calC^{\otimes} \in \Pre^{L,\otimes}_{\mathrm{st}}$ be a nilpotent symmetric monoidal category. Suppose that $\calC^{\otimes}$ is $\Sigma$-free. Then the left adjoint $\oblv \colon \cCAlg^{\mathrm{nil}}_{\mathrm{d.p}}(\calC^{\otimes}) \to \calC$ induces the natural equivalence
$$ \Sp(\cCAlg^{\mathrm{nil}}_{\mathrm{d.p}}(\calC^{\otimes})) \xrightarrow{\oblv} \Sp(\calC) \simeq \calC $$
between stabilizations.
\end{prop}

\begin{proof}
By Theorem~\ref{theorem: nil dp basic properties}, the category $\cCAlg^{\mathrm{nil}}_{\mathrm{d.p}}(\calC^{\otimes})$ is presentable. Therefore, $\Sp(\cCAlg^{\mathrm{nil}}_{\mathrm{d.p}}(\calC^{\otimes}))$ is equivalent to the colimit
$$ \colim\left(\cCAlg^{\mathrm{nil}}_{\mathrm{d.p}}(\calC^{\otimes}) \xrightarrow{\Sigma} \cCAlg^{\mathrm{nil}}_{\mathrm{d.p}}(\calC^{\otimes}) \xrightarrow{\Sigma} \cCAlg^{\mathrm{nil}}_{\mathrm{d.p}}(\calC^{\otimes}) \to \ldots\right)$$
in the category $\Pre^L$. Since the category $\cCAlg^{\mathrm{nil}}_{\mathrm{d.p}}(\calC^{\otimes})$ is comonadic over $\calC$, the functor $\Sigma \colon \cCAlg^{\mathrm{nil}}_{\mathrm{d.p}}(\calC^{\otimes}) \to \cCAlg^{\mathrm{nil}}_{\mathrm{d.p}}(\calC^{\otimes})$ is induced by the following morphism of comonads
$$\sigma\colon \Sym_{\calC}(-) \to \Sigma^{-1} \Sym_{\calC}(\Sigma(-)). $$
By Lemma~\ref{lemma: sigma-free suspention}, the morphism $\sigma$ factors as follows
$$\sigma\colon \Sym_{\calC}(-) \to \id_{\calC} \to \Sigma^{-1} \Sym_{\calC}(\Sigma(-)), $$
where $\id_\calC$ is the identity comonad. Hence, the suspension functor $$\Sigma \colon \cCAlg^{\mathrm{nil}}_{\mathrm{d.p}}(\calC^{\otimes}) \to \cCAlg^{\mathrm{nil}}_{\mathrm{d.p}}(\calC^{\otimes})$$ factors as follows
$$\Sigma\colon \cCAlg^{\mathrm{nil}}_{\mathrm{d.p}}(\calC^{\otimes}) \xrightarrow{\oblv} \calC \xrightarrow{\Sigma} \calC \xrightarrow{\triv} \cCAlg^{\mathrm{nil}}_{\mathrm{d.p}}(\calC^{\otimes}).$$
This implies the assertion.
\end{proof}

\begin{dfn}\label{definition: nil dp smash product}
Let $\calC^{\otimes} \in \Pre^{L,\otimes}_{\mathrm{st}}$ be a nilpotent symmetric monoidal category. Then we write 
$$\Sp(\cCAlg^{\mathrm{nil}}_{\mathrm{d.p}}(\calC^{\otimes})^{\times}) \to N\mathrm{Surj}$$ for the \emph{stabilization} of the Cartesian symmetric monoidal structure $$\cCAlg^{\mathrm{nil}}_{\mathrm{d.p}}(\calC^{\otimes})^{\times} \to N\mathrm{Surj},$$ see~\cite[Definition~6.2.4.12]{Heuts21}. In particular, the category $\Sp(\cCAlg^{\mathrm{nil}}_{\mathrm{d.p}}(\calC^{\otimes})^{\times})$ is a stable (non-unital) $\infty$-operad whose underlying category is the stabilization $\Sp(\cCAlg^{\mathrm{nil}}_{\mathrm{d.p}}(\calC^{\otimes}))$, see~\cite[Definition~6.2.4.10]{HigherAlgebra} and~\cite[Definition~2.11]{Heuts21}
\end{dfn}

\begin{cor}\label{corollary: nil dp smash product}
Let $\calC^{\otimes} \in \Pre^{L,\otimes}_{\mathrm{st}}$ be a nilpotent symmetric monoidal category. Suppose that $\calC$ is $\Sigma$-free. Then the right adjoint $\mathfrak{C}_{\mathrm{d.p}} \colon \calC \to \cCAlg^{\mathrm{nil}}_{\mathrm{d.p}}(\calC^{\otimes})$ induces the natural equivalence
$$\calC^{\otimes} \xrightarrow{\mathfrak{C}_{\mathrm{d.p}}} \Sp(\cCAlg^{\mathrm{nil}}_{\mathrm{d.p}}(\calC^{\otimes})^\times) $$
between stable $\infty$-operads.
\end{cor}

\begin{proof}
By Theorem~\ref{theorem: nil dp is differentiable} and \cite[Propositions~6.2.4.14 and 6.2.4.15]{HigherAlgebra}, the stabilization $\Sp(\cCAlg^{\mathrm{nil}}_{\mathrm{d.p}}(\calC^{\otimes})^\times)$ exists and unique. By Proposition~\ref{proposition: nil dp stabilization nil}, the underlying stable category of $\Sp(\cCAlg^{\mathrm{nil}}_{\mathrm{d.p}}(\calC^{\otimes})^\times)$ is $\calC$, and it suffices to show that the monoidal structures coincide. By~\cite[Definition~6.2.4.10]{HigherAlgebra}, the $\infty$-operad $\Sp(\cCAlg^{\mathrm{nil}}_{\mathrm{d.p}}(\calC^{\otimes})^\times)$ is corepresentable and it is defined by $k$-th symmetric tensor products $$\otimes^k\colon \calC^{\times k} \to \calC, \;\; k\geq 1,$$ see~\cite[Remark~6.2.4.4]{HigherAlgebra}, which are (multi-)derivatives of the functor
\begin{align}\label{equation: nil dp n tensor product}
(X_1,\ldots,X_k)\in \calC^{\times k} &\mapsto \mathfrak{C}_{\mathrm{d.p}}(X_1)\times \ldots \times \mathfrak{C}_{\mathrm{d.p}}(X_k)\\ &\simeq \mathfrak{C}_{\mathrm{d.p}}(X_1\oplus \ldots \oplus X_k) \in \cCAlg^{\mathrm{nil}}_{\mathrm{d.p}}(\calC^{\otimes}), \nonumber
\end{align}
see~\cite[Definition~6.2.1.1]{HigherAlgebra}. Note that the derivative of the functor~\eqref{equation: nil dp n tensor product} coincides with the derivative of its composite with the forgetful functor $\oblv$. In other words, $\otimes^k$, $k\geq 1$ is the derivative of the functor
\begin{align*}
(X_1,\ldots,X_k)\in \calC^{\times k} &\mapsto \oblv\circ\mathfrak{C}_{\mathrm{d.p}}(X_1\oplus \ldots \oplus X_k) \\
&\simeq \bigoplus_{m\geq 1} (X_1\oplus \ldots \oplus X_k)^{\otimes m}_{h\Sigma_m}\in \calC.
\end{align*}
By the binomial formula, the derivative of the last functor is $$(X_1,\ldots,X_k)\mapsto X_1\otimes\ldots \otimes X_k.$$
\end{proof} 

\begin{cor}\label{corollary: nil dp lax adjunction}
Let $\calC^{\otimes} \in \Pre^{L,\otimes}_{\mathrm{st}}$ be a nilpotent symmetric monoidal category. Suppose that $\calC^{\otimes}$ is $\Sigma$-free. Then the right adjoint 
$$\mathfrak{C}_{\mathrm{d.p}} \colon \calC^{\otimes} \to \cCAlg^{\mathrm{nil}}_{\mathrm{d.p}}(\calC^{\otimes})^{\times}$$ 
inherits a lax symmetric monoidal enhancement such that the induced maps
$$\mathfrak{C}_{\mathrm{d.p}}(A) \times \mathfrak{C}_{\mathrm{d.p}}(B) \to \mathfrak{C}_{\mathrm{d.p}}(A\otimes B), \;\; A,B \in \calC $$
coincide with the natural maps~\eqref{equation: pre lax free nil dp}. Similarly, the forgetful functor
$$\oblv\colon \cCAlg^{\mathrm{nil}}_{\mathrm{d.p}}(\calC^{\otimes})^{\times} \to \calC^{\otimes} $$
inherits an oplax symmetric monoidal enhancement such that the induced maps $$\oblv(X\times Y) \to \oblv(X)\otimes \oblv(Y), \;\; X,Y \in \cCAlg^{\mathrm{nil}}_{\mathrm{d.p}}(\calC^{\otimes}) $$
coincide with~\eqref{equation: pre oplax oblv nil dp}.
\end{cor}

\begin{proof}
By Corollary~\ref{corollary: nil dp smash product}, the adjoint pair~\eqref{equation: nil dp free oblv} is equivalent to the adjoint pair 
$$
\begin{tikzcd}
\Sigma^{\infty}: \cCAlg^{\mathrm{nil}}_{\mathrm{d.p}}(\calC^{\otimes})^{\times} \arrow[shift left=.6ex]{r}
&\Sp(\cCAlg^{\mathrm{nil}}_{\mathrm{d.p}}(\calC^{\otimes})^{\times})\simeq \calC^{\otimes}:\Omega^{\infty} \arrow[shift left=.6ex,swap]{l}
\end{tikzcd}
$$
of stabilization. By~\cite[Definition~6.2.4.12]{HigherAlgebra}, the right adjoint functor $\Omega^{\infty} \simeq \mathfrak{C}_{\mathrm{d.p}}$ is a map of $\infty$-operads, and so is lax symmetric monoidal. By the proof of Corollary~\ref{corollary: nil dp smash product}, the lax symmetric monoidal structure for the functor $\mathfrak{C}_{\mathrm{d.p}}$ is given by the map~\eqref{equation: pre lax free nil dp}. We obtain the assertion about the left adjoint $\oblv$ by the adjunction.
\end{proof}

\begin{prop}\label{proposition: pronilpotent dp is limit of nil dp}
Let $\calC^{\otimes}\in \Pre^{L,\otimes}_{\mathrm{st}}$ be a pro-nilpotent symmetric monoidal category exhibited as the limit $\calC \simeq \lim \calC_j$. Then the natural functor
\begin{equation}\label{equattion: pronilpotent coalgebras is limit}
\cCAlg^{\mathrm{nil}}_{\mathrm{d.p}}(\calC^{\otimes}) \to \lim_{j\in \N^{op}} \cCAlg^{\mathrm{nil}}_{\mathrm{d.p}}(\calC^{\otimes}_j)
\end{equation}
is an equivalence. In particular, the projection functors
$$f_j\colon \cCAlg^{\mathrm{nil}}_{\mathrm{d.p}}(\calC^{\otimes}) \to \cCAlg^{\mathrm{nil}}_{\mathrm{d.p}}(\calC^{\otimes}_j), \;\; j\geq 1 $$
preserve arbitrary small limits and colimits.
\end{prop}

\begin{proof}
We observe that the forgetful functor 
$$\oblv\colon \lim_{j\in \N^{op}} \cCAlg^{\mathrm{nil}}_{\mathrm{d.p}}(\calC^{\otimes}_j) \to \lim_{j\in \N^{op}} \calC_j \simeq \calC$$
satisfies the conditions of the Barr--Beck--Lurie theorem~\cite[Theorem~4.7.3.5]{HigherAlgebra}, i.e. $\oblv$ is a conservative left adjoint which preserves limits of $\oblv$-split cosimplicial diagrams. Therefore, the limit category $\lim\cCAlg^{\mathrm{nil}}_{\mathrm{d.p}}(\calC^{\otimes}_j)$ is comonadic over $\calC$ and the obtained comonad is given by the formula
\begin{equation}\label{equation: pronilpotent comonad}
X\in \calC \mapsto  \prod_{k\geq 1} X^{\otimes k}_{h\Sigma_k}.
\end{equation}
Finally, the functor~\eqref{equattion: pronilpotent coalgebras is limit} induces the natural transformation~\eqref{equation: symmetric comonad pro-nilpotent} from the comonad $\Sym_{\calC}\simeq \bigoplus_{k\geq 1} (-)^{k}_{h\Sigma_k}$ of Theorem~\ref{theorem: comonad for nilpotent divided power coalgebras} to the comonad~\eqref{equation: pronilpotent comonad}. By Lemma~\ref{lemma: basic properties of pro-nilpotent}, this transformation is an equivalence.
\end{proof}

\begin{cor}\label{corollary: pronilpotent stabilization}
Let $\calC^{\otimes}\in \Pre^{L,\otimes}_{\mathrm{st}}$ be a pro-nilpotent symmetric monoidal category exhibited as the limit $\calC^{\otimes} \simeq \lim \calC_j^{\otimes}$. Suppose that $\calC^{\otimes}$ and each $\calC_j^{\otimes}$, $j\geq 0$ are $\Sigma$-free. Then 
\begin{enumerate}
\item the category $\cCAlg^{\mathrm{nil}}_{\mathrm{d.p}}(\calC^{\otimes})$ is presentable and differentiable;
\item the right adjoint $\mathfrak{C}_{\mathrm{d.p}}\colon \calC \to \cCAlg^{\mathrm{nil}}_{\mathrm{d.p}}(\calC^{\otimes})$ induces a symmetric monoidal equivalence
$$\calC^{\otimes} \xrightarrow{\mathfrak{C}_{\mathrm{d.p}}} \Sp(\cCAlg^{\mathrm{nil}}_{\mathrm{d.p}}(\calC^{\otimes})^\times) $$
between stable $\infty$-operads;
\item the forgetful functor $$\oblv\colon \cCAlg^{\mathrm{nil}}_{\mathrm{d.p}}(\calC^{\otimes})^{\times} \to \calC^{\otimes}$$ inherits an oplax symmetric monoidal enhancement such that the induced maps $$\oblv(X\times Y) \to \oblv(X)\otimes \oblv(Y), \;\; X,Y \in \cCAlg^{\mathrm{nil}}_{\mathrm{d.p}}(\calC^{\otimes}) $$
coincide with~\eqref{equation: pre oplax oblv nil dp}.
\end{enumerate}
\end{cor}

\begin{proof}
By~\cite[Proposition~5.5.3.13]{HTT} and Proposition~\ref{proposition: pronilpotent dp is limit of nil dp}, the category of nilpotent divided power coalgebras $\cCAlg^{\mathrm{nil}}_{\mathrm{d.p}}(\calC^{\otimes})$ is presentable. Moreover, by Proposition~\ref{proposition: pronilpotent dp is limit of nil dp}, the projection functors
$$f_j\colon \cCAlg^{\mathrm{nil}}_{\mathrm{d.p}}(\calC^{\otimes}) \to \cCAlg^{\mathrm{nil}}_{\mathrm{d.p}}(\calC^{\otimes}_j), \;\; j\geq 1 $$
are jointly conservative and preserve both finite limits and sequential colimits. Now, the first part follows by Theorem~\ref{theorem: nil dp is differentiable}. 

For the second part, let 
$$g_j \colon \cCAlg^{\mathrm{nil}}_{\mathrm{d.p}}(\calC^{\otimes}_j) \to \cCAlg^{\mathrm{nil}}_{\mathrm{d.p}}(\calC^{\otimes}) $$
denote the left adjoint of $f_j$, $j\geq 1$. Then
$$\colim_{j\in \N} \cCAlg^{\mathrm{nil}}_{\mathrm{d.p}}(\calC^{\otimes}_j) \xrightarrow{\simeq} \cCAlg^{\mathrm{nil}}_{\mathrm{d.p}}(\calC^{\otimes}) \in \Pre^L $$
is an equivalence, where colimit is computed in the category $\Pre^L$. Since the stabilization of a presentable category is the colimit (in $\Pre^L$) along suspensions (see Section~\ref{section: spectrum objects}), we have
$$\Sp(\cCAlg^{\mathrm{nil}}_{\mathrm{d.p}}(\calC^{\otimes}))\simeq \colim_{j\in \N}\Sp(\cCAlg^{\mathrm{nil}}_{\mathrm{d.p}}(\calC^{\otimes}_j)) \simeq \colim_{j\in \N}\calC_j \simeq \calC \in \Pre^L $$
by Proposition~\ref{proposition: nil dp stabilization nil}. Finally, the last part follows by Corollary~\ref{corollary: nil dp smash product} and Corollary~\ref{corollary: nil dp lax adjunction}.
\end{proof}

\begin{rmk}\label{remark: stabilization any case}
Let $\calC^{\otimes} \in \Pre^{L,\otimes}_{\mathrm{st}}$ be a stable presentably symmetric monoidal category. Assume that $\cCAlg^{\mathrm{nil}}_{\mathrm{d.p}}(\calC^{\otimes})$ is differentiable, e.g. $\calC$ is (pro-)nilpotent. Then~\cite[Corollary~6.2.2.16]{HigherAlgebra} shows that the stabilization of the \emph{comonadic} adjunction~\eqref{equation: nil dp free oblv} is the adjoint pair
\begin{equation}\label{equation: oblv sym stable}
\begin{tikzcd}
\partial \oblv: \Sp(\cCAlg^{\mathrm{nil}}_{\mathrm{d.p}}(\calC^{\otimes})) \arrow[shift left=.6ex]{r}
&\Sp(\calC)\simeq \calC:\partial\mathfrak{C}_{\mathrm{d.p}} \arrow[shift left=.6ex,swap]{l}
\end{tikzcd}
\end{equation}
such that $\partial\mathfrak{C}_{\mathrm{d.p}}$ is fully faithful. However, in opposite to the case of~\emph{monadic} adjunction, \cite[Corollary~6.2.2.16]{HigherAlgebra} does not guarantee that the adjoint pair~\eqref{equation: oblv sym stable} is an equivalence. At the time of writing, we are not aware if~\eqref{equation: oblv sym stable} is an equivalence provided $\calC^{\otimes}$ is (pro-)nilpotent but not $\Sigma$-free. 
\end{rmk}

Recall that the opposite of a symmetric monoidal category $\calC^{\otimes}$ is again symmetric monoidal, see e.g~\cite[Remark~2.4.2.7]{HigherAlgebra}.
\begin{dfn}[Section~3.1 in \cite{Elliptic-I}]\label{definition: abstract coalgebras}
Let $\calC^{\otimes}$ be a symmetric monoidal category. The \emph{category of (cocommutative and non-unital) coalgebras} $\CoAlg(\calC^{\otimes})$ in the category $\calC^{\otimes}$ is defined as follows
$$\CoAlg(\calC^{\otimes})=\Alg(\calC^{\otimes,op})^{op},$$
where $\Alg(\calC^{\otimes,op})$ is the category of (commutative and non-unital) algebras in the opposite symmetric monoidal category $\calC^{\otimes,op}$. 
\end{dfn}

\begin{prop}\label{proposition:coalg is presentable}
Let $\calC^{\otimes} \in \Pre^{L,\otimes}$ be a presentably symmetric monoidal category. Then the category $\cCAlg(\calC^{\otimes})$ is presentable as well. Moreover, the forgetful functor $$\oblv \colon \cCAlg(\calC^{\otimes}) \to \calC$$ is conservative, it preserves colimits and admits a right adjoint
$$\mathfrak{C}\colon \calC \to \cCAlg(\calC^{\otimes}). $$
\end{prop}

\begin{proof}
See~\cite[Section~3.1]{Elliptic-I}.
\end{proof}


\begin{thm}\label{theorem: pronil dp sigma-free comonadic}
Let $\calC^{\otimes}\in \Pre^{L,\otimes}_{\mathrm{st}}$ be a $\Sigma$-free pro-nilpotent symmetric monoidal category. Then the adjoint pair
$$
\begin{tikzcd}
\oblv: \CoAlg(\calC^{\otimes}) \arrow[shift left=.6ex]{r}
&\calC:\mathfrak{C} \arrow[shift left=.6ex,swap]{l}
\end{tikzcd}
$$
is comonadic. Moreover, the comonad $\Gamma_{\mathrm{d.p}} = \oblv \circ \mathfrak{C}$ evaluates on objects as follows
$$\Gamma_{\mathrm{d.p}}(X) = \oblv \circ \mathfrak{C}(X) \simeq  \bigoplus_{n\geq 1} X^{\otimes n}_{h\Sigma_n} \in \calC , \;\; X\in \calC. $$
\end{thm}

\begin{proof}
Let $\Pro(\calC)$ be the category of \emph{pro-objects} in $\calC$, see e.g.~\cite[Definition~3.1.1]{DAGXIII}. We think of objects of $\Pro(\calC)$ as cofiltered families of objects from~$\calC$. We warn the reader that since the category $\calC$ is not small, we consider $\Pro(\calC)$ as the $\infty$-category defined in the larger universe. 

By~\cite[Lemma~2.3]{KST19}, the category of pro-objects $\Pro(\calC)$ is stable. By~\cite[Remark~2.4.2.7 and Proposition~4.8.1.10]{HigherAlgebra}, $\Pro(\calC)$ inherits a symmetric monoidal structure $\Pro(\calC)^{\otimes}$ in which the tensor product commutes with arbitrary small limits at each variable. By~\cite[Lemma~5.1]{BB24}, the forgetful functor
$$\oblv_{\Pro}\colon \CoAlg(\Pro(\calC)^\otimes) \to \Pro(\calC) $$
is comonadic and it admits a right adjoint $\mathfrak{C}_{\Pro}$ such that
\begin{equation}\label{equation: pro comonad}
\oblv_{\Pro} \circ \mathfrak{C}_{\Pro}(X) \simeq \prod_{n\geq 1} (X^{\otimes n})^{h\Sigma_n} \in \Pro(\calC), \;\; X\in \Pro(\calC), 
\end{equation}
see also~\cite[Proposition 3.1.3.13]{HigherAlgebra}.

Recall that there exists a fully faithful functor
$$c\colon \calC\to \Pro(\calC) $$
which sends an object to the constant cofiltered family. By~\cite[Lemma~2.3]{KST19}, the category $\Pro(\calC)$ has all small limits and colimits, and the functor $c$ is left exact and preserves small colimits. Moreover, the functor $c$ is strong symmetric monoidal.

Since we assumed that the category $\calC^{\otimes}$ is pro-nilpotent and $\Sigma$-free, the category of pro-objects $\Pro(\calC)^{\otimes}$ is pro-nilpotent and $\Sigma$-free as well. Hence, for $X\in \calC$, we have
\begin{align*}
\oblv_{\Pro} \circ \mathfrak{C}_{\Pro}(cX) \simeq \prod_{n\geq 1} ((cX)^{\otimes n})^{h\Sigma_n} \simeq \bigoplus_{n\geq 1} ((cX)^{\otimes n})_{h\Sigma_n} \simeq c\Big(\bigoplus_{n\geq 1} (X^{\otimes n})_{h\Sigma_n} \Big),
\end{align*}
see Lemma~\ref{lemma: basic properties of pro-nilpotent} and Remark~\ref{remark: Sigma-free and Tate}. In other words, the comonad~\eqref{equation: pro comonad} preserves constant objects. This implies that the functor 
\begin{equation}\label{equation: pro composite}
\CoAlg(\calC^{\otimes}) \hookrightarrow \CoAlg(\Pro(\calC)^{\otimes})
\end{equation}
is fully faithful and a pro-coalgebra $C\in \CoAlg(\Pro(\calC)^{\otimes})$ belongs to the essential image of~\eqref{equation: pro composite} if and only if $\oblv_{\Pro}(C) \in \Pro(\calC)$ is equivalent to a constant object. This implies the theorem.
\end{proof}

Let $\calC^{\otimes}\in \Pre^{L,\otimes}_{\mathrm{st}}$ be a pro-nilpotent symmetric monoidal category exhibited as the limit $\calC^{\otimes} \simeq \lim \calC_j^{\otimes}$. Suppose that $\calC^{\otimes}$ and each $\calC_j^{\otimes}$, $j\geq 0$ are $\Sigma$-free. Then, by Corollary~\ref{corollary: pronilpotent stabilization}, the oplax symmetric monoidal functor
$$\oblv\colon \cCAlg^{\mathrm{nil}}_{\mathrm{d.p}}(\calC^{\otimes})^{\times} \to \calC^{\otimes} $$
induces the comparison functor
\begin{equation}\label{equation: nil dp comparison}
\zeta\colon \cCAlg^{\mathrm{nil}}_{\mathrm{d.p}}(\calC^{\otimes}) \to \CoAlg(\calC^{\otimes}).
\end{equation}
By~\cite[Proposition~3.2.4.7]{HigherAlgebra} and Proposition~\ref{proposition: cartesian product in nil dp}, the comparison functor $\zeta$ preserves finite products.

\begin{cor}\label{corollary: pronil dp sigma-free same}
Let $\calC^{\otimes}\in \Pre^{L,\otimes}_{\mathrm{st}}$ be a pro-nilpotent symmetric monoidal category exhibited as the limit $\calC^{\otimes} \simeq \lim \calC_j^{\otimes}$. Suppose that $\calC^{\otimes}$ and each $\calC_j^{\otimes}$, $j\geq 0$ are $\Sigma$-free. Then the comparison functor
$$\zeta\colon  \cCAlg^{\mathrm{nil}}_{\mathrm{d.p}}(\calC^{\otimes}) \to \CoAlg(\calC^{\otimes})$$
is an equivalence.
\end{cor}

\begin{proof}
The functor $\zeta$ induces the natural coalgebra map
$$
\varphi^X\colon \zeta \mathfrak{C}_{\mathrm{d.p}}(X) \to \mathfrak{C}(X) \in \CoAlg(\calC^{\otimes}), \;\; X\in \calC.
$$
By applying the forgetful functor $\oblv \colon \CoAlg(\calC^{\otimes}) \to \calC$ to $\varphi$, we obtain a morphism of comonads
$$\Phi\colon \Sym_{\calC}(X) \to \Gamma_{\mathrm{d.p}}(X) \in \calC, \;\; X\in \calC. $$
By Theorem~\ref{theorem: pronil dp sigma-free comonadic}, it suffices to show that $\Phi$ is an equivalence. By Lemma~\ref{lemma: basic properties of pro-nilpotent}, it is enough to show that the induced map $$\partial_i\Phi\colon \partial_i(\Sym_{\calC}) \to \partial_i(\Gamma_{\mathrm{d.p}})$$ between $i$-th derivatives is an equivalence for all $i\geq 0$. 

However, $\partial_i\Phi$ is the multi-linearization of the natural transformation
$$\oblv \circ \varphi^{X_1\oplus \ldots \oplus X_i}\colon \oblv(\zeta \mathfrak{C}_{\mathrm{d.p}}(X_1\oplus \ldots \oplus X_i)) \to \oblv(\mathfrak{C}(X_1 \oplus \ldots \oplus X_i)),$$
where $X_1,\ldots,X_i \in \calC$, $i\geq 1$. By Proposition~\ref{proposition: cartesian product in nil dp} and since $\zeta$ preserves finite products, we observe 
$$\partial_i\Phi \simeq (\partial_1\Phi)^{\otimes i}, \;\; i\geq 1. $$
Finally, $\partial_1\Phi$ is an equivalence because $\zeta$ is compatible with forgetful functors.
\end{proof}

\subsection{Coalgebras in symmetric sequences}\label{section: coalgebras in symmetric sequences}

Recall that the category of symmetric sequences $\sseq(\calC)$ in an $\infty$-category $\calC$ is the functor category $\sseq(\calC)=\Fun(\FB,\calC)$, see Section~\ref{section: H-Lie-algebras}, and $\sseq$ is the category $\sseq(\Mod_{\kk})$.

Suppose that the category $\calC \in \Pre^{L,\otimes}_{\mathrm{st}}$ is stable presentably (non-unital) symmetric monoidal. We endow the ($1$-)category $\FB$ with the symmetric monoidal structure given by disjoint unions. Then the category of symmetric sequences $\sseq(\calC)$ inherits a (non-unital) symmetric monoidal structure $\sseq(\calC)^{\circledast}$ given by the \emph{Day convolution} product (see e.g.~\cite[Section~2.2.6]{HigherAlgebra} and~\cite{Glasman16}):
$$- \circledast - \colon \sseq(\calC)\times \sseq(\calC) \to \sseq(\calC). $$
Informally, the Day convolution is given as follows
\begin{equation}\label{equation: day convolution, formula}
(X\circledast Y)(I)\simeq\bigoplus_{I_1\sqcup I_2\cong I}X(I_1)\otimes Y(I_2), \;\; I_1,I_2,I\in \FB, 
\end{equation}
where $X,Y\in \sseq(\calC)$ are symmetric sequences in $\calC$. Elsewhere below, we denote by $\sseq^{\circledast}$ the category $\sseq=\sseq(\Mod_{\kk})$ equipped with the Day convolution symmetric monoidal product. 

We denote by $\sseq_0(\calC)$ the full subcategory of $\sseq(\calC)$ spanned by the symmetric sequences $X\in \sseq(\calC)$ such that $X_0\simeq 0$. We note that the subcategory $\sseq_0(\calC)$ is closed under the Day convolution product. As before, we write $\sseq_0^{\circledast}$ for $\sseq_0(\Mod_{\kk})$ equipped with the Day convolution product. 

\begin{prop}\label{proposition: symmetric sequences are nilpotent}
Let $\calC^{\otimes} \in \Pre^{L,\otimes}_{\mathrm{st}}$ be a stable presentably symmetric monoidal category. Then the symmetric monoidal category $\sseq_0(\calC)^{\circledast}$ is pro-nilpotent.
\end{prop}

\begin{proof}
Fix $n\geq 0$. Let $\sseq_0^{\leq n}(\calC)$ be the full subcategory of $\sseq_0(\calC)$ spanned by the symmetric sequences $X \in \sseq_0(\calC)$ such that $X(I)\simeq 0$ for $|I|\geq n+1$. Note that the inclusion
$$\sseq_0^{\leq n}(\calC) \hookrightarrow \sseq_0(\calC) $$
admits a left (and right) adjoint $L^{\leq n}\colon \sseq_0(\calC) \to \sseq_0^{\leq n}(\calC)$ such that 
$$
L^{\leq n}(X)(I) \simeq
\begin{cases}
X(I) & \mbox{if $|I|\leq n$,}\\
0 & \mbox{otherwise}
\end{cases}
$$
for a symmetric sequence $X\in \sseq_0(\calC)$. By the formula~\eqref{equation: day convolution, formula}, the localization $L^{\leq n}$ is compatible with the Day convolution (see~\cite[Definition~2.2.1.6]{HigherAlgebra}), i.e. if $L^{\leq n}(X)\simeq 0$ for $X\in \sseq_0(\calC)$, then $L^{\leq n}(X\circledast Y)\simeq 0$ for any $Y\in \sseq_0(\calC)$. Therefore, by~\cite[Proposition~2.2.1.9]{HigherAlgebra}, the category $\sseq_0^{\leq n}(\calC)$ inherits the symmetric monoidal structure $\sseq_0^{\leq n}(\calC)^{\circledast_n}$ with the monoidal product $$X\circledast_n Y \simeq L^{\leq n}(X\circledast Y), \;\;X,Y\in \sseq_0^{\leq n}(\calC).$$ Moreover, the functors
$$L^{\leq n} \colon \sseq_0(\calC)^{\circledast} \to \sseq_0^{\leq n}(\calC)^{\circledast_n}, \; n \geq 0 $$
are symmetric monoidal.

Finally, we have $$\sseq_0(\calC)^{\circledast} \simeq \lim_n \sseq^{\leq n}_0(\calC)^{\circledast_n} \in \Pre^{L,\otimes}_{\mathrm{st}}$$ in the category of stable presentably symmetric monoidal categories. Moreover, $\sseq^{\leq 0}_0(\calC)\simeq 0$, the transition functors $$L^{\leq n}_m \colon \sseq^{\leq m}_0(\calC) \to\sseq^{\leq n}_0(\calC), \;\; m\geq n$$ preserve arbitrary limits, and $X\circledast_n Y\simeq 0$ for any $X\in \ker(L^{\leq n}_{n-1}) \subset \sseq_0^{\leq n}(\calC)$ and $Y \in \sseq_0^{\leq n}(\calC)$, $n\geq 0$. In other words, the category $\sseq_0(\calC)^\circledast$ satisfies the conditions of Definition~\ref{definition: pro-nilpotent category}.
\end{proof}

\begin{prop}\label{proposition: symmetric sequences are sigma free}
Let $\calC^{\otimes} \in \Pre^{L,\otimes}_{\mathrm{st}}$ be a stable presentably symmetric monoidal category. Then the symmetric monoidal category $\sseq_0(\calC)^{\circledast}$ is $\Sigma$-free.
\end{prop}

\begin{proof}
By Definition~\ref{definition: Sigma-free}, we have to find functors
$$\widetilde{T}^n \colon \sseq_0(\calC) \to \sseq_0(\calC)\hookrightarrow \Fun(\FB,\calC) $$
such that $\Coind^{\Sigma_n} \circ \widetilde{T}^n(X)\simeq X^{\otimes n}$ for any $X\in \sseq_0(\calC)$ and $n\geq 1$. 

Fix $n\geq 1$. Let $\underline{n}=\{1,\ldots, n\}$ be the set with $n$ elements. For $I \in \FB$, let $\Surj(I,\underline{n})$ be the set of surjection between $I$ and $\underline{n}$. Then, by the formula~\eqref{equation: day convolution, formula}, we have a natural $\Sigma_n$-equivariant equivalence
$$X^{\otimes n}(I)\simeq  \bigoplus_{f\in \Surj(I,\underline{n})} X(f^{-1}(1)) \otimes\ldots \otimes X(f^{-1}(n)), \;\; I\in \FB.$$
Here the group $\Sigma_n$ acts on the right hand side by permuting the summands through the left $\Sigma_n$-action on the finite set $\Surj(I,\underline{n})$.

Note that the group $\Sigma_n$ acts \emph{freely} on the set $\Surj(I,\underline{n})$, $I\in \FB$. Let $\mathrm{Part}(I,\underline{n})$ denote the quotient $\Surj(I,\underline{n})/\Sigma_n$ and fix a section
$$s\colon \mathrm{Part}(I,\underline{n}) \to \Surj(I,\underline{n}). $$
Then, for $I\in \FB$, we observe
\begin{align*}
X^{\otimes n}(I)&\simeq  \bigoplus_{\sigma \in \Sigma_n}\left(\bigoplus_{\lambda \in \mathrm{Part}(I,\underline{n})} X((\sigma s(\lambda))^{-1}(1)) \otimes\ldots \otimes X((\sigma s(\lambda))^{-1}(n))\right)\\
&\simeq\Coind^{\Sigma_n}\left(\bigoplus_{\lambda \in \mathrm{Part}(I,\underline{n})} X((s(\lambda))^{-1}(1)) \otimes\ldots \otimes X((s(\lambda))^{-1}(n))\right)
\end{align*}
\end{proof}

\begin{cor}\label{corollary: nil dp = coalg for sseq}
Let $\calC^{\otimes} \in \Pre^{L,\otimes}_{\mathrm{st}}$ be a stable presentably symmetric monoidal category. Then the comparison functor
$$\zeta\colon \cCAlg^{\mathrm{nil}}_{\mathrm{d.p}}(\sseq_0(\calC)^\circledast) \xrightarrow{\simeq} \CoAlg(\sseq_0(\calC)^\circledast) $$
is an equivalence. In particular, the adjoint pair
$$
\begin{tikzcd}
\oblv: \CoAlg(\sseq_0(\calC)^\circledast) \arrow[shift left=.6ex]{r}
&\sseq_0(\calC):\mathfrak{C} \arrow[shift left=.6ex,swap]{l}
\end{tikzcd}
$$
is comonadic and 
\begin{equation}\label{equation: cofree coalgebra in sseq}
\oblv \circ \mathfrak{C}(X)(I) \simeq \bigoplus_{n\geq 1}^{|I|}\left( \bigoplus_{f\in \Surj(I,\underline{n})} X(f^{-1}(1)) \otimes\ldots \otimes X(f^{-1}(n)) \right)_{h\Sigma_n},
\end{equation}
where $X\in \sseq_0(\calC)$, $I\in \FB$, and $\underline{n}=\{1,\ldots,n\}$ is a finite set.
\end{cor}

\begin{proof}
We apply Corollary~\ref{corollary: pronil dp sigma-free same} to the category $\sseq_0(\calC)^{\circledast}$ by using Proposition~\ref{proposition: symmetric sequences are nilpotent} and Proposition~\ref{proposition: symmetric sequences are sigma free}.
\end{proof}

\begin{dfn}\label{definition: coalgebras in Mod}
Let us denote the category $\CoAlg(\sseq_0(\Mod_{\kk})^{\circledast})$ of coalgebras in symmetric sequences of chain complexes by $\CoAlg(\sseq_0^{\circledast})$.
\end{dfn}

\begin{rmk}\label{remark:coalg=left comodule}
One can show that the category $\cCAlg(\sseq_0^{\circledast})$ of coalgebras is equivalent to the category of \emph{left comodules over the cocommutative cooperad} $\bfComm_*\in \sseq$, $\bfComm_*(I)=\kk$, $I\in \FB$. However, we will avoid to use this description in this paper, therefore we will only give the intuition here; the detailed proof will be presented elsewhere.

Informally, a coalgebra $C\in \cCAlg(\sseq^{\circledast}_0)$ is a symmetric sequence equipped with a (symmetric) comultiplication map 
$$\mu\colon C \to C\circledast C \in \sseq_0. $$
Using the formula~\eqref{equation: day convolution, formula}, we observe that the map $\mu$ is given by the series of $(\Sigma_m\times \Sigma_{n-m})$-equivariant comultiplication maps
$$\mu_{2;m,n-m}\colon C_n \to C_m \otimes C_{n-m}, \;\; 0\leq m\leq n, \; n\geq 1. $$
By iterating the map $\mu_2$, we obtain the series of equivariant maps
\begin{equation}\label{equation: comultiplication maps}
\mu_{k;n_1,\ldots, n_k}\colon C_n \to C_{n_1} \otimes \ldots \otimes C_{n_k}, \;\; n_1+\ldots + n_k=n, \; k\geq 1
\end{equation}
which satisfy certain coassociativity relations. This data seems to recover a left $\bfComm_*$-comodule on $C_*\in\sseq_0$.
\end{rmk}

Since the monoidal structure $\sseq(\calC)_0^{\circledast}$ is not compatible with limits, the forgetful functor $\oblv\colon \CoAlg(\sseq_0(\calC)^{\circledast}) \to \sseq_0(\calC)$ does not commute with arbitrary totalizations. However, due to Corollary~\ref{corollary: nil dp = coalg for sseq}, the functor $\oblv$ commutes with totalizations of a special kind.

\begin{prop}\label{proposition: coalgebras and almost finite totalizations}
Let $\calC^{\otimes} \in \Pre^{L,\otimes}_{\mathrm{st}}$ be a stable presentably symmetric monoidal category and let $C^{\bullet} \in \Fun(\Delta,\CoAlg(\sseq_0(\calC)^{\circledast}))$ be a cosimplicial object in the category $\CoAlg(\sseq_0(\calC)^{\circledast})$. Assume that for any $k\geq 0$ there exists $n(k)\in \N$ such that the natural map
$$(\mathrm{Tot}(\oblv C^{\bullet}))(I) \xrightarrow{\simeq} (\mathrm{Tot}^{N} (\oblv C^\bullet))(I), \;\; I\in \FB, \; |I|=k $$
is an equivalence in $\calC$ for all $N\geq n(k)$. Then the map
$$\oblv(\mathrm{Tot}(C^\bullet)) \to \Tot(\oblv C^{\bullet}) \in \sseq_0(\calC)$$
is also an equivalence.
\end{prop}

\begin{proof}
Follows from Theorem~\ref{theorem: nil dp basic properties}.
\end{proof}

\begin{dfn}\label{definition: trivial nil dp coalgebra}
Let $\calC^{\otimes} \in \Pre^{L,\otimes}_{\mathrm{st}}$ be a stable presentably symmetric monoidal category. Suppose that $\calC$ is compactly generated. Let us denote by
$$\triv\colon \sseq_0(\calC) \to \cCAlg^{\mathrm{nil}}_{\mathrm{d.p}}(\sseq_0(\calC)^{\circledast})\simeq \CoAlg(\sseq_0(\calC)^{\circledast}) $$
the \emph{functor of a trivial coalgebra} induced by the comonad map $\id_{\calC} \to \Sym_{\sseq_0}(\calC)$, see Theorem~\ref{theorem: comonad for nilpotent divided power coalgebras}. Note that $\oblv \circ \triv \simeq \id_{\calC}$.
\end{dfn}

Suppose that the symmetric monoidal category $\calC^{\otimes}$ is unital and $\mathbbm{1}\in \calC$ is the unit object, then the monoidal category $\sseq(\calC)^{\circledast}$ is again unital with the unit $\mathbbm{1}_*$ given by $\mathbbm{1}_0\simeq \mathbbm{1}$ and $\mathbbm{1}_n\simeq \ast$ for $n\geq 1$. In particular, $\mathbbm{1}_*\in \sseq(\calC)^{\circledast}$ is an algebra object and the slice category $\sseq(\calC)_{/\mathbbm{1}_*}$ is endowed with the canonical symmetric monoidal structure such that the forgetful functor
$$\sseq(\calC)_{/\mathbbm{1}_*} \to \sseq(\calC) $$
is symmetric monoidal. 
\begin{dfn}\label{definition: sseq1}
We write $\sseq_1(\calC)$ for the full subcategory of $\sseq(\calC)_{/\mathbbm{1}_*}$ spanned by the symmetric sequences $X\to \mathbbm{1}_* \in \sseq(\calC)_{/\mathbbm{1}_*}$ such that $X_0\xrightarrow{\simeq} \mathbbm{1}$ is an equivalence. Note that $\sseq_1(\calC)$ is closed under the Day convolution and we write $\sseq_1(\calC)^{\circledast}$ for the induced symmetric monoidal subcategory. 
\end{dfn}

\begin{rmk}\label{remark: aug sseq1 = sseq0}
There is a natural equivalence
$$\CoAlg(\sseq_0(\calC)^\circledast) \simeq \CoAlg^{\mathrm{aug}}(\sseq_1(\calC)^{\circledast})$$
between the category $\CoAlg^{\mathrm{aug}}(\sseq_1(\calC)^{\circledast})$ of \emph{coaugmented} cocommutative coalgebras in the unital symmetric monoidal category $\sseq_1(\calC)^\circledast$ and the category $\CoAlg(\sseq_0(\calC)^\circledast)$ of \emph{non-unital} coalgebras in $\sseq_0(\calC)^\circledast$. This equivalence sends $C\in \CoAlg(\sseq_0(\calC)^\circledast)$ to $C\oplus \mathbbm{1}_*$.
\end{rmk}

\begin{prop}\label{proposition: + to day oplax monoidal sseq}
Let $\calC^{\otimes} \in \Pre^{L,\otimes}_{\mathrm{st}}$ be a stable presentably unital symmetric monoidal category. Suppose that $\calC$ is compactly generated. There exists an oplax symmetric monoidal functor
$$\mathfrak{i}\colon \sseq_0(\calC)^{\oplus} \to \sseq_1(\calC)^{\circledast} $$
such that
\begin{enumerate}
\item $\mathfrak{i}(X) \simeq X\oplus \mathbbm{1}_*$;
\item the induced functor
$$\sseq_0(\calC) \xrightarrow{\simeq} \CoAlg(\sseq_0(\calC)^{\oplus}) \xrightarrow{\mathfrak{i}} \CoAlg(\sseq_1(\calC)^{\circledast}) $$
is equivalent to the composite
\begin{equation}\label{equation: triv composite}
\sseq_0(\calC) \xrightarrow{\triv} \CoAlg(\sseq_0(\calC)^{\circledast}) \to \CoAlg(\sseq_1(\calC)^{\circledast}).
\end{equation}
\end{enumerate}
\end{prop}

\begin{proof}
Since $\calC$ is a stable category, the symmetric monoidal structure $\sseq_0(\calC)^{\oplus}$ is Cartesian, see~\cite[Definition~2.4.1.1]{HigherAlgebra}. Finally, we observe that~\cite[Theorem~2.4.3.18]{HigherAlgebra} applied to the composite~\eqref{equation: triv composite} constructs the desired oplax symmetric monoidal functor $\mathfrak{i}$.
\end{proof}

\begin{rmk}\label{remark: oplax triv}
Let $X,Y \in \sseq_0(\calC)$ and let $I\in \FB$ be a non-empty finite set. Then the oplax symmetric monoidal functor $\mathfrak{i}$ defines natural transformations
\begin{equation}\label{equation: oplax triv}
\mathfrak{i}(X\oplus Y)(I) \to(\mathfrak{i}(X) \circledast \mathfrak{i}(Y))(I). 
\end{equation}
Note that the left hand side is $X(I)\oplus Y(I)$. By the formula~\eqref{equation: day convolution, formula}, the right hand side contains the direct summand
$$\mathfrak{i}(X)(I)\otimes \mathfrak{i}(Y)(\emptyset) \oplus  \mathfrak{i}(X)(\emptyset)\otimes \mathfrak{i}(Y)(I) \simeq X(I)\otimes \mathbbm{1} \oplus \mathbbm{1} \otimes Y(I) \simeq X(I) \oplus Y(I). $$
Finally, the map~\eqref{equation: oplax triv} is equivalent to the inclusion of the described direct summand. 
\end{rmk}


Since $\oblv \circ \triv \simeq \id_{\sseq_0}$, the functor
$$\triv \colon \sseq_0(\calC) \to \CoAlg(\sseq_0(\calC)^{\circledast})$$
preserves colimits, and so, it admits a right adjoint
\begin{equation}\label{equation: prim}
\prim\colon \CoAlg(\sseq_0(\calC)^{\circledast}) \to \sseq_0(\calC).
\end{equation}
More explicitly, $\prim$ is the totalization of the \emph{cobar construction} $\Cobar^\bullet(C) \in \Fun(\Delta,\sseq_0(\calC))$, where
$$\Cobar^m(C) \simeq (\oblv\circ \mathfrak{C})^{\circ m} \circ \oblv(C), \;\; C\in \CoAlg(\sseq_0(\calC)^{\circledast}), \; m\geq 0. $$
The adjoint pair $\triv \dashv \prim$ defines a monad acting on the category $\sseq_0(\calC)$. We will describe this monad in the case $\calC=\Mod_{\kk}$.

We denote by $\bfLie^{\mathrm{sp}}(\sseq_0^{\circledast})$ the category of \emph{spectral Lie algebras in $\sseq_0^{\circledast}$}; that is the category of algebras over the operad $\bfLie^{\mathrm{sp}}_*$, see Definition~\ref{definition: linear spectral lie operad}. Informally, a spectral Lie algebra $L\in \bfLie^{\mathrm{sp}}(\sseq_0^{\circledast})$ is a symmetric sequence $L\in\sseq_0$ equipped with the maps
\begin{equation}\label{equation: lie multiplication 1}
m_k\colon \bfLie^{\mathrm{sp}}_k \otimes_{h\Sigma_k} L^{\circledast k} \to L \in \sseq_0, \;\; k\geq 1
\end{equation}
which satisfy certain associativity relations. By the formula~\eqref{equation: day convolution, formula}, the map $m_k$ packages the following $\Sigma_{n_1,\ldots, n_k}$-equivariant maps
\begin{equation}\label{equation: lie multiplicatons}
m_{k;n_1,\ldots, n_k}\colon \bfLie^{\mathrm{sp}}_k \otimes L_{n_1} \otimes \ldots \otimes L_{n_k} \to L_n \in \Mod_{\kk},
\end{equation}
where $n_1+\ldots + n_k=n$ is an (unordered) partition of $\underline{n}$ and $\Sigma_{n_1,\ldots, n_k} \subset \Sigma_n$ is its stabilization subgroup, $k\geq 1$.  In other words, a spectral Lie algebra in $\sseq_0^{\circledast}$ models left $\bfLie^{\mathrm{sp}}_*$-modules in the monoidal category $\sseq_0^{\circ}$ equipped with the composition product.

Next, we recall from~\cite[Section~3.3]{GF12} that there exists an adjoint pair
\begin{equation}\label{equation: koszul adjunction}
\begin{tikzcd}
Q^{\mathrm{enh}}: \bfLie^{\mathrm{sp}}(\sseq_0^\circledast) \arrow[shift left=.6ex]{r}
&\CoAlg^{\mathrm{nil}}_{\mathrm{d.p}}(\sseq_0^\circledast)\simeq \CoAlg(\sseq_0^\circledast) :\prim^{\mathrm{enh}}, \arrow[shift left=.6ex,swap]{l}
\end{tikzcd}
\end{equation}
such that the composite $\oblv\circ \prim^{\mathrm{enh}}\simeq \prim$ is the left adjoint to $\triv$, see Definition~\ref{definition: trivial nil dp coalgebra}. In particular, we have
\begin{equation}\label{equation: prim triv Lie}
\prim \circ \triv(X)(I) \simeq \bigoplus_{n\geq 1}^{|I|}\left( \bigoplus_{f\in \Surj(I,\underline{n})} \bfLie^{\mathrm{sp}}_n \otimes X(f^{-1}(1)) \otimes\ldots \otimes X(f^{-1}(n)) \right)_{h\Sigma_n},
\end{equation}
where $X\in \sseq_0(\calC)$, $I\in \FB$, and $\underline{n}=\{1,\ldots,n\}$ is a finite set.

Usually, it is difficult to compute $\prim(C)$ of $C\in \cCAlg(\sseq_0^{\circledast})$, but here we present a simple approximation for $\prim(C)$, which will be useful for us later.

\begin{dfn}\label{definition: fiber of comultiplication}
Let $n,k\in \N$ and $C\in \cCAlg(\sseq_0^{\circledast})$ be a cocommutative coalgebra in $\sseq_0^{\circledast}$. We write $\prim^{k;n}(C)\in \Mod_{\kk}^{B\Sigma_{nk}}$ for the fiber of the comultiplication map $\mu^{k;n,\ldots, n}$ (see~\eqref{equation: comultiplication maps})
\begin{align*}
\prim^{k;n}(C)_{nk} &= \fib (C \xrightarrow{\mu} (C^{\circledast k})^{h\Sigma_k})_{nk}\\
&=\fib (C_{nk} \xrightarrow{\mu_{k;n,\ldots, n}} \mathrm{Ind}^{\Sigma_{nk}}_{\Sigma_k \wr \Sigma_n} C_n \otimes \ldots \otimes C_n) \in \Mod_{\kk}^{B\Sigma_{nk}}. 
\end{align*}
\end{dfn}

We construct a natural \emph{corona cell map}
\begin{equation}\label{equation: corona cell map, general}
c^{k;n}_C\colon \prim(C)_{nk} \to \prim^{k;n}(C)_{nk} \in \Mod_{\kk}^{B\Sigma_{nk}}.
\end{equation}
By the discussion above, we have the projection map 
\begin{equation}\label{equation: corona, projection}
\prim(C)_{nk} \xrightarrow{\sim} \Tot(\Cobar^{\bullet}(C))_{nk} \to \mathrm{Tot}^{1}(\Cobar(C)^{\bullet})_{nk},
\end{equation}
where, by~\eqref{equation: cofree coalgebra in sseq}, the right hand side is equivalent to the fiber $\fib(\mu)$ of the map
$$\mu=\oplus \mu_{j}\colon C_{nk} \to \bigoplus_{j> 1}^{nk}\left( \bigoplus_{f\in \Surj(\underline{nk},\underline{j})} C_{|f^{-1}(1)|} \otimes\ldots \otimes C_{|f^{-1}(j)|} \right)_{h\Sigma_j} $$
given by the sum of all comultiplication maps $\mu_{j; n_1,\ldots, n_j}$ of~\eqref{equation: comultiplication maps} with $1<j\leq nk$. The fiber $\fib(\mu)$ projects to the fiber $\fib(\mu_{k;n,\ldots,n})$ of the comultiplication map $\mu_{k;n,\ldots,n}$, and we define the corona cell map~\eqref{equation: corona cell map, general} as the composite of the map~\eqref{equation: corona, projection} with the just described projection. 

\begin{exmp}\label{example: lie and corona cell}
Suppose that $n=1$, $k\geq 2$, and $C=\triv(\kk)\in \CoAlg(\sseq_0^{\circledast})$ is the trivial coalgebra concentrated in degree one. Then we have 
$$\prim^{k;1}(C)_{k}\simeq \Sigma^{-1}\kk,$$
and since $\prim(C)\simeq \bfLie^{\mathrm{sp}}_* \in \sseq_0$, we obtain a map 
\begin{equation}\label{equation: corona cell map, local}
c^{k;1}\colon \Sigma^{1-k}\bfLie_k\otimes \mathbf{sgn}^{\otimes k} \to \Sigma^{-1}\kk \in \Mod_{\kk}^{B\Sigma_k}.
\end{equation}
Finally, we note that the map~\eqref{equation: corona cell map, local} is equivalent to the map~\eqref{equation: corona cell map, kjaer}.



\end{exmp}

\subsection{Derivatives and coderivatives}\label{section: derivaties and coderivatives}

Let $\Sp^{\omega} \subset \Sp$ denote the full subcategory of \emph{compact} spectra.
\begin{lmm}\label{lemma: right adjoint to partial}
Let $\calC \in \Pre^{L}_{\mathrm{st}}$ be a stable presentable category. Then, for every $n\geq 0$, the functor
$$\partial_n\colon \Fun^{\omega}(\Sp,\calC) \simeq \Fun(\Sp^{\omega},\calC) \to \calC^{B\Sigma_n}$$
admits a right adjoint $R_n$ given by the formula
$$R_n(A)(X) \simeq (A\otimes X^{\otimes n})^{h\Sigma_n}, \;\; A\in \calC^{B\Sigma_n}, \;\; X\in \Sp^{\omega}. $$
\end{lmm}

\begin{proof}
Fix $n\geq 0$. Since the category $\calC$ is stable, the functor $\partial_n$ preserves small colimits by~\cite[Remark~6.1.1.31]{HigherAlgebra}. Therefore, by the adjoint functor theorem~\cite[Corollary~5.5.2.9]{HTT}, $\partial_n$ admits a right adjoint. It is enough to show that
$$\nat(F, R_n(A)) \simeq \Map_{\calC^{B\Sigma_n}}(\partial_n(F),A)$$
for any functor $F\in \Fun(\Sp^{\omega},\calC)$ and $A\in \calC^{B\Sigma_n}$. Note that $R_n(A) \in \Fun(\Sp^{\omega},\calC)$ is $n$-excisive, so
$$\nat(F, R_n(A)) \simeq \nat(P_nF, R_n(A)).$$
Therefore, by~\cite[Proposition 6.1.4.14]{HigherAlgebra}, it suffices to show that the space $\nat(P_{n-1}F,R_n(A))$ is contractible. The latter follows by~\cite[Corollary~6.1.1.17]{HigherAlgebra}, see also~\cite[Remark~6.1.6.28]{HigherAlgebra}.
\end{proof}

Let $\calC^{\otimes} \in \Pre^{L,\otimes}_{\mathrm{st}}$ be a symmetric monoidal category. Then we write $\Fun(\Sp^{\omega},\calC)^{\otimes}$ for the functor category endowed with the objectwise monoidal structure; that is $(F\otimes G)(X) \simeq F(X)\otimes G(X) \in \calC$ for functors $F,G \in \Fun(\Sp^{\omega},\calC)$ and $X\in \Sp^\omega$.

\begin{lmm}\label{lemma: right adjoint to sseq partial}
Let $\calC^{\otimes} \in \Pre^{L,\otimes}_{\mathrm{st}}$ be a stable unital presentably symmetric monoidal category. Then, there exists a lax symmetric monoidal functor
$$R \colon \sseq(\calC)^{\circledast} \to \Fun(\Sp^\omega,\calC)^{\otimes}$$
such that $R(A)\simeq \prod_{n\geq 0}R_n(A_n)$. In particular, $R$ is the right adjoint to the functor
$$\partial_*\colon \Fun(\Sp^{\omega},\calC) \to \sseq(\calC) $$
which assigns to a functor $F\in \Fun(\Sp^{\omega},\calC)$ its symmetric sequence of derivatives $\partial_*(F)\in \sseq_0(\calC)$, see e.g.~\cite[Example~6.3.0.5]{HigherAlgebra}.
\end{lmm}

\begin{proof}
By~\cite[Example~2.2.6.9]{HigherAlgebra}, the category $\CoAlg(\sseq(\calC)^{\circledast})$ is equivalent to the category of oplax monoidal functors $\Fun^{\mathrm{oplax}}((\FB)^{\sqcup}, \calC^{\otimes})$. Since the symmetric monoidal category $(\FB)^{\sqcup}$ is freely generated by a one object (namely, $\underline{1}$), we observe that the full subcategory $$\Fun^{\otimes}((\FB)^{\sqcup}, \calC^{\otimes}) \subset \Fun^{\mathrm{oplax}}((\FB)^{\sqcup}, \calC^{\otimes})$$
of strong symmetric monoidal functors is equivalent to $\calC$. Let $$\Delta\colon \Sp^\omega \to \CoAlg(\sseq(\calC)^{\circledast})^{op}$$ denote the composite
$$\Delta\colon \Sp^\omega \xrightarrow{\D} \Sp^{\omega,op} \xrightarrow{X\mapsto X\otimes \mathbbm{1}_{\calC}} \calC^{op} \simeq \Fun^{\otimes}((\FB)^{\sqcup}, \calC^{\otimes})^{op} \hookrightarrow \CoAlg(\sseq(\calC)^{\circledast})^{op},$$
where $\D\colon \Sp^{\omega} \to \Sp^{\omega,op}$ is the Spanier--Whitehead duality and $\mathbbm{1}_\calC \in \calC$ is the monoidal unit. Informally, $\Delta(X)_* \in \CoAlg(\sseq(\calC)^{\circledast})$, $X\in \Sp^{\omega}$ is the symmetric sequence given by $\Delta(X)_n\simeq  \D(X^{\otimes n}) \otimes \mathbbm{1}_\calC$, $n\geq 0$ with obvious comultiplication maps.

Note that the category $\sseq(\calC)^{\circledast}$ is $\calC^{\otimes}$-linear. Therefore the coalgebra object $\Delta(X)$, $X\in \Sp^{\omega}$ corepresents a lax monoidal functor
$$\delta_X = \Map_{\sseq(\calC)}(\Delta(X), -) \colon \sseq(\calC)^{\circledast} \to \calC^{\otimes} $$
such that $$\delta_X(A)\simeq \prod_{n\geq 0} \Map_{\calC^{B\Sigma_n}}(\D(X^{\otimes n})\otimes \mathbbm{1}_\calC, A_n) \simeq \prod_{n\geq 0} (A\otimes X^{\otimes n})^{h\Sigma_n}$$
for $A\in \sseq(\calC)$. Finally, the composite
$$\Sp^{\omega} \xrightarrow{\Delta}  \CoAlg(\sseq(\calC)^{\circledast})^{op} \xrightarrow{\mathscr{Y}_{\calC}} \Fun^{\mathrm{lax}}(\sseq(\calC)^{\circledast},\calC^{\otimes})$$
induces the desired functor $R$. Here $\mathscr{Y}_{\calC}$ is $\calC$-linear co-Yoneda functor.
\end{proof}

By Lemma~\ref{lemma: right adjoint to sseq partial}, the functor
$$\partial_*\colon \Fun^{\omega}(\Sp,\calC)^{\otimes} \to \sseq(\calC)^{\circledast} $$
inherits an oplax monoidal enhancement. In particular, there is a natural map
\begin{equation}\label{equation: partial oplax}
\6_*(F\otimes G) \to \6_*(F)\circledast \6_*(G), \;\; F,G \in \Fun^{\omega}(\Sp,\calC)
\end{equation}
which is the left adjoint to the composite
\begin{equation}\label{equation: partial oplax adjoint}
F\otimes G \to R\partial_* F\otimes R\partial_* G \to R(\partial_* F \circledast \partial_* G).
\end{equation}

\begin{lmm}\label{lemma: sseq partial is strict}
Let $\calC^{\otimes} \in \Pre^{L,\otimes}_{\mathrm{st}}$ be a stable unital presentably symmetric monoidal category. Then the oplax symmetric monoidal functor
$$\6_*\colon \Fun^{\omega}(\Sp,\calC)^{\otimes} \to \sseq(\calC)^{\circledast} $$
is strong monoidal. More precisely, the natural maps~\eqref{equation: partial oplax} are equivalences for all $F,G\in \Fun^{\omega}(\Sp,\calC)$.
\end{lmm}

\begin{proof}
Note that the map~\eqref{equation: partial oplax adjoint} is an equivalence if the functor $F$ (resp. $G$) is $n$-homogeneous (resp. $m$-homogeneous) and has the form $F(X)\simeq A\otimes X^{\otimes n}$ (resp. $G(X)\simeq B\otimes X^{\otimes m}$) for $X\in \Sp$ and $A\in \calC$ (resp. $B\in \calC$). This implies that the map~\eqref{equation: partial oplax} is also an equivalence for such functors. Since the homogeneous functors are generated by the functors of the previous form under colimits and $\6_*$ commutes with colimits, the map~\eqref{equation: partial oplax} is an equivalence provided $F$ and $G$ are homogeneous.

Next, we show 
$$\6_*(F\otimes G) \to \6_*(F)\circledast \6_*(G) $$
is an equivalence for an arbitrary pair $F,G \in \Fun^{\omega}(\Sp,\calC)$. By~\cite[Lemma~6.10]{GoodwillieIII}, see also~\cite[Corollary~6.4]{GoodwillieIII} and~\cite[Remark~6.1.3.23]{HigherAlgebra}, the product $F\otimes G$ is $n$-reduced provided the functor $F$ is $n$-reduced and $G$ is arbitrary, $n\geq 0$. This implies that the natural map
$$P_n(F\otimes G) \to P_n(P_n(F)\otimes G) $$
is an equivalence. Therefore, we can assume that $F$ and $G$ are excisive (of some degree). In this case, the functors $F$ and $G$ are finite colimits of homogeneous functors (via their Taylor towers), so the map~\eqref{equation: partial oplax} is an equivalence by the previous step.
\end{proof}

Let $\Fun^{\omega}_*(\Sp, \sL)$ be the category of \emph{reduced} finitary functors from the category of spectra $\Sp$ to the category $\sL$ of simplicial restricted Lie algebras. We write $$\mathbbm{1}\in \Fun^{\omega}(\Sp, \Mod_{\kk})$$ for the \emph{constant functor} valued at the monoidal unit $\kk \in \Mod_{\kk}$. Let us denote by 
$$\Fun^{\omega}_{\mathbbm{1}}(\Sp, \Mod_{\kk}) \subset \Fun^{\omega}(\Sp, \Mod_{\kk})_{/\mathbbm{1}}$$
for the full subcategory of the slice category spanned by the finitary functors $F\to \mathbbm{1}$ such that $F(0)\xrightarrow{\simeq} \kk$ is an equivalence. We consider two composites
\begin{equation}\label{equation: linear derivatives}
\bar{\6}_*\colon \Fun^{\omega}_*(\Sp, \sL) \xrightarrow{\6_*} \sseq_0(\Sp(\sL)) \xrightarrow{\widetilde{C}_*} \sseq_0
\end{equation}
and
\begin{equation}\label{equation: linear coderivatives}
\bar{\6}^*\colon \Fun^{\omega}_*(\Sp, \sL) \xrightarrow{C_*} \Fun^{\omega}_{\mathbbm{1}}(\Sp, \Mod_{\kk}) \xrightarrow{\6_*} \sseq_1.
\end{equation}

Note that the category $\Fun^{\omega}_*(\Sp,\sL)$ inherits the Cartesian monoidal structure $\Fun^{\omega}_*(\Sp,\sL)^{\times}$;
that is $(F\times G)(X) \simeq F(X)\times G(X)\in \sL$ for $F,G\in\Fun^{\omega}_*(\Sp,\sL)$ and $X\in\Sp$.

\begin{prop}\label{proposition: coderivatives are symmetric monoidal}
The functor $\bar{\6}^*\colon \Fun^{\omega}_*(\Sp,\sL)^{\times} \to \sseq_1^{\circledast}$ inherits strong symmetric monoidal enhancement if we endow the category $\sseq_1$ with the Day convolution product $\sseq_1^{\circledast}$. In particular, there is an equivalence
$$\bar{\6}^*(F\times G) \simeq \bar{\6}^*(F)\circledast \bar{\6}^*(G) \in \sseq_1.$$
\end{prop}

\begin{proof}
We endow the category $\Fun^{\omega}_{\mathbbm{1}}(\Sp,\Mod_{\kk})$ with the symmetric monoidal structure $\Fun^{\omega}_{\mathbbm{1}}(\Sp,\Mod_{\kk})^{\otimes}$ given by the objectwise tensor product. 
Then, by the K\"{u}nneth theorem~\cite[Corollary~4.2.5]{koszul}, the functor 
$$C_*\colon \Fun^{\omega}_*(\Sp,\sL)^{\times} \to \Fun^{\omega}_{\mathbbm{1}}(\Sp,\Mod_{\kk})^{\otimes} $$
is symmetric monoidal. By Lemma~\ref{lemma: sseq partial is strict}, the functor $\bar{\6}^*$ is strong symmetric monoidal as the composite of strong symmetric monoidal functors.
\end{proof}

By Proposition~\ref{proposition: coderivatives are symmetric monoidal},  the functor $\bar{\6}^*$ induces the functor 
$$\bar{\6}^*\colon \CoAlg(\Fun^{\omega}_*(\Sp,\sL)^\times) \to \CoAlg(\sseq_1^\circledast) $$
between the categories of cocommutative coalgebras. Moreover, the diagonal map $$\Delta\colon F\to F\times F$$ defines the canonical coalgebra structure on a functor $F\in \Fun^{\omega}_*(\Sp,\sL)$; that is, the diagonal map defines the functor
$$\Fun^{\omega}_*(\Sp,\sL) \to \CoAlg(\Fun^{\omega}_*(\Sp,\sL)^\times). $$
Combining these two observations together, we obtain the functor
\begin{equation}\label{equation: coalgebra on coderivatives}
\Fun^{\omega}_*(\Sp,\sL) \to \CoAlg(\sseq_1^{\circledast}). 
\end{equation}
Finally, we note that $\bar{\6}^*(0) \simeq \kk \in \CoAlg(\sseq_1^{\circledast})$, i.e. the functor~\eqref{equation: linear coderivatives} evaluated at the constant zero functor $0\in \Fun^{\omega}_*(\Sp,\sL)$ is the monoidal unit $\mathbbm{1}_*\in \sseq^\circledast$. Therefore the functor~\eqref{equation: coalgebra on coderivatives} factorizes as follows
$$\Fun^{\omega}_*(\Sp,\sL) \to \CoAlg^{\mathrm{aug}}(\sseq_1^\circledast) \to \CoAlg(\sseq_1^\circledast). $$
Let us denote by $\AC$ the left functor in the factorization above
\begin{equation}\label{equation: Arone-Chong coderivatves}
\AC\colon \Fun^{\omega}_*(\Sp,\sL) \to \CoAlg^{\mathrm{aug}}(\sseq_1^\circledast) \simeq \CoAlg(\sseq_0^\circledast).
\end{equation}

\begin{rmk}\label{remark: AC products}
By Proposition~\ref{proposition: coderivatives are symmetric monoidal} and~\cite[Proposition~3.2.4.7]{HigherAlgebra}, we observe that the functor $\AC$ preserves finite products.
\end{rmk}

Our next goal is to show that $\AC$ preserves all fiber squares, at least for \emph{analytic} functors.

\begin{dfn}[\cite{Goodwillie_analytic}]\label{definition: analytic functor}
A reduced functor $F\in \Fun^{\omega}_*(\Sp,\sL)$ is called~\emph{analytic} if there are constants $b,\rho, \kappa\geq 0$ such that for all spectra $X$, $\conn(X)\geq \kappa$ the projection map
$$F(X) \to P_k F(X), \;\; k\geq 0 $$
is $(b-k(\rho-1)+(k+1)\conn(X))$-connected. 
\end{dfn}

We write $\Fun^{\omega}_{an}(\Sp,\sL)$ for the full subcategory of $\Fun^{\omega}_*(\Sp,\sL)$ spanned by analytic functors. Since the excisive approximations $P_k$, $k\geq 0$ commute with finite limits, the subcategory is $\Fun^{\omega}_{an}(\Sp,\sL)$ closed under finite limits.

Consider a fiber square 
\begin{equation}\label{equation: fiber square of functors}
\begin{tikzcd}
F \arrow{r} \arrow{d}
&F_1 \arrow{d} \\
F_2 \arrow{r}
&F_{12}
\end{tikzcd}
\end{equation}
in the category $\Fun^{\omega}_{an}(\Sp,\sL)$; we would like to show that this square remains Cartesian in the category $\CoAlg(\sseq_0^\circledast)$ after applying the functor $\AC$. 

Recall from~\cite{Rector70} that the functor $F$ is equivalent to the totalization $\mathrm{Tot}(T^{\bullet})$ of the cosimplicial object $T^\bullet \in \Fun(\Delta, \Fun^{\omega}_{an}(\Sp,\sL))$ such that 
\begin{equation}\label{equation: rector}
T^q\simeq F_1\times F_{12}^{\times q}\times F_2, \;\; q\geq 0.
\end{equation}

\begin{lmm}\label{lemma: totalization and chains}
The natural map
$$C_*(\mathrm{Tot}(T^{\bullet}X)) \xrightarrow{\simeq} \mathrm{Tot} (C_*(T^\bullet X)) \in \Mod_{\kk}$$
is an equivalence as soon as $X\in \Sp$ is sufficiently highly connected.
\end{lmm}

\begin{proof}
By the assumption, all simplicial restricted Lie algebras $$F(X), F_1(X), F_2(X), F_{12}(X) \in \sL$$ are connected as soon as $X\in \Sp$, $\conn(X)\geq \kappa$ for some $\kappa\geq 0$. By~\cite[Theorems~C and~D]{koszul}, we observe that the square
\begin{equation*}
\begin{tikzcd}
C_*(F(X)) \arrow{r} \arrow{d}
&C_*(F_1(X)) \arrow{d} \\
C_*(F_2(X)) \arrow{r}
&C_*(F_{12}(X))
\end{tikzcd}
\end{equation*}
is a fiber square in the category $\sCA_1$ of $1$-reduced simplicial truncated coalgebras (see Section~\ref{section: monoidal structures}) if $X\in \Sp$ and $\conn(X)\geq \kappa$. Since the functor $C_*\colon \sL \to \sCA_1$ commutes with finite products, it suffices to show that the forgetful functor
$$\oblv \colon \sCA_1 \to \Mod_{\kk}^{\geq 0} $$
commutes with certain totalizations. Namely, it suffices to show that
$$\oblv(\mathrm{Tot}(C_*(T^{\bullet} X))) \to \mathrm{Tot}(\oblv(C_*(T^{\bullet} X))) \in \Mod_{\kk}^{\geq 0}, \;\; \conn(X)\geq \kappa$$
is an equivalence. The latter follows from the convergence of the Eilenberg--Moore spectral sequence, see e.g.~\cite[Theorem~4.8]{Peroux_rigid}.
\end{proof}

\begin{lmm}\label{lemma:connectivity of fibers in rector model}
Let $T^\bullet \in \Fun(\Delta, \Fun^{\omega}_{an}(\Sp,\sL))$ be the cosimplicial object~\eqref{equation: rector} associated with the fiber square~\eqref{equation: fiber square of functors}. Then there exists $\kappa\geq 0$ such that $$\conn(N^q C_*(T^\bullet X)) \geq q+q\cdot\conn(X), \;\; X\in \Sp,\; q\geq 1$$ as soon as $\conn(X)\geq \kappa$. 
\end{lmm}

\begin{proof}
We note that the codegeneracy maps in the cosimplicial object~\eqref{equation: rector} are given by projections. Thus, by the equation~\eqref{equation:total fiber}, we obtain that $$N^qC_*( T^\bullet X) \simeq C_*(F_1 (X))\otimes \widetilde{C}_*(F_{12}(X))^{\otimes q} \otimes C_*(F_2(X)),\;\; X\in \Sp, \; q\geq 1.$$
By the assumption, there exists $\kappa\geq 0$ such that $\conn(F_{12}(X)) \geq \conn(X)$ if $\conn(X)\geq \kappa$. Since $$\conn(\widetilde{C}_*(F_{12}(X)))=\conn(F_{12}(X))+1,$$ the lemma follows.
\end{proof}

\begin{prop}\label{proposition: AC commutes with fiber sequences}
The functor $\AC\colon \Fun^{\omega}_{an}(\Sp,\sL) \to \CoAlg(\sseq_0^\circledast)$ preserves fiber squares.
\end{prop}

\begin{proof}
Recall that $T^\bullet \in \Fun(\Delta, \Fun^{\omega}_{an}(\Sp,\sL))$ is the cosimplicial object~\eqref{equation: rector} associated with the fiber square~\eqref{equation: fiber square of functors}. Since the functor $\AC$ preserves finite products, it suffices to show that the natural map
$$\AC(\mathrm{Tot}(T^\bullet)) \to \mathrm{Tot}(\AC T^\bullet) \in \CoAlg(\sseq_0^\circledast)$$
is an equivalence. The forgetful functor $\oblv\colon \CoAlg(\sseq_0^\circledast) \to \sseq_0$ is conservative, hence it is enough to show that the map
$$\6_* C_*(\mathrm{Tot}(T^{\bullet})) \to \oblv(\mathrm{Tot}(\AC T^{\bullet})) \in \sseq_0 $$
is an equivalence. 

By Lemma~\ref{lemma:connectivity of fibers in rector model}, we observe that, for each $k\geq 1$, there exists $N(k)\geq 0$ such that the map
$$\mathrm{Tot}(\6_kC_*(T^{\bullet})) \to \mathrm{Tot}^N(\6_kC_*(T^{\bullet})) $$
is an equivalence as soon as $N>N(k)$. By Proposition~\ref{proposition: coalgebras and almost finite totalizations}, we obtain that the map
$$\oblv(\mathrm{Tot}(\AC T^{\bullet})) \to  \mathrm{Tot}(\oblv(\AC T^{\bullet})) \simeq  \mathrm{Tot}(\6_* C_*(T^{\bullet})) \in \sseq_0$$
is an equivalence. Finally, by Lemma~\ref{lemma: totalization and chains} and Lemma~\ref{lemma:connectivity of fibers in rector model}, the map 
$$\6_kC_*(\mathrm{Tot}(T^\bullet)) \to \mathrm{Tot}(\6_kC_*(T^\bullet)), \;\; k\geq 1$$
is an equivalence, which implies the proposition.
\end{proof}

At the end of the section, we use Proposition~\ref{proposition: AC commutes with fiber sequences} to enhance the sequence of derivatives $\bar{\6}_*(F) \in \sseq_0$ of an analytic functor $F\in \Fun^{\omega}_{an}(\Sp,\sL)$ with a spectral Lie algebra structure. This will be done by induction on the Taylor tower.

\begin{lmm}\label{lemma: coderivatives of infinite loops}
Let $H\in \Fun^{\omega}_{*}(\Sp,\sL)$ be an $n$-homogeneous functor, $n\geq 1$. Then there is an equivalence 
$$\AC(H) \xrightarrow{\sim} \mathfrak{C}(\bar{\6}_nH) \in \CoAlg(\sseq_0^\circledast).$$
Here $\mathfrak{C}\colon \sseq_0 \to \CoAlg(\sseq_0^\circledast)$ is a cofree coalgebra from Corollary~\ref{corollary: nil dp = coalg for sseq}.
\end{lmm}

\begin{proof}
Since the functor $H$ is homogeneous, there is an $n$-homogeneous functor $\mathbb{H}\colon \Sp \to \Sp(\sL)$ such that $H\simeq \deloop \mathbb{H}.$ We note that the counit map $$\susp\deloop \mathbb{H} \to \mathbb{H}$$ induces the natural transformation
\begin{equation}\label{equation: nat transformation inf loops, eq1}
\widetilde{C}_*(H)\simeq \widetilde{C}_*(\susp\deloop \mathbb{H}) \to \widetilde{C}_*(\mathbb{H}).
\end{equation}
Moreover, since the functor $\widetilde{C}_*\colon \Sp(\sL) \to \Mod_{\kk}$ commutes with arbitrary colimits, there are natural equivalences
\begin{equation}\label{equation: nat transformation inf loops, eq2}
\bar{\6}_*H \simeq \widetilde{C}_* \6_*(H) \simeq \widetilde{C}_* \6_*(\mathbb{H}) \simeq \6_*\widetilde{C}_*(\mathbb{H}) \in \sseq_0.
\end{equation}
By combining~\eqref{equation: nat transformation inf loops, eq1} and~\eqref{equation: nat transformation inf loops, eq2} together, we obtain  the natural transformation
$$\oblv (\AC H) \simeq {\6}_*\widetilde{C}_*(H) \to \6_*\widetilde{C}_*(\mathbb{H}) \simeq \bar{\6}_*H \in \sseq_0, $$
which is an equivalence in arity $n$. The adjoint pair from Corollary~\ref{corollary: nil dp = coalg for sseq} gives the map
\begin{equation}\label{equation: infinite loops, eq1}
\AC H \to \mathfrak{C}(\bar{\6}_*H)\in \CoAlg(\sseq_0^\circledast)
\end{equation}
of cocommutative coalgebras in symmetric sequences, which is again an equivalence in arity $n$. We show that the map~\eqref{equation: infinite loops, eq1} is an equivalence in all arities.

By Lemma~\ref{lemma: precomposition with homogeneous, stable} (resp. formula~\eqref{equation: cofree coalgebra in sseq}), the left (resp. right) hand side is concentrated in arities $nk$, $k\geq 1$. Fix $k\geq 1$. We show that~\eqref{equation: infinite loops, eq1} is an equivalence in arity $nk$. By the formula~\eqref{equation: cofree coalgebra in sseq} again, it suffices to show that the comultiplication map
$$\mu_* \colon \AC(H) \to (\AC(H)^{\circledast k})^{h\Sigma_k} \in \sseq_0$$
is an equivalence in arity $nk$. By the construction (see Proposition~\ref{proposition: coderivatives are symmetric monoidal} and the discussion after), $\mu_{nk}$ is equivalent to $\6_{nk}(\mu^H)$, where $\mu^H$ is the comultiplication 
$$\mu^H \colon \widetilde{C}_*(H) \to \widetilde{C}_*(H^{\times k})^{h\Sigma_k} \simeq (\widetilde{C}_*(H)^{\otimes k})^{h\Sigma_k},$$
see Section~\ref{section: tate diagonal}. Finally, $\6_{nk}(\mu^H)$ is an equivalence by Lemma~\ref{lemma: precomposition with homogeneous, stable} and \cite[Lemma~B.3]{Heuts21}.
\end{proof}

\begin{lmm}\label{lemma: reduction to excessive, analytic}
Let $F\in \Fun^{\omega}_{an}(\Sp,\sL)$ be an analytic functor, then the natural map 
$$\AC F \xrightarrow{\sim} \lim_n \AC P_nF \in \CoAlg(\sseq_0^\circledast)$$
is an equivalence.
\end{lmm}

\begin{proof}
Since $F$ is analytic, there exist constants $b,\rho,\kappa\geq 0$ such that
the map 
$$F(X) \to P_nF(X), \;\; X\in \Sp $$
is $(b-n(\rho-1)+(n+1)\conn(X))$-connected as soon as $\conn(X)\geq \kappa$. Therefore the induced map 
$$C_*(F(X)) \to C_*(P_nF(X)) \in \Mod_{\kk}, \;\; X\in \Sp $$
is $(b-n(\rho-1)+(n+1)\conn(X))$-connected as soon as $\conn(X)\geq \kappa$, and so, the map $$\bar{\6}^kF \simeq \6_k C_*(F) \to \6_k C_*(P_nF) \simeq \bar{\6}^k P_nF$$
is an equivalence if $k\leq n$. This implies the lemma.
\end{proof}

Note that the unit map for the adjoint pair (see~\eqref{equation: quillen adjunction, abxi triv} and~\cite[Section~4.2]{koszul})
$$
\begin{tikzcd}
\widetilde{C}_*: \sL \arrow[shift left=.6ex]{r}
&\Mod_{\kk} :\triv \circ\Sigma^{-1} \arrow[shift left=.6ex,swap]{l}
\end{tikzcd}
$$
induces the natural map 
\begin{equation}\label{equation: derivatives to coderivatives}
\bar{\6}_*F \to \bar{\6}^*F \in \sseq_0, \;\; F\in \Fun^{\omega}_{an}(\Sp, \sL).
\end{equation}
We show that the map~\eqref{equation: derivatives to coderivatives} is well-behaved with respect to the Cartesian monoidal structure on $\Fun^{\omega}_{an}(\Sp, \sL)$. Namely, the functor $\bar{\6}_*$ preserves finite products as each functor in the composite~\eqref{equation: linear derivatives} preserves finite products. Therefore, it defines a symmetric monoidal functor
$$\bar{\6}_*\colon \Fun^{\omega}_*(\Sp,\sL)^{\times} \to \sseq_0^{\oplus}. $$
By combining this symmetric monoidal functor with the tautological oplax symmetric monoidal functor
$$\mathfrak{i}\colon \sseq_0^{\oplus} \to \sseq_1^\circledast $$
from Proposition~\ref{proposition: + to day oplax monoidal sseq}, we obtain the oplax symmetric monoidal functor
\begin{equation}\label{equation: usual partial triv precur}
\mathfrak{i}\circ \bar{\6}_*\colon \Fun^{\omega}_*(\Sp,\sL)^{\times} \to \sseq_1^{\circledast}.
\end{equation}

\begin{lmm}\label{lemma: oplax nat transformation}
The natural transformation $\bar{\6}_* \to \bar{\6}^*$ induces a \emph{monoidal} natural transformation
$$\mathfrak{i}\circ \bar{\6}_* \to \bar{\6}^* $$
between oplax symmetric monoidal functors from $\Fun^{\omega}_*(\Sp,\sL)^{\times}$ to $\sseq_1^{\circledast}$.
\end{lmm}

\begin{proof}
Fix $n\geq 1$. It suffices to show that the diagram
$$
\begin{tikzcd}
\bar{\6}_n(F\times G) \arrow{rr} \arrow{d}{\simeq}
& & \bar{\6}^n(F\times G) \arrow{d}{\simeq} \\
\bar{\6}_n(F) \oplus \bar{\6}_n(G) \arrow{r}
&\bar{\6}^n(F) \oplus \bar{\6}^n(G) \arrow[hookrightarrow]{r}
& \displaystyle\bigoplus_{k+m=n} \Ind_{\Sigma_k\times \Sigma_m}^{\Sigma_n} \bar{\6}^k(F) \otimes \bar{\6}^m(G).
\end{tikzcd}
$$
commutes for all $F,G\in \Fun^{\omega}_*(\Sp,\sL)$, see Remark~\ref{remark: oplax triv}. By~\cite[Proposition~1.3.1(2)]{AroneChing11}, the natural map
$$P_n(C_*F) \to P_n(C_*(P_nF)) $$
is an equivalence; the relevant part of the proof of~\cite[Proposition~1.3.1(2)]{AroneChing11} are the two paragraphs on p. 31, ibid., which works as well in the more general setting of~\cite[Section~6]{HigherAlgebra}. Therefore, we can assume that $F, G$  are $n$-excisive. Finally, by the functoriality of the diagram, we can assume that $F,G$ are $n$-homogeneous, for which the diagram commutes by Lemma~\ref{lemma: precomposition with homogeneous, stable}.
\end{proof}

By Proposition~\ref{proposition: + to day oplax monoidal sseq} and Lemma~\ref{lemma: sseq partial is strict}, the map~\eqref{equation: derivatives to coderivatives} lifts to the map 
\begin{equation}\label{equation: derivatives to coderivatives, eq1}
\eta\colon \triv(\bar{\6}_*F) \to \AC F\in \CoAlg(\sseq_0^\circledast)
\end{equation}
of cocommutative coalgebras in $\sseq_0$. Finally, by the adjunction
$$
\begin{tikzcd}
\triv: \sseq_0 \arrow[shift left=.6ex]{r}
&\CoAlg(\sseq_0^\circledast) :\prim, \arrow[shift left=.6ex,swap]{l}
\end{tikzcd}
$$
it induces the map
\begin{equation}\label{equation: derivatives to prim of coderivatives}
\bar{\6}_*F \to \prim(\AC F).
\end{equation}

\begin{thm}\label{theorem:AC conjecture for analytic functors}
The map~\eqref{equation: derivatives to prim of coderivatives} is an equivalence for an arbitrary analytic functor $F\in \Fun^{\omega}_{an}(\Sp,\sL)$.
\end{thm}

\begin{proof}
First, the theorem is true for any homogeneous functor. Indeed, suppose that $F \in \Fun^{\omega}_{an}(\Sp,\sL)$ is $n$-homogeneous, then $$\prim(\AC F) \simeq \prim(\mathfrak{C}(\bar{\6}_n(F))) \simeq \bar{\6}_n(F) \in \sseq_0,$$
see Lemma~\ref{lemma: coderivatives of infinite loops}. Therefore, the map~\eqref{equation: derivatives to prim of coderivatives} is an equivalence in arities $\neq n$. In the arity $n$, we have 
$$\prim(\AC F)_n \xrightarrow{\simeq} \bar{\6}^n(F) $$
is an equivalence and the composite $$\bar{\6}_n(F) \to \prim(\AC F)_n \to \bar{\6}^n(F)$$ is the map~\eqref{equation: derivatives to coderivatives}. This map is an equivalence in arity $n$ for an $n$-homogeneous functor $F$ by Lemma~\ref{lemma: precomposition with homogeneous, stable}.

Second, since the functor $\prim$ is a right adjoint, it commutes with all limits. Therefore, by Lemma~\ref{lemma: reduction to excessive, analytic}, we can assume that the functor $F$ is $n$-excisive for some $n\geq 1$. Finally, we finish the proof by inductively applying Proposition~\ref{proposition: AC commutes with fiber sequences} to the fiber sequences
\begin{equation*}
P_nF \to P_{n-1}F \to BD_nF, \;\; n\geq 1. \qedhere
\end{equation*}
\end{proof}

\begin{rmk}\label{remark: arone-ching conjecture}
Theorem~\ref{theorem:AC conjecture for analytic functors} is inspired by~\cite[Theorem~4.3.3]{AroneChing11} for functors from the category of spectra to the category of spaces. We notice that~\cite[Theorem~4.3.3]{AroneChing11} holds for \emph{any} reduced functor which preserves filtered colimits, and we expect that Theorem~\ref{theorem:AC conjecture for analytic functors} is also true in such generality.
\end{rmk}

By combining Theorem~\ref{theorem:AC conjecture for analytic functors} with the adjunction~\eqref{equation: koszul adjunction}, we observe that the symmetric sequence of Goodwillie derivatives $\bar{\6}_*F \in \sseq_0$ ($F$ is analytic) have a canonical spectral Lie algebra structure (with respect to the Day convolution). In particular, there are equivariant maps
\begin{equation}\label{equation: lie multiplicatons, derivatives}
m_{k;n_1,\ldots, n_k}\colon \bfLie^{\mathrm{sp}}_k \otimes \bar{\6}_{n_1}F \otimes \ldots \otimes \bar{\6}_{n_k}F \to \bar{\6}_nF \in \Mod_{\kk},
\end{equation}
where $n_1+\ldots + n_k=n$ and $k\geq 1$.

\begin{exmp}\label{example: Lie action maps, free simplicial restricted}
Suppose that $F=\free(H\kk\otimes \tau_{\geq 0}(-))\colon \Sp \to \sL$ is the free simplicial restricted Lie algebra functor. Then, by Corollary~\ref{corollary:derivatives of free}, the functor $F$ is analytic. We have $\bar{\6}_nF\simeq \Sigma \bfLie_n\in \Mod_{\kk}^{B\Sigma_n}$, $n\geq 1$ and 
$$
\bar{\6}^n F \simeq
\begin{cases}
\Sigma \kk & \mbox{if $n=1$,}\\
0 & \mbox{otherwise.}
\end{cases}
$$
Note that the coalgebra $\AC F \in \CoAlg(\sseq_0^{\circledast})$ is trivial by arity reasons. Therefore, $
\bar{\6}_*(F) \in \bfLie^{\mathrm{sp}}(\sseq_0^\circledast)$ is a \emph{free} spectral Lie algebra. This implies that the Lie multiplications maps~\eqref{equation: lie multiplicatons, derivatives} for the functor $F=\free(H\kk\otimes \tau_{\geq 0}(-))$ are homotopy equivalent to the multiplication maps in the Lie operad $\bfLie_*\in \sseq(\Vect_{\kk})$.
\end{exmp}

\subsection{Connecting map}\label{section: connecting map}
We fix natural numbers $n,k\geq 1$ throughout this section. Recall that $\prim\colon \CoAlg(\sseq_0^{\circledast}) \to \sseq_0$ is the right adjoint~\eqref{equation: prim} and $\prim$ factors as follows
$$\prim\colon  \CoAlg(\sseq_0^{\circledast}) \xrightarrow{\prim^{\mathrm{enh}}} \bfLie^{\mathrm{sp}}(\sseq_0^{\circledast}) \xrightarrow{\oblv} \sseq_0,$$
see~\eqref{equation: koszul adjunction}. In this section, we compute $\Prim^{\mathrm{enh}}$ for an extension of a cofree coalgebra by another cofree coalgebra.
\begin{thm}\label{theorem: two cell coalgebra}
Let $F\xrightarrow{p} E \xrightarrow{q} B$ be a fiber sequence in $\CoAlg(\sseq_0^{\circledast})$. Suppose that $E\simeq \mathfrak{C}(E_n)$ and $B\simeq \mathfrak{C}(B_{nk})$ are cofree on their arity $n$ and $nk$ parts, respectively. Then
\begin{enumerate}
\item the comultiplication $E \to (E^{\circledast k})^{h\Sigma_{k}} \simeq (E^{\circledast k})_{h\Sigma_{k}}$ induces an equivalence $E_{nk} \xrightarrow{\simeq} (E_n^{\circledast k})_{h\Sigma_k}$;
\item the map $p$ induces an equivalence $\prim(F)_n \xrightarrow{\simeq} \prim(E)_n\simeq E_n$;
\item the map $\Omega B \to F$ induces an equivalence $$\Sigma^{-1}B_{nk} \simeq \prim(\Omega B)_{nk} \xrightarrow{\simeq} \prim(F)_{nk};$$
\item the following diagram (in $\Mod^{B\Sigma_{nk}}_{\kk}$) commutes
$$
\begin{tikzcd}[column sep=large, row sep=large]
\Sigma^{-1}(\prim(F)_n^{\circledast k})_{h\Sigma_k} \arrow{d}[swap]{\simeq} & \bfLie^{\mathrm{sp}}_k\otimes_{h\Sigma_k} \prim(F)^{\circledast k}_n \arrow{r}{m} \arrow{l}[swap]{c^{k;1}\otimes \id} 
& \prim(F)_{nk} \\
\Sigma^{-1}(E_n^{\circledast k})_{h\Sigma_k} &\Sigma^{-1}E_{nk} \arrow{r}{\Sigma^{-1}q_{nk}} \arrow{l}[swap]{\simeq} & \Sigma^{-1}B_{nk} \arrow{u}{\simeq},
\end{tikzcd}
$$
where $m=m_{k}$ is the Lie multiplication map~\eqref{equation: lie multiplication 1} and $c^{k;1}$ is the corona cell map~\eqref{equation: corona cell map, local}.
\end{enumerate}
\end{thm}

\begin{proof}
The first part follows by Corollary~\ref{corollary: nil dp = coalg for sseq}. Since $\prim$ is a right adjoint, the second and the third parts follow. The last part is more involved. 

First, the cosimplicial model for the fiber of the map $q$ in $\CoAlg(\sseq_0^{\circledast})$ implies that the sequence 
$$F_{l}\xrightarrow{p_{l}} E_{l} \xrightarrow{q_{l}} B_{l} \in \Mod^{B\Sigma_l}_{\kk} $$
is actually a fiber sequence for $l\leq nk$. Moreover, we have a commutative diagram
$$ 
\begin{tikzcd}
F_{nk} \arrow{d}[swap]{\mu^F} \arrow{r}{p_{nk}} & E_{nk} \arrow{r}{q_{nk}} \arrow{d}[swap]{\mu^E}  
& B_{nk} \arrow{d}[swap]{\mu^B} \\
(F_{n}^{\circledast k})_{h\Sigma_k} \arrow{r}{p_{n}^{\circledast k}} & (E_{n}^{\circledast k})_{h\Sigma_k} \arrow{r} & (B_{n}^{\circledast k})_{h\Sigma_k}\simeq 0
\end{tikzcd}
$$
of comultiplications in $\Mod^{B\Sigma_{nk}}_{\kk}$. Since $\mu^E$ and $p_n$ are equivalences, we identify $p_{nk}$ with the comultiplication $\mu^F$, i.e. $p_{nk}\simeq (\mu^E)^{-1} \circ p_n^{\circledast k} \circ \mu^F$.

Second, recall the natural transformation
$$c^{k;n}_{-}\colon\prim(-)_{nk} \to \prim^{k;n}(-) \in \Mod_{\kk}^{B\Sigma_{nk}},$$
see~\eqref{equation: corona cell map, general}. Note that $c^{k;n}_{-}$ is an equivalence when evaluated at the cofree algebras $E$ and $B$. As both $\prim$ and $\prim^{k;n}$ preserve fiber sequences, $c^{k;n}_F$ is an equivalence as well. Moreover, the Lie multiplication 
$$m \colon \bfLie^{\mathrm{sp}}_k\otimes_{h\Sigma_k} \prim(F)^{\circledast k}_n \to \prim(F)_{nk}$$ on $\prim(F)$ is induced from the counit map
$$\gamma\colon \triv\circ \prim(F) \to F $$
by applying $\prim(-)_{nk}$ to both sides and by using the equivalence~\eqref{equation: prim triv Lie}. Since $c^{k;n}_F$ is an equivalence, it suffices to compute $\prim^{k;n}(\gamma)$. Indeed, by Example~\ref{example: lie and corona cell}, the Lie multiplication $m$ is the precomposition of $\prim^{k;n}(\gamma)$ with the corona cell map $c^{k;1}\otimes \id$.

Under the identifications from the first three parts of the statement, the map $\gamma_n \colon \prim(F)_n \to F_n $ is the identity $\id\colon E_n \to E_n$ and $$\gamma_{nk} \colon \Sigma^{-1}B_{nk} \simeq \prim(F)_{nk} \to F_{nk} $$ is the fiber of the map $p_{nk}$. Then, by the definition, $\prim^{k;n}(\gamma)$ is the induced map between horizontal fibers in the diagram
$$
\begin{tikzcd}
\Sigma^{-1}B_{nk} \arrow{r}{0} \arrow{d}[swap]{\gamma_{nk} = \fib(p_{nk})}  
& E_{nk} \simeq (E_{n}^{\circledast k})_{h\Sigma_k} \arrow{d}[swap]{\gamma_n^{\circledast k} = \id} \\
F_{nk} \arrow{r}{p_{nk}} & E_{nk}\simeq (E_{n}^{\circledast k})_{h\Sigma_k}.
\end{tikzcd}
$$
Therefore, the map $\prim^{k;n}(\gamma)$ is given by
$$(\Sigma^{-1}\alpha,\id)\colon \Sigma^{-1}B_{nk}\oplus \Sigma^{-1}E_{nk} \to  \Sigma^{-1}E_{nk}, $$
where $\alpha$ is the cofiber of $p_{nk}$. Finally, by the definition, $\cofib(p_{nk}) \simeq q_{nk}$, which implies the statement.
\end{proof}

\begin{dfn}\label{defininition: two stage}
An analytic functor $F\colon \Sp\to\sL$ has \emph{two stages} if and only if there are natural numbers $n,k\geq 1$ such that
\begin{enumerate}
\item $F$ is $r$-excisive for $r=nk$;
\item $P_{r-1}F$ is $n$-homogeneous.
\end{enumerate}
\end{dfn}
Let $F\in \Fun^{\omega}_{an}(\Sp,\sL)$ be a functor with two stages, then there is the fiber sequence
\begin{equation}\label{equation:two stage fiber sequence}
F \to D_nF \xrightarrow{\delta} B D_r F, \;\; r=nk.
\end{equation}
Note that the functors $D_nF$ and $BD_r F$ are homogeneous. So, they admit functorial deloopings 
$$D_nF\simeq \deloop \D_nF \;\; \text{and} \;\; BD_rF\simeq \deloop \Sigma \D_rF.$$
Finally, let
\begin{equation}\label{equation: adjoint to attaching}
\tilde{\delta}\colon \susp\deloop\D_nF \to \Sigma\D_r F 
\end{equation}
denote the adjoint map to the natural transformation 
$$\delta \colon \deloop \D_nF \to \deloop \Sigma \D_rF. $$

\begin{cor}\label{corollary:lie action, corona, adjoint}
Let $F\in \Fun^{\omega}_{an}(\Sp,\sL)$ be a functor with two stages. Then the following diagram commutes
$$
\begin{tikzcd}[column sep=large, row sep=large]
\bfLie^{\mathrm{sp}}_k\otimes \bar{\6}_n(F)^{\otimes k} \arrow{r}{m} \arrow{d}[swap]{c^{k;1}\otimes \id} 
& \bar{\6}_r(F) \\
\Sigma^{-1}\bar{\6}_n(F)^{\otimes k} \arrow{ru}[swap]{\Sigma^{-1}\bar{\6}_r(\tilde{\delta})},
\end{tikzcd}
$$
where $m=m_{k;n,\ldots,n}$ is the $\Sigma_{k}\wr\Sigma_n$-equivariant Lie multiplication map~\eqref{equation: lie multiplicatons, derivatives}, $c^{k;1}$ is the corona cell map~\eqref{equation: corona cell map, local}, and $\bar{\6}_r(\tilde{\delta})$ is induced by the adjoint map~\eqref{equation: adjoint to attaching}.
\end{cor}


\begin{proof}
By Proposition~\ref{proposition: AC commutes with fiber sequences},  the fiber sequence~\eqref{equation:two stage fiber sequence} induces the fiber sequence 
$$\AC(F) \to \AC(D_n F) \xrightarrow{\AC(\delta)} \AC(BD_rF) \in \CoAlg(\sseq_0^\circledast) $$
of coalgebras. By Lemma~\ref{lemma: coderivatives of infinite loops}, $\AC(D_n(F)) \simeq \mathfrak{C}(\bar{\6}_n(F))$ and $\AC(BD_r(F)) \simeq \mathfrak{C}(\Sigma\bar{\6}_r(F))$ are cofree on their arity $n$ and $r=nk$ parts, respectively. Finally, the assertion follows by Theorem~\ref{theorem:AC conjecture for analytic functors}, Theorem~\ref{theorem: two cell coalgebra}, and Lemma~\ref{lemma: precomposition with homogeneous, stable}.
\end{proof}

\subsection{Main example}\label{section: main example} Let $n\geq 1$. We consider the functor $$L_\xi\free\colon \Mod^{\geq 1}_{\kk} \to \sLxi\hookrightarrow \sL$$ of the $\kk$-complete free simplicial restricted Lie algebra. Recall from the end of Section~\ref{section: goodwillie tower of a free lie algebra} that we associated with $L_\xi\free$ two truncated functors $P_{[pn;n]}(\Omega^{p-2}L_\xi\free)$ and $\widetilde{P}_n(\Omega^{p-2}L_\xi\free)$. We recall that there is the fiber sequence~\eqref{equation:Dpk to Dk}:
\begin{equation*}
D_{pn}(\Omega^{p-2}L_\xi\free) \to \widetilde{P}_{n}(\Omega^{p-2}L_\xi\free) \to D_{n}(\Omega^{p-2}L_\xi\free) \xrightarrow{\delta_n} BD_{pn}(\Omega^{p-2}L_\xi\free).
\end{equation*}
We have the adjoint map $\tilde{\delta}_n$ to the natural transformation $\delta_n$
$$\tilde{\delta}_n\colon \Sigma^\infty\Omega^\infty\D_n(\Omega^{p-2}L_\xi\free) \to \Sigma\D_{pn}(\Omega^{p-2}L_\xi\free)$$
which induces the natural transformation
\begin{equation}\label{equation: stable piece, differential}
\D_{pn}(\tilde{\delta}_n) \colon \suspfreexi(\Sigma^{-1}(\Sigma^{3-p}\mathbf{L}_n(V))^{\otimes p}_{h\Sigma_p}) \to \suspfreexi(\Sigma^{3-p}\mathbf{L}_{pn}(V)), \;\; V\in \Mod^{+}_{\kk}.
\end{equation}
Here $\mathbf{L}_n(V) = \bfLie_n\otimes_{h\Sigma_n}V^{\otimes n}\in \Mod_{\kk}$, see~\eqref{equation: lie powers}. We use Corollary~\ref{corollary:lie action, corona, adjoint} to compute the map 
$$\D_{pn}(\tilde{\delta}_n)_*\colon \pi_*\suspfreexi(\Sigma^{-1}(\Sigma^{3-p}\mathbf{L}_n(V))^{\otimes p}_{h\Sigma_p}) \to \pi_*\suspfreexi(\Sigma^{3-p}\mathbf{L}_{pn}(V)). $$
induced by $\D_{pn}(\tilde{\delta}_n)$.

\begin{thm}[$p$ is odd]\label{theorem: the adjoint of delta_n}
Let $V\in \Mod^{+}_{\kk}$ and let $w\in \pi_q(\mathbf{L}_n(V))$, $n\geq 1$ be a homotopy class. Consider the homotopy class
$$x = \sigma^{-1} \beta^{\e}Q^{i}(\sigma^{3-p}w) \otimes \nu_I \in \pi_*(\suspfreexi(\Sigma^{-1}(\Sigma^{3-p}\mathbf{L}_n(V))^{\otimes p}_{h\Sigma_p})),$$
where $i\geq 0$, $\e\in \{0,1\}$, and $\nu_I\in \Lambda$. Then 
$$\D_{pn}(\tilde{\delta}_n)_*(x)=(-1)^{q\e}\sigma^{3-p}\bQe{\e}{i}(w)\otimes \nu_{I}\in \pi_*(\suspfreexi(\Sigma^{3-p}\mathbf{L}_{pn}(V))).$$
In particular, $x\in \ker(\D_{pn}(\tilde{\delta}_n)_*)$ if and only if $\beta^{\e}Q^{i}(\sigma^{3-p}w)\in \ker(\sigma^{p-2})$, where $\ker(\sigma^{p-2})\subset \pi_*(\Sigma^{-1}(\Sigma^{3-p}\mathbf{L}_n(V))^{\otimes p}_{h\Sigma_p})$ is the kernel of the desuspension map
$$\sigma^{p-2}\colon (\Sigma^{3-p}\mathbf{L}_n(V))^{\otimes p}_{h\Sigma_p} \to \Sigma^{2-p}(\Sigma \mathbf{L}_n(V))^{\otimes p}_{h\Sigma_p}.$$
\end{thm}

\begin{proof}
We prove the theorem in several steps. First, by Proposition~\ref{proposition:suspention and natural transformations} and Remark~\ref{remark:inverse to the equivalence}, we observe that it suffices to compute the induced map 
\begin{align*}
\widetilde{H}_*(\D_{pn}(\tilde{\delta}_n))\colon \pi_*((\Sigma^{3-p}\mathbf{L}_n(V))^{\otimes p}_{h\Sigma_p})&\simeq  \widetilde{H}_*(\suspfreexi(\Sigma^{-1}(\Sigma^{3-p}\mathbf{L}_n(V))^{\otimes p}_{h\Sigma_p})) \\
&\to \widetilde{H}_*(\suspfreexi(\Sigma^{3-p}\mathbf{L}_{pn}(V))) \simeq \pi_*(\Sigma^{4-p}\mathbf{L}_{pn}(V))
\end{align*}
on the homology groups. Namely, we have to show that $\widetilde{H}_*(\D_{pn}(\tilde{\delta}_n))$ maps the element $\beta^{\e}Q^{i}(\sigma^{3-p}w)$ to $(-1)^{q\e}\sigma^{4-p}\bQe{\e}{i}(w)$. 

Second, by Definition~\ref{definition: dll-classes} of the Dyer--Lashof--Lie classes, the corona cell map~\eqref{equation: corona cell map, local}
$$c^{p;1}\colon \Sigma^{1-p}\bfLie_p\otimes \mathbf{sgn} \to \Sigma^{-1}\kk \in \Mod^{B\Sigma_{p}}_{\kk} $$
induces the map
$$c\colon \Sigma^2 \mathbf{L}_p(\Sigma^{2-p}\mathbf{L}_n(V)) \to (\Sigma^{3-p}\mathbf{L}_n(V))^{\otimes p}_{h\Sigma_p} $$
such that $$c_*(\sigma^2\bQe{\e}{i}(\sigma^{2-p}w))= \beta^{\e}Q^i(\sigma^{3-p}w)\in \pi_*((\Sigma^{3-p}\mathbf{L}_n(V))^{\otimes p}_{h\Sigma_p}). $$
Therefore it is enough to compute the composite $\widetilde{H}_*(\D_{pn}(\tilde{\delta}_n)) \circ c_*$.

Third, we consider the composite functors
\begin{align*}
\widetilde{F}_n&\colon \Sp \xrightarrow{H\kk\otimes \tau_{\geq 1}(-)} \Mod_{\kk}^{\geq 1} \xrightarrow{\widetilde{P}_n(\Omega^{p-2}L_\xi\free)} \sL, \\
F_n&\colon \Sp \xrightarrow{H\kk\otimes \tau_{\geq 1}(-)} \Mod_{\kk}^{\geq 1} \xrightarrow{P_n(L_\xi\free)} \sL, \\
F&\colon \Sp \xrightarrow{H\kk\otimes \tau_{\geq 1}(-)} \Mod_{\kk}^{\geq 1} \xrightarrow{L_\xi\free} \sL,
\end{align*}
where $n\geq 1$. By Corollary~\ref{corollary:derivatives of free}, the functors $\widetilde{F}_n, F_n, F$ are analytic. By Theorem~\ref{theorem:AC conjecture for analytic functors} and the adjunction~\ref{equation: koszul adjunction}, the span of the natural transformations
$$\Omega^{p-2}F \xrightarrow{} \Omega^{p-2} F_n  \xleftarrow{} \widetilde{F}_n  $$
induces the span
\begin{equation}\label{equation: the adjoint of delta_n, span}
\Sigma^{2-p} \bar{\6}_* F \xrightarrow{} \Sigma^{2-p} \bar{\6}_* F_n  \xleftarrow{} \bar{\6}_*\widetilde{F}_n 
\end{equation}
in the category $\bfLie^{\mathrm{sp}}(\sseq_0^{\circledast})$.
Again, by Corollary~\ref{corollary:derivatives of free}, we have
$$
\bar{\6}_l \widetilde{F}_n \simeq
\begin{cases}
\Sigma^{3-p} \bfLie_n & \mbox{if $l=n$,}\\
\Sigma^{3-p} \bfLie_{pn} & \mbox{if $l=pn$,}\\
0 & \mbox{otherwise.}
\end{cases}
$$
By Example~\ref{example: Lie action maps, free simplicial restricted} and the diagram~\eqref{equation: the adjoint of delta_n, span}, we obtain that the $(\Sigma_p \wr \Sigma_n)$-equivariant Lie multiplication map~\eqref{equation: lie multiplicatons}:
\begin{align*}
m=m_{p;n,\ldots, n} \colon \bfLie^{\mathrm{sp}}_p\otimes (\bar{\6}_n\widetilde{F}_n)^{\otimes p} \simeq \Sigma \bfLie_p \otimes (\Sigma^{-1}\bar{\6}_n \widetilde{F}_n)^{\otimes p} \to \bar{\6}_{pn} \widetilde{F}_{n}
\end{align*}
is homotopy equivalent to the composite 
$$\Sigma \bfLie_p \otimes (\Sigma^{2-p}\bfLie_n)^{\otimes p} \xrightarrow{\sigma^{2-p}} \Sigma^{3-p} \bfLie_p\otimes \bfLie_n^{\otimes p}\xrightarrow{\Sigma^{3-p}\widetilde{m}} \Sigma^{3-p}\bfLie_{pn}, $$
where $\sigma^{2-p}$ is the canonical suspension map, and $\widetilde{m}$ is the multiplication map in the Lie operad $\bfLie_*\in \sseq(\Vect_{\kk})$.

Finally, since the functor $H\kk\otimes \tau_{\geq 1}(-)\colon \Sp \to \Mod_{\kk}^{\geq 1}$ is essentially surjective, we deduce by Corollary~\ref{corollary:lie action, corona, adjoint} that
$$\widetilde{H}_*(\D_{pn}(\tilde{\delta}_n)) \circ c_* = (\Sigma m)_*. $$
By Definition~\ref{definition: dll operations} of the Dyer--Lashof--Lie operations, the theorem follows.
\end{proof}

The ``$p=2$''-analog of Theorem~\ref{theorem: the adjoint of delta_n} reads as follows; its proof is almost identical and we leave it to the reader to complete the details.

\begin{thm}[$p=2$]\label{theorem: the adjoint of delta_n, p=2}
Let $V\in \Mod^{+}_{\kk}$ and let $w\in \pi_q(\mathbf{L}_n(V))$, $n\geq 1$ be a homotopy class. Consider the homotopy class
$$x = \sigma^{-1} Q^{i}(\sigma w) \otimes \lambda_I \in \pi_*(\suspfreexi(\Sigma^{-1}(\Sigma \mathbf{L}_n(V))^{\otimes 2}_{h\Sigma_2})),$$
where $i\geq q+1$, and $\lambda_I\in \Lambda$. Then 
$$\D_{2n}(\tilde{\delta}_n)_*(x)= \sigma \Qe{i}(w)\otimes \lambda_{I}\in \pi_*(\suspfreexi(\Sigma \mathbf{L}_{2n}(V))).$$ 
\end{thm}

\section{Applications}\label{section: application}

In this section we prove our main results Theorem~\ref{theorem: B, intro} (Theorem~\ref{theorem: whitehead conjecture}), Theorem~\ref{theorem: C, intro} (Proposition~\ref{proposition: differential, length 1}), Theorem~\ref{theorem: D, intro} (Theorems~\ref{theorem: lin formula, odd p} and~\ref{theorem: lin formula, p = 2}), Theorem~\ref{theorem: E,intro} (Lemmas~\ref{lemma:GSS differential, positive columns} and~\ref{lemma:GSS differential, positive columns, p =2}), and Theorem~\ref{theorem: F, intro} (Theorem~\ref{theorem: dll action on lie}).

In Section~\ref{section: lin formula} we define the \emph{renormalized (algebraic) Goodwillie spectral sequence}~\eqref{equation: RAGSS}. In Proposition~\ref{proposition: differential, length 1} we combine together Theorems~\ref{theorem:james-hopf and dyer-lashof}, \ref{theorem: the adjoint of delta_n}, and~\ref{theorem: the adjoint of delta_n, p=2} to evaluate the $\tilde{d}_2$-differentials in~\eqref{equation: RAGSS} on the generators of the algebra~$\Lambda$, which will give Theorem~\ref{theorem: C, intro}. In Lemma~\ref{lemma: leibniz rule and james} we combine together the Leibniz rule (Theorem~\ref{theorem:leibniz rule in GSS}) and the James periodicity (Theorem~\ref{theorem: main theorem, james periodicity}) to obtain the recursive formulas for the $\tilde{d}_2$-differential evaluated on the admissible monomials. These formulas imply \emph{Lin's formulas} (Theorem~\ref{theorem: D, intro}) in Theorems~\ref{theorem: lin formula, odd p} and~\ref{theorem: lin formula, p = 2}. We present an example of a calculation with the Lin formula in Example~\ref{example: lin formula, p=2}.

In Propositions~\ref{proposition: differential, adem, p=2} and~\ref{proposition: differential, adem, p is odd} we evaluate the $\tilde{d}_2$-differential on the quadratic admissible monomials of the $\Lambda$-algebra. In Lemmas~\ref{lemma:GSS differential, positive columns} and~\ref{lemma:GSS differential, positive columns, p =2} we express the $\tilde{d}_2$-differentials in~\eqref{equation: RAGSS} in terms of those differentials which originate from the $0$-th row and the Dyer--Lashof--Lie operations of Definition~\ref{definition: dll operations}. By combining these results together we prove Theorem~\ref{theorem: dll action on lie}, which states that the Dyer--Lashof--Lie operations satisfy \emph{Adem-type relations}~\eqref{equation: Lie-Adem relation, p is odd, eq1}, \eqref{equation: Lie-Adem relation, p is odd, eq2}, and~\eqref{equation: Lie-Adem relation, p is 2, eq1}. For the sake of completeness, we also discuss in Theorem~\ref{theorem: homotopy of free H-Lie algebra} the homotopy groups of a \emph{free $H_{\infty}$-$\bfLie_*$-algebra $F_{\bfLie}(W)$}, cf.~\cite{Antolin20} and~\cite{Kjaer18}.

In Corollaries~\ref{cor: goodwillie differential lowers the filtation, p is odd} and~\ref{cor: goodwillie differential lowers the filtation, p is 2} we show that the $\tilde{d}_2$-differentials in~\eqref{equation: RAGSS} preserve  a \emph{certain multiindex filtration} and compute their leading terms. This observation essentially proves Theorem~\ref{theorem: whitehead conjecture} and Theorem~\ref{theorem: B, intro}. In Remark~\ref{remark: lin-nishida conjecture} we discuss degeneration properties of the Goodwillie spectral sequence~\eqref{equation: AGSS} for $\free(W)$, $\dim\pi_*(W)>1$. 

We conclude the paper with the \emph{Lin--Nishida conjecture} (Remark~\ref{remark: lin-nishida conjecture}) which concerns the interaction of the differentials in the Adams and Goodwillie spectral sequences. 

Throughout this section $V\simeq \Sigma^l \kk\in \Mod_{\kk}^{\geq 0}$ is a \emph{$1$-dimensional} graded vector space. We write $\iota_l\in \pi_l(V)$ for the canonical generator.

\subsection{Lin's formula}\label{section: lin formula}
Recall that the Taylor tower~\eqref{equation: algebraic tower} of the functor $$\free\colon \Mod_{\kk}^{\geq 0} \to \sL$$ induces the spectral sequence~\eqref{equation: AGSS}:
$$E^1_{t,n}(V) = \pi_t(D_n(\free)(V)) \Rightarrow \pi_t(\free(V)) $$ 
with the differentials $d_r$ acting as follows $d_r\colon E^r_{t,n}(V) \to E^r_{t-1,n+r}(V)$. By Proposition~\ref{proposition: vanishing of LnV}, Corollary~\ref{corollary:derivatives of free}, and the assumption $\dim\pi_*(V)=1$, we have
$$E^1_{*,n}(V)\cong\pi_*(\bfLie_n\otimes_{h\Sigma_n}V^{\otimes n})\otimes \Lambda =0$$
if $n\neq p^h, 2p^h$, $h\geq 0$. In order to ease the notation, we consider the \emph{renormalized} Goodwillie spectral sequence:
\begin{equation}\label{equation: RAGSS}
\widetilde{E}^r_{t,m}(V)=
\begin{cases}
E^r_{t,p^h}(V)
& \mbox{if $p$ is odd, $m=2h$,}\\
E^r_{t,2p^h}(V)
& \mbox{if $p$ is odd, $m= 2h+1$,}\\
E^r_{t,2^h}(V)
& \mbox{if $p=2$, $m=2h$,}\\
0
& \mbox{otherwise.}
\end{cases}
%
\end{equation}
with the differentials $\tilde{d}_r\colon \widetilde{E}^r_{t,m}(V) \to \widetilde{E}^r_{t-1,m+r}(V)$ given by the (possibly) non-trivial differentials $d_{r'}$ in the spectral sequence~\eqref{equation: AGSS}. For instance, if $p$ is odd, then $$\tilde{d}_1=d_1\colon \widetilde{E}^1_{t,0}(V)=E^1_{t,1}(V) \to E^1_{t-1,2}(V)=\widetilde{E}^1_{t-1,1}(V)$$ 
and
$$\tilde{d}_2=d_{p-1}\colon \widetilde{E}^2_{t,0}(V) = E^2_{t,1}(V) \to E^2_{t-1,p}(V) = \widetilde{E}^2_{t-1,2}(V).  $$

Finally, we note that the renormalized spectral sequence~\eqref{equation: RAGSS} still abuts to the homotopy groups $\pi_*(\free(V))$.

\begin{prop}\label{proposition: differential, length 1}
Suppose that $V \simeq \Sigma^{l}\kk$, and $l\geq 0$. Then $\tilde{d}_1=0$ and 
$$\tilde{d}_2(\iota_l\otimes \nu^{\e}_{i})  = 
\begin{cases}
\bQe{\e}{i}(\iota_l) \otimes 1
& \mbox{if $p$ is odd, $i\geq 0$, $\e\in \{0,1\}$,}\\
\Qe{i}(\iota_l) \otimes 1
& \mbox{if $p=2$ and $i\geq 0$.}
\end{cases}
$$
\end{prop}

\begin{proof}
By Corollary~\ref{corollary: exponential vanishing}, we have $\tilde{d}_1=0$.

Note that $\widetilde{E}^1_{t,2}(V)=E^1_{t,p}(V)\cong (\pi_*(\mathbf{L}_p(V))\otimes \Lambda)_t=0$ if $t\leq l-1$. Therefore $$\tilde{d}_2(\iota_l\otimes \nu_0^0)=0=\Qe{0}(\iota_l).$$
Hence, we can assume that $i\geq 1$. Moreover, by Remark~\ref{remark: completed agss}, we can replace the spectral sequence~\eqref{equation: AGSS} with its $\xi$-adically completed version~\eqref{equation: completed agss}. In other words, it suffices to compute $\hat{d}_{(p-1)}(\iota_{l}\otimes \nu_{i}^{\e})$. Finally, we replace the spectral sequence~\eqref{equation: completed agss} with the spectral sequence
\begin{equation}\label{equation: shifted complete AGSS}
'\widehat{E}^1_{t,n}(V)=\pi_t(D_n(\Omega^{p-2}L_\xi\free)(V)) \Rightarrow \pi_t(\Omega^{p-2}L_\xi\free(V)), \;\; V\in\Mod_{\kk}^{\geq 1}.
\end{equation}
associated with the Taylor tower of the functor $\Omega^{p-2}L_\xi\free\colon \Mod_{\kk}^{\geq 1} \to \sLxi$. We write $d'_r\colon {}'\widehat{E}^r_{t,n} \to {}'\widehat{E}^r_{t-1,n+r}$  for the differentials in the spectral sequence~\eqref{equation: shifted complete AGSS}.

We note that $'\widehat{E}^1_{t,n}(V)\cong \widehat{E}^1_{t+p-2,n}(V)$, $V\in\Mod_{\kk}^{\geq 1}$, $t,n\geq 0$. Moreover, under these isomorphisms, we have 
\begin{equation}\label{equation: shifted differentials}
\hat{d}_{p-1}(x)=\sigma^{p-2}d'_{p-1}(\sigma^{2-p}x)
\end{equation}
for any $x\in \widehat{E}^1_{t,n}(V)$, $t\geq p-2$, $n\geq 0$. Due to the splitting of Corollary~\ref{corollary: splitting}, the differential $d'_{p-1}=\delta_{1*}$ is induced by the connecting map 
$$\delta_1\colon D_1(\Omega^{p-2}L_\xi\free)(V) \to BD_p(\Omega^{p-2}L_\xi\free)(V), $$
see~\eqref{equation:Dpk to Dk}.
Therefore we have the following identity
$$\tilde{d}_2(\iota_l\otimes \nu^\e_{i})=\sigma^{p-3}\delta_{1*}(\sigma^{2-p}\iota_l\otimes \nu_{i}^{\e}) \in \widetilde{E}^1_{t-1,2}(V). $$
By Proposition~\ref{proposition:james-hopf and adjoint}, there is an equivalence
$$\delta_1\simeq \deloop\D_{p}(\tilde{\delta}_1) \circ j_p^{\Sigma^{3-p}V}, $$
where $j_p$ is the James--Hopf map, see Definition~\ref{definition:james-hopf map}. Therefore
\begin{align*}
\tilde{d}_2(\iota_l\otimes \nu_i^{\e}) &= \sigma^{p-3}\delta_{1*}(\sigma^{2-p}\iota_l \otimes \nu^{\e}_i) \nonumber \\
&=\sigma^{p-3}\D_p(\tilde{\delta}_1)_*\circ j_{p*}(\sigma^{-1}(\sigma^{3-p}\iota_l)\otimes \nu_i^{\e}) \nonumber \\
&=\sigma^{p-3}\D_p(\tilde{\delta}_1)_*((-1)^{l\e}\sigma^{-1}\beta^{\e}Q^i(\sigma^{3-p}\iota_l)\otimes 1) \nonumber \\
&=\sigma^{p-3}\sigma^{3-p}\bQe{\e}{i}(\iota_l) \otimes 1 \nonumber \\
&= \bQe{\e}{i}(\iota_l)\otimes 1. \nonumber
\end{align*}
Here the third identity follows by Theorem~\ref{theorem:james-hopf and dyer-lashof} and the fourth one follows by Theorems~\ref{theorem: the adjoint of delta_n} and~\ref{theorem: the adjoint of delta_n, p=2}.
\end{proof}

Using Theorem~\ref{theorem:leibniz rule in GSS}, we obtain a more general version of the previous proposition. Namely, let $x \in \pi_q\suspfree(\Sigma^{l}\kk)$ and let us consider $x$ as the map of suspension spectra 
$$\overline{x}\colon \suspfree(\Sigma^{q}\kk) \to \suspfree(\Sigma^{l}\kk)$$
such that $\overline{x}_*(\iota_q)=x\in \pi_q\suspfree(\Sigma^l \kk)$, where $\iota_q\in \pi_q\suspfree(\Sigma^q \kk)$ is the canonical generator. We recall that the $\Lambda$-algebra acts on the right on the homotopy groups $\pi_*(X)$, $X\in \Sp(\sL)$.

\begin{lmm}\label{lemma: GSS differential on hopf elements and leibniz}
Let $x\in \pi_q\suspfree(V)$ be a homotopy class. Then, if $p$ is odd,
$$\tilde{d}_{2}(x\nu^{\e}_{i})=\tilde{d}_{2}(x)\nu^{\e}_{i}+\D_p(\overline{x})_*(\bQe{\e}{i}(\iota_q))$$ in $\pi_{q+2i(p-1)-1-\e}\suspfree(\bfLie_p\otimes_{h\Sigma_p} V^{\otimes p})$,
where $i\geq 1$, $\e\in\{0,1\}$. 
Similarly, if $p=2$,
$$\tilde{d}_{2}(x\lambda_{i})=\tilde{d}_{2}(x)\lambda_{i}+\D_2(\overline{x})_*(\Qe{i}(\iota_q))$$ in $\pi_{q+i-1}\suspfree(\bfLie_2 \otimes_{h\Sigma_2}V^{\otimes 2})$,
where $i\geq 0$.
\end{lmm}

\begin{proof}
The lemma follows by Theorem~\ref{theorem:leibniz rule in GSS} and Proposition~\ref{proposition: differential, length 1}.
\end{proof}

Let $I=(i_1,\e_1,\ldots,i_s,\e_s)$ (resp. $I=(i_1,\ldots, i_s)$ if $p=2$) be a sequence of numbers, where $\e_1,\ldots,\e_s\in\{0,1\}$. We set $\nu_I=\nu_{i_1}^{\e_1}\cdot \ldots \cdot \nu^{\e_s}_{i_s}\in \Lambda$, resp. $\lambda_{I}=\lambda_{i_1}\cdot \ldots \cdot \lambda_{i_s}\in\Lambda$. 

\begin{cor}\label{corollary: differential, stability}
Let $x\in \pi_q\suspfree(V)$ and let $I=(i_1,\e_1,\ldots,i_s,\e_s)$ (resp. $I=(i_1,\ldots,i_s)$) be an admissible sequence. Suppose that $q>2i_1$ (resp. $q>i_1+1$). Then
$$\tilde{d}_2(x\nu_I)=\tilde{d}_2(x)\nu_I, \;\; \text{resp. $\tilde{d}_2(x\lambda_I)=\tilde{d}_2(x)\lambda_I$.} $$
\end{cor}

\begin{proof}
The proof goes by induction on $s$. If $s=1$, the statement follows from Lemma~\ref{lemma: GSS differential on hopf elements and leibniz} and the last part of Corollary~\ref{corollary: kjaer basis of free lie algebra}.

Assume that the corollary holds for the sequence $I'=(i_1,\e_1,\ldots,i_{s-1},\e_{s-1})$ (resp. $I'=(i_1,\ldots, i_{s-1})$), we will prove it for the sequence $I$. By the inductive assumption, it suffices to show that 
$$2i_s<q+2(i_1+\ldots +i_{s-1})(p-1)-\e_1-\ldots - \e_{s-1},$$ resp. $i_s<q+i_{1}+\ldots +i_{s-1}-1.$ Since the sequence $I$ is admissible, the last inequalities hold.
\end{proof}

\begin{lmm}\label{lemma: differential, excess}
Let $I=(i_1,\e_1,\ldots,i_s,\e_s)$ (resp. $I=(i_1,\ldots,i_s)$) be an admissible sequence. Then
$$\tilde{d}_2(\iota_l\otimes \nu_I)=\sum_{\alpha} \bQe{\e_{\alpha}}{i_{\alpha}}(\iota_l)\otimes \nu_{\alpha}, \;\; \text{resp. $\tilde{d}_2(x\otimes \lambda_I)=\sum_{\alpha} \Qe{i_{\alpha}}(\iota_l)\otimes \lambda_{\alpha},$} $$
where $\nu_{\alpha}\in \Lambda$ (resp. $\lambda_{\alpha} \in \Lambda$) and $i_{\alpha}\leq i_1$.
\end{lmm}

\begin{proof}
We prove the lemma only for an odd prime $p$; the argument for $p=2$ is almost identical and we leave it to the reader to complete the details. By Lemma~\ref{lemma: GSS differential on hopf elements and leibniz} and the induction on $s$, it is enough to show that
$$\D_p(\overline{x})_*(\bQe{\e_s}{i_s}(\iota_q)) = \sum_{\alpha} \bQe{\e_{\alpha}}{i_{\alpha}}(\iota_l)\otimes \lambda_{\alpha},\;\; \lambda_\alpha\in\Lambda, \;\; i_\alpha \leq i_1, $$
where $\overline{x}\colon \suspfree(\Sigma^q \kk) \to \suspfree(\Sigma^l \kk)$ is the map of suspension spectra which represents $x=\iota_l\otimes \nu_{i_1}^{\e_1}\cdot \ldots \cdot \nu^{\e_{s-1}}_{i_{s-1}}$. Here $q= l+ 2(p-1)(i_1+\ldots + i_{s-1})-\e_1-\ldots -\e_{s-1}$. Finally, by Remark~\ref{remark: image of filtration}, we observe that for every $\alpha$,
$$2i_{\alpha} \leq 2i_s-(q-l)\leq 2i_1,$$
which proves the lemma.
\end{proof}

Recall the $\Lambda$-linear shift operators $$\psi_k\colon \pi_*(\D_p(\suspfree V)) \to \pi_{*-\bar{k}}(\D_p(\suspfree V)), \;\; V\in \Mod_{\kk}, k\in\Z$$
from Definition~\ref{definition: shift operators}.

\begin{lmm}\label{lemma: leibniz rule and james}
Let $x\in \pi_q\suspfree(V)$. Then there exists a constant $N=N(q,i,\e)$ such that for any $k\in \N$, $\nu_p(k)>N$ the following identity holds
$$\tilde{d}_2(x\nu_i^{\e})=\tilde{d}_2(x)\nu_{i}^{\e}+\psi_{k}(\tilde{d}_2(x\nu_{i+k}^{\e})), $$
where $i\geq 0$, $\e\in\{0,1\}$. Similarly, for $p=2$,
$$\tilde{d}_2(x\lambda_i)=\tilde{d}_2(x)\lambda_{i}+\psi_{k}(\tilde{d}_2(x\lambda_{i+k})), $$
where $i\geq 0$.
\end{lmm}

\begin{proof}
We prove the lemma only for an odd prime $p$; the argument for $p=2$ is almost identical and we leave it to the reader to complete the details. Let $N'=N(\bQe{\e}{i}(\iota_q))$ be the constant from Theorem~\ref{theorem: main theorem, james periodicity}. Then for any $k\in\N$, $\nu_p(k)>N'$, we have
\begin{equation}\label{equation: leibniz, james, eq1}
\D_p(\overline{x})_*(\bQe{\e}{i}(\iota_q)) = \psi_k(\D_p(\overline{x})_*(\bQe{\e}{i+k}(\iota_q))). 
\end{equation}

By combining the equation~\eqref{equation: leibniz, james, eq1}  with the Leibniz rule of Lemma~\ref{lemma: GSS differential on hopf elements and leibniz}, we obtain
\begin{align*}
\psi_k\tilde{d}_2(x\nu_{i+k}^{\e}) &=\psi_k(\tilde{d}_2(x)\nu_{i+k}^{\e})+\psi_k\D_p(\overline{x})_*(\bQe{\e}{i+k}(\iota_q)) \\
&=\psi_k(\tilde{d}_2(x)\nu_{i+k}^{\e})+\D_p(\overline{x})_*(\bQe{\e}{i}(\iota_q))
\end{align*}
for all $k\in\N$ such that $\nu_p(k)>N'$. Note that, if $k>q$, then 
$$\psi_k(\tilde{d}_2(x)\nu_{i+k}^{\e})=0, $$
and so, by applying again Lemma~\ref{lemma: GSS differential on hopf elements and leibniz}, we have
$$\tilde{d}_2(x\nu_{i}^{\e})=\tilde{d}_2(x)\nu_i^{\e}+ \D_p(\overline{x})_*(\bQe{\e}{i}(\iota_q))= \tilde{d}_2(x)\nu_i^{\e}+ \psi_k(\tilde{d}_2(x\nu_{i+k}^{\e})).$$
Finally, we can assume that $N=N(q,i,\e) = N'+q$.
\end{proof}

\begin{rmk}\label{remark: james and leibniz, constant}
By Remark~\ref{remark: james periodicity, bound}, one can take the constant $N=N(q,i,\e)$ in Lemma~\ref{lemma: leibniz rule and james} to be any greater than $2i$ (resp. $i$ if $p=2$).
\end{rmk}

Let $I=(i_1,\e_1,\ldots,i_s,\e_s)$ (resp. $I=(i_1,\ldots, i_s)$ if $p=2$) be a sequence of numbers, where $\e_1,\ldots,\e_s\in\{0,1\}$. For $1\leq t\leq s$ and $k\in \N$, we denote by $I(t,k)$ the new sequence
$$I(t,k)=(i_1,\e_1,\ldots,i_t+k,\e_t,\ldots,i_s,\e_s), \;\text{resp.} \; I(t,k)=(i_1,\ldots,i_t+k,\ldots,i_s).$$

\begin{thm}[Lin's formula, odd $p$]\label{theorem: lin formula, odd p} Suppose that $p$ is odd, $V\simeq \Sigma^l\kk$, $l\geq 0$. Let $I=(i_1,\e_1,\ldots,i_s,\e_s)$ be an admissible sequence. Then 
\begin{equation}\label{equation: lin, admissible, eq0}
\tilde{d}_2(\iota_l\otimes \nu_{I}) = \bQe{\e_1}{i_1}(\iota_l)\otimes \nu_{I'}+\sum_{\alpha}c_{\alpha}\bQe{\e_1(\alpha)}{i_{1}(\alpha)}(\iota_l)\otimes \nu_{I'(\alpha)},
\end{equation}
where $I'=(i_2,\e_2,\ldots,i_s,\e_s)$, $I'(\alpha)=(i_2(\alpha),\e_2(\alpha),\ldots, i_s(\alpha),\e_s(\alpha))$ is an admissible sequence of length $s-1$, and the constant $c_{\alpha}\in \F_p$ is the coefficient for $\nu_{I(\alpha,1,k)}$, $I(\alpha,1,k)=(i_1(\alpha)+k,\e_1(\alpha), I'(\alpha))$
in the admissible expansion of 
\begin{equation}\label{equation: lin, admissible, eq1}
\sum_{t=2}^{s}\nu_{I(t,k)}=\sum_{t=2}^{s}\nu_{i_1}^{\e_1}\cdot \ldots \cdot \nu_{i_t+k}^{\e_t}\cdot \ldots \cdot \nu_{i_s}^{\e_s} \in \Lambda.
\end{equation}
Here the constants $c_{\alpha}$ are independent of $k$ for $\nu_p(k)>N$, where $N=N(I)$ depends only on $I$.
\end{thm}

\begin{proof}
By Lemma~\ref{lemma: leibniz rule and james} and Proposition~\ref{proposition: differential, length 1}, we have
\begin{align*}
\tilde{d}_2(\iota_l\otimes \nu_I)&=\tilde{d}_2(\iota_l\otimes \nu_{i_1}^{\e_1}\cdot \ldots \cdot \nu_{i_s}^{\e_s})\\
&=\bQe{\e_1}{i_1}(\iota_{l})\otimes \nu_{I'}+ \psi_k \left(\sum_{t=2}^s\tilde{d}_2(\iota_l\otimes \nu_{i_1}^{\e_1}\cdot \ldots \cdot \nu_{i_t+k}^{\e_t})\nu_{i_{t+1}}^{\e_{t+1}}\cdot\ldots \cdot \nu_{i_{s}}^{\e_s} \right)
\end{align*}
as soon as $\nu_p(k)\gg 0$. Increasing $k$ if needed, we can assume that $$2i_{t+1}<l+k+2(p-1)(i_1+\ldots+i_{t})-\e_1-\ldots-\e_t,\;\; 1 \leq t\leq s-1.$$
Therefore, by Corollary~\ref{corollary: differential, stability}, we obtain
\begin{align}\label{equation: lin, proof, eq2}
\tilde{d}_2(\iota_l\otimes \nu_I) &=  \bQe{\e_1}{i_1}(\iota_{l})\otimes \nu_{I'} \\
&+ \psi_k\tilde{d}_2 \left(\sum_{t=2}^s\iota_l\otimes \nu_{i_1}^{\e_1}\cdot \ldots \cdot \nu_{i_t+k}^{\e_t}\cdot\nu_{i_{t+1}}^{\e_{t+1}}\cdot\ldots \cdot \nu_{i_{s}}^{\e_s} \right). \nonumber
\end{align}

Consider the admissible expansion of the sum~\eqref{equation: lin, admissible, eq1}
\begin{equation}\label{equation: lin, proof, eq3}
\sum_{t=2}^{s} \nu_{I(t,k)} = \sum_{\alpha\in \mathcal{A}} c_{\alpha}\nu_{I(\alpha,1,k)} +\sum_{\beta \in \mathcal{B}} c_{\beta} \nu_{I(\beta)} \in \Lambda, 
\end{equation}
where $c_{\alpha},c_{\beta}\in \F_p$, the sequences $I(\alpha,1,k)=(i_1(\alpha)+k,\e_1,\ldots, i_s(\alpha),\e_s(\alpha))$ and $I(\beta)=(i_1(\beta),\e_1(\beta),\ldots, i_s(\beta), \e_s(\beta))$ are admissible, $i_{1}(\alpha)\geq 0$, $\alpha \in \mathcal{A}$, and $i_1(\beta)<k$, $\beta\in \mathcal{B}$. Then, by Lemma~\ref{lemma: differential, excess}, we have 
\begin{equation}\label{equation: lin, proof, eq4}
\psi_k\tilde{d}_2(\iota_l\otimes \nu_{I(\beta)})=0, \;\; \beta\in \mathcal{B}. 
\end{equation}
By increasing $k$ if needed again, we can assume that $2i_2(\alpha)<i_1(\alpha)+k+l$, and so, by Corollary~\ref{corollary: differential, stability}, we observe
\begin{equation}\label{equation: lin, proof, eq5}
\tilde{d}_2(\iota_l\otimes \nu_{I(\alpha,1,k)})=\bQe{\e_1(\alpha)}{i_1(\alpha)+k}(\iota_l)\otimes \nu_{I'(\alpha)}, \;\; \alpha\in \mathcal{A}.
\end{equation}

We finish the proof of the formula~\eqref{equation: lin, admissible, eq0} by combining together the formulas~\eqref{equation: lin, proof, eq2}, \eqref{equation: lin, proof, eq3}, \eqref{equation: lin, proof, eq4}, and~\eqref{equation: lin, proof, eq5}. Finally, we note that the constants $c_{\alpha}$ in~\eqref{equation: lin, admissible, eq0} are independent of $k$ for $\nu_p(k)>N$ because the left hand side of~\eqref{equation: lin, admissible, eq0} is independent of $k$ and the terms in the right side are linearly independent, see Sections~\ref{section: lambda algebra} and~\ref{section: H-Lie-algebras}.
\end{proof}

The proof of the ``$p=2$''-version of Theorem~\ref{theorem: lin formula, odd p} is almost identical and we leave it to the reader to complete the details. 

\begin{thm}[Lin's formula, $p=2$]\label{theorem: lin formula, p = 2} Suppose that $p=2$, $V\simeq \Sigma^l\kk$, $l\geq 0$. Let $I=(i_1,\ldots,i_s)$ be an admissible sequence. Then 
$$\tilde{d}_2(\iota_l\otimes \lambda_{I}) = \Qe{i_1}(\iota_l)\otimes \lambda_{I'}+\sum_{\alpha}c_{\alpha}\Qe{i_{1}(\alpha)}(\iota_l)\otimes \lambda_{I'(\alpha)}, $$
where $I'=(i_2,\ldots,i_s)$, $I'(\alpha)=(i_2(\alpha),\ldots, i_s(\alpha))$ is an admissible sequence of length $s-1$, and $c_{\alpha}\in \F_2$ is the coefficient for $\lambda_{I(\alpha,1,k)}$, $I(\alpha,1,k)=(i_1(\alpha)+k, I'(\alpha))$
in the admissible expansion of 
$$\sum_{t=2}^{s}\lambda_{I(t,k)}=\sum_{t=2}^{s}\lambda_{i_1}\cdot \ldots \cdot \lambda_{i_t+k}\cdot \ldots \cdot \lambda_{i_s} \in \Lambda. $$
Here the constants $c_{\alpha}$ are independent of $k$ for $\nu_2(k)>N$, where $N=N(I)$ depends only on $I$. \qed
\end{thm}

\begin{rmk}\label{remark: lin formula, leading term}
The leading term in the formula of Theorem~\ref{theorem: lin formula, odd p} (resp.~\ref{theorem: lin formula, p = 2}) can be thought as the \emph{stable Hopf invariant} of the class $\nu_I$ (resp. $\lambda_I$), cf.~\cite[Theorem~4.3.1]{Behrens12}.
\end{rmk}

\begin{rmk}\label{remark:lin formula, constant}
Remark~\ref{remark: james and leibniz, constant} shows that one can take the ``shifting'' constant $N=N(i_1,\e_1,\ldots,i_s,\e_s)$ (resp. $N=N(i_1,\ldots, i_s)$) in Theorem~\ref{theorem: lin formula, odd p} (resp. in Theorem~\ref{theorem: lin formula, p = 2}) to be any greater than $2(i_2+\ldots+i_s)$ (resp. $i_2+\ldots + i_s$). However, more careful examination of the proof of Theorems~\ref{theorem: lin formula, odd p} and~\ref{theorem: lin formula, p = 2} can show that the estimate given here is extremely non-optimal.
\end{rmk}

We provide an example for the formula in Theorem~\ref{theorem: lin formula, p = 2}.

\begin{exmp}\label{example: lin formula, p=2}
Suppose that $p=2$ and $l\geq 0$. We compute $$\tilde{d}_2(\iota_{l}\otimes \lambda_3\lambda_6)\in \pi_*\suspfree(\bfLie_2\otimes_{h\Sigma_2}V^{\otimes 2})\cong \pi_*(\bfLie_2\otimes_{h\Sigma_2}V^{\otimes 2})\otimes \Lambda.$$
Let $N=7>6$ and let $k=256$, $\nu_2(k)=8>N$. Using the Adem relation~\eqref{equation: lambda, adem3, p=2}, we find the following admissible expansion for the non-admissible monomial $\lambda_3\lambda_{6+k}=\lambda_3\lambda_{262}$:
\begin{align*}
\lambda_3\lambda_{262}&= \lambda_{258} \lambda_{7}+\lambda_{257} \lambda_{8}\\
&+\lambda_{255} \lambda_{10}+\lambda_{251} \lambda_{14}+\lambda_{243} \lambda_{22}+\lambda_{227} \lambda_{38}+\lambda_{195} \lambda_{70}+\lambda_{131} \lambda_{134}. 
\end{align*}
We are interested only in terms whose first multiplier has index greater or equal than $k=256$. Those are only $\lambda_{2+256}\lambda_7$ and $\lambda_{1+256}\lambda_8$. Therefore
$$\tilde{d}_2(\iota_{l}\otimes\lambda_3\lambda_6)=\Qe{3}(\iota_l)\otimes \lambda_6 + \Qe{2}(\iota_l)\otimes \lambda_7 + \Qe{1}(\iota_l)\otimes \lambda_8.$$
\end{exmp}

We note that the formula for the differential $\tilde{d}_2$ in Theorems~\ref{theorem: lin formula, odd p} and~\ref{theorem: lin formula, p = 2} does not depend on the degree $l$. This allows us to express the James--Hopf map (see Definition~\ref{definition:james-hopf map})
$$j_{p*}\colon \Sigma^{-1}\pi_*(W)\otimes \Lambda \to \pi_*(\Sigma^{-1}W^{\otimes p}_{h\Sigma_p}) \otimes \Lambda, \;\; W\in \Mod_{\kk}^{\geq 0} $$
in terms of the universal differential 
\begin{equation}\label{equation: universal differential}
d_{p-1}\colon \kk\otimes \Lambda \to \pi_*((\bfLie_p \otimes \kk^{\otimes p})_{h\Sigma p}) \otimes \Lambda
\end{equation}
in the spectral sequence $E^1_{*,*}(\kk)$, see~\eqref{equation: AGSS}.

\begin{prop}[$p$ is odd]\label{proposition: james-hopf and differential, reserve}
Let $W\in \Mod^{\geq 0}_{\kk}$. Suppose that 
$$d_{p-1}(\iota_0\otimes \nu) = \sum_{\alpha}\bQe{\e(\alpha)}{i(\alpha)}(\iota_0)\otimes \nu_{\alpha} \in E^1_{*,p}(\kk), $$
where $i(\alpha)\geq 0$, $\e(\alpha)\in \{0,1\}$, and $\nu, \nu_{\alpha}\in \Lambda$. Then
$$j_{p*}(\sigma^{-1}w\otimes \nu)=\sum_{\alpha}(-1)^{q\e(\alpha)}\sigma^{-1}\beta^{\e(\alpha)}Q^{i(\alpha)}(w)\otimes \nu_{\alpha} + \ker(\sigma^{p-2})\otimes \Lambda,$$
where $w\in \pi_q(W)$ and $\ker(\sigma^{p-2}) \subset \pi_*(\Sigma^{-1}W^{\otimes p}_{h\Sigma_p})$ is the kernel of the map induced by the suspension
$$\sigma^{p-2}\colon \Sigma^{-1}(W^{\otimes p}_{h\Sigma_p}) \to \Sigma^{1-p}(\Sigma^{p-2}W)^{\otimes p}_{h\Sigma_p}.$$
\end{prop}

\begin{proof}
Since the James--Hopf map is functorial, we can assume that $\dim\pi_*(W)=1$. Let $W\simeq \Sigma^l \kk$, $q=l$, and $w= \iota_l \in\pi_l(\Sigma^l \kk)$. We set $V=\Sigma^{p-3}W\simeq \Sigma^{l+p-3}\kk$ and we consider the spectral sequence~\eqref{equation: AGSS}
$$E^1_{*,*}(V) \Rightarrow \pi_*(\free (V)). $$
As in Proposition~\ref{proposition: differential, length 1}, we have
$$\sigma^{3-p} d_{p-1}(\iota_{l+p-3}\otimes \nu) = \delta_{1*}(\sigma^{2-p} \iota_{l+p-3}\otimes \nu) = \D_p(\tilde{\delta}_1)_*\circ j_{p*}(\sigma^{-1}\iota_l\otimes \nu). $$
By the assumption, the left hand side is equal to 
$$\sigma^{3-p} d_{p-1}(\iota_{l+p-3}\otimes \nu) = \sum_{\alpha} \sigma^{3-p}\bQe{\e(\alpha)}{i(\alpha)}(\iota_{l+p-3})\otimes \nu_{\alpha}.$$
Finally, the proposition follows by Theorem~\ref{theorem: the adjoint of delta_n}.
\end{proof}

The ``$p=2$''-analog of Proposition~\ref{proposition: james-hopf and differential, reserve} reads as follows; its proof is almost identical and we leave to the reader to complete the details.

\begin{prop}[$p=2$]\label{proposition: james-hopf and differential, reserve, p=2}
Let $W\in \Mod^{\geq 0}_{\kk}$. Suppose that 
$$d_{1}(\iota_0\otimes \lambda) = \sum_{\alpha}\Qe{i(\alpha)}(\iota_0)\otimes \lambda_{\alpha} \in E^1_{*,2}(\kk), $$
where $i(\alpha)\geq 1$ and $\lambda, \lambda_{\alpha}\in \Lambda$. Then
$$j_{2*}(\sigma^{-1}w\otimes \nu)=\sum_{\alpha}\sigma^{-1}Q^{i(\alpha)}(w)\otimes \lambda_{\alpha}, \;\; w\in \pi_q(W). $$
\end{prop}

\subsection{Adem-type relations}\label{section: adem-type relations} Recall the following property of the mod-$p$ binomial coefficients.

\begin{lmm}\label{lemma: binomial, congruence}
Let $b\geq 0$ and $a<0$. Then 
$$\binom{a+k}{b} \equiv (-1)^b\binom{b-a-1}{b} \mod p $$
for all $k\in \N$ such that $\nu_p(k) > \lceil \log_p(b) \rceil$.
\end{lmm}

\begin{proof} 
Follows by Lucas's theorem.
\end{proof}

Using Lemma~\ref{lemma: binomial, congruence}, we can provide more explicit formulas in Theorems~\ref{theorem: lin formula, odd p} and~\ref{theorem: lin formula, p = 2} in the case of monomials of length $2$.

\begin{prop}\label{proposition: differential, adem, p=2}
Suppose that $p=2$, $V\simeq \Sigma^l\kk$, $l\geq 0$. Let $\lambda_i\lambda_j \in \Lambda$ be an admissible monomial, i.e. $j\leq 2i$. Then 
\begin{align*}
\tilde{d}_2(\iota_l\otimes \lambda_i\lambda_j) &= \Qe{i}(\iota_l)\otimes \lambda_j \\
&+\sum_{m=2i+1}^{i+j-1}\binom{2m-2i-j-1}{m-2i-1}\Qe{i+j-m}(\iota_l)\otimes\lambda_m.
\end{align*}
\end{prop}

\begin{proof}
By Theorem~\ref{theorem: lin formula, p = 2}, we have to find the admissible expansion of $\lambda_i\lambda_{j+k}$, $\nu_2(k)\gg 0$. By the Adem relation~\eqref{equation: lambda, adem3, p=2}, we have
\begin{align*}
\lambda_i\lambda_{j+k} &= \sum_{m=2i+1}^{j+k-1}\binom{j+k-m-1}{m-2i-1}\lambda_{j+k-m+i}\lambda_m \\
&=\sum_{m=2i+1}^{i+j}\binom{j+k-m-1}{m-2i-1}\lambda_{j+k-m+i}\lambda_m\\ 
&+ \text{terms which start with $\lambda_\alpha$, $\alpha<k$.}
\end{align*}
Since the sequence $(i,j)$ is admissible, we have 
$$j-m-1\leq j-(2i+1)-1\leq -2<0. $$
Therefore, by Lemma~\ref{lemma: binomial, congruence},
\begin{align*}
\tilde{d}_2(\iota_l\otimes \lambda_i\lambda_j) &= \Qe{i}(\iota_l)\otimes \lambda_j +\sum_{m=2i+1}^{i+j}\binom{j+k-m-1}{m-2i-1}\Qe{i+j-m}(\iota_l)\otimes\lambda_m \\
&=\Qe{i}(\iota_l)\otimes \lambda_j +\sum_{m=2i+1}^{i+j}\binom{2m-2i-j-1}{m-2i-1}\Qe{i+j-m}(\iota_l)\otimes\lambda_m.
\end{align*}
Finally, we drop the last term $\Qe{0}(\iota_l)\otimes\lambda_{i+j}$, since $\Qe{0}(\iota_l)$ always vanishes.
\end{proof}

\begin{prop}\label{proposition: differential, adem, p is odd}
Suppose that $p$ is odd, $V\simeq \Sigma^l\kk$, $l\geq 0$. Let $\lambda_i\nu^{\e}_j \in \Lambda$, $\e\in\{0,1\}$ be an admissible monomial, i.e. $j< pi$. Then 
\begin{align*}
\tilde{d}_2(\iota_l\otimes \lambda_i\nu^{\e}_j) &= \bQe{}{i}(\iota_l)\otimes \nu^{\e}_j \\
&+(-1)^{\e}\sum_{m = pi}^{i+j-1}\binom{p(m-i)-(p-1)j+\e-1}{m-pi} \bQe{\e}{j+i-m}(\iota_l)\otimes \lambda_m \\
&+(\e-1)\sum_{m=pi}^{i+j-1}\binom{p(m-i)-(p-1)j}{m-pi} \bQe{}{j+i-m}(\iota_l)\otimes \mu_m.
\end{align*}
Similarly, let $\mu_i\nu_j^{\e} \in \Lambda$, $\e\in\{0,1\}$ be an admissible monomial, i.e. $j\leq pi$. Then 
\begin{align*}
\tilde{d}_2(\iota_l\otimes \mu_i\nu^{\e}_j) &= \Qe{i}(\iota_l)\otimes \nu^{\e}_j \\
&-\sum_{m = pi+1}^{i+j-1}\binom{p(m-i)-(p-1)j-1}{m-pi-1} \Qe{j+i-m}(\iota_l)\otimes \nu_m^{\e}.
\end{align*}
\end{prop}

\begin{proof}
We prove the proposition only for monomials $\lambda_i\nu^{\e}_j$; the proof for monomials $\mu_i\nu^{\e}_j$ is identical and we leave it to the reader to complete the details.
By Theorem~\ref{theorem: lin formula, odd p}, we have to find the admissible expansion of $\lambda_i\nu^{\e}_{j+k}$, $\nu_p(k)\gg 0$. By the Adem relation~\eqref{equation: lambda, adem1, p is odd}, we have
\begin{align*}
\lambda_i\nu^{\e}_{j+k} &= \sum_{m=pi}^{i+j}(-1)^{m+i+\e}\binom{(p-1)(j+k-m)-\e}{m-pi}\nu^{\e}_{j+k-m+i}\lambda_m \\
&+(1-\e)\sum_{m=pi}^{i+j}(-1)^{m+i+1}\binom{(p-1)(j+k-m)-1}{m-pi}\lambda_{j+k-m+i}\mu_m\\
&+ \text{terms which start with $\lambda_\alpha$ or $\mu_{\alpha}$, $\alpha<k$.}
\end{align*}
Since the monomial $\lambda_i\nu_j^{\e}$ is admissible, we have 
$$(p-1)(j-m)\leq (p-1)(j-pi)<0. $$
By Lemma~\ref{lemma: binomial, congruence}, we observe
$$(-1)^{m+i+\e}\binom{(p-1)(j+k-m)-\e}{m-pi}=(-1)^{\e}\binom{p(m-i)-(p-1)j+\e-1}{m-pi},$$
$$(-1)^{m+i+1}\binom{(p-1)(j+k-m)-1}{m-pi}=-\binom{p(m-i)-(p-1)j}{m-pi}. $$
Therefore
\begin{align*}
\tilde{d}_2(\iota_l\otimes \lambda_i\nu^{\e}_j) &= \bQe{}{i}(\iota_l)\otimes \nu^{\e}_j \\
&+(-1)^{\e}\sum_{m = pi}^{i+j}\binom{p(m-i)-(p-1)j+\e-1}{m-pi} \bQe{\e}{j+i-m}(\iota_l)\otimes \lambda_m \\
&+ (\e-1)\sum_{m=pi}^{i+j}\binom{p(m-i)-(p-1)j}{m-pi} \bQe{}{j+i-m}(\iota_l)\otimes \mu_m.
\end{align*}
Finally, we drop the last terms $\Qe{0}(\iota_l)\otimes\lambda_{i+j}$ and $\bQe{}{0}(\iota_l)\otimes \mu_m$, since $\Qe{0}(\iota_l)$, $\bQe{}{0}(\iota_l)$ always vanish. 

The proof for $\mu_i\nu_j^{\e}$ uses the Adem relation~\eqref{equation: lambda, adem2, p is odd} instead of~\eqref{equation: lambda, adem1, p is odd}.
\end{proof}

\begin{dfn}[$p$ is odd]\label{definition: Dyer--Lashof--Lie algebra, p is odd}
The \emph{Dyer--Lashof--Lie algebra} $\calR_p$ is the $\F_p$-algebra generated by the Dyer--Lashof--Lie operations $\bQe{\e}{i}$, $i\geq 1$, $\e\in\{0,1\}$ subject to the following relations
\begin{align}\label{equation: Lie-Adem relation, p is odd, eq1}
\bQe{\e}{j}\cdot \bQe{}{i} &=(-1)^{\e+1}\sum_{m = pi}^{i+j-1}\binom{p(m-i)-(p-1)j+\e-1}{m-pi} \bQe{}{m}\cdot\bQe{\e}{j+i-m} \\
&+(1-\e)\sum_{m=pi}^{i+j-1}\binom{p(m-i)-(p-1)j}{m-pi} \Qe{m}\cdot\bQe{}{j+i-m}, \nonumber
\end{align}
where $j<pi$, $\e\in\{0,1\}$, and
\begin{align}\label{equation: Lie-Adem relation, p is odd, eq2}
\bQe{\e}{j}\cdot\Qe{i} = \sum_{m = pi+1}^{i+j-1}\binom{p(m-i)-(p-1)j-1}{m-pi-1} \bQe{\e}{m}\cdot \Qe{j+i-m},
\end{align}
where $j\leq pi$, $\e\in\{0,1\}$.
\end{dfn}

\begin{dfn}[$p=2$]\label{definition: Dyer--Lashof--Lie algebra, p is 2}
The \emph{Dyer--Lashof--Lie algebra} $\calR_2$ is the $\F_2$-algebra generated by the Dyer--Lashof--Lie operations $\Qe{i}$, $i\geq 1$ subject to the following relations
\begin{align}\label{equation: Lie-Adem relation, p is 2, eq1}
\Qe{j}\cdot \Qe{i} =\sum_{m=2i+1}^{i+j-1}\binom{2m-2i-j-1}{m-2i-1}\Qe{m}\cdot\Qe{i+j-m}.
\end{align}
where $j\leq 2i$.
\end{dfn}

\begin{rmk}\label{remark: dll, summation over the empty range}
In the formulas~\eqref{equation: Lie-Adem relation, p is odd, eq1}, \eqref{equation: Lie-Adem relation, p is odd, eq2}, and~\eqref{equation: Lie-Adem relation, p is 2, eq1}, the summation over the empty range is meant to be zero. In particular,
$$ 
0 =
\begin{cases}
\Qe{j}\cdot \Qe{i}, & \mbox{$j\leq i+1$, $p=2$}\\
\bQe{\e}{j}\cdot \bQe{}{i},& \mbox{$j\leq (p-1)i$, $\e\in \{0,1\}$, $p$ is odd}\\
\bQe{\e}{j}\cdot \Qe{i}, 
& \mbox{$j\leq (p-1)i+1$, $\e\in \{0,1\}$, $p$ is odd.}
\end{cases}
$$
\end{rmk}

\begin{dfn}\label{definition: module over dll}
A \emph{left $\calR_p$-module} $M_*=\bigoplus_{q\geq 0} M_q$ is a graded vector space over $\kk$ equipped with an $\F_p$-linear left action of the Dyer--Lashof--Lie algebra $\calR_p$ such that
\begin{enumerate}
\item $\bQe{\e}{i}(M_q)\subset M_{q+2i(p-1)-\e-1}, i\geq 1, \e\in\{0,1\}$ (resp. $\Qe{i}(M_q)\subset M_{q+i-1}$, $i\geq 1$ if $p=2$);
\item $\bQe{\e}{i}(\alpha x)=\alpha^p\bQe{\e}{i}(x)$ (resp. $\Qe{i}(\alpha x)=\alpha^2 \Qe{i}(x)$), $\alpha\in \kk$, $x\in M_q$.
\end{enumerate}
\end{dfn}

\begin{dfn}\label{definition: unstable module}
An $\calR_p$-module $M_*$ is \emph{unstable} if $\bQe{\e}{i}(x)=0$ (resp. $\Qe{i}(x)=0$) for all $x\in M_q$ and $2i \leq q$ (resp. $i\leq q$).
\end{dfn}

In the sequel, we want to demonstrate that the homotopy groups $\pi_*(L)$ of an $H_\infty$-$\bfLie_*$-algebra $L$ form a left unstable $\calR_p$-module. For that we need to compute the differentials $d_r\colon E^1_{*,n}(W) \to E^1_{*,n+r}(W)$ in the spectral sequence~\eqref{equation: AGSS} with $n\geq 2$; the case $n=1$ is done in Theorems~\ref{theorem: lin formula, odd p} and~\ref{theorem: lin formula, p = 2}.

\begin{lmm}[$p$ is odd]\label{lemma:GSS differential, positive columns}
Let $W \in \Mod^{\geq 0}_\kk$, $w\in \pi_*(\bfLie_n\otimes_{\Sigma^n}W^{\otimes n})$, and let $\nu \in \Lambda$. Suppose that the universal differential~\eqref{equation: universal differential}
$$d_{p-1}(\iota_0\otimes \nu) = \sum_{\alpha}\bQe{\e(\alpha)}{i(\alpha)}(\iota_0)\otimes \nu_{\alpha} \in E^1_{*,p}(\kk),$$
where $i(\alpha)\geq 0$, $\e(\alpha)\in\{0,1\}$, $\nu_{\alpha} \in \Lambda$. Then 
$$d_{(p-1)n}(w\otimes \nu)= \sum_{\alpha}\bQe{\e(\alpha)}{i(\alpha)}(w)\otimes \nu_{\alpha} \in E^{1}_{*,pn}(W).$$
\end{lmm}

\begin{proof}
We proceed as in Proposition~\ref{proposition: differential, length 1}. We can assume that $\nu \in \Lambda_{r,*}$, $r>0$, and so, we have the following identity
$$d_{p-1}(w\otimes \nu)=\sigma^{p-3}\delta_{n*}(\sigma^{2-p}w\otimes \nu) \in E^1_{t-1,pn}(W), $$
where $\delta_{n*}\colon \pi_t(D_n(\Omega^{p-2}L_\xi\free)(W)) \to \pi_{t-1}(BD_{pn}(\Omega^{p-2}L_\xi\free)(W))$ is the homomorphism induced by the connecting map
$$\delta_n\colon D_n(\Omega^{p-2}L_\xi\free)(W) \to BD_{pn}(\Omega^{p-2}L_\xi\free)(W), $$
see~\eqref{equation:Dpk to Dk}. By Proposition~\ref{proposition:james-hopf and adjoint}, there is an equivalence
$$\delta_n\simeq \deloop\D_{pn}(\tilde{\delta}_n) \circ j_p^{\Sigma^{3-p}\mathbf{L}_nW}, $$
where $j_p$ is the James--Hopf map, see Definition~\ref{definition:james-hopf map}. Finally, the lemma follows from Proposition~\ref{proposition: james-hopf and differential, reserve} and Theorem~\ref{theorem: the adjoint of delta_n}.
\end{proof}

\begin{rmk}\label{remark: higher differential}
The proof of Lemma~\ref{lemma:GSS differential, positive columns} may be surprising as it demonstrates the importance of the James--Hopf map $j_p$ for our calculations. Indeed, we only used the formal properties of the differentials $d_{p-1}$ to determine $j_p$, which can then be used to determine the differentials $d_{(p-1)n}$ for higher $n\geq 1$.
\end{rmk}

The ``$p=2$''-analog of Lemma~\ref{lemma:GSS differential, positive columns} reads as follows; its proof is almost identical and we leave to the reader to complete the details.

\begin{lmm}[$p=2$]\label{lemma:GSS differential, positive columns, p =2}
Let $W \in \Mod^{\geq 0}_\kk$, $w\in \pi_*(\bfLie_n\otimes_{\Sigma^n}W^{\otimes n})$, and let $\lambda \in \Lambda$. Suppose that the universal differential~\eqref{equation: universal differential}
$$d_{1}(\iota_0\otimes \lambda) = \sum_{\alpha}\Qe{i(\alpha)}(\iota_0)\otimes \lambda_{\alpha} \in E^1_{*,2}(\kk),$$
where $i(\alpha)\geq 0$ and $\lambda_{\alpha} \in \Lambda$. Then 
$$d_{n}(w\otimes \lambda)= \sum_{\alpha}\Qe{i(\alpha)}(w)\otimes \lambda_{\alpha} \in E^{1}_{*,2n}(W). $$ 
\end{lmm}

\begin{thm}\label{theorem: dll action on lie}
Given an $H_\infty$-$\bfLie_*$-algebra $L$ in $\Mod_{\kk}^{\geq 0}$, the action of the Dyer--Lashof--Lie operations $\bQe{}{i}$, $\Qe{i}$ (resp. $\Qe{i}$ if $p=2$) makes $\pi_*(L)$ into a left unstable $\calR_p$-module.
\end{thm}

\begin{proof}
We prove the theorem only for an odd prime $p$; the argument for $p=2$ is almost identical and we leave it to the reader to complete the details. 

Given a class $x \in \pi_l(L)$, we want to show that the Dyer--Lashof--Lie operations $\bQe{\e}{i}$ obey the relations~\eqref{equation: Lie-Adem relation, p is odd, eq1} and~\eqref{equation: Lie-Adem relation, p is odd, eq2} when acting on $x$. Let $V=\Sigma^l \kk \in \Mod_{\kk}$. Then the class $x$ represents a map
$$\overline{x}\colon F_{\bfLie}(V)=\bigoplus_{n\geq 0} \bfLie_n\otimes_{h\Sigma_n} V^{\otimes n} \to L$$
of $H_\infty$-$\bfLie_*$-algebras, where $F_{\bfLie}(V)$ is the free $H_\infty$-$\bfLie_*$-algebra. By the naturality of the Dyer--Lashof--Lie operations, we can assume that $L=F_{\bfLie}(V)$ is the free $H_\infty$-$\bfLie_*$-algebra and $x=\iota_l \in \pi_l(F_{\bfLie}(V))$.

Consider the vector space $\pi_*(F_{\bfLie}(V))\otimes \Lambda$. By Corollary~\ref{corollary:derivatives of free}, we have 
$$\pi_*(F_{\bfLie}(V))\otimes \Lambda\cong \widetilde{E} ^2_{*,*}(V),$$
where $\widetilde{E}^2_{*,*}(V)$ is the second page of the renormalized spectral sequence~\eqref{equation: RAGSS}. In particular, the vector space $\pi_*(F_{\bfLie(V)})\otimes \Lambda$ is endowed with the differential $\tilde{d}_2$, $\tilde{d}_2^2=0$.

Consider the element $y=\iota_l\otimes \lambda_i\nu^{\e}_j \in \pi_*(F_{\bfLie(V)})\otimes \Lambda$, $j<pi$, $\e\in\{0,1\}$. On the one hand, we have $\tilde{d}_2^2(y)=0$. On the other hand, by Proposition~\ref{proposition: differential, adem, p is odd}, we have
\begin{align*}
0&=\tilde{d}^2_2(\iota_l\otimes \lambda_i\nu^{\e}_j) = \tilde{d}_2(\bQe{}{i}(\iota_l)\otimes \nu^{\e}_j) \\
&+(-1)^{\e}\sum_{m = pi}^{i+j-1}\binom{p(m-i)-(p-1)j+\e-1}{m-pi} \tilde{d}_2(\bQe{\e}{j+i-m}(\iota_l)\otimes \lambda_m) \\
&+(\e-1)\sum_{m=pi}^{i+j-1}\binom{p(m-i)-(p-1)j}{m-pi} \tilde{d}_2(\bQe{}{j+i-m}(\iota_l)\otimes \mu_m).
\end{align*}
By Lemma~\ref{lemma:GSS differential, positive columns} and Proposition~\ref{proposition: differential, length 1}, we have 
$$\tilde{d}_2(\bQe{}{i}(\iota_l)\otimes \nu_j^{\e}) = \bQe{\e}{j}\cdot \bQe{}{i}(\iota_l)\otimes 1,$$ 
$$\tilde{d}_2(\bQe{\e}{i+j-m}(\iota_l)\otimes \lambda_m) = \bQe{}{m}\cdot \bQe{\e}{i+j-m}(\iota_l) \otimes 1,$$
$$\tilde{d}_2(\bQe{}{i+j-m}(\iota_l)\otimes \mu_m) = \Qe{m}\cdot \bQe{}{i+j-m}(\iota_l) \otimes 1, $$
where $pi\leq m\leq i+j-1$. This implies that the Dyer--Lashof--Lie operations enjoy the relation~\eqref{equation: Lie-Adem relation, p is odd, eq1}. The second part of Proposition~\ref{proposition: differential, adem, p is odd} yields the relation~\eqref{equation: Lie-Adem relation, p is odd, eq2}. Finally, if $p=2$, we obtain the relation~\eqref{equation: Lie-Adem relation, p is 2, eq1} by Proposition~\ref{proposition: differential, adem, p=2}.
\end{proof}

\begin{rmk}\label{remark: adem relations, history}
Theorem~\ref{theorem: dll action on lie} is new only for odd primes. For $p=2$, we refer the reader to~\cite[Proposition~5.9]{Antolin20} and~\cite[Theorem~1.5.1]{Behrens12}.
\end{rmk}

\begin{rmk}\label{remark: adem relations, transfers}
Another strategy to obtain the relations~\eqref{equation: Lie-Adem relation, p is odd, eq1},~\eqref{equation: Lie-Adem relation, p is odd, eq2},~\eqref{equation: Lie-Adem relation, p is 2, eq1} is to follow the proof of~\cite[Theorem~3.16]{AM99} and to compute the image of transfer map $$\tau \colon H_*(\Sigma_{p^2},\F_p) \to H_*(\Sigma_p \wr \Sigma_p,\F_p)$$
between the group homology. Here $\Sigma_p\wr \Sigma_p$ is the wreath product of $(\Sigma_p)^{\times p} \rtimes \Sigma_p$; $\Sigma_p \wr \Sigma_p$ is a subgroup of the symmetric group $\Sigma_{p^2}$ of index coprime to $p$. For $p=2$, the transfer map $\tau$ was computed in~\cite{Priddy_transfers} (see also~\cite[Theorem~7.1 and Example~7.6]{Kuhn85}). For odd primes, the formulas for $\tau$ are unknown. However, reversing the strategy, one can try instead to compute $\tau$ by using the obtained relations~\eqref{equation: Lie-Adem relation, p is odd, eq1},~\eqref{equation: Lie-Adem relation, p is odd, eq2},~\eqref{equation: Lie-Adem relation, p is 2, eq1}. We are going to return to this question in a future work.
\end{rmk}

\begin{dfn}\label{definition: unstable dll-algebra}
An \emph{unstable $\calR_p$-Lie algebra} is a left $\calR_p$-module $L_*$ which is a graded Lie algebra such that
\begin{enumerate}
\item if $p=2$, $[x,\Qe{i}(y)]=0$, $x,y\in L_*$, $i\geq 0$.
\item if $p$ is odd, $[x,\bQe{\e}{i}(y)]=0$, $x,y\in L_*$, $i\geq 0$, $\e\in\{0,1\}$.
\item if $p=2$, $\Qe{q+1}(x)=[x,x]$, $x\in L_q$.
\item if $p=3$, $\bQe{}{q}(x)=[[x,x],x]$, $x\in L_{2q-1}$.
\end{enumerate}
\end{dfn}

\begin{thm}\label{theorem: dll+lie action on lie}
Given an $H_\infty$-$\bfLie_*$-algebra $L$ in $\Mod^{\geq 0}_{\kk}$, the $\calR_p$-action on $\pi_*(L)$ and the natural graded Lie bracket on $\pi_*(L)$ give $\pi_*(L)$ the structure of an unstable $\calR_p$-Lie algebra.
\end{thm}

\begin{proof}
By Theorem~\ref{theorem: dll action on lie}, it suffices to check that the  Lie bracket vanishes on Dyer--Lashof--Lie classes and the bottom Dyer--Lashof--Lie classes can be expressed in terms of the Lie bracket if $p=2,3$. For the first part, we refer to~\cite[Lemma~6.5]{Antolin20} if $p=2$ and to~\cite[Proposition~3.7]{Kjaer18} if $p>2$. For the second part, we refer to~\cite[Lemma~6.4]{Antolin20} if $p=2$ and to~\cite[Remark~6.6]{Zhang21} if $p=3$.
\end{proof}

\begin{rmk}\label{remark: kjaer's error}
We note that~\cite[Corollary~4.7]{Kjaer18} states that the triple Lie brackets $[[v,v],v]$, $v\in \pi_*(L)$ always vanish, even if $p=3$. However, the argument there contains a mistake which was pointed out by A.~Zhang in~\cite[Remark~6.6]{Zhang21}.
\end{rmk}

\begin{dfn}\label{definition: cu sequence}
A (possibly void) sequence $J=(j_1,\delta_1,\ldots, j_s,\delta_s)$ (resp. $J=(j_1,\ldots, j_s)$) is called \emph{completely unadmissible (abbrv. CU)} if $$j_t> p j_{t+1} -\delta_{t+1}, \;\; \text{resp. $j_t>2 j_{t+1}$, $1\leq t \leq s-1$}.$$ The \emph{excess} of $J$, denoted by $e(J)$, is defined by
$$e(J)=2j_s, \;\; \text{resp. $e(J)=j_s$,} \;\; e(\emptyset)=-1. $$
\end{dfn}

Given a sequence $J=(j_1,\delta_1,\ldots,j_s,\delta_s)$ (resp. $J=(j_1,\ldots,j_s)$), we write $\Qe{J} \in \calR_p$ for the monomial $$\Qe{J}=\bQe{\delta_1}{j_1}\cdot \ldots \cdot \bQe{\delta_s}{j_s} \in \calR_p \;\; \text{(resp. $\Qe{J}=\Qe{j_1}\cdot\ldots \cdot \Qe{j_s}\in \calR_2$)}.$$ We note that the Adem relations~\eqref{equation: Lie-Adem relation, p is odd, eq1}, \eqref{equation: Lie-Adem relation, p is odd, eq2}, and~\eqref{equation: Lie-Adem relation, p is 2, eq1} imply that the algebra $\calR_{p}$ is spanned by the set of completely unadmissible monomials $\Qe{J}$.

\begin{prop}\label{proposition: homotopy lie, basis, 1-dim}
Let $V\simeq \Sigma^l \kk \in \Mod_{\kk}$, $l\geq 0$. Then the homotopy groups $\pi_*(\bfLie_n\otimes_{h\Sigma_n} V^{\otimes n})$, $n\geq 1$ have a basis
\begin{enumerate}
\item $\Qe{J}(\iota_l)$, $J=(j_1,\ldots,j_h)$ is a CU sequence, $e(J)>l$ if  $p=2$ and $n=2^h$;
\item $\Qe{J}(\iota_l)$, $J=(j_1,\delta_1,\ldots,j_h,\delta_h)$ is a CU sequence, $e(J)> l$ if $p$ is odd and $n=p^h$;
\item $\Qe{J}([\iota_l,\iota_l])$, $J=(j_1,\delta_1,\ldots,j_h,\delta_h)$ is a CU sequence, $e(J)> 2l$ if $p$ is odd, $l$ is odd, and $n=2p^h$.
\end{enumerate}
Otherwise, the homotopy groups $\pi_*(\bfLie_n\otimes_{h\Sigma_n} V^{\otimes n})$ vanish.
\end{prop}

\begin{proof}
For $p=2$, see~\cite[Theorem~1.5.1]{Behrens12} and~\cite[Section~7.1]{Antolin20}. For $p$ is odd, we refer to~\cite[Proposition~4.2 and Corollary~4.6]{Kjaer18}.
\end{proof}

\begin{dfn}\label{definition: totally isotropic Lie algebra}
A graded Lie algebra $L_*$ is \emph{$p$-isotropic} if
\begin{enumerate}
\item for $p=2$, $[x,x]=0$, $x\in L_*$,
\item for $p=3$, $[[x,x],x]=0$, $x\in L_{*}$.
\item for $p\geq 5$, any graded Lie algebra is $p$-isotropic.
\end{enumerate}
\end{dfn}
In other words, a Lie algebra is $p$-isotropic if and only if it has no non-vanishing self-brackets of length $\geq p$.

Let $W\in \Vect_{\kk}^{\mathrm{gr}}$ be a graded vector space. Then, by acting as in~\cite[Proposition~7.4]{Antolin20}, one can show that the free unstable $\calR$-Lie algebra $\mathcal{UL}(W)$ generated by $W$ has a basis $\Qe{J}(w)$, where $J$ is a CU sequence, $e(J)>|w|$, and $w$ is a basis element of the free \emph{$p$-isotropic} graded Lie algebra $\bfLie^{\text{$p$-}\mathrm{iso}}(W)$. Then, by~\cite[Theorem~5.2]{Kjaer18} and~\cite[Theorem~7.1]{Antolin20}, we obtain the following theorem.

\begin{thm}\label{theorem: homotopy of free H-Lie algebra}
Let $W\in \Vect_{\kk}^{\mathrm{gr}}$ be a graded vector space, and let $F_{\bfLie}(W)\in \mathrm{Ho}(\Mod_{\kk})$ be the free $H_\infty$-$\bfLie_*$-algebra spanned by $W$. Then the natural map
$$\mathcal{UL}(W) \xrightarrow{\cong} \pi_*(F_{\bfLie}(W)) = \bigoplus_{n\geq 1} \pi_*(\bfLie_n\otimes_{\Sigma_n}W^{\otimes n}) $$
is an isomorphism. Moreover, let $J$ be a CU sequence of length $h$ and let $w=[w_1,\ldots, w_r]  \in \bfLie^{\text{$p$-}\mathrm{iso}}(W)$ be a Lie word of length $r$. Then, under the isomorphism above, a basis element $\Qe{J}(w)\in \mathcal{UL}(W)$ belongs to the summand $$\pi_*(\bfLie_{rp^h}\otimes_{\Sigma_{rp^h}}W^{\otimes rp^h}) \subset \pi_*(F_{\bfLie}(W)).$$
\qed
\end{thm}

\begin{rmk}\label{remark: relations are exhaustive}
Theorem~\ref{theorem: homotopy of free H-Lie algebra} implies that the list of quadratic relations~\eqref{equation: Lie-Adem relation, p is odd, eq1}, \eqref{equation: Lie-Adem relation, p is odd, eq2}, and~\eqref{equation: Lie-Adem relation, p is 2, eq1} between the Dyer--Lashof--Lie operations is exhaustive. In other words, any other relation follows by the relations given in Definitions~\ref{definition: Dyer--Lashof--Lie algebra, p is odd} and~\ref{definition: Dyer--Lashof--Lie algebra, p is 2}.
\end{rmk}

\subsection{Whitehead conjecture}\label{section: whitehead conjecture}

Let $x\in \Lambda$ be an element of the algebra $\Lambda$. We write $$\overline{x}\colon \suspfree(\Sigma^q \kk) \to \suspfree(\Sigma^l \kk) \in \sL$$
for the map of suspension spectra which represents a class $\iota_l\otimes x\in \pi_q(\suspfree(\Sigma^l \kk))$. The next lemma and proposition are improvements of Lemma~\ref{lemma: differential, excess}.

\begin{lmm}\label{lemma: Dp lowers the filtration}
Let $(i_1,\e_1,i_2,\e_2)$, $\e_1,\e_2\in \{0,1\}$, $i_1,i_2\geq 0$ (resp. $(i,j)$, $i,j\geq 0$) be an admissible sequence. Then, if $p$ is odd,
$$\D_p(\overline{\nu^{\e_1}_{i_1}})_*(\bQe{\e_2}{i_2}(\iota_q)) = \sum_{\alpha} \bQe{\e(\alpha)}{i(\alpha)}(\iota_l)\otimes \nu_\alpha \in \pi_*(\bfLie_p\otimes_{\Sigma_p}(\Sigma^l \kk)^{\otimes p})\otimes \Lambda, $$
where $\e(\alpha)\in\{0,1\}$ and $0\leq i(\alpha)< i_1$. Similarly, if $p=2$,
$$\D_2(\overline{\lambda_{i}})_*(\Qe{j}(\iota_q)) = \sum_{\alpha} \Qe{i(\alpha)}(\iota_l) \otimes \lambda_{\alpha} \in \pi_*(\bfLie_2\otimes_{\Sigma_2}(\Sigma^l \kk)^{\otimes 2})\otimes \Lambda, $$
where $0\leq i(\alpha)<i$.
\end{lmm}

\begin{proof}
The lemma follows from Propositions~\ref{proposition: differential, adem, p=2}, \ref{proposition: differential, adem, p is odd}, and Lemma~\ref{lemma: GSS differential on hopf elements and leibniz} by applying the differential $\tilde{d}_2$ to the admissible monomial $\iota_l\otimes \nu_{i_1}^{\e_1}\nu^{\e_2}_{i_2} \in \pi_*(\suspfree(\Sigma^l \kk))$ (resp. $\iota_l\otimes \lambda_{i}\lambda_{j} \in \pi_*(\suspfree(\Sigma^l \kk))$).
\end{proof}

\begin{prop}\label{proposition: goodwillie differential lowers the filtation}
Let $I=(i_1,\e_1,\ldots, i_s,\e_s)$ (resp. $I=(i_1,\ldots,i_s)$) be an admissible sequence. Then, if $p$ is odd, 
$$\tilde{d}_2(\iota_l\otimes \nu_I) = \bQe{\e_1}{i_1}(\iota_l)\otimes \nu_{I'} + \sum_{\alpha}\bQe{\e(\alpha)}{i(\alpha)}(\iota_l)\otimes \nu_{\alpha}, $$
where $I'=(i_2,\e_2, \ldots, i_s,\e_s)$, $\e(\alpha)\in\{0,1\}$, $i(\alpha)<i_1$, and $\nu_{\alpha} \in \Lambda$. Similarly, if $p=2$,
$$\tilde{d}_2(\iota_l\otimes \lambda_I) = \Qe{i_1}(\iota_l)\otimes \lambda_{I'} + \sum_{\alpha}\Qe{i(\alpha)}(\iota_l)\otimes \lambda_{\alpha}, $$
where $I'=(i_2, \ldots, i_s)$, $i(\alpha)<i_1$, and $\lambda_{\alpha} \in \Lambda$.
\end{prop}

\begin{proof}
We prove the proposition only for an odd prime $p$; the argument for $p=2$ is almost identical and we leave it to the reader to complete the details. The proof goes by the induction on $s$. If $s=1$, then the statement follows from Proposition~\ref{proposition: differential, length 1}.

Assume that the proposition holds for the sequence $I'=(i_2,\e_2,\ldots, i_s,\e_s)$, i.e.
$$\tilde{d}_2(\iota_l\otimes \nu_{I'}) = \bQe{\e_2}{i_2}(\iota_l)\otimes \nu_{I''} + \sum_{\gamma}\bQe{\e(\gamma)}{i(\gamma)}(\iota_l)\otimes \nu_{\gamma}, $$
where $I''=(i_3,\e_3, \ldots, i_s,\e_s)$, $i(\gamma)<i_2$. Then, by Theorem~\ref{theorem:leibniz rule in GSS}, we have
\begin{align*}
\tilde{d}_2(\iota_l\otimes \nu_I) &= \tilde{d}_2(\iota_l \otimes \nu^{\e_1}_{i_1}\nu_{I'})\\
&= \bQe{\e_1}{i_1}(\iota_l)\otimes \nu_{I'} + \sum_{\gamma}\D_p(\overline{\nu^{\e_1}_{i_1}})_*(\bQe{\e(\gamma)}{i(\gamma)}(\iota_q))\otimes \nu_{\gamma}\\
&=\bQe{\e_1}{i_1}(\iota_l)\otimes \nu_{I'} + \sum_{\gamma}\D_p(\overline{\nu^{\e_1}_{i_1}})_*(\bQe{\e(\gamma)}{i(\gamma)}(\iota_q))\otimes \nu_{\gamma}.
\end{align*}
By the inductive assumption, the sequences $(i_1,\e_1,i(\gamma),\e(\gamma))$ are admissible, and so, Lemma~\ref{lemma: Dp lowers the filtration} implies the proposition.
\end{proof}

\begin{cor}[$p$ is odd]\label{cor: goodwillie differential lowers the filtation, p is odd}
Let $I=(i_1,\e_1,\ldots, i_s,\e_s)$ be an admissible sequence, and let $J=(j_1,\delta_1,\ldots, j_h, \delta_h)$ be a completely unadmissible sequence. Then 
$$
\tilde{d}_2(\Qe{J}(\iota_l)\otimes \nu_I) = 
\begin{cases}
\Qe{(i_1,\e_1,J)}(\iota_l)\otimes \nu_{I'} + \sum_{\alpha}\Qe{J(\alpha)}(\iota_l)\otimes \nu_{\alpha}, & \mbox{$(i_1,\e_1,J)$ is CU,}\\
\sum_{\alpha}\Qe{J(\alpha)}(\iota_l)\otimes \nu_{\alpha},& \mbox{otherwise,}\\
\end{cases}
$$
where $I'=(i_2,\e_2, \ldots, i_s,\e_s)$, $J(\alpha)=(j_1(\alpha), \delta_1(\alpha),\ldots,j_{h+1}(\alpha),\delta_{h+1}(\alpha))$ are CU sequences, $j_1(\alpha)<j_1$, and $\nu_{\alpha} \in \Lambda$. 
\end{cor}

\begin{proof}
By Lemma~\ref{lemma:GSS differential, positive columns} and Proposition~\ref{proposition: goodwillie differential lowers the filtation}, we obtain
$$\tilde{d}_2(\Qe{J}(\iota_l)\otimes \nu_I) =  \bQe{\e_1}{i_1}\cdot \Qe{J}(\iota_l)\otimes \nu_{I'} + \sum_{\alpha}\bQe{\e(\alpha)}{i(\alpha)}\cdot \Qe{J}(\iota_l)\otimes \nu_{\alpha},$$
where $i(\alpha)<i_1.$ 

Next, let $J'=(j'_1,\delta'_1,\ldots,j'_{h+1},\delta'_{h+1})$ be any sequence. Then, by applying the relations~\eqref{definition: Dyer--Lashof--Lie algebra, p is odd} and~\eqref{equation: Lie-Adem relation, p is odd, eq2} inductively, we get
$$\Qe{J'}=\sum_{\gamma}\Qe{J'(\gamma)} \in \calR_{p},$$ where the sequences $J'(\gamma)=(j'_1(\gamma),\delta'_1(\gamma),\ldots,j'_{h+1}(\gamma),\delta'_{h+1}(\gamma))$ are completely unadmissible and $j'_1(\gamma)<j'_1$. This implies the corollary.
\end{proof}

The ``$p=2$''-analog of Corollary~\ref{cor: goodwillie differential lowers the filtation, p is odd} reads as follows.
\begin{cor}[$p=2$]\label{cor: goodwillie differential lowers the filtation, p is 2}
Let $I=(i_1,\ldots, i_s)$ be an admissible sequence, and let $J=(j_1,\ldots, j_h)$ be a completely unadmissible sequence. Then
$$
\tilde{d}_2(\Qe{J}(\iota_l)\otimes \lambda_I) =
\begin{cases}
\Qe{(i_1,J)}(\iota_l)\otimes \lambda_{I'} + \sum_{\alpha}\Qe{J(\alpha)}(\iota_l)\otimes \lambda_{\alpha}, & \mbox{$(i_1,J)$ is CU,}\\
\sum_{\alpha}\Qe{J(\alpha)}(\iota_l)\otimes \lambda_{\alpha},& \mbox{otherwise,}\\
\end{cases}
$$
where $I'=(i_2, \ldots, i_s)$, $J(\alpha)=(j_1(\alpha),\ldots,j_{h+1}(\alpha))$ are CU sequences, $j_1(\alpha)<j_1$, and $\lambda_{\alpha} \in \Lambda$. \qed
\end{cor}

\begin{thm}[Algebraic Whitehead conjecture]\label{theorem: whitehead conjecture}
Suppose that $V\simeq \Sigma^l \kk$, $l\geq 0$. Then the renormalized Goodwillie spectral sequence~\eqref{equation: RAGSS}
\begin{equation*}
\widetilde{E}^2_{t,m}(V)\Rightarrow \pi_t(\free(V))
\end{equation*}
degenerates at the third page. More precisely, $\widetilde{E}^3_{t,m}(V)=0$ for $m\geq 2$ and $t\geq 0$.
\end{thm}

\begin{proof}
We prove the proposition only for an odd prime $p$ and arbitrary $l$; the argument for $p=2$ is almost identical and we leave them to the reader to complete the details.

By Proposition~\ref{proposition: homotopy lie, basis, 1-dim}, the complex $ C_\bullet(V)=(\widetilde{E}^2_{*,*}(V),\tilde{d}_2)$, $C_h=\widetilde{E}^2_{*,2h}$ has a basis $\Qe{J}(\iota_l)\otimes \nu_I,$
where $I=(i_1,\e_1,\ldots, i_s,\e_s)$ is an admissible sequence, $J=(j_1,\delta_1,\ldots, a_h,\delta_h)$ is a completely unadmissible sequence, and $e(J)>l$.

We denote by $C_\bullet(V)_n, n\geq 0$ the vector subspace of $C_\bullet(V)$ spanned by the elements $\Qe{J}(\iota_l)\otimes \nu_I$ such that $l(I)+l(J)=n$. By Corollary~\ref{corollary: agss is trigraded}, the differential $\tilde{d}_2$ preserves each subspace $C_{\bullet}(V)_n, n\geq 0$.

Given a sequence $K=(k_1,\kappa_1,\ldots,k_n, \kappa_n)$, $\kappa_1,\ldots,\kappa_n\in\{0,1\}$, we write $\overline{K}$ for the subsequence $\overline{K}=(k_1,\ldots, k_n)$ and we write $\overline{K}^{op}$ for the reversed sequence $\overline{K}^{op}=(k_n,\ldots, k_1)$.

Let us consider the lexicographical order on the set $\mathcal{I}_n$ of sequences of length $n$: $\alpha=(\alpha_1,\ldots,\alpha_n) \preceq \beta=(\beta_1,\ldots, \beta_n)\in \mathcal{I}_n$ if and only if there exists $r$ such that $$\alpha_1=\beta_1,\ldots, \alpha_{r-1}=\beta_{r-1}\;\; \text{and} \;\; \alpha_r<\beta_r.$$
Define an $\mathcal{I}_n$-valued increasing filtration $F_{\alpha}C_\bullet(V)_n$, $\alpha \in\mathcal{I}_n$ on the vector space $C_\bullet(V)_n$ by the rule:
$$F_{\alpha}C_\bullet(V)_n = \mathrm{span}(\Qe{J}(\iota_l)\otimes \nu_I \; |\; (\overline{J}^{op},\overline{I})\preceq \alpha), $$
where $(\overline{J}^{op},\overline{I})$ is the concatenation of the sequences $\overline{J}^{op}$ and $\overline{I}$. By Corollary~\ref{cor: goodwillie differential lowers the filtation, p is odd}, the differential $\tilde{d}_2$ preserves subspaces $F_{\alpha}C_\bullet(V)_n$, $\alpha\in \mathcal{I}_n$.

Again, by Corollary~\ref{cor: goodwillie differential lowers the filtation, p is odd}, the associated graded complex $gr_{\alpha}C_\bullet(V)_n=F_{\alpha}/F_{\alpha'}$ (where $\alpha'$ precedes $\alpha$ in $\mathcal{I}_n$) has the induced differential given by the rule 
$$
d(\Qe{J}(\iota_l)\otimes \nu_I)=
\begin{cases}
\Qe{(i_1,\e_1,J)}(\iota_l)\otimes \nu_{I'} & \mbox{if $I=(i_1,\e_1,I')$ and $(i_1,\e_1,J)$ is CU,}\\
0 & \mbox{otherwise.}
\end{cases}
$$
Hence, by a straightforward computation, the cohomology groups $$H^h(gr_{\alpha}C_\bullet(V)_n, d) =0$$ vanish as soon as $\alpha\in \mathcal{I}_n$, $n\geq 0$, and $h\geq 1$. Therefore the cohomology groups $\widetilde{E}^3_{*,2h}(V)=H^h(C_\bullet(V),\tilde{d}_2)$ also vanish for $h\geq 1$.

If $p$ and $l$ are odd, then the second page $\widetilde{E}^2_{*,*}(V)$ is the direct sum of two complexes $(\widetilde{E}^2_{*,\mathrm{even}}(V), \tilde{d}_2)$ and $(\widetilde{E}^2_{*,\mathrm{odd}}(V), \tilde{d}_2)$, see Lemma~\ref{lemma:GSS differential, positive columns}. By Proposition~\ref{proposition: homotopy lie, basis, 1-dim}, the complex $(\widetilde{E}^2_{*,\mathrm{even}}(V), \tilde{d}_2)$ (resp. $(\widetilde{E}^2_{*,\mathrm{odd}}(V), \tilde{d}_2)$) has a basis $$\Qe{J}(\iota_l)\otimes \nu_I \;\; \text{(resp. $\Qe{J}([\iota_l,\iota_l]) \otimes \nu_I)$,}$$ where $I=(i_1,\e_1,\ldots, i_s,\e_s)$ is an admissible sequence, $J=(j_1,\delta_1,\ldots, a_h,\delta_h)$ is a completely unadmissible sequence, and $e(J)>l$ (resp. $e(J)>2l)$. The same argument with the multiindex filtration as for even $l$ shows that the complexes $(\widetilde{E}^2_{*,\mathrm{even}}(V), \tilde{d}_2)$ and $(\widetilde{E}^2_{*,\mathrm{odd}}(V), \tilde{d}_2)$ are contractible.
\end{proof}

\begin{cor}\label{corollary: Einfty of RAGSS}
Suppose that $V\simeq \Sigma^l \kk$, $l\geq 0$. Then the $\widetilde{E}^{\infty}$-term of the renormalized Goodwillie spectral sequence~\eqref{equation: RAGSS} has a basis 
\begin{enumerate}
\item $\iota_l\otimes \lambda_I \in \widetilde{E}^{\infty}_{*,0}(V)$, $I=(i_1,\ldots,i_s)$ is an admissible sequence, and $i_1\leq l$ if  $p=2$;
\item $\iota_l \otimes \nu_I\in \widetilde{E}^{\infty}_{*,0}(V)$, $I=(i_1,\e_1,\ldots,i_h,\e_h)$ is an admissible sequence, $2i_1\leq l$ if $p$ is odd;
\item $[\iota_l,\iota_l] \otimes \nu_I\in \widetilde{E}^{\infty}_{*,1}(V)$, $I=(i_1,\e_1,\ldots,i_h,\e_h)$ is an admissible sequence, $2i_1\leq 2l$ if $p,l$ are odd.
\end{enumerate}
\end{cor}

\begin{proof}
By Theorem~\ref{theorem: whitehead conjecture}, we have
$$\widetilde{E}^{\infty}_{*,m}(V) = \widetilde{E}^3_{*,m}(V)  = \ker(\tilde{d}_2\colon E^2_{*,m}(V) \to E^2_{*-1,m+2}(V)), \;\; m=0,1,$$
and $\widetilde{E}^{\infty}_{*,m}=0$ if $m\geq 2$. By Corollaries~\ref{cor: goodwillie differential lowers the filtation, p is odd},~\ref{cor: goodwillie differential lowers the filtation, p is 2}, and Lemma~\ref{lemma:GSS differential, positive columns}, we observe that the kernel $\ker(\tilde{d}_2)$ is spanned by the elements described above.
\end{proof}

\begin{rmk}\label{remark: whitehead conjecture, wedges}
Let $W\in \Mod^{\geq 0}_{\kk}$ be a simplicial vector space such that $\dim(\pi_* W)>1$. Then the renormalized Goodwillie spectral sequence~\eqref{equation: RAGSS} does not make sense, and so does Theorem~\ref{theorem: whitehead conjecture}. However, one can still observe certain degeneration phenomena for the (non-renormalized) Goodwillie spectral sequence~\eqref{equation: AGSS}
$$E^1_{*,n}(W)=\pi_{*}(\bfLie_n\otimes_{h\Sigma_n}W^{\otimes n}) \otimes \Lambda \Rightarrow \pi_*(\free(W)). $$
Let $E^1_*(W)=\bigoplus_{n\geq 1} E^1_{*,n}(W)$. By Corollary~\ref{corollary: exponential vanishing}, there exists a well-defined endomorphism $$d\colon E^1_*(W) \to E^1_*(W)$$ such that
$$d|_{E^1_{*,n}(W)} =d_{(p-1)n} \colon E^1_{*,n}(W) \to E^1_{*,pn}(W) \subset E^1_{*}(W). $$
Note that $d^2=0$. 

By Theorem~\ref{theorem: homotopy of free H-Lie algebra}, the $E^1_*(W)$-term has a basis $\Qe{J}(w)\otimes \nu_I \in E^1_{*}(W)$, where $J$ is CU sequence, $w\in \bfLie^{\text{$p$-}\mathrm{iso}}(\pi_*W)$ is a Lie word, $e(J)>|w|$, and $I$ is an admissible sequence, $\nu_I\in \Lambda$. Then, we set $E^1_*(W)_h$ to be the vector subspace of $E^1_*(W)$ spanned by all elements $\Qe{J}(w)\otimes \nu_I \in E^1_{*}(W)$ such that $l(J)=h$. Then 
$$E^1_*(W) = \bigoplus_{h\geq 0} E^1_*(W)_h \;\; \text{and} \;\; d(E^1_*(W)_h) \subset E^1_*(W)_{h+1},$$
see Lemma~\ref{lemma:GSS differential, positive columns}. By the same argument with the multiindex filtration as in Theorem~\ref{theorem: whitehead conjecture}, we obtain that the complex $(E^1_*(W)_\bullet,d)$ is contractible.

Furthermore, as in Corollary~\ref{corollary: Einfty of RAGSS}, we obtain that $H^0(E^1_*(W)_\bullet,d)$ has a basis
\begin{enumerate}
\item $w\otimes \lambda_I \in H^0(E^1_*(W)_\bullet,d)\cap E^{1}_{*,r2^s}(W)$, $w\in \bfLie^{\mathrm{iso}}(\pi_*W)$ is a Lie word of length $r$ in the free isotropic Lie algebra $\bfLie^{\mathrm{iso}}(\pi_*W)$, $I=(i_1,\ldots,i_s)$ is an admissible sequence, and $i_1\leq |w|$ if $p=2$;
\item  $w\otimes \nu_I \in H^0(E^1_*(W)_\bullet,d)\cap E^{1}_{*,rp^s}(W)$, $w\in \bfLie^{\text{$p$-}\mathrm{iso}}(\pi_*W)$ is a Lie word in the free $p$-isotropic Lie algebra $\bfLie^{\text{$p$-}\mathrm{iso}}(\pi_*W)$, $l(w)=r$, $I=(i_1,\e_1,\ldots,i_h,\e_h)$ is an admissible sequence, and $2i_1\leq |w|$ if $p$ is odd.
\end{enumerate}

By e.g. the (algebraic) Hilton--Milnor theorem~\cite[Example~8.7.4]{Neisendorfer10}, one can see that the listed classes are permanent. So, $E^\infty$-term $E^\infty_{*,*}(W)$ is isomorphic to $H^0(E^1_*(W)_\bullet,d)$. In particular, if $x\in E^1_{*,n}(W)$ is any class such that 
$d_{(p-1)n}(x)=0$, then either $x\in E^1_{*,n}(W)$ persists to a non-trivial class on the $E_\infty$-page, or $x=d_{r(p-1)p^{h-1}}(y)$ for some $y\in E^1_{*,rp^{h-1}}(W)$.

Finally, we note that the last computation is consistent with the result of the computations for the homotopy groups $\pi_*(\free (W))$ of the free simplicial restricted Lie algebra $\free (W)$ given by~\cite[Theorem~8.5]{BC70}
and~\cite[Proposition~13.2]{Wellington82}.
\end{rmk}

\begin{rmk}[Lin--Nishida conjecture]\label{remark: lin-nishida conjecture}
Assume for simplicity that $V= \Sigma^l\kk$, $l\geq 0$ and $l$ is even if $p$ is odd. By Corollary~\ref{corollary: agss is trigraded}, the first page $\widetilde{E}^1_{*,*}(V)$ of the renormalized spectral sequence~\eqref{equation: RAGSS} is trigraded:
$$\widetilde{E}^1_{t,2h}(V) = E^1_{t,p^h}(V)=\bigoplus_{s\geq 0} E^1_{t,p^h,s}(V)\cong \bigoplus_{s\geq 0}\big(\pi_*(\mathbf{L}_{p^h}(V))\otimes \Lambda_s\big)_t, \;\; t,h\geq 0. $$
In order to keep track of degrees, we will change the notation and we will write $$d_{h}(s)\colon E^1_{t,p^h,s}(V)\to E^1_{t-1,p^{h+1},s-1}(V)$$ for the first non-trivial differential $\tilde{d}_2=d_{(p-1)p^h}$, see Corollary~\ref{corollary: agss is trigraded}.

By~\cite{AM99}, the homotopy groups $\pi_*(\mathbf{L}_{p^h}(V))$, $h\geq 0$ are isomorphic to the shifted homology groups of certain spectra $\D_{p^h}(S^{l+1})\in \Sp$:
\begin{equation}\label{equation: lin-nishida, iso}
\pi_*(\mathbf{L}_{p^h}(\Sigma^l \kk)) \cong \Sigma^{-1}H_*(\D_{p^h}(S^{l+1});\kk)), \;\; h\geq 0. 
\end{equation}
By~\cite{6authors}, there are Adams spectral sequences
$${}^{\mathrm{ASS}}E^1_{t,s}(h,l) = \big(\Sigma^{-1}H_*(\D_{p^h}(S^{l+1});\F_p)\otimes \Lambda_{s}\big)_t \Rightarrow \pi_{t-1}(\D_{p^h}(S^{l+1}))$$
for all $h,l\geq 0$. We write $$\delta_{s}(h)\colon E^1_{t,s}(h,l) \to E^1_{t-1,s+1}(h,l), \;\; h,l\geq 0$$ for the differentials acting on the first page of the Adams spectral sequence. We also recall from~\cite{6authors} that the lambda algebra $\Lambda$ is equipped with the natural differential $\delta\colon \Lambda\to \Lambda$ and the differentials $\delta_s(h)$, $h,s\geq 0$ can be expressed explicitly in terms of $\delta$ and the right action of the Steenrod algebra on the homology groups $H_*(\D_{p^h}(S^{l+1});\kk)$. Moreover, the differentials $\delta_s(h)$, $h,s\geq 0$ completely encode the right module structure on $H_*(\D_{p^h}(S^{l+1});\kk)$ over the Steenrod algebra.

We conjecture that, under the isomorphisms~\eqref{equation: lin-nishida, iso}, the following diagrams
\begin{equation}\label{equation: lin-nishida, diagram}
\begin{tikzcd}[column sep=large]
E^1_{t,p^h,s}(\Sigma^l\kk) \arrow{r}{d_{h}(s)} \arrow{d}[swap]{\delta_{s}(h)}
&E^1_{t-1,p^{h+1},s-1}(\Sigma^l \kk)\arrow{d}{\delta_{s-1}(h+1)} \\
E^1_{t-1,p^h,s+1}(\Sigma^l\kk) \arrow{r}{d_{h}(s+1)}
&E^1_{t-2,p^{h+1},s}(\Sigma^l\kk)
\end{tikzcd}
\end{equation}
anticommute for all $t,h,l\geq 0$ and $s\geq 1$. By the Nishida relations~\cite[Theorem~I.1.1(9)]{CLM76}, the diagram~\eqref{equation: lin-nishida, iso} anticommutes if $s=1$, $h=0$. For $h\geq 1$, the action of the Steenrod operations on the CU monomials is more complicated (see e.g.~\cite[Remark~1.4.5]{Behrens12}), and it is much harder to verify the conjecture in this case. 

If $p=2$, $h=0$, and $s$ is arbitrary, then the diagram~\eqref{equation: lin-nishida, diagram} (anti)commutes by~\cite[Proposition~2.1(1)]{Lin81}. We expect that the diagram~\eqref{equation: lin-nishida, diagram} can be categorified and be interpreted by the functor calculus.
\end{rmk}

\bibliographystyle{alpha}
\bibliography{references}

\end{document}